\documentclass{memo-l}

\usepackage{amsmath}
\usepackage{amssymb}
\usepackage{amsthm} 
\usepackage{epsf}
\usepackage{graphicx}
\usepackage{esint} 
\usepackage{dsfont}
\usepackage{hyperref}
\hypersetup{backref,  pdfpagemode=FullScreen,  colorlinks=true} 
\usepackage{color}
\usepackage{enumerate}

\newtheorem{theorem}{Theorem}[chapter]
\newtheorem{lemma}[theorem]{Lemma}

\theoremstyle{definition}
\newtheorem{definition}[theorem]{Definition}

\theoremstyle{remark}
\newtheorem{remark}[theorem]{Remark}

\numberwithin{section}{chapter}
\numberwithin{equation}{chapter}

\newcommand{\nn}{\nonumber}
\newcommand{\ms}{\medskip}

\newcommand{\R}{\mathbb{R}}
\renewcommand{\H}{\mathcal H}

\newcommand{\bN}{\mathbb{N}}

\renewcommand{\O}{{\mathcal  O}}
\renewcommand{\d}{\partial}
\newcommand{\dist}{\,\mathrm{dist}}

\newcommand{\sm}{\setminus}
\newcommand{\supp}{\mathrm{supp}}
\newcommand{\diam}{\mathrm{diam}}

\newcommand{\wt}{\widetilde}
\newcommand{\wh}{\widehat}

\newcommand{\ub}{\underbar}

\newcommand{\cW}{{\mathcal  W}}

\newcommand{\1}{{\mathds 1}}

\newcommand{\dr}{\partial}
\DeclareMathOperator*{\osc}{osc}
\newcommand{\ds}{\displaystyle}
\newcommand{\WW}{W_r}

\newcommand{\g}{\mathfrak g}

\newcommand{\rn}{{\mathbb{R}^n}}

\newcommand{\RR}{{\mathbb{R}}}
\newcommand{\NN}{{\mathbb{N}}}
\newcommand{\ZZ}{{\mathbb{Z}}}

\newcommand{\eps}{\varepsilon}

\newcommand{\A}{\mathcal{A}}

\newcommand{\po}{{\partial\Omega}}

\newcommand{\bp}{\noindent {\it Proof}.\,\,}
\newcommand{\ep}{\hfill$\Box$ \vskip 0.08in}

\DeclareMathOperator{\diver}{div}

\makeindex

\begin{document}

\frontmatter

\title{Elliptic Theory for Sets with Higher Co-dimensional Boundaries}

\author{Guy David}
\address{Univ Paris-Sud, Laboratoire de Math\'ematiques, UMR8628, Orsay, 91405, France}
\curraddr{}
\email{guy.david@math.u-psud.fr}

\author{Joseph Feneuil}
\address{Department of Mathematics, Temple University, 1805 N. Broad Street, Philadelphia, PA 19122, USA}
\curraddr{}
\email{joseph.feneuil@temple.edu}

\author{Svitlana Mayboroda}
\address{School of Mathematics, University of Minnesota, Minneapolis, MN 55455, USA}
\curraddr{}
\email{svitlana@math.umn.edu}
\thanks{This research was supported in part by Fondation Jacques Hadamard and by CNRS. The first author was supported in part by the ANR,
programme blanc GEOMETRYA ANR-12-BS01-0014, the European Community Marie Curie grant MANET 607643 and H2020 grant GHAIA 777822, and the Simons Collaborations in MSP grant 601941, GD. The third author was  supported  by the Alfred P. Sloan Fellowship, the NSF INSPIRE Award DMS 1344235, NSF CAREER Award DMS 1220089, the NSF RAISE-TAQ grant DMS 1839077, and the Simons Foundation grant 563916, SM.}
\thanks{We would like to thank the Department of Mathematics at Universit\'e Paris-Sud, the Ecole des Mines, and the Mathematical Sciences Research Institute (NSF grant DMS 1440140) for warm hospitality.}
\thanks{Finally, we would like to thank Alano Ancona for stimulating discussions at the early stages of the project and for sharing with us the results of his work.}

\date{February 25, 2017}



\begin{abstract}
Many geometric and analytic properties of sets hinge on the properties of elliptic measure, notoriously missing for sets of higher co-dimension. The aim of this manuscript is to develop a version of elliptic theory, associated to a linear PDE, which ultimately yields a notion analogous to that of the harmonic measure, for sets of codimension higher than 1.

To this end, we turn to degenerate elliptic equations. Let $\Gamma \subset \R^n$ be an Ahlfors regular set of dimension $d<n-1$ (not necessarily integer) and $\Omega = \R^n \setminus \Gamma$. Let $L = - \diver A\nabla$ be a degenerate elliptic operator with measurable coefficients such that the ellipticity constants of the matrix $A$ are bounded from above and below by a multiple of $\dist(\cdot, \Gamma)^{d+1-n}$. We define weak solutions; prove trace and extension theorems in suitable weighted Sobolev spaces; establish the maximum principle, De Giorgi-Nash-Moser estimates, the Harnack inequality, the H\"older continuity of solutions (inside and at the boundary). We define the Green function and provide the basic set of pointwise and/or $L^p$ estimates for the Green function and for its gradient. With this at hand, we define harmonic measure associated to $L$, establish its doubling property, non-degeneracy, change-of-the-pole formulas, and, finally, the comparison principle for local solutions. 

In another article to appear, we will prove that when $\Gamma$ is the graph of a Lipschitz function with small Lipschitz constant, we can find an elliptic operator $L$ for which the harmonic measure given here is absolutely continuous with respect to the $d$-Hausdorff measure on $\Gamma$ and vice versa. It thus extends Dahlberg's theorem to some sets of codimension higher than~1.
\end{abstract}

\maketitle

\tableofcontents


\mainmatter
\chapter{Introduction}
\label{S1}

The past few years have witnessed remarkable progress
in the study of relations between regularity properties 
of the harmonic measure $\omega$ on the boundary of a domain of $\R^n$ (for instance, 
its absolute continuity with respect to the Hausdorff measure $\H^{n-1}$) 
and the regularity of the domain (for instance, rectifiability properties  of the boundary). 
In short, the emerging philosophy is that  the  
rectifiability of the boundary is necessary for the absolute continuity of $\omega$ with respect to 
$\H^{n-1}$, and that rectifiability along with suitable connectedness assumptions is sufficient. 
Omitting for now precise definitions, let us recall the main results in this regard. 
The celebrated 1916 theorem of F. \& M. Riesz has established the 
absolute continuity of the harmonic measure for a simply connected domain in the complex plane, 
with a rectifiable boundary \cite{RR}. The quantifiable analogue of this result 
(the $A^\infty$ property of harmonic measure) was obtained by Lavrent'ev in 1936 \cite{Lv} 
and the local version, pertaining to subsets of a rectifiable curve which is a boundary of a simply connected planar domain, 
was proved by 
Bishop and Jones in 1990 \cite{BJ}. In the latter work the authors also showed that 
some connectedness is necessary for the absolute continuity of $\omega$ with respect to $\H^{n-1}$, 
for there exists a planar set with  rectifiable boundary for which the harmonic measure is singular with respect to $\H^{n-1}$. 

The situation in higher dimensions, $n\geq 3$, is even more complicated. 
The absolute continuity of $\omega$ with respect to $\H^{n-1}$ was first established by Dahlberg 
on Lipschitz graphs \cite{Dah} and was then extended to non-tangentially accessible (NTA) domains with Ahlfors regular boundary  in \cite{DJ}, \cite{Se}, and to more general NTA domains in \cite{Ba}. Roughly speaking, the non-tangential accessibility is an 
assumption of quantifiable connectedness, 
which requires the presence of  
interior and exterior corkscrew points, as well as Harnack chains. Ahlfors regularity simply postulates that the measure of intersection with the boundary of every ball of radius $r$ centered at the boundary is proportional to $r^{n-1}$, i.e., that the boundary is in a certain sense $n-1$ dimensional (we will provide a careful definition below).
Similarly to the lower-dimensional case, counterexamples 
show that some topological restrictions are needed for the
absolute continuity of $\omega$ with respect to $\H^{n-1}$ \cite{Wu}, \cite{Z}. 
Much more recently, in \cite{HM1}, \cite{HMU}, \cite{AHMNT}, the authors proved that, in fact, for sets with Ahlfors regular boundaries, under a (weaker) 1-sided NTA assumption, the
uniform rectifiability of the boundary is equivalent to the complete set of NTA  conditions 
and hence, is equivalent to the
absolute continuity of harmonic measure with respect to the Lebesgue measure. 
Finally, in 2015 the full converse, ``free boundary"  
result was obtained and established that rectifiability is necessary for  the
absolute continuity of harmonic measure with respect to $\H^{n-1}$ in any dimension $n\geq 2$ (without any additional topological assumptions) \cite{AHM3TV}. It was proved simultaneously that for a complement of an $(n-1)$-Ahlfors regular set the $A^\infty$ property of harmonic measure yields uniform rectifiability of the boundary \cite{HLMN}. Shortly after, it was established that in an analogous setting $\eps$-approximability and Carleson measure estimates for bounded harmonic functions are equivalent to uniform rectifiability \cite{HMM1}, \cite{GMT}, and that analogous results hold for more general elliptic operators \cite{HMM2}, \cite{AGMT}.

The purpose of this work is to start the investigation
of similar properties for domains with a lower-dimensional boundary $\Gamma$.  To the best of our knowledge, the only known approach to elliptic problems on domains with higher co-dimensional boundaries is by means of the $p$-Laplacian operator and its generalizations \cite{NDbook}, \cite{LN}. In \cite{LN} the authors worked with an associated Wiener capacity, defined $p$-harmonic measure, and established boundary Harnack inequalities for Reifenberg flat sets of co-dimension higher than one. Our goals here are different.

\medskip

We shall systematically assume that $\Gamma$ is Ahlfors-regular of some dimension
$d < n-1$, which {\it does not need to be an integer}. 
This means that there is a constant $C_0 \geq 1$ such that
\begin{equation} \label{1.1}
C_0^{-1} r^d \leq \H^d(\Gamma \cap B(x,r)) \leq C_0 r^d
\ \text{ for } x\in \Gamma \text{ and } r > 0.
\end{equation}
We want to define an analogue of the
harmonic measure, that will be defined on $\Gamma$ and
associated to a divergence form operator on $\Omega = \R^n \sm \Gamma$. We still write
the operator as $L = - \rm{div} A \nabla$, with $A : \Omega \to \mathbb{M}_n(\R)$, 
and we write the ellipticity condition with a different homogeneity, i.e., we require that for some 
$C_1 \geq 1$, 
\begin{eqnarray} \label{1.2.1}
&& \dist(x,\Gamma)^{n-d-1} A(x)\xi \cdot \zeta \leq C_1  |\xi| \,|\zeta|
\ \text{ for } x\in \Omega \text{ and } \xi, \zeta \in \R^n, \\[4pt]
\label{1.2.2}
&& \dist(x,\Gamma)^{n-d-1} A(x)\xi \cdot \xi \geq C_1^{-1}  |\xi|^2
\ \text{ for } x\in \Omega \text{ and } \xi \in \R^n.
\end{eqnarray}
The effect of this normalization should be to incite the analogue of the Brownian motion here
to get closer to the boundary with the right probability;  for instance if $\Gamma = \R^d \subset \R^n$
and $A(x) = \dist(x,\Gamma)^{-n+d+1} I$, it turns out that the effect of $L$ on functions $f(x,t)$ 
that are radial in the second variable $t\in \R^{n-d}$ is the same as for the Laplacian on $\R^{d+1}_+$. 
In some sense, we create Brownian travelers  which treat $\Gamma$ as a ``black hole": they detect more mass and they are more attracted to $\Gamma$ than a standard Brownian traveler governed by the Laplacian would be. 

The purpose of the present manuscript is to develop, with merely these assumptions, a comprehensive elliptic theory. We solve the Dirichlet problem for $Lu = 0$, 
prove the maximum principle, the De Giorgi-Nash-Moser estimates and the 
Harnack inequality for solutions, use this to define a harmonic measure associated to $L$, show that it is doubling, 
and prove the comparison principle for positive $L$-harmonic functions that vanish at the boundary. Let us discuss the details.

We first introduce some notation. Set $\delta(x) = \dist(x,\Gamma)$
and $w(x) = \delta(x)^{-n+d+1}$ for $x\in \Omega = \R^n \sm \Gamma$,
and denote by $\sigma$ the restriction to $\Gamma$ of $\H^d$.  Denote by $W=\dot W^{1,2}_w(\Omega)$ the weighted Sobolev space of functions 
$u \in L^1_{loc}(\Omega)$
whose distribution gradient in $\Omega$ lies in $L^2(\Omega,w)$:
\begin{equation}\label{defW}
W=\dot W^{1,2}_w(\Omega):=\{u \in L^1_{loc}(\Omega): \,\nabla u\in L^2(\Omega,w)\},
\end{equation}
and set
$\|u\|_W = \big\{\int_\Omega |\nabla u(x)|^2 w(x) dx \big\}^{1/2}$ for $f\in W$.
Finally denote 
by ${\mathcal M}(\Gamma)$ the set of measurable functions on $\Gamma$ and then set
\begin{equation}\label{defH}
H=\dot H^{1/2}(\Gamma) :=\left\{g\in {\mathcal M}(\Gamma): \,\int_\Gamma\int_\Gamma {|g(x)-g(y)|^2 \over |x-y|^{d+1}}d\sigma(x) d\sigma(y)<\infty\right\}.
\end{equation}

Before we solve Dirichlet problems we construct two bounded linear operators
$T : W \to H$ (a trace operator) and $E : H \to W$ (an extension operator), such that
$T \circ E = I_H$. The trace of $u \in W$ is such that for $\sigma$-almost every $x\in \Gamma$,
\begin{equation}\label{defT1.6}
Tu(x) = \lim_{r \to 0} \fint_{B(x,r)} u(y) dy :=  \lim_{r \to 0} {1 \over |B(x,r)|} \int u(y) dy,
\end{equation}
and even, analogously to the Lebesgue density property,
\begin{equation}\label{defLebdp1}
\lim_{r \to 0} \fint_{B(x,r)} |u(y)-Tu(x)| dy = 0.
\end{equation}
Similarly, we check that if $g\in H$, then 
\begin{equation}\label{defLebdp2}
\lim_{r \to 0} \fint_{\Gamma \cap B(x,r)} |g(y)-g(x)| d\sigma(y) = 0
\end{equation}
for $\sigma$-almost every $x\in \Gamma$.
We typically use the fact that $|u(x)-u(y)| \leq \int_{[x,y]} |\nabla u| d\mathcal L^1$ for almost all choices
of $x$ and $y \in \Omega$, for which we can use the absolute continuity of $u\in W$ on 
(almost all) line segments, plus the important fact that, by \eqref{1.1}, 
$\Gamma \cap \ell = \emptyset$ for almost every line $\ell$. 

Note that the latter geometric fact is enabled specifically by the higher co-dimension ($d<n-1$), even though our boundary can be quite wild. In fact, a stronger property holds in the present setting and gives, in particular, Harnack chains. 
There exists a constant $C>0$, that depends only
on $C_0$, $n$, and $d < n-1$, such that for $\Lambda \geq 1$ and $x_1, x_2 \in \Omega$ 
such that $\dist(x_i,\Gamma) \geq r$ and $|x_1-x_2| \leq \Lambda r$, we can find two points
$y_i \in B(x_i,r/2)$ such that $\dist([y_1,y_2],\Gamma) \geq C^{-1} \Lambda^{-d/(n-d-1)} r$.
That is, there is a thick tube in $\Omega$ that connects the two $B(x_i,r/2)$.

Once we have trace and extension operators, we deduce from the Lax-Milgram theorem that for $g\in H$, 
there is a unique weak solution $u\in W$ of $Lu=0$ such that $Tu = g$. 
For us a weak solution is a function $u\in W$ such that
\begin{equation}\label{defwsol}
\int_\Omega A(x)\nabla u(x) \cdot \nabla \varphi(x) dx = 0
\end{equation}
for all $\varphi \in C_0^\infty (\Omega)$, the space of infinitely differentiable functions which are compactly 
supported in $\Omega$. 

Then we follow the Moser iteration scheme to study the weak solutions of $Lu=0$, as we would do
in the standard elliptic case in codimension $1$.  
This leads to the quantitative boundedness (a.k.a. Moser bounds) and the quantitative H\"older continuity
(a.k.a. De Giorgi-Nash estimates), in an interior or boundary ball $B$,
of any weak solution of $Lu = 0$ in $2B$ such that $Tu = 0$ on $\Gamma \cap 2B$ 
when the intersection is non-empty.  Precise estimates will be given later in the introduction. 
The boundary estimates are trickier, because we do not have the conventional ``fatness" of the complement 
of the domain, and it is useful to know beforehand that suitable versions of Poincar\'e and Sobolev inequalities hold. For instance,
\begin{equation} \label{1.4}
\fint_{B(x,r)} |u(y)| dy \leq C r^{-d} \int_{B(x,r)} |\nabla u(y)| w(y) dy
\end{equation} 
for $u \in W$, $x\in \Gamma$, and $r>0$ such that $Tu = 0$ on $\Gamma \cap B(x,r)$
and, if $m(B(x,r))$ denotes $\int_{B(x,r)} w(y) dy$,
\begin{equation} \label{1.5}
\Big\{\frac1{m(B(x,r))}\int_{B(x,r)} \Big|u(y) - \fint_{B(x,r)} u \Big|^p w(y) dy\Big\}^{1/p}
\leq C r \Big\{ \frac{1}{m(B(x,r))}\int_{B(x,r)} |\nabla u(y)|^2 w(y) dy \Big\}^{1/2}
\end{equation}
for $u\in W$,  $x\in \overline\Omega=\R^n$, $r>0$, and 
$p\in \left[1,\frac{2n}{n-2}\right]$ (if $n\geq 3$) or $p\in [1,+\infty)$ (if $n=2$).

A substantial portion of the proofs lies in the
analysis of the newly defined Sobolev spaces. It is important to note, in particular, that we prove the density of smooth functions on $\R^n$ (and not just $\Omega$)
in our weighted Sobolev space $W$.
That is,
for any function $f$ in $W$, there exists a sequence $(f_k)_{k\geq 1}$ of functions in $C^\infty(\rn) \cap W$ such that $\|f- f_k\|_W$ tends to 0 and $f_k$ converges to $f$ in $L^1_{loc}(\rn)$. 
In codimension $1$, this sort of property, just like \eqref{1.4} or \eqref{1.5},
typically requires a fairly nice boundary, 
e.g., Lipschitz, and it is quite remarkable that here they all hold in the
complement of any Ahlfors-regular set. 
This is, of course, a fortunate outcome of working with lower dimensional boundary: 
we can guarantee ample access to the boundary (cf., e.g., the Harnack ``tubes" discussed above), 
which turns out to be sufficient despite the absence of traditionally required ``massive complement". 
Or rather one could say that the boundary itself is sufficiently ``massive" from the PDE point of view, due to our carefully chosen equation and corresponding function spaces.

With all these ingredients, we can follow the standard proofs for elliptic divergence form operators.
When $u$ is a solution to $Lu = 0$ in a ball $2B \subset \Omega$, the De Giorgi-Nash-Moser estimates and the Harnack inequality in the ball $B$ don't depend on the properties of the boundary $\Gamma$ and thus can be proven as in the case of codimension 1.
When $B \subset \R^n$ is a ball centered on $\Gamma$ and $u$ is a weak solution to $Lu = 0$ in $2B$ whose trace satisfies $Tu = 0$ on $\Gamma \cap 2B$, the quantitative boundedness and the quantitative H\"older continuity of the solution $u$ are expressed with the help of the weight $w$. There holds, if $m(2B) = \int_{2B} w(y) dy$,
\begin{equation} \label{MoserIntro}
\sup_B u \leq C \left( \frac{1}{m(2B)} \int_{2B} |u(y)|^2 w(y) dy \right)^{1/2} 
\end{equation} 
and, for any $\theta \in (0,1]$, 
\begin{equation} \label{HolderIntro}
\sup_{\theta B} u \leq C \theta^\alpha \sup_{B} u \leq C \theta^\alpha \left( \frac{1}{m(2B)} 
\int_{2B} |u(y)|^2 w(y) dy \right)^{1/2}, 
\end{equation} 
where $\theta B$ denotes the ball with same center as $B$ but whose radius is multiplied by $\theta$, and $C,\alpha>0$ are constants that depend only on the dimensions $d$ and $n$, the Ahlfors constant $C_0$ and the ellipticity constant $C_1$.

We establish then the existence and uniqueness of a Green function $g$, which is roughly speaking a positive function on $\Omega \times \Omega$ such that, for all $y\in \Omega$, the function $g(.,y)$ solves $Lg(.,y) = \delta(y)$ and $Tg(.,y) = 0$. In particular, the following pointwise estimates are shown: 
\begin{equation} \label{GreenIntro}
0 \leq g(x,y) \leq \left\{\begin{array}{ll} C|x-y|^{1-d} & \text{ if } 4|x-y| \geq \delta(y) \\ 
\frac{C|x-y|^{2-n}}{w(y)} & \text{ if } 2|x-y| \leq \delta(y), \, n\geq 3 \\
\frac{C_\epsilon}{w(y)} \left(\frac{\delta(y)}{|x-y|} \right)^\epsilon & \text{ if } 2|x-y| \leq \delta(y), \, n=2,\end{array} \right.
\end{equation} 
where $C>0$ depends on $d$, $n$, $C_0$, $C_1$ and $C_\epsilon>0$ depends on $d$, $C_0$, $C_1$, $\epsilon$.
When $n\geq 3$, the pointwise estimates can be gathered to a single one, and may look more natural for the reader: if $m(B) = \int_B w(y) dy$, 
\begin{equation} \label{GreenIntro2}
0 \leq g(x,y) \leq C\frac{|x-y|^{2}}{m(B(x,|x-y|))}
\end{equation} 
whenever $x,y\in \Omega$. The bound in the case where $n=2$ and $2|x-y| \leq \delta(y)$ can surely be improved into a logarithm bound, but the bound given here is sufficient for 
our purposes.
Also, our results hold for any $d$ and any $n$ such that $d<n-1$,  
(i.e., even in the cases where $n=2$ or $d\leq 1$), 
which proves that Ahlfors regular domains are `Greenian sets' in our adapted elliptic theory. Note that contrary to the codimension 1 case, the notion of the fundamental solution in $\RR^n$ is not accessible,
since the distance to the boundary of $\Omega$ is an integral part of the definition of $L$.

We use the Harnack inequality, the De Giorgi-Nash-Moser estimates, as well as a suitable version of the maximum principle, to solve the Dirichlet problem for continuous functions with compact support on $\Gamma$, and then to
define harmonic measures $\omega^x$ for $x\in \Omega$ (so that $\int_\Gamma g d\omega^x$ 
is the value at $x$ of the solution of the Dirichlet problem for $g$). Note that we do not need an analogue of the Wiener criterion (which normally guarantees that solutions with continuous data are continuous 
up
to the boundary and allows one to define the harmonic measure), as we have already proved a stronger property, that solutions are H\"older continuous 
up
to the boundary. 
Then, following the ideas of \cite[Section 1.3]{KenigB}, we prove the following properties 
on the harmonic measure $\omega^x$. First, the non-degeneracy of the harmonic measure states 
that if $B$ is a ball centered on $\Gamma$,
\begin{equation} \label{ndhmi1}
\omega^x(B \cap \Gamma) \geq C^{-1}
\end{equation} 
whenever $x \in \Omega \cap \frac12 B$ and 
\begin{equation} \label{ndhmi2}
\omega^x(\Gamma \setminus B) \geq C^{-1}
\end{equation} 
whenever $x \in \Omega \setminus 2B$, the constant $C>0$ depending as previously on $d$, $n$, $C_0$ and $C_1$. Next, let us recall that
any boundary ball has a corkscrew point, that is for any ball $B= B(x_0,r) \subset \R^n$ centered on $\Gamma$, there exists $\Delta_B \in B$ such that $\delta(\Delta_B)$ is bigger than $\epsilon r$,  where $\epsilon>0$ depends only on $d$, $n$ and $C_0$. 
With this definition in mind, we
compare the harmonic 
measure with the Green function: for any ball $B$ of radius $r$ centered on $\Gamma$,
\begin{equation} \label{Grandhm1}
C^{-1} r^{1-d} g(x,\Delta_B) \leq \omega^x(B \cap \Gamma) \leq C r^{1-d} g(x,\Delta_B)
\end{equation}
for any $x\in \Omega \setminus 2B$ and
\begin{equation} \label{Grandhm2}
C^{-1} r^{1-d} g(x,\Delta_B) \leq \omega^x(\Gamma \setminus B) \leq C r^{1-d} g(x,\Delta_B)
\end{equation}
for any $x\in \Omega \cap \frac12 B$ which is far enough from $\Delta_B$, say $|x-\Delta_B| \geq \epsilon r/2$, where $\epsilon$ is the constant used to define $\Delta_B$. The constant $C>0$ in \eqref{Grandhm1} and \eqref{Grandhm2} depends again only on $d$, $n$, $C_0$ and $C_1$. 
The estimates \eqref{Grandhm1} and \eqref{Grandhm2} can be seen as weak versions of the `comparison principle', which deal only with the Green functions and the harmonic measure and which can be proven by using the specific properties of the latter objects. The inequalities \eqref{Grandhm1} and \eqref{Grandhm2} 
are essential for the proofs of
the next three results. 

The first one is the doubling property of
the harmonic measure, 
which 
guarantees that, if $B$ is a ball centered on $\Gamma$, $\omega^x(2B \cap \Gamma) \leq C \omega^x(B \cap \Gamma)$ whenever $x\in \Omega \setminus 4B$. It has an interesting counterpart: $\omega^x(\Gamma \setminus B) \leq C \omega^x(\Gamma \setminus 2B)$ whenever $x\in \Omega \cap \frac12B$. 

The second one is the change-of-the-pole estimates, which can be stated as
\begin{equation} \label{Chpole}
C^{-1} \omega^{\Delta_B}(E) \leq \frac{\omega^x(E)}{\omega^x(\Gamma \cap B)} \leq C \omega^{\Delta_B}(E)
\end{equation}
when $B$ is a ball
centered on $\Gamma$, $E \subset B\cap \Gamma$ 
is a Borel set, 
and $x \in \Omega \setminus 2B$. 

The last result is the comparison principle, that says that if $u$ and $v$ are
positive weak solutions of $Lu = Lv = 0$ such that $Tu = Tv = 0$ on $2B \cap\Gamma$, 
where $B$ is a ball centered on $\Gamma$,
then $u$ and $v$ are comparable in $B$, i.e.,
\begin{equation} \label{cpintro}
\sup_{z\in B \sm \Gamma} \frac{u(z)}{v(z)} \leq C \inf_{z\in B \sm \Gamma} \frac{u(z)}{v(z)}.
\end{equation}
In each case, i.e., for the doubling property of the harmonic measure, the change of pole, or the comparison principle, the constant $C>0$ depends only on $d$, $n$, $C_0$ and $C_1$.

It is difficult to survey a history of the subject that is 
so
classical (in the co-dimension one case). In that setting, that is, in co-dimension one and reasonably nice geometry, 
e.g., of Lipschitz domains,  the results have largely become folklore and we often follow the exposition in 
standard texts \cite{GT}, \cite{HLbook}, \cite{Mazya11}, \cite{MZbook}, \cite{StaCRAS}, \cite{GW}, \cite{CFMS}. 
The general order of development is inspired by \cite{KenigB}. Furthermore, let us point out that while 
the invention of a harmonic measure which serves the higher co-dimensional boundaries, 
which is associated to a linear PDE, and which is absolutely continuous with respect to the Lebesgue measure 
on reasonably nice sets, is the main focal point of our work, various versions of degenerate elliptic operators 
and weighted Sobolev spaces have of course appeared in the literature over the years. 
Some versions of some of the results  listed above or similar ones can  be found, e.g., in \cite{Ancona}, \cite{FKS}, \cite{Hajlasz}, \cite{HK00}, \cite{NDbook}, \cite{Kilpelai}, \cite{JW}. However, 
the presentation here is fully self-contained, and since we did not rely on previous work, 
we hope to be forgiven for not providing a detailed review of the corresponding literature. 
Also, the context of the present paper often makes it possible to have much simpler proofs than 
a more general setting of not necessarily Ahlfors regular sets.
It is perhaps worth pointing out that we work with {\it homogeneous} Sobolev spaces. Unfortunately, those are much less popular in the literature that their non-homogeneous counterparts, while they are more suitable for PDEs on unbounded domains. 

\medskip

As outlined in \cite{DFM-CRAS}, ,
we intend in subsequent publications to take stronger assumptions, 
both on the geometry of $\Gamma$ and the choice of $L$,
and prove that the harmonic measure defined here is absolutely continuous with respect to 
$\H^d_{\vert \Gamma}$. For instance, we
will assume that $d$ is an integer and $\Gamma$ is the graph 
of a Lipschitz function $F : \R^d \to \R^{n-d}$, with a small enough Lipschitz constant.
As for $A$, we will
assume that $A(x) = D(x)^{-n+d+1} I$ for $x\in \Omega$, with 
\begin{equation} \label{1.3}
D(x) = \Big\{ \int_\Gamma |x-y|^{-d-\alpha} d\H^d(y) \Big\}^{-1/\alpha}
\end{equation}
for some constant $\alpha > 0$. Notice that because of \eqref{1.1}, $D(x)$ is equivalent to 
$\delta(x)$; when $d=1$ we can also take $A(x) = \delta(x)^{-n+d+1} I$, but
when $d \geq 2$ we do not know whether $\delta(x)$ is smooth enough to work. 
In \eqref{1.3}, we could also replace $\H^d$ with another Ahlfors-regular measure on $\Gamma$. 

With these additional assumptions we will prove that the harmonic measure described above 
is absolutely continuous with respect to $\H^d_{\vert \Gamma}$, with a density which is 
a Muckenhoupt $A_\infty$ weight. In other words, we shall establish an analogue of 
Dahlberg's result \cite{Dah} for domains with a higher co-dimensional boundary given by a Lipschitz graph 
with a small Lipschitz constant.
It is not so clear what is the right condition for this in terms of $A$, 
but the authors still hope that 
a good condition on $\Gamma$ is its uniform rectifiability. 
Notice that in remarkable contrast with the case of codimension 1, we do not state an additional quantitative connectedness condition on $\Omega$, 
such as the Harnack chain condition in codimension $1$; this is because 
such conditions are automatically satisfied when 
$\Gamma$ is Ahlfors-regular with a large codimension. 

The present paper is aimed at giving a fairly pleasant general framework for studying a version of the
harmonic measure in the context of Ahlfors-regular sets $\Gamma$ of codimension larger than $1$, 
but it will probably be interesting and hard to understand well the relations between the geometry 
of $\Gamma$, the regularity properties of $A$ (which has to be linked to $\Gamma$ through 
the distance function), 
and the regularity properties of the associated harmonic measure.

\chapter{The
Harnack Chain Condition and the
Doubling Property}

\label{SGP}
We keep the same notation as in Section \ref{S1}, concerning $\Gamma \subset \R^n$,
a closed set that satisfies \eqref{1.1} for some $d < n-1$, $\Omega = \R^n \sm \Gamma$, then 
$\sigma = \H^d_{\vert \Gamma}$, $\delta(z) = \dist(z,\Gamma)$, 
and the weight $w(z) = \delta(z)^{d+1-n}$.  

Let us add the notion of measure. The measure $m$ is defined on (Lebesgue-)measurable subset of $\R^n$ by $m(E) = \int_E w(z) dz$. We may write $dm(z)$ for $w(z) dz$. 
Since $0<w<+\infty$ a.e. in $\R^n$, $m$ and the Lebesgue measure are mutually absolutely continuous, that is they have the same zero sets. Thus there is no need to specify the measure when using the expressions {\em almost everywhere} and {\em almost every}, both abbreviated {\em a.e.}; conversely the expression $a.e.$ without any measure refers to {\em Lebesgue almost everywhere} or equivalently {\em $m$-almost everywhere}.

\medskip

In the sequel of the article, $C$ will denote a real number (usually big) that can vary from one line to another. The parameters which the constant $C$ depends on are either obvious from context or recalled. Besides, the notation $A\approx B$ will be used to replace $C^{-1}A \leq B \leq CA$. 

\medskip

This section is devoted to the proof of the very first geometric properties on the space $\Omega$ and the weight $w$. We will prove in particular that $m$ is a doubling measure and $\Omega$ satisfies the Harnack chain condition.

\medskip

First, let us prove the Harnack chain condition we stated in Section \ref{S1}.

\begin{lemma}\label{lHC} Let $\Gamma$ be a $d$-Ahlfors regular set in $\RR^n$, $d<n-1$, that is, assume that \eqref{1.1} is satisfied. Then there exists a constant $c>0$, that depends only
on $C_0$, $n$, and $d < n-1$, such that for $\Lambda \geq 1$ and $x_1, x_2 \in \Omega$ 
such that $\dist(x_i,\Gamma) \geq r$ and $|x_1-x_2| \leq \Lambda r$, we can find two points
$y_i \in B(x_i,r/3)$ such that $\dist([y_1,y_2],\Gamma) \geq c \Lambda^{-d/(n-d-1)} r$.
That is, there is a thick tube in $\Omega$ that connects the two $B(x_i,r/3)$.
\end{lemma}

\bp Indeed, suppose $x_2 \neq x_1$, set $\ell = [x_1,x_2]$, and denote by $P$ the vector hyperplane 
with a direction orthogonal to $x_2-x_1$. Let $\varepsilon \in (0,1)$ be small, to be chosen soon.
We can find $N \geq C^{-1} \varepsilon^{1-n}$ points $z_j\in P \cap B(0,r/3)$, such that
$|z_j-z_k| \geq 4\varepsilon r$ for $j\neq k$. Set $\ell_j = z_j + \ell$, and suppose that
$\dist(\ell_j,\Gamma) \leq \varepsilon r$ for all $j$. Then we can find points $w_j \in \Gamma$
such that $\dist(w_j,\ell_j) \leq \varepsilon r$. Notice that the balls $B_j = B(w_j,\varepsilon r)$
are disjoint because $\dist(\ell_j,\ell_k) \geq 4 \varepsilon r$, and by \eqref{1.1}
\[N C_0^{-1}(\varepsilon r)^d \leq \sum_j \sigma(B_j) 
= \sigma\big(\bigcup_j B_j \big) \leq \sigma(B(w, 2r+|x_2-x_1|))
\leq C_0 (2+\Lambda)^d r^d\]
where $w$ is any of the $w_j$. Thus 
$\varepsilon^{1-n} \varepsilon^d \leq C C_0^{2} \Lambda^d$ (recall that $\Lambda\geq 1$), 
a contradiction if we take $\varepsilon \leq c \Lambda^{-d /(n-d-1)}$,
where $c >0$ depends on $C_0$ too. Thus we can find $j$ such that
$\dist(\ell_j,\Gamma) \geq \varepsilon r$, and the desired conclusion holds with
$y_i = x_i + z_j$.
\ep 

\begin{lemma}\label{lHC2}
Let $\Gamma \subset \R^n$ be an Ahlfors regular set of dimension $d$, $d<n-1$. 
Then for any $\Lambda \geq 1$, we can find a large integer $N_\Lambda$ and a small real $\epsilon_\Lambda$ such that for any couple of points $x_1,x_2 \in \Omega$ satisfying $\dist(x_i,\Gamma) \geq r$ and $|x_1-x_2| \leq \Lambda r$, there exists a sequence of balls $(B_i)_{0\leq i \leq N_\Lambda}$ - called Harnack chain - verifying:
\begin{enumerate}[(i)]
\item $B_0$ is centered on $x_1$ and $B_{N_\Lambda}$ is centered on $x_2$,
\item for $0 \leq i \leq N_\Lambda$, $3B_i \subset \Omega$ and $B_i$ has radius at least $\epsilon_\Lambda r$,
\item for $1 \leq i \leq N_\Lambda$, $B_{i-1} \cap B_i \neq \emptyset$.
\end{enumerate}
\end{lemma}

\bp
This lemma is a corollary of Lemma \ref{lHC}. First, in the proof, $c$ shall be the constant given by Lemma \ref{lHC}. 

Let $\Lambda \geq 1$ and $x_1,x_2$ as in the lemma. Let $y_1,y_2$ be the associated points given by Lemma \ref{lHC}, which thus verify $y_1 \in B(x_1,r/3)$, $y_2 \in B(x_2,r/3)$, and $\dist([y_1,y_2],\Gamma) \geq c\Lambda^{-d/(n-d-1)}r$.

From there, we define $\epsilon_\Lambda$ as $c\Lambda^{-d/(n-d-1)}/3$ and $N_\Lambda$ as the smallest integer bigger than $(\Lambda +2)r/\epsilon_\Lambda +2$. The sequence $(B_i)_{0\leq i \leq N_\Lambda}$ is then constructed as follows. First, $B_0$ is $B(x_1,r/3)$. Then $(B_i)_{1\leq i < N_\Lambda-1}$ are balls centered on the segment $[y_1,y_2]$ and of radius $\epsilon_\Lambda r$ -  precisely $B_i$ is centered on $[(N_\Lambda - i - 1) y_1 + (i-1)y_2]/[N_\Lambda - 2]$ - and finally $B_{N_\Lambda}$ is $B(x_2,r/3)$.
\ep

Then, we give estimates on the weight $w$.

\begin{lemma} \label{lwest}
There exists $C>0$ such that 
\begin{enumerate}[(i)]
\item for any $x\in \R^n$ and any $r>0$ satisfying $\delta(x) \geq 2r$,
\begin{equation} \label{ISob7b}
C^{-1} r^n w(x) \leq m(B(x,r))= \int_{B(x,r)} w(z) dz \leq C r^n w(x),
\end{equation}
\item for any $x\in \R^n$ and any $r>0$ satisfying $\delta(x) \leq 2r$,
\begin{equation} \label{ISob8b}
C^{-1} r^{d+1} \leq m(B(x,r)) = \int_{B(x,r)} w(z) dz \leq Cr^{d+1}.
\end{equation}
\end{enumerate}
\end{lemma}

\begin{remark}
In the above lemma, the estimates are different if $\delta(x)$ is bigger or smaller than $2r$. Yet the critical ratio $\frac{\delta(x)}{r}=2$ is not relevant: for any $\alpha >0$, we can show as well that \eqref{ISob7b} holds whenever $\delta(x) \geq \alpha r$ and \eqref{ISob8b} holds if $\delta(x) \leq \alpha r$, with a constant $C$ that depends on $\alpha$. 

Indeed, we can replace $2$ by $\alpha$ if we can prove that for any $K>1$ there exists $C>0$ such that for any $x\in \R^n$ and $r>0$ satisfying 
\begin{equation} \label{rdoubling1}
K^{-1} r \leq \delta(x) \leq Kr
\end{equation}
we have
\begin{equation} \label{rdoubling2}
C^{-1} r^{d+1} \leq r^n w(x) \leq Cr^{d+1}.
\end{equation}
However, since $w(x) = \delta(x)^{d+1-n}$, \eqref{rdoubling1} implies $w(x) \approx r^{d+1-n}$ which in turn gives \eqref{rdoubling2}.
\end{remark}

\bp First suppose that $\delta(x)\geq 2r$. Then for 
any $z\in B(x,r)$, $\frac12 \delta(x)\leq \delta(z) \leq \frac32 \delta(x)$
and hence
$C^{-1} w(x) \leq w(z) \leq Cw(x)$; 
\eqref{ISob7b} follows. 

The lower bound in \eqref{ISob8b} is also fairly easy, just note that when $\delta(x) \leq 2r$, $\delta(z) \leq 3r$ for any $z\in B(x,r)$ and hence
\begin{equation} \label{ISob8d}
m(B(x,r)) \geq \int_{B(x,r)} (3r)^{1+d-n} dz \geq C^{-1} r^{d+1}.
\end{equation}

Finally we
check the upper bound in \eqref{ISob8b}. We claim that for any $y\in \Gamma$ and any $r>0$,
\begin{equation} \label{ldoub1}
m(B(y,r)) = \int_{B(y,r)} \delta(\xi)^{d+1-n} d\xi \leq C r^{d+1}.
\end{equation}
From the claim, let us prove the upper bound in \eqref{ISob8b}. Let $x\in \R^n$ and $r>0$ be such that $\delta(x) \leq 2r$. 
Thus there exists $y\in \Gamma$ such that $B(x,r) \subset B(y,3r)$ and thanks to \eqref{ldoub1}
\begin{equation} \label{ISob8c}
m(B(x,r)) \leq \int_{B(y,3r)} w(z) dz \leq C(3r)^{d+1} \leq Cr^{d+1},
\end{equation}
which gives the upper bound in \eqref{ISob8b}. 

Let us now prove the claim. By translation invariance, we can choose $y=0 \in \Gamma$. Note that $\delta(\xi) \leq r$ in the domain
of integration. Let us evaluate the measure of the set 
$Z_k = \big\{\xi \in B(0,r) \, ; \, 2^{-k-1}r < \delta(\xi) \leq 2^{-k} r \big\}$.
We use \eqref{1.1} to cover $\Gamma \cap B(0,2r)$ with less than
$C 2^{kd}$ balls $B_j$ of radius $2^{-k}r$ centered on $\Gamma$; then $Z_k$
is contained in the union of the $3B_j$, so $|Z_k| \leq C 2^{kd} (2^{-k}r)^n$
and $\int_{\xi \in Z_k} \delta(\xi)^{1+d-n} d\xi
\leq C 2^{kd} (2^{-k}r)^n (2^{-k}r)^{d+1-n} = C 2^{-k} r^{d+1}$.
We sum over $k \geq 0$ and get \eqref{ldoub1}.
\ep 

A consequence of Lemma~\ref{lwest} is that $m$ is a doubling measure, that is for any $x\in \R^n$ and $r>0$, one has $m(B(x,2r)) \leq C m(B(x,r))$ for a constant $C>0$ independent of $x$ and $r$. Actually, we can prove the following stronger fact: for any $x\in \R^n$ and any $r>s>0$, there holds
\begin{equation} \label{doublinggen}
C^{-1} \left( \frac rs \right)^{d+1} \leq \frac{m(B(x,r))}{m(B(x,s))} \leq C \left( \frac rs \right)^n.
\end{equation} 
Three cases may happen. First, $\delta(x) \geq 2r \geq 2s$ and then with \eqref{ISob7b},
\begin{equation} \label{Idoub3}
\frac{m(B(x,r))}{m(B(x,s))} \approx \frac{r^n w(x)}{s^n w(x)} =  \left( \frac rs \right)^n.
\end{equation}
Second, $ \delta(x) \leq 2s \leq 2r$. In this case, note that \eqref{ISob8b} implies
\begin{equation} \label{ldoub4}
\frac{m(B(x,r))}{m(B(x,s))} \approx \frac{r^{d+1}}{s^{d+1}} = \left( \frac rs \right)^{d+1}.
\end{equation}
At last, $2s \leq \delta(x) \leq 2r$. Note that \eqref{ISob7b} and \eqref{ISob8b} yield
\begin{equation} \label{ldoub5}
\frac{m(B(x,r))}{m(B(x,s))} \approx \frac{r^{d+1}}{s^nw(x)}.
 \end{equation} 
Yet, $2s \leq \delta(x) \leq 2r$ implies $C^{-1} r^{d+1-n} \leq w(x) \leq Cs^{d+1-n}$ and thus 
\begin{equation} \label{ldoub6}
C^{-1} \left( \frac rs \right)^{d+1} \leq \frac{m(B(x,r))}{m(B(x,s))} \leq C \left( \frac rs \right)^n.
 \end{equation} 
which finishes the proof of \eqref{doublinggen}. 

One can see that the coefficients $d+1$ and $n$ are optimal in \eqref{doublinggen}. The fact that the volume of a ball with radius $r$ is not equivalent to $r^\alpha$ for some $\alpha>0$ will cause some difficulties. For instance, regardless of the choice of $p$, we cannot have a Sobolev embedding $W \hookrightarrow L^p$ and we have to settle for the Sobolev-Poincar\'e inequality \eqref{1.5}.

\medskip

Another consequence of Lemma~\ref{lwest} is that for any $x\subset \R^n$, any $r>0$, and any nonnegative function $g\in L^1_{loc}(\R^n)$,
\begin{equation} \label{L1byL1w}
\frac{1}{|B(x,r)|} \int_{B(x,r)} g(z) dz \leq C \frac{1}{m(B(x,r))} \int_B g(z) w(z) dz.
 \end{equation} 
Indeed, the inequality \eqref{L1byL1w} holds if we can prove that 
\begin{equation} \label{ldoub8}
\frac{m(B(x,r))}{|B(x,r)|} \leq Cw(z) \qquad \forall z\in B(x,r).
\end{equation}
This latter fact can be proven as follows: 
\begin{equation} \label{ldoub7}
\frac{m(B(x,r))}{|B(x,r)|} \leq \frac{m(B(z,2r))}{|B(x,r)|} \leq Cr^{-n} m(B(z,2r))
 \end{equation} 
If $\delta(z) \geq 4r$, then Lemma~\ref{lwest} gives $r^{-n} m(B(z,2r)) \leq Cw(z)$. If $\delta(z) \leq 4r$, then $w(z) \geq C^{-1} r^{d+1-n}$ and Lemma~\ref{lwest} entails $r^{-n} m(B(z,2r)) \leq Cr^{d+1-n} \leq Cw(z)$. In both cases, we obtain \eqref{ldoub8} and thus \eqref{L1byL1w}.   

\medskip

We end the section with a corollary of Lemma~\ref{lwest}.

\begin{lemma} \label{lwA2}
The weight $w$ is in the $\mathcal A_2$-Muckenhoupt class,  i.e. there exists $C>0$ such that for any ball $B\subset \R^n$, 
\begin{equation} \label{lconv25} 
\fint_B  w(z) dz  \fint_B w^{-1}(z) dz \leq C.
\end{equation}
\end{lemma}

\bp Let $B:=B(x,r)$. If $\delta(x) \geq 2r$, then for any $z\in B(x,r)$, $C^{-1} w(x) \leq w(z) \leq Cw(x)$ and thus $\fint_B  w(z)dz \cdot  \fint_B w^{-1}(z)dz \leq C w(x) w^{-1}(x) = C$.
If $\delta(x) \leq 2r$, then \eqref{ISob8b} implies that $\fint_B  w(z)dz \leq C r^{-n} r^{d+1} = Cr^{d+1-n}$. Besides, for any $z\in B(x,r)$, $\delta(z) \leq 3r$ and hence $w^{-1}(z) \leq C r^{n-d-1}$. It follows that if $\delta(x) \leq 2r$, $\fint_B  w(z)dz \cdot  \fint_B w^{-1}(z)dz \leq C$. The assertion \eqref{lconv25} follows.
\ep

\chapter{Traces}
\label{S2}

The weighted Sobolev space $W= \dot W_w^{1,2}(\Omega)$ and $H = \dot H^{1/2}(\Gamma)$ are defined as in Section \ref{S1} (see \eqref{defW}, \eqref{defH}). Let us give a precison. Any $u\in W$ has a distributional derivative in $\Omega$ that belongs to $L^2(\Omega,w)$, that is there exists a vector valued function $v\in L^2(\Omega,w)$ such that for any $\varphi \in C^\infty_0(\Omega,\R^n)$
\begin{equation} \label{defDD}
\int_\Omega v \cdot \varphi \, dx = -\int_\Omega u \diver\varphi \, dx.
\end{equation}
This definition make sense since $v \in L^2(\Omega,w) \subset L^1_{loc}(\Omega)$. For the proof of the latter inclusion, use for instance Cauchy-Schwarz inequality and  \eqref{L1byL1w}.

\medskip

The aim of the section is to state and prove a trace theorem. But for the moment, let us keep discussing about the space $W$. 
We say that $u$ is absolutely continuous on lines in $\Omega$ if there exists $\bar u$ which coincides with $u$ a.e. such that for almost every line $\ell$
(for the usual invariant measure on the Grassman manifold, but we can also say, given any choice of direction $v$ and 
and a vector hyperplane plane $P$ transverse to $v$, for the line 
$x+\R v$ for almost every $x \in P$), we have the following properties.
First, the restriction of $\bar u$ to $\ell \cap \Omega$ (which makes sense, for a.e. line $\ell$, and is measurable, by Fubini) is absolutely continuous, which means that it is differentiable almost  everywhere on $\ell \cap \Omega$ and is the indefinite integral of its derivative on each component of
$\ell \cap \Omega$. By the natural identification, the derivative in question is obtained from
the distributional gradient of $u$.

\begin{lemma}\label{eqACL}
Every $u \in W$ is absolutely continuous on lines in $\Omega$.
\end{lemma}

\bp This lemma can be seen as a consequence of \cite[Theorem 1.1.3/1]{Mazya11} since 
the absolute continuity on lines is a local property and, thanks to \eqref{L1byL1w}, 
$W \subset \{u \in L^1_{loc}(\Omega), \, \nabla u \in L^2_{loc}(\Omega)\}$. 
Yet, the proof of Lemma~\ref{eqACL} is classical: since the property is local, 
it is enough to check the property on lines parallel to a fixed vector $e$, and when $\Omega$ is 
the product of $n$ intervals, one of which is parallel to $e$.
This last amounts to using the definition of the distributional gradient, testing on product functions,
and applying Fubini.
In addition, the derivative of $u$ on almost every line $\ell$ of direction $e$ coincides with 
$\nabla u \cdot e$ almost everywhere on $\ell$.
\ep

\begin{lemma}\label{W=barW}
We have the following equality of spaces
\begin{equation}\label{defWg}
W=\{u \in L^1_{loc}(\rn), \,\nabla u\in L^2(\R^n,w)\},
\end{equation}
where the derivative of $u$ is taken in the sense of distribution in $\R^n$, that is for any $\varphi \in C^\infty_0(\R^n,\R^n)$,
\[\int \nabla u \cdot \varphi \, dx= -\int u \diver\varphi \, dx. \]
\end{lemma}

\bp Here and in the sequel, we will constantly use the fact that with $\Omega = \R^n \sm \Gamma$ and
because \eqref{1.1} holds with $d<n-1$, 
\begin{equation} \label{2.1}
\text{ almost every line $\ell$ is contained in $\Omega$.} 
\end{equation}
Let us recall that it means that given any choice of direction $v$ and a vector hyperplane $P$ transverse to $v$, the line $x+ \R v \subset \Omega$ for almost every $x\in P$. In particular, for almost every $(x,y)\in (\R^n)^2$, there is a unique line going through $x$ and $y$ and this line is included in $\Gamma$.

Lemma~\ref{eqACL} and \eqref{2.1} implies that $u\in W$ is actually absolutely continuous on lines in $\R^n$, i.e. any $u \in W$ (possibly modified on a set of zero measure) is absolutely continuous on almost every line $\ell \subset \R^n$. 
As we said before, $\nabla u = (\dr_1 u, \dots, \dr_n u)$, the distributional gradient of $u$ in $\Omega$, equals the `classical' gradient of $u$ defined in the following way. 
If $e_1=(1,0,\dots,0)$ is the first coordinate vector, then $\dr_1 u(y,z)$ is the derivative at the point $y$ of the function $u_{|(0,z)+\R e_1}$, the latter quantity being defined for almost every $(y,z)\in \R\times \R^{n-1}$ because $u$ is absolutely continuous on lines in $\R^n$. If $i>1$, $\dr_i u(x)$ is defined in a similar way.

As a consequence, for almost any $(y,z) \in \R^n \times \R^n$,  $u(z) - u(y) = \int_{0}^1 (z-y)\cdot \nabla u(y+t(z-y)) dt$ and hence, 
\begin{equation} \label{2.2}
|u(y)-u(z)| \leq \int_{0}^1 |z-y||\nabla u(y+t(z-y))| dt.
\end{equation} 
Let us integrate this for $y$ in a ball $B$. We get that for almost every $z\in \R^n$,
\begin{equation} \label{2.3}
\fint_{y \in B} |u(y)-u(z)| dy \leq \fint_{y \in B}\int_{0}^1 |z-y||\nabla u(y+t(z-y))| dt.
\end{equation}
Let us further restrict to the case $z\in B=B(x,r)$; the change of variable $\xi = z+t(y-z)$ shows that
\begin{equation} \begin{split}
\fint_{y \in B} |u(y)-u(z)| dy & = \int_0^1 \fint_{y\in B}  |y-z| |\nabla u (z+t(y-z))| dy\, dt \\
& = \int_0^1 \frac{1}{|B|}\int_{\xi \in B(z+t(x-z),tr)} \frac{|z-\xi|}t |\nabla u(\xi)| \frac{d\xi}{t^n} dt \\
& = \int_{\xi \in B} |\nabla u (\xi)| \frac{|z-\xi|}{|B(z,r)|} d\xi \int_{|z-\xi|/2r}^1 \frac{dt}{t^{n+1}}\\
& \leq 2^n |B(0,1)|^{-1}\int_{\xi \in B} |\nabla u(\xi)| |z-\xi|^{1-n} d\xi, \\
\end{split} \end{equation}
where the last but one line is due to the fact that 
$\xi \in B(z+t(x-z),tr)$ is equivalent to $|\xi-z-t(x-z)| \leq tr$, 
 which forces
$|\xi-z| \leq tr + t|x-z| \leq 2rt$.
Therefore,
for almost any $z\in B$,
\begin{equation} \label{2.4}
 \fint_{y \in B} |u(y)-u(z)| dy
\leq C\int_{\xi \in B} |\nabla u(\xi)| |z-\xi|^{1-n} d\xi,
\end{equation}
where $C$ depends on $n$, but not on $r$, $u$, or $z$. With a second integration on $z\in B=B(x,r)$, we obtain
\begin{equation} \label{2.4a}
  \fint_{z \in B}\fint_{y \in B} |u(y)-u(z)| dy \, dz
\leq C\int_{\xi \in B} |\nabla u(\xi)| \fint_{z\in B} |z-\xi|^{1-n} dz \, d\xi \leq Cr\fint_{\xi \in B} |\nabla u(\xi)| d\xi.
\end{equation}
By H\"older's
inequality and \eqref{L1byL1w}, the right-hand side is bounded (up to a constant depending on $r$) by $\| u \|_W$. As a consequence,
\begin{equation} \label{2.4b}
  \fint_{z \in B}\fint_{y \in B} |u(y)-u(z)| 
\leq C_r \|u\|_W < +\infty.
\end{equation}
and thus, by Fubini's lemma, $\fint_{y \in B} |u(y)-u(z)| <+\infty$ for a.e. $z\in B$. In particular, the quantity $\fint_{y \in B} |u(y)|$ is finite for any ball $B\subset \R^n$, that is $u\in L^1_{loc}(\R^n)$. 

\medskip

Since $L^1_{loc}(\R^n) \subset L^1_{loc}(\Omega)$, we just proved that $W = \{u \in L^1_{loc}(\R^n), \, \nabla u \in L^2(\Omega,w)\}$, where $\nabla u=(\dr_1 u,\dots,\dr_n u)$ is distributional gradient on $\Omega$. Let $u\in W$. 
Since $\Gamma$ has zero measure, $\nabla u \in L^2(\R^n,w)$ and thus it suffices to check that $u$ has actually a distributional derivative in $\R^n$ and that this derivative equals $\nabla u$. 
However, the latter fact is a simple consequence of \cite[Theorem 1.1.3/2]{Mazya11}, because $u$ is absolutely continuous on lines in $\R^n$. 
The proof of Maz'ya's result is basically the following:  for any $i\in \{1,\dots,n\}$ and any $\phi \in C^\infty_0(\R^n)$, an integration by part gives $\int u \dr_i \phi = - \int (\dr_i u)  \phi$. 
The two integrals in the latter equality make sense since both $u$ and $\dr_i u$ are in $L^1_{loc}(\R^n)$; the integration by part is possible because $u$ is absolutely continuous on almost every line.
\ep

\begin{remark}
An important by-product of the proof is that Lemma~\ref{eqACL} can be improved into: for any $u\in W$ (possibly modified on a set of zero measure) and almost every line $\ell \subset \R^n$, $u_{|\ell}$ is absolutely continuous. This property will be referred 
to as (ACL).
\end{remark}

\begin{theorem}\label{tTr} There exists a bounded linear operator
$T : W \to H$ (a trace operator) with the following properties. The trace of $u \in W$ is such that, 
for $\sigma$-almost every $x\in \Gamma$, 
\begin{equation}\label{defT}
Tu(x) = \lim_{r \to 0} \fint_{B(x,r)} u(y) dy :=  \lim_{r \to 0} {1 \over |B(x,r)|} \int u(y) dy 
\end{equation}
and, analogously to the Lebesgue density property,
\begin{equation}\label{eqLeb}
\lim_{r \to 0} \fint_{B(x,r)} |u(y)-Tu(x)| dy = 0 .
\end{equation}
\end{theorem}

\bp 
First, we want bounds on $\nabla u$ near $x\in \Gamma$, so we set
\begin{equation} \label{2.5}
M_r(x) = \fint_{B(x,r)} |\nabla u|^2 \, dz
\end{equation}
and estimate $\int_\Gamma M_r(x) d\sigma(x)$.
We cover $\Gamma$ by balls $B_j = B(x_j,r)$ centered on $\Gamma$
such that the $2B_j = B(x_j,2r)$ have bounded overlap (we could even make the
$B(x_j,r/5)$ disjoint), and notice that for $x\in B_j$,
\begin{equation} \label{2.6}
M_r(x) \leq C r^{-n} \int_{2B_j} |\nabla u|^2 \, dz.
\end{equation}
We sum and get that
\begin{eqnarray} \label{2.7}
\int_\Gamma M_r(x) d\sigma(x) &\leq& \sum_j \int_{B_j} M_r(x) d\sigma(x)
\leq C \sum_j \sigma(B_j) \sup_{x\in B_j} M_r(x)
\nn\\
&\leq& C \sum_j  \sigma(B_j) r^{-n} \int_{2B_j} |\nabla u|^2 \, dz 
\leq C r^{d-n} \sum_j \int_{2B_j} |\nabla u|^2 \, dz \nn\\
& \leq & C r^{d-n} \int_{\Gamma(2r)}  |\nabla u|^2 \, dz
\end{eqnarray}
because the $2B_j$ have bounded overlap and where $\Gamma(2r)$ denotes a 
$2r$-neighborhood of $\Gamma$.
Next set
\begin{equation} \label{2.8}
N(x) = \sum_{k \geq 0} 2^{-k} M_{2^{-k}}(x);
\end{equation}
then
\begin{eqnarray} \label{2.9}
\int_\Gamma N(x) d\sigma(x)
&=& \sum_{k \geq 0} 2^{-k} \int_\Gamma M_{2^{-k}}(x) d\sigma(x)
\leq C \sum_{k \geq 0} 2^{k(n-d-1)} \int_{\Gamma(2^{-k+1})} |\nabla u(z)|^2 dz
\nn\\
&\leq& C \int_{\Gamma(2)}  |\nabla u(z)|^2 a(z) dz,
\end{eqnarray}
where $a(z) = \sum_{k \geq 0} 2^{k(n-d-1)} \1_{z\in \Gamma(2^{-k+1})}$.
For a given $z \in \Omega$, $z\in \Gamma(2^{-k+1})$ only for $k$ so small that
$\delta(z) \leq 2^{-k+1}$.
The largest values of $2^{k(n-d-1)}$ are for $k$ as large as possible, 
when $2^{-k} \approx \delta(z)$; thus $a(z) \leq C\delta(z)^{-n+d+1}= Cw(z)$, and
\begin{equation} \label{2.10}
\int_\Gamma N(x) d\sigma(x) 
\leq C \int_{\Gamma(2)}  |\nabla u(z)|^2 w(z) dz.
\end{equation}

Our trace function $g=Tu$ will be defined as the limit of the functions $g_r$, where 
\begin{equation} \label{2.11}
g_r(x) = \fint_{z\in B(x,r)} u(z)dz.
\end{equation}
Our aim is to use the estimates established in the proof of Lemma~\ref{W=barW}.
Notice that for $x\in \Gamma$ and $r > 0$,
\begin{equation} \label{2.12} \begin{split}
\fint_{z\in B(x,r)} |u(z)-g_r(x)| dz 
& = \fint_{z\in B(x,r)} \Big|u(z)- \fint_{\xi \in B(x,r)}u(y) dy \Big| dz \\
& \leq \fint_{z\in B(x,r)} \fint_{y\in B(x,r)} \Big|u(z)- u(y) \Big|dy \,  dz. \\
\end{split} \end{equation}
By \eqref{2.4a},
\begin{eqnarray} \label{2.14}
\fint_{z\in B(x,r)} |u(z)-g_r(x)| dz
& \leq &  \fint_{z\in B(x,r)} \fint_{y\in B(x,r)} \Big|u(z)- u(y) \Big|dy \,  dz \nn \\
&\leq& C r^{-n+1} \int_{\xi \in B(x,r)} |\nabla u(\xi)| d\xi.
\end{eqnarray}
Thus for $r/10 \leq s \leq r$,
\begin{eqnarray} \label{2.15}
|g_s(x)-g_r(x)|  &=& \Big|\fint_{z\in B(x,s)} u(z) dz -g_r(x) \Big|
\leq \fint_{z\in B(x,s)} |u(z)-g_r(x)| dz 
\nn\\&\leq& C r \fint_{\xi \in B(x,r)} |\nabla u(\xi)| d\xi \leq C r M_{r}(x)^{1/2}.
\end{eqnarray}
Set $\Delta_r(x) = \sup_{r/10 \leq s \leq r} |g_s(x)-g_r(x)|$; we just proved that
$\Delta_r(x) \leq C r M_{r}(x)^{1/2}$. Let $\alpha \in (0,1/2)$ be given.
If $N(x) < +\infty$, we get that
\begin{eqnarray} \label{2.16}
\sum_{k \geq 0} 2^{\alpha k} \Delta_{2^{-k}}(x)
&\leq& C \sum_{k \geq 0} 2^{\alpha k} 2^{-k} M_{2^{-k}}(x)^{1/2}
\nn\\
&\leq& C \Big\{\sum_{k \geq 0} 2^{-k} M_{2^{-k}}(x) \Big\}^{1/2}
\Big\{\sum_{k \geq 0} 2^{2\alpha k} 2^{-k}  \Big\}^{1/2} \leq C N(x)^{1/2} < +\infty.
\end{eqnarray}
Therefore, $\sum_{k \geq 0} \Delta_{2^{-k-2}}(x)$ converges (rather fast), and since \eqref{2.10} implies that $N(x)< +\infty$ for $\sigma$-almost every $x\in \Gamma$, it follows that
there exists
\begin{equation} \label{2.17}
g(x) = \lim_{r \to 0} g_r(x)  \ \text{ for $\sigma$-almost every } x\in \Gamma.
\end{equation}
In addition, we may integrate (the proof of) \eqref{2.16} and get that for $2^{-j-1} < r \leq 2^{-j}$,
\begin{eqnarray} \label{2.18}
\|g - g_r\|_{L^2(\sigma)}^2 
&=& \int_\Gamma |g(x) - g_r(x)|^2 d\sigma(x)
\leq \int_\Gamma \Big\{ \sum_{k \geq j} \Delta_{2^{-k}}(x) \Big\}^2 d\sigma(x)
\nn\\
&\leq& C 2^{-2\alpha j} \int_\Gamma \Big\{ \sum_{k \geq j} 
2^{\alpha j}\Delta_{2^{-k}}(x) \Big\}^2 d\sigma(x)
\nn\\
&\leq& C r^{2\alpha} \int_\Gamma N(x) d\sigma(x) \leq C r^{2\alpha} \|u\|_W^2
\end{eqnarray}
by \eqref{2.17} and the definition of $\Delta_r(x)$, then \eqref{2.16} and \eqref{2.10}.
Thus $g_r$ converges also (rather fast) to $g$ in $L^2$. Let us make an additional remark. 
Fix $r>0$ and $\alpha \in (0,1/2)$. For any ball $B$ centered on $\Gamma$,
\begin{equation} \label{2.18a}
\|g\|_{L^1(B,\sigma)} \leq C_B \|g-g_r\|_{L^2(\sigma)} + \|g_r\|_{L^1(B,\sigma)}
\end{equation}
by H\"older's inequality. 
The first term is bounded with \eqref{2.18}. Use \eqref{1.1} and Fubini's theorem
to bound the second one by $C_r \|g\|_{L^1(\wt B)}$, where $\wt B$ is a large ball 
(that depends on $r$ and contains $B$).
As a consequence,
\begin{equation} \label{L1locsigma}
\text{ for any $u\in W$, $g=T u \in L^1_{loc}(\Gamma,\sigma)$.}
\end{equation}

This completes the
definition of the trace $g = T(u)$.
We announced (as a Lebesgue property) that
\begin{equation} \label{2.19}
\lim_{r \to 0} \fint_{B(x,r)} |u(y)-Tu(x)| dy
\end{equation}
for $\sigma$-almost every $x\in \Gamma$, and indeed
\begin{eqnarray} \label{2.20}
\fint_{B(x,r)} |u(y)-Tu(x)| dy 
&=& \fint_{B(x,r)} |u(y)-g(x)| dy \leq |g(x)-g_r(x)| +  \fint_{B(x,r)} |u(y)-g_r(x)|\,dy
\nn\\
&\leq& |g(x)-g_r(x)| +  C r \fint_{B(x,3r)} |\nabla u|\, dy
\leq |g(x)-g_r(x)| +  C r M_{4r}(x)^{1/2}
\end{eqnarray}
by \eqref{2.14} and the second part of \eqref{2.15}. The first part tends to $0$ for 
$\sigma$-almost every $x\in \Gamma$, by \eqref{2.17}, and the second part tends to $0$
as well, because $N(x) <+\infty$ almost everywhere and by the definition \eqref{2.8}.

Next we show that $g=Tu$ lies in the Sobolev space $H = H^{1/2}(\Gamma)$, i.e., that
\begin{equation} \label{2.21}
\|g\|_H^2 = \int_\Gamma\int_\Gamma {|g(x)-g(y)|^2 \over |x-y|^{d+1}} d\sigma(x) d\sigma(y)
< +\infty.
\end{equation}
The simplest will be to prove uniform estimates on the $g_r$, and then go to the limit. Let us
fix $r>0$ and consider the integral
\begin{equation} \label{2.22}
I(r) = \int_{x\in\Gamma}\int_{y\in\Gamma ; |y-x| \geq r} 
{|g_r(x)-g_r(y)|^2 \over |x-y|^{d+1}} d\sigma(x) d\sigma(y).
\end{equation}
Set $Z_k(r) = \big\{ (x,y)\in \Gamma \times \Gamma \, ; \, 2^k r \leq |y-x| < 2^{k+1}r \big\}$
and $I_k(r) = \int\int_{Z_k(r)} {|g_r(x)-g_r(y)|^2 \over |x-y|^{d+1}} d\sigma(x) d\sigma(y)$.
Thus $I(r) = \sum_{k\geq 0} I_k(r)$ and 
\begin{equation} \label{2.23}
I_k(r) \leq (2^kr)^{-d-1} \int\int_{Z_k(r)} |g_r(x)-g_r(y)|^2 d\sigma(x) d\sigma(y).
\end{equation}
Fix $k \geq 0$, set $\rho = 2^{k+1}r$, and observe that for $(x,y) \in Z_k(r)$,
\begin{eqnarray} \label{2.24}
|g_r(x)-g_\rho(y)| &=& \Big|\fint_{z \in B(x,r)}\fint_{\xi \in B(y,\rho)} [u(z)-u(\xi)] dzd\xi \Big|
\leq \fint_{z \in B(x,r)} 
\fint_{\xi \in B(y,\rho)} |u(z)-u(\xi)| d\xi
dz
\nn\\
&\leq&3^n \fint_{z \in B(x,r)} \Big\{ \fint_{\xi \in B(z,3\rho)} |u(z)-u(\xi)| d\xi\Big\}dz
\nn\\
&\leq& C \rho^n \fint_{z \in B(x,r)} 
\fint_{\zeta \in B(z,3\rho)} |\nabla u(\zeta)| |z-\zeta|^{1-n}d\zeta dz
\end{eqnarray}
because $B(y,\rho) \subset B(z,3\rho)$ and by \eqref{2.4}. We apply Cauchy-Schwarz,
with an extra bit $|z-\zeta|^{-\alpha}$, where $\alpha >0$ will be taken small,
and which will be useful for convergence later
\begin{eqnarray} \label{2.25}
|g_r(x)-g_\rho(y)|^2 
&\leq& C\rho^{2n} \Big\{\fint_{z \in B(x,r)} 
\fint_{\zeta \in B(z,3\rho)} |\nabla u(\zeta)|^2 |z-\zeta|^{1-n+\alpha} \Big\}
\nn\\&\,& \hskip6cm
\Big\{ \fint_{z \in B(x,r)} \fint_{\zeta \in B(z,3\rho)} |z-\zeta|^{1-n-\alpha} \Big\}
\nn\\
&\leq& C \rho^{n+1-\alpha} \fint_{z \in B(x,r)} 
\fint_{\zeta \in B(z,3\rho)} |\nabla u(\zeta)|^2 |z-\zeta|^{1-n+\alpha} d\zeta dz.
\end{eqnarray} 
The same computation, with $g_r(y)$, yields
\begin{equation} \label{2.26}
|g_r(y)-g_\rho(y)|^2 
\leq C \rho^{n+1-\alpha} 
\fint_{z \in B(y,r)} \fint_{\zeta \in B(z,3\rho)} |\nabla u(\zeta)|^2 |z-\zeta|^{1-n+\alpha}
d\zeta dz.
\end{equation}
We add the two and get an estimate for $|g_r(x)-g_r(y)|^2$, which we can integrate
to get that
\begin{eqnarray} \label{2.27}
I_k(r) &\leq& C \rho^{-d-1}\rho^{n+1-\alpha} \int\int_{(x,y) \in Z_k(r)} 
\fint_{z \in B(x,r)} \fint_{\zeta \in B(z,3\rho)} |\nabla u(\zeta)|^2 |z-\zeta|^{1-n+\alpha}
d\zeta dz d\sigma(x)d\sigma(y)
\nn\\
&\leq& C \rho^{-d-\alpha} r^{-n} \int\int_{(x,y) \in Z_k(r)} 
\int_{z \in B(x,r)} \int_{\zeta \in B(z,3\rho)} |\nabla u(\zeta)|^2 |z-\zeta|^{1-n+\alpha}
d\zeta dz d\sigma(x)d\sigma(y)
\end{eqnarray}
by \eqref{2.23}, \eqref{2.25}, and \eqref{2.26}, and where we can drop the part that comes
from \eqref{2.26} by symmetry. We integrate in $y \in \Gamma$ such that
$2^k r \leq |x-y| \leq 2^{k+1}r$ and get that 
\begin{eqnarray} \label{2.28}
I_k(r) &\leq& C \rho^{-\alpha} r^{-n} \int_{x \in \Gamma} 
\int_{z \in B(x,r)} \int_{\zeta \in B(z,3\rho)} |\nabla u(\zeta)|^2 |z-\zeta|^{1-n+\alpha}
d\zeta dz d\sigma(x)
\nn\\
&\leq& C \int_{\zeta\in \Omega} |\nabla u(\zeta)|^2 h_k(\zeta) d\zeta,
\end{eqnarray}
with 
\begin{equation} \label{2.29}
h_k(\zeta) = \rho^{-\alpha} r^{-n} \int_{x \in \Gamma}\int_{z \in B(x,r) \cap B(\zeta,3\rho)} 
|z-\zeta|^{1-n+\alpha} dzd\sigma(x).
\end{equation}

We start with the contribution $h_k^0(\zeta)$ of the region where $|x-\zeta| \geq 2r$, 
where the computation is simpler because $|z-\zeta| \geq {1\over 2} |x-\zeta|$ there.
We get that  
\begin{eqnarray} \label{2.30}
h_k^0(\zeta) &\leq& C \rho^{-\alpha} r^{-n} 
\int_{x \in \Gamma}\int_{z \in B(x,r) \cap B(\zeta,3\rho)} 
|x-\zeta|^{1-n+\alpha} dzd\sigma(x)
\nn\\
&\leq& C \rho^{-\alpha} \int_{x \in \Gamma \cap B(\zeta,4\rho)} |x-\zeta|^{1-n+\alpha} d\sigma(x).
\end{eqnarray}
With $\zeta$, $r$, and $\rho$ fixed, $h_k^0(\zeta)$ vanishes unless $\delta(\zeta)
= \dist(\zeta,\Gamma) < 4\rho$. The region where $|x-\zeta|$ is of the order of $2^m\delta(\zeta)$,
$m \geq 0$, contributes less than $C (2^m\delta(\zeta))^{d+1-n+\alpha}$ to the integral
(because $\sigma$ is Ahlfors-regular). If $\alpha$ is chosen small enough, the exponent is still negative,
the largest contribution comes from $m=0$, and 
$h_k^0(\zeta) \leq C \rho^{-\alpha} \delta(\zeta)^{d+1-n+\alpha}$.
Recall that $\rho = 2^{k+1} r$, and $k$ is such that $\delta(\zeta) < 4\rho$; 
we sum over $k$ and get that
\begin{equation} \label{2.31}
\sum_k h_k^0(\zeta) 
\leq C \sum_{k \geq 0 \, ; \, \delta(\zeta) < 4\rho} \rho^{-\alpha} \delta(\zeta)^{d+1-n+\alpha}
\leq C \delta(\zeta)^{d+1-n},
\end{equation}
because this time the smallest values of $\rho$ give the largest contributions.
We are left with 
\begin{equation} \label{2.32}
h_k^1(\zeta)=h_k(\zeta)-h_k^0(\zeta)
= \rho^{-\alpha} r^{-n} \int_{x \in \Gamma \cap B(\zeta,2r)}\int_{z \in B(x,r) \cap B(\zeta,3\rho)} 
|z-\zeta|^{1-n+\alpha} dzd\sigma(x).
\end{equation}
Notice that $|z-\zeta| \leq |z-x| + |x-\zeta| \leq 3r$; 
we use the local Ahlfors-regularity to get rid of the integral on $\Gamma$, and 
get that
\begin{equation} \label{2.33}
h_k^1(\zeta) \leq C \rho^{-\alpha} r^{-n} r^{d}
\int_{z \in B(\zeta,3r)} |z-\zeta|^{1-n+\alpha} dz
\leq C \rho^{-\alpha} r^{d+1-n+\alpha}.
\end{equation}
We sum over $k$ and get that 
$\sum_k h_k^1
(\zeta) \leq C r^{d+1-n} \leq C \delta(\zeta)^{d+1-n}$, because
if $\delta(\zeta) \geq 2r$, we simply get that $h_k^1(\zeta) = 0$ for all $k$, 
because $\Gamma \cap B(\zeta,2r) = \emptyset$ and by \eqref{2.32}.
Altogether, $\sum_k h_k(\zeta) \leq C  \delta(\zeta)^{d+1-n}$, and 
\begin{equation} \label{2.34}
I(r) = \sum_{k} I_k(r) \leq C \sum_{k}\int_{\zeta\in \Omega} |\nabla u(\zeta)|^2  h_k(\zeta) d\zeta
\leq C \int_{\zeta\in \Omega} |\nabla u(\zeta)|^2 \delta(\zeta)^{d+1-n} d\zeta
= C \|u\|_W^2
\end{equation}
by definition of the $I_k(r)$, then \eqref{2.28} and the definition of $W$.
We may now look at the definition \eqref{2.22} of $I(r)$, let $r$ tend to $0$, and get that
\begin{equation} \label{2.35}
\|g\|_H \leq C \|u\|_W^2
\end{equation}
by Fatou's lemma, as needed for the trace theorem.
\ep 
\chapter{Poincar\'e Inequalities}
\label{SPoincare}

\begin{lemma}\label{lpBry} Let $\Gamma$ be a $d$-ADR set in $\RR^n$, $d<n-1$, that is, assume that \eqref{1.1} is satisfied. Then 
\begin{equation} \label{1.4-bis}
\fint_{B(x,r)} |u(y)| dy \leq C r^{-d} \int_{B(x,r)} |\nabla u(y)| w(y) dy
\end{equation} 
for $u \in W$, $x\in \Gamma$, and $r>0$ such that $Tu = 0$ on $\Gamma \cap B(x,r)$.
\end{lemma}

\bp To simplify the notation we assume that $x=0$.

We should of course observe
that the right-hand side of \eqref{1.4-bis} is finite. Indeed, recall that Lemma~\ref{lwest} gives
\begin{equation} \label{2.45}
\int_{\xi \in B(0,r)} w(\xi) d\xi \leq C r^{1+d};
\end{equation}
then by Cauchy-Schwarz
\begin{eqnarray} \label{2.46}
 r^{-d} \int_{\xi \in B(0,r)} |\nabla u(\xi)| w(\xi) d\xi
 &\leq&  r^{-d} \Big\{\int_{\xi \in B(0,r)} |\nabla u(\xi)|^2 w(\xi) \, d\xi \Big\}^{1/2}
 \Big\{\int_{\xi \in B(0,r)} w(\xi)\, d\xi \Big\}^{1/2}
 \nn\\
&\leq&  r^{1-d \over 2} 
\Big\{\int_{\xi \in B(0,r)} |\nabla u(\xi)|^2 w(\xi)\, d\xi \Big\}^{1/2}.
\end{eqnarray}
The homogeneity still looks a little weird because of the weight (but things become
simpler if we think that $\delta(\xi)$ is of the order of $r$), but at least the right-hand side is finite
because $u\in W$.

Turning to the proof of \eqref{1.4-bis}, to avoid complications with the fact that \eqref{2.2} and \eqref{2.3} do not necessarily
hold $\sigma$-almost everywhere on $\Gamma$, let us use the $g_s$ again.
We first prove that for $s < r$ small, 
\begin{equation} \label{2.36}
\fint_{y\in B(0,r)}\fint_{x\in \Gamma \cap B(0,r/2)} 
|u(y)-g_s(x)| d\sigma(x)\,dy \leq C r^{-d} \fint_{B(0,r)} |\nabla u(\xi)| \delta(\xi)^{1+d-n} d\xi.
\end{equation}
Denote by $I(s)$ the left-hand side. By \eqref{2.11},
\begin{equation} \label{2.37}
I(s) \leq \fint_{y\in B(0,r)}\fint_{x\in \Gamma \cap B(0,r/2)} \fint_{z\in B(x,s)}
|u(y)-u(z)| dz\,d\sigma(x)\,dy.
\end{equation}
For $x$ fixed, we can still prove as in \eqref{2.4} that
\begin{equation} \label{2.39}
\fint_{y\in B(0,r)} |u(y)-u(z)| dy \leq  C\int_{B(0,r)} |\nabla u(\xi)| |z-\xi|^{1-n} d\xi
\end{equation}
(for $x\in \Gamma \cap B(0,r/2)$ and $z\in B(x,s)$, there is even a bilipschitz change
of variable that sends $z$ to $0$ and maps $B(0,r)$ to itself). We are left with
\begin{equation} \label{2.40}
I(s) \leq C \fint_{x\in \Gamma \cap B(0,r/2)} \fint_{z\in B(x,s)}
\int_{\xi \in B(0,r)} |\nabla u(\xi)| |z-\xi|^{1-n} d\xi dz d\sigma(x).
\end{equation}
The main piece of the integral will again be called $I_0(s)$, where we integrate in the region
where $|\xi-x| \geq 2s$ and hence $|z-\xi|^{1-n} \leq 2^n |x-\xi|^{1-n}$. Thus
\begin{eqnarray} \label{2.41}
I_0(s) &\leq& C \int_{\xi \in B(0,r)}
\fint_{x\in \Gamma \cap B(0,r/2)} \fint_{z\in B(x,s)}
 |\nabla u(\xi)| |x-\xi|^{1-n} dz d\sigma(x) d\xi 
 \nn\\ 
 &\leq& C r^{-d} \int_{\xi \in B(0,r)} \int_{x\in \Gamma \cap B(0,r/2)} 
 |\nabla u(\xi)| |x-\xi|^{1-n} d\sigma(x) d\xi 
 \nn\\  &\leq& 
 C r^{-d} \int_{\xi \in B(0,r) \sm \Gamma} |\nabla u(\xi)| h(\xi) d\xi,
\end{eqnarray}
where for $\xi \in B(0,r) \sm \Gamma$ we set
\begin{equation} \label{2.42}
h(\xi) = \int_{x\in \Gamma \cap B(0,r/2)} |x-\xi|^{1-n} d\sigma(x)
\leq C \delta(\xi)^{1-n+d}
\end{equation}
where for the last inequality we cut the domain of integration into pieces where
$|x-\xi| \approx 2^m \delta(\xi)$ and use \eqref{1.1}. For the other piece of \eqref{2.40}
where $|\xi-x| < 2s$, we get the integral
\begin{eqnarray} \label{2.43}
I_1(s) &\leq& C r^{-d}s^{-n} \int_{\xi \in B(0,r)} \int_{x\in \Gamma \cap B(0,r/2) \cap B(\xi,2s)} 
\int_{z\in B(x,s)} |\nabla u(\xi)| |z-\xi|^{1-n} d\xi dz d\sigma(x)
\nn\\
&\leq& C  r^{-d}s^{d-n} \int_{\xi \in B(0,r) ; \delta(\xi) \leq 2s} 
\int_{z\in B(\xi,3s)}|\nabla u(\xi)| |z-\xi|^{1-n} d\xi dz 
\nn\\
&\leq& C  r^{-d}s^{1+d-n} \int_{\xi \in B(0,r) ; \delta(\xi) \leq 2s} 
|\nabla u(\xi)| d\xi 
\leq C r^{-d} \int_{\xi \in B(0,r)} |\nabla u(\xi)| \delta(\xi)^{1+d-n} d\xi .
\end{eqnarray}
Altogether
\begin{equation} \label{2.44}
I(s) \leq C r^{-d} \int_{\xi \in B(0,r)} |\nabla u(\xi)| \delta(\xi)^{1+d-n} d\xi ,
\end{equation}
which is \eqref{2.36}. When $s$ tends to $0$, $g_s(x)$ tends to 
$g(x) = Tu(x)=0$ for $\sigma$-almost every $x \in \Gamma \cap B(0,r/2)$, 
and we get \eqref{1.4-bis}
by Fatou. \ep

\begin{lemma}\label{lSob} Let $\Gamma$ be a $d$-ADR set in $\RR^n$, $d<n-1$, that is, 
assume that \eqref{1.1} is satisfied. 
Let $p\in \left[1,\frac{2n}{n-2}\right]$ (or $ p\in [1,+\infty)$ if $n=2$). 
Then for any $u\in W$, $x\in \R^n$ and $r>0$
\begin{equation} \label{1.5-bis}
\Big\{\frac1{m(B(x,r))}\int_{B(x,r)} \left|u(y) - u_{B(x,r)} \right|^p w(y) dy\Big\}^{1/p}
\leq C r \Big\{ \frac{1}{m(B(x,r))}\int_{B(x,r)} |\nabla u(y)|^2 w(y) dy \Big\}^{1/2},
\end{equation}
where $u_B$ denotes either $\fint_B u\, dz$ or $m(B)^{-1} \int_B u \, w(z)dz$. If $x\in \Gamma$ and, in addition, $Tu=0$ on $\Gamma\cap B(x, r)$ then 
\begin{equation} \label{1.5-bisbis}
\Big\{r^{-d-1}\int_{B(x,r)} |u(y)|^p w(y) dy\Big\}^{1/p}
\leq C r \Big\{ r^{-d-1} \int_{B(x,r)} |\nabla u(y)|^2 w(y) dy \Big\}^{1/2}.
\end{equation}
\end{lemma}

\bp In the proof, we will use $dm(z)$ for $w(z) dz$ and hence, for instance $\int_B u \, dm$ denotes $\int_B u(z) w(z) dz$.
We start with the following inequality.
Let $p\in [1,+\infty)$. If $u\in L^p_{loc}(\R^n,w) \subset L^1_{loc}(\R^n)$, then for any ball 
$B$, 
\begin{equation} \label{equivuB}
\int_B \left| u - \fint_B u \, dy\right|^p dm \approx  \int_B \left| u - \frac1{m(B)}\int_B u \, dm \right|^p dm.
\end{equation}
First we bound the left-hand side. We introduce  $m(B)^{-1} \int_B u \, dm$ inside
the absolute values and then use the triangle inequality:
\begin{eqnarray} \label{ISobA} 
\int_B \left| u - \fint_B u\, dy \right|^p dm 
&\leq & C  \int_B \left| u(z) - \frac1{m(B)}\int_B u \, dm \right|^p dm + C  m(B) \left| \fint_B u - \frac1{m(B)}\int_B u \, dm \right|^p \nn\\
& \leq & C\int_B \left| u(z) - \frac1{m(B)}\int_B u \, dm \right|^p dm +  C \frac{m(B)}{|B|} \int_B \left|u - \frac1{m(B)}\int_B u \, dm \right|^p \nn\\
& \leq & C \int_B \left| u(z) - \frac1{m(B)}\int_B u \, dm \right|^p dm, 
\end{eqnarray}
where the last line is due to \eqref{L1byL1w}. The reverse estimate is quite immediate
\begin{equation} \label{ISobB}  \begin{split} 
& \int_B \left| u -  \frac1{m(B)}\int_B u \, dm \right|^p dm \\
&\qquad \leq C  \int_B \left| u(z) - \fint_B u \, dy\right|^p dm +  C m(B) \left| \frac1{m(B)}\int_B u \, dm 
- \fint_B u \, dy\right|^p \\
&\qquad \leq C  \int_B \left| u(z) - \fint_B u \, dy \right|^p dm +  C m(B) \left| \frac1{m(B)}\int_B \left( u  - \fint_B u \, dy \right) \, dm\right|^p \\
& \qquad \leq C \int_B \left| u(z) - \fint_B u \, dy \right|^p dm,
\end{split} \end{equation}
which finishes the proof of \eqref{equivuB}. 

In the sequel of the proof, we write $u_B$ for $m(B)^{-1} \int_B u \, dm$. 
Thanks to \eqref{equivuB}, it suffices to prove \eqref{1.5-bis} only for this particular choice of $u_B$.
We now want
to prove a (1,1) Poincar\'e inequality, that is
\begin{equation} \label{11Poincare}
\int_{B} \left|u(z) - u_B\right| dm \leq Cr \int_{B} |\nabla u(z)| dm.
\end{equation}
for any $u\in W$ and any ball $B\subset \R^n$ of radius $r$. In particular, $u\in L^1_{loc}(\R^n,w)$.

Let $B \subset \R^n$ of radius $r$. Recall first that thanks to Lemma~\ref{W=barW}, $\fint_B u$ makes sense for every ball $B$. If we prove for $u\in W$ the estimate
\begin{equation} \label{11Poincareb}
\int_{B} \left|u - \fint_B u \, dy \right| dm \leq Cr \int_{B} |\nabla u| dm,
\end{equation}
then \eqref{11Poincare} will follows. Indeed, assume \eqref{11Poincareb} holds for any ball $B$. 
The left-hand
of \eqref{11Poincareb} is then bounded, up to a constant, by $r\|u\|_W$ and is thus finite. 
Therefore,  
for any ball $B$, $\int_B |u| dm \leq \int_B |u-\fint_B u| dm + |\fint_B u\, dy| < +\infty$, i.e. $u \in L^1_{loc}(\R^n,w)$. 
But now, $u \in L^1_{loc}(\R^n,w)$, so we can use \eqref{equivuB}. 
Together with \eqref{11Poincareb}, it implies 
\eqref{11Poincare}.

We want to prove \eqref{11Poincareb}. The estimate \eqref{2.4} yields
\begin{equation} \label{ISob1} \begin{split}
 \int_{B} \left|u - \fint_{B} u \, dy\right| dm & \leq C \int_{B} \int_{B} |\nabla u (\xi)| |z-\xi|^{1-n} w(z) d\xi \, dz \\
& \leq C \int_{B} |\nabla u (\xi)| \int_{B(\xi,2r)} |z-\xi|^{1-n} w(z) dz\, d\xi
\end{split} \end{equation}
and thus it remains to check that for 
$\xi \in \R^n$ and  
$r>0$,
\begin{equation} \label{ISob2}
I = \int_{B(\xi,2r)} |z-\xi|^{1-n} w(z) dz \leq Crw(\xi).
\end{equation}
First, note that if $\delta(\xi)\geq 4r$, then for all $z\in B(\xi,2r)$, one has 
\[\frac12 \delta(\xi) \leq \delta(\xi) - 2r \leq \delta(z) \leq \delta(\xi)+2r \leq 2\delta(\xi)\]
and so $w(z)$ is equivalent to $w(\xi)$. 
Thus $I \leq C w(\xi) \int_{B(\xi,2r)} |z-\xi|^{1-n} dz \leq C r w(\xi)$. 
It remains to prove the case $\delta(\xi) < 4r$. We split $I$ into $I_1+I_2$ where, for $I_1$, the domain of integration is restrained to $B(\xi,\delta(\xi)/2)$. For any $z \in B(\xi,\delta(\xi)/2)$, we have $w(z) \leq C w(\xi)$ and thus
\begin{equation} \label{ISob3}
I_1 \leq C w(\xi) \int_{B(\xi,\delta(\xi)/2)}  |z-\xi|^{1-n} dz \leq C w(\xi) \delta(\xi) \leq C r w(\xi).
\end{equation}
It remains to bound $I_2$. In order to do it, we decompose the remaining domain into annuli $C_j(\xi):=\{z\in \R^n, \, 2^{j-1}\delta(\xi) \leq |\xi-z| \leq 2^j \delta(\xi)\}$. We write $\kappa$ for the smallest integer bigger than $\log_2(2r/\delta(\xi))$, which is the highest value for which $C_\kappa \cap B(\xi,2r)$ is non-empty. We have
\begin{equation} \label{ISob4} \begin{split}
I_2 \leq C \sum_{j=0}^{\kappa} 
2^{j(1-n)} \delta(\xi)^{1-n} \int_{C_j(\xi)} dm(z) 
\leq C  \sum_{j=0}^{\kappa}
2^{j(1-n)} \delta(\xi)^{1-n} m(B(\xi,2^j\delta(\xi))).
\end{split} \end{equation}
The ball $B(\xi,2^j\delta(\xi))$ is close to $\Gamma$ and thus Lemma~\ref{lwest} gives that the quantity $m(B(\xi,2^j\delta(\xi)))$ is bounded, up to a constant, by $2^{j(d+1)} \delta(x)^{d+1}$. We deduce, since $2+d-n \leq 1$, that
\begin{equation} \label{ISob5} \begin{split}
I_2 & \leq C \sum_{j=0}^{\kappa}  
2^{j(2+d-n)} \delta(\xi)^{2+d-n}  \leq C \delta(\xi)^{2+d-n} 
\sum_{j=0}^{\kappa}
2^{j} \\ 
& \leq C \delta(\xi)^{2+d-n} \left( \frac{2r}{\delta(\xi)} \right) \leq C r \delta(\xi)^{1+d-n} = Crw(\xi),
\end{split} \end{equation}
which ends the proof of \eqref{ISob2} and thus also the one of the Poincar\'e inequality \eqref{11Poincare}.

\medskip

Now we want
to establish \eqref{1.5-bis}. The quickest way to do it is to use some results of 
Haj{\l}asz
and Koskela. We say that $(u,g)$ forms a Poincar\'e pair if $u$ is in $L^1_{loc}(\R^n,w)$, $g$ is positive and measurable and for any ball $B\subset \R^n$ of radius $r$, we have
\begin{equation} \label{HK1} 
m(B)^{-1}\int_{B} |u(z) - u_B| dm(z) \leq Cr m(B)^{-1} \int_{B} g \,  dm(z).
\end{equation}

In this context, Theorem 5.1 (and Corollary 9.8) in \cite{HK00} states that the 
Poincar\'e inequality \eqref{HK1} can be improved into a Sobolev-Poincar\'e inequality. More precisely, if $s$ is such that, for any ball $B_0$ of radius $r_0$, any $x\in B_0$ and any $r\leq  r_0$,
\begin{equation} \label{doublings} 
\frac{m(B(x,r))}{m(B_0)} \geq C^{-1} \left( \frac{r}{r_0}\right)^s
\end{equation}
then \eqref{HK1} implies for any $1<q<s$
\begin{equation} \label{HK2} 
\left( m(B)^{-1}\int_{B} |u(z) - u_B|^{q^*} dm(z) \right)^\frac1{q^*} \leq Cr \left( m(B)^{-1} \int_{B} g^q dm(z) \right)^\frac1q
\end{equation}
where $q^*= \frac{qs}{s-q}$ and $B$ is a ball of radius $r$. Combined with H\"older's inequality, we get
\begin{equation} \label{HK3} 
\left( m(B)^{-1}\int_{B} |u(z) - u_B|^{p} dm(z) \right)^\frac1{p} \leq Cr \left( m(B)^{-1} \int_{B} g^2 dm(z) \right)^\frac12
\end{equation}
for any $p\in [1,2s/(s-2)]$ if $s> 2$ or any $p<+\infty$ if $s\leq 2$.

We will use the result of Haj{\l}asz
and Koskela with $g= |\nabla u|$. We need to check the assumptions of their result. The bound \eqref{HK1} is exactly \eqref{11Poincare} and we already proved it. 
The second and last thing we need to verify is that \eqref{doublings} holds with $s=n$. This fact is an easy consequence of \eqref{doublinggen}. 
Indeed, if $B_0$ is a ball of radius $r_0$, $x\in B_0$ and $r\leq r_0$
\begin{equation} \label{ISobC}
\frac{m(B(x,r))}{m(B_0)} \geq \frac{m(B(x,r))}{m(B(x,2r_0))}.
\end{equation}
Yet, \eqref{doublinggen} implies that $\frac{m(B(x,r))}{m(B(x,2r_0))}$ is bounded from below by $C^{-1}(\frac r{2r_0})^n$, that is $C^{-1}(\frac r{r_0})^n$. Then
\begin{equation} \label{doublingn} 
\frac{m(B(x,r))}{m(B_0)} \geq C^{-1} \left( \frac{r}{r_0}\right)^n,
\end{equation}
which is the desired conclusion. We deduce that \eqref{HK3} holds with $g=|\nabla u|$ and for any $p\in [1,\frac{2n}{n-2}]$ ($1\leq p< +\infty$ if $n=2$), which is exactly \eqref{1.5-bis}.

\medskip

To finish to prove the lemma, it remains to establish \eqref{1.5-bisbis}. Let $B=B(x,r)$ be a ball centered on $\Gamma$.
However, since $x\in \Gamma$, \eqref{ISob8b} entails that $m(B)$ is equivalent to $r^{d+1}$. Thus, thanks to \eqref{1.5-bis} and Lemma~\ref{lpBry},
\begin{eqnarray} \label{ISob12}
\left( r^{-d-1} \int_{B} |u(z)|^p dm(z) \right)^\frac1p 
&\leq& C \left( m(B)^{-1}\int_{B} \Big|u(z) - \fint_{B} u \, dy\Big|^p dm(z) \right)^\frac1p 
+ C \fint_{B} |u(z)| dz  \nn\\
&\leq& Cr \left( r^{-d-1} \int_{B} |\nabla u|^2 dm(z) \right)^\frac12,
\end{eqnarray}
which proves Lemma~\ref{lSob}.
\ep

\begin{remark} \label{rPoincare}
If $B \subset \R^n$ and $u \in W$ \ub{is supported in $B$}, then for any $p\in [1,2n/(n-2)]$ 
(or $p\in [1,+\infty)$ if $n=2$), there holds
\begin{equation} \label{rPoincare1}
\Big\{\frac1{m(B)}\int_{B} \left|u(y)\right|^p w(y) dy\Big\}^{1/p}
\leq C r \Big\{ \frac{1}{m(B)}\int_{B} |\nabla u(y)|^2 w(y) dy \Big\}^{1/2}.
\end{equation}
That is, we can choose $u_B = 0$ in \eqref{1.5-bis}.

To prove \eqref{rPoincare1}, the main idea is that we can replace in \eqref{1.5-bis} the quantity 
$u_B = m(B)^{-1} \int_B u$ by the average $u_{\bar B}$, where $\bar B$ is a ball near $B$.
We choose for $\bar B$ a ball with same radius as $B$ and contained in $3B \setminus B$,
because this way  $u_{\bar B} = 0$ since $u$ is supported in $B$. Then
\begin{eqnarray} \label{rPoincare2} 
\Big\{\frac1{m(3B)}\int_{3B} |u(y)|^p w(y) dy\Big\}^{1/p}  
&=& \Big\{\frac1{m(3B)}\int_{3B} |u(y) - u_{\bar B}|^p w(y) dy\Big\}^{1/p} \nn\\
&\leq& \Big\{\frac1{m(3B)}\int_{3B} |u(y) - u_{3B}|^p w(y) dy\Big\}^{1/p} + |u_{3B} - u_{\bar B}|
\end{eqnarray}
Yet, using Jensen's inequality and then H\"older's
inequality, $|u_{3B} - u_{\bar B}|$ is bounded by $\Big\{\frac1{m(\bar B)}\int_{3B} |u(y) - u_{3B}|^p w(y) dy\Big\}^{1/p}$. If we use in addition the doubling property given by \eqref{doublingn}, we get that $|u_{3B} - u_{\bar B}|$ is bounded by $\Big\{\frac1{m(3B)}\int_{3B} |u(y) - u_{3B}|^p w(y) dy\Big\}^{1/p}$, that is,
\begin{equation} \label{rPoincare2bis} \begin{split}
\Big\{\frac1{m(3B)}\int_{3B} |u(y)|^p w(y) dy\Big\}^{1/p}  
\leq \Big\{\frac1{m(3B)}\int_{3B} |u(y) - u_{3B}|^p w(y) dy\Big\}^{1/p}.
\end{split} \end{equation}
We conclude thanks to \eqref{1.5-bis} and the doubling property \eqref{doublinggen}.
\end{remark}
\chapter{Completeness
and Density of Smooth Functions}
\label{Scompleteness}

In later sections we shall work with various dense classes.
We prepare the job in this section, with a little bit of work on function spaces and approximation arguments. 
Most results in this section are basically unsurprising, except perhaps the fact that when $d \leq 1$,
the test functions are dense in $W$ (with no decay condition at infinity).

Let $\dot W$ be the factor space $W/\RR$, equipped with the norm $\|\cdot\|_W$. The elements of $\dot W$ are classes $\dot u=\{u+c\}_{c\in\RR}$, where $u\in W$.  
\begin{lemma}\label{lcomp}
The space $\dot W$ is complete.
In particular, if a sequence of elements of $W$, $\{v_k\}_{k=1}^\infty$, and $u\in W$ are such that $\|v_k-u\|_W \to 0$ as $k\to \infty$, then there exist constants $c_k\in \RR$ such that $v_k-c_k \to u$ in $L^1_{loc}(\rn).$
\end{lemma}

\bp
Let $(\dot u_k)_{k\in \bN}$ be a Cauchy sequence in $\dot W$. We need to show
that
\begin{enumerate}[(i)]
\item for every sequence $(v_k)_{k\in \bN}$ in $W$, with
$v_k \in \dot u_k$ for 
$k\in \bN$, there exists $u\in W$ and $(c_k)_{k\in \bN}$ such that $v_k - c_k \to u$ in $L^1_{loc}(\R^n)$ and
\begin{equation}\label{aa}
\lim_{k\to \infty} \|v_k - u\|_W = 0 ;
\end{equation}
\item if $u$ and $u'$ are such that there exist $(v_k)_{k\in \bN}$ and $(v'_k)_{k\in \bN}$ 
such that 
$v_k,v'_k \in \dot u_k $ for all
$k\in \bN$ and 
\begin{equation}\label{bb}
\lim_{k\to \infty} \|v_k - u\|_W = \lim_{k\to \infty} \|v'_k - u'\|_W = 0,
\end{equation}
then $\dot u' = \dot u$.
\end{enumerate}

First assume
that (i) is true and let us prove (ii). Let $u$, $u'$, $(v_k)_{k\in \bN}$ and $(v'_k)_{k\in \bN}$ be such that $v_k,v'_k \in \dot u_k$ for any $k\in \bN$ 
and \eqref{bb} holds.
Then
the sequence $(\nabla v_k- \nabla v'_k)_{k\in \bN}$ converges in $L^2(\Omega,w)$ to $\nabla (u -u')$ on one hand and is constant equal to 0 on the other hand. Thus $\nabla (u-u')  = 0$ and $u$ and $u'$ differ only by a constant, hence $\dot u' = \dot u$.

\medskip
Now we prove (i). By translation invariance, we may assume that $0 \in \Gamma$.
Let the $v_k \in \dot u_k$ be given, and choose $c_k = \fint_{B(0,1)} v_k$. 
We want to show that $v_k- c_k$ converges in $L^1_{loc}(\R^n)$. 

Set $B_j = B(0,2^j)$ for $j \geq 0$; let us check that for $f\in W$ and $j \geq 0$,
\begin{equation} \label{Co1} 
\fint_{B_j} \Big|  f(x) - \fint_{B_0} f(y) \, dy\Big|\, dx \leq C 2^{(n+1)j} \|f\|_W.
\end{equation}
Set $m_j = \fint_{B_j} f\, dy$; observe that
\begin{eqnarray}\label{c1}
\fint_{B_j} |f-m_j| \, dx&\leq& \frac{1}{m(B_j)}\int_{B_j} |f(x)-m_j| w(x) dx 
\nn\\
&\leq& C 2^j m(B_j)^{-1/2} \Big\{\int_{B_j} |\nabla f(y)|^2 w(y)dy\Big\}^{1/2}
\\
&\leq& C 2^j m(B_j)^{-1/2} ||f||_W \leq C 2^j ||f||_W
\nn
\end{eqnarray}
by \eqref{L1byL1w}, the Poincar\'e inequality \eqref{1.5-bis} with $p=1$,
and a brutal estimate using \eqref{ISob8b}, our assumption that $0 \in \Gamma$, and 
the fact that $B_j \supset B_0$. In addition, 
\begin{equation}\label{c2} \begin{split}
|m_0-m_j| & = \Big|\fint_{B_0} f(x) \, dx -m_j \Big| \leq \fint_{B_0} |f(x)-m_j|\, dx \\
&  \leq 2^{jn} \fint_{B_j} |f(x)-m_j|\, dx
\leq C 2^{(n+1)j} \|f\|_W
\end{split}\end{equation}
by \eqref{c1}. Finally 
\begin{equation}\label{c3} \begin{split}
\fint_{B_j} \Big|  f(x) - \fint_{B_0} f(y)\, dy \Big|\, dx & = \fint_{B_j} \Big|  f(x) - m_0 \Big|\, dx \leq \fint_{B_j} |f(x)-m_j|\, dx + |m_0-m_j| \\
& \leq C 2^{(n+1)j} ||f||_W,
\end{split} \end{equation}
as needed for \eqref{Co1}.

\medskip

Return
to the convergence of $v_k$. 
Recall that 
$c_k = \fint_{B_1}  v_k$. 
By \eqref{Co1} with $f = v_k - c_k - v_l + c_l$ (so that $m_0 = 0$),
$v_k - c_k$ is a Cauchy sequence in $L^1_{loc}(B_j)$ for each $j \geq 0$, 
hence there exists $u^j \in L^1(B_j)$ such that $v_k - c_k$ converges to $u^j$. 
By uniqueness of the limit, we have that
for $1\leq j \leq j_0$,
\begin{equation} \label{Co2}
u^{j_0} = u^j \ \text{ a.e. in } B_j
\end{equation}
and thus we can define a function $u$ on $\R^n$ as $u(x) = u^j(x)$ if $x \in B_j$. 
By construction $u\in L^1_{loc}(\R^n)$ and $v_k - c_k \to u$ in $L^1_{loc}(\R^n)$. 

It remains to show that $u$ is actually in $W$ and $v_k \to u$ in $W$. 
First, since $L^2(\Omega,w)$ is complete, there exists $V$ such that $\nabla v_k$ 
converges to $V$ in $L^2(\Omega,w)$. 
Then observe that for
$\varphi \in C^\infty_0(B_j \setminus \Gamma,\R^n)$,
\[\int_{B_j} V \cdot \varphi  = \lim_{k\to \infty} \int_{B_j} \nabla v_k \cdot \varphi = -\lim_{k\to \infty} \int_{B_j} (v_k-c_k) \diver\varphi  = -\int_{B_j} u^j \diver\varphi.\]
Hence by definition of a 
weak derivative, 
\[\nabla u = \nabla u^j = V \ \text{ a.e. in } B_j.\]
Since the result holds for any $j\geq 1$,
\[\nabla u = V \ \text{ a.e. in } \R^n,\]
that is, by construction of $V$, $u \in W$ and $\|v_k - u\|_W$ converges to 0.
\ep

\begin{lemma}\label{lcomp2}
The space
\begin{equation} \label{2.47}
W_0 = \big\{ u\in W \, ; \, Tu = 0 \big\},
\end{equation}
equipped with the scalar product $\left< u,v \right>_W := \int_\Omega \nabla u(z) \cdot \nabla v(z) \, dm(z)$ (and the norm $\|.\|_W$) is a Hilbert space.

Moreover, for any ball $B$ centered on $\Gamma$, the set 
\begin{equation} \label{2.47b}
W_{0,B} = \big\{ u\in W \, ; \, Tu = 0   \ \text{ $\H^d$-almost everywhere on } \Gamma \cap B \big\},
\end{equation}
equipped with the scalar product $\left<.,. \right>_W$, is also a Hilbert space.
\end{lemma}

 \bp
Observe that $W_0$ and $W_{0,B}$ are no longer spaces of functions defined modulo an additive constant.
That is, if $f\in W_0$ (or $W_{0,B}$) is a constant $c$, then $c=0$ because \eqref{defT} says that $Tu = c$
almost everywhere on $\Gamma$.
Thus $\|.\|_W$ is really a norm on $W_0$ and $W_{0,B}$, and we only need to prove that these spaces 
are complete.
We first prove this for $W_{0,B}$;
the case of $W_0$ will be easy deal with afterwards.

Let $B$ be a ball centered on $\Gamma$, and consider $W_{0,B}$. 
By translation and dilation
invariance of the result, we can 
assume that $B=B(0,1)$.

Let $(v_k)_{k\in \bN}$ be a Cauchy sequence of functions in $W_{0,B}$. 
We want first to show that $v_k$ has a limit in $L^1_{loc}(\R^n)$ and $W$.
We use Lemma~\ref{lcomp} and so there exists $\bar u\in W$ and $c_k\in \R$ such that 
\begin{equation} \label{Co5} 
\|v_k - \bar u\|_W \to 0
\end{equation}
and
\begin{equation} \label{Co6} 
v_k - c_k \to \bar u \ \  \text{ in } L^1_{loc}(\R^n).
\end{equation}
By looking at the proof of Lemma~\ref{lcomp}, we can take $c_k = \fint_B v_k$. Let us prove that 
$(c_k)$ is a Cauchy sequence in $\R$.
We have for any $k,l\geq 0$
\begin{equation} \label{Co7} 
|c_k - c_l| \leq \fint_B |v_k(z) - v_l(z)|\, dz \leq C m(B)^{-1}  \int_B |v_k(z) - v_l(z)| dm(z)
\end{equation}
with \eqref{L1byL1w}. Since $T(v_k-v_l) = 0$ on $B$ and since we assumed that $B$ is the unit ball, Lemma~\ref{lSob} entails
\begin{equation} \label{Co8} 
|c_k - c_l| \leq C \|v_k-v_l\|_W.
\end{equation}
Since $(v_k)_{k\in \bN}$ is a Cauchy sequence in W, $(c_k)_{k\in \bN}$ is a Cauchy sequence in $\R$ and thus converges to some value $c\in \R$. 
Set $u = \bar u - c$. We deduce from \eqref{Co6} that
\begin{equation} \label{Co9} 
v_k\to u \ \  \text{ in } L^1_{loc}(\R^n),
\end{equation}
and since $u$ and $\bar u$ differ only from a constant, \eqref{Co5} can be rewritten as
\begin{equation} \label{Co10} 
\|v_k - u\|_W \to 0.
\end{equation}

\medskip
We still need to show that $u\in W_{0,B}$, i.e., that 
$Tu = 0$ a.e. on $B$. We will actually prove something a bit stronger. We claim that
if $u,v_k\in W$, then the convergence of $v_k$ to $u$ in both $W$ and $L^1_{loc}(\R^n)$ implies the convergence of the traces $Tv_k \to Tu$ in $L^1_{loc}(\Gamma,\sigma)$.
That is,
\begin{equation} \label{CoA}
v_k \to u \text{ in $W$ and in $L^1_{loc}(\R^n)$ } \Longrightarrow Tv_k \to Tu \text{ in $L^1_{loc}(\Gamma,\sigma)$}.
\end{equation}

Recall that by \eqref{L1locsigma}, $Tf \in L^1_{loc}(\Gamma,\sigma)$ whenever $f\in W$.
Our result, that is $Tu = 0$ a.e. on $B$, follows easily from the claim: we already established that $v_k \to u$ in $W$ and in $L^1_{loc}(\R^n)$ and thus \eqref{CoA} gives that 
$\|Tu\|_{L^1(B,\sigma)}
= \lim_{k\to \infty} \|Tv_k\|_{L^1(B,\sigma)} = 0$, 
i.e., that
$Tu = 0$ $\sigma$-a.e. in $B$.

We turn to the proof of \eqref{CoA}. 
Since $T$ is linear, we may subtract $u$, and assume that $v_k$ tends to $0$ and $u=0$.
Let us use the notation of Theorem \ref{tTr}, and set $g^k = Tv_k$ and $g^k_r(x) = \fint_{B(x,r)} v_k(z)\, dz$.
Since $\|v_k\|_W$ tends to $0$, we may assume without loss of generality that $\|v_k\|_W \leq 1$
for $k\in \bN$. We want to prove that for every ball $\wt B\subset \R^n$ centered on $\Gamma$
and every $\epsilon >0$, we can find $k_0$ such that
\begin{equation}\label{dd}
\|g^k\|_{L^1(\wt B,\sigma)} \leq \epsilon \ \ \text{ for } k \geq k_0.
\end{equation}
We may also assume that the radius of $\wt B$ is larger than $1$ 
(as it makes \eqref{dd} harder to prove).

Fix $\wt B$ and $\varepsilon$ as above, and $\alpha\in (0,1/2)$, and observe that for $r\in (0,1)$,
\begin{equation} \label{Co11} \begin{split}
\int_{\wt B} \left|g^k\right| d\sigma 
&\leq  \int_{\wt B} |g^k - g^k_r| d\sigma + \int_{\wt B} |g^k_r|d\sigma \\
& \leq  C(\wt B) \|g^k-g^k_r\|_{L^2(\sigma)} 
+ \int_{x \in \wt B} \fint_{y \in B(x,r)} |v_k(y)| dy d\sigma(x) \\
& \leq C(\wt B,\alpha) r^{2\alpha} \|v_k\|_W + C r^{d-n} \int_{2\wt B} |v_k(y)| dy, 
\end{split}\end{equation}
where for the last line we used \eqref{2.18}, Fubini, and the condition \eqref{1.1} on $\Gamma$.
Recall that $\|v_k\|_W \leq 1$;
we choose $r$ so small that $C(\wt B,\alpha) r^{2\alpha} \|u\|_W \leq \epsilon/2$, and
since by assumption $v_k$ tends to $0$ in $L^1_{loc}$, we can find $k_0$
such that $C r^{d-n} \int_{2\wt B} |v_k(y)| dy \leq \varepsilon/2$ for $k \geq k_0$, as needed for \eqref{dd}.

This completes the proof of \eqref{CoA}, and we have seen that the completeness of $W_{0,B}$ follows.
Since $W_0$ is merely an intersection of spaces $W_{0,B}$, it is complete as well, and
Lemma \ref{lcomp2} follows.
\ep

\begin{lemma} \label{lconvol}
Choose a non-negative function $\rho \in C^\infty_0(\R^n)$ such that $\int \rho = 1$ and 
$\rho$ is supported in $\overline{B(0,1)}$. Furthermore let $\rho$ be radial and nonincreasing, 
i.e. $\rho(x) = \rho(y) \geq \rho(z)$ if $|x|=|y|\leq |z|$. 
Define $\rho_\epsilon$, for $\epsilon>0$,  by $\rho_\epsilon(x) = \epsilon^{-n} \rho(\epsilon^{-1} x)$. 
For every $u\in W$, we have:
\begin{enumerate}[(i)]
\item $\rho_\epsilon * u \in C^\infty(\R^n)$ for every $\epsilon>0$;
\item If $x\in \R^n$ is a Lebesgue point of $u$, then $\rho_\epsilon * u(x) \to u(x)$ as $\epsilon\to 0$; 
in particular, $\rho_\epsilon * u \to u$ a.e. in $\R^n$;
\item $\nabla (\rho_\epsilon*u) = \rho_\epsilon*\nabla u$ for $\varepsilon > 0$;
\item $\lim_{\epsilon \to 0} \|\rho_\epsilon* u - u\|_W = 0$;
\item $\rho_\epsilon* u \to u$ in $L^1_{loc}(\R^n)$.
\end{enumerate}
\end{lemma}

\bp Recall that $W \subset L^1_{loc}(\R^n)$ (see Lemma~\ref{W=barW}). 
Thus conclusions (i) and (ii) are classical and can be found as Theorem 1.12 in \cite{MZbook}.

Let $u\in W$ and write $u_\epsilon$ for $\rho_\epsilon * u$. 
We have seen that $u_\epsilon \in C^\infty(\R^n)$,
so $\nabla u_\epsilon$ is defined on $\R^n$. 
One would like to say that $\nabla u_\epsilon = \rho_\epsilon * \nabla u$, i.e. point (iii). 
Here
$\nabla u_\epsilon$ is the classical gradient of $u_\epsilon$ on $\R^n$, thus {\em a fortiori } 
also the distributional gradient on $\R^n$ of $u_\epsilon$.
That is, for any $\varphi \in C^\infty_0(\R^n,\R^n)$, there holds
\begin{equation} \label{lconv22} \begin{split}
\int_{\R^n} \nabla u_\epsilon \cdot \varphi  & 
= - \int_{\R^n} u_\epsilon(x) \diver\varphi(x) dx
= - \int_{\R^n} \int_{\R^n} \rho_\epsilon(y) u(x-y) \diver \varphi(x) dy \, dx \\
& = \int_{B(0,\epsilon)} \rho_\epsilon(y) \left( - \int_{\R^n} u(z) \diver \varphi(z+y) dz \right) dy.
\end{split} \end{equation}
The function $\varphi$ 
lies in $C^\infty_0(\R^n,\R^n)$, and so does,
for any $y\in \R^n$, the function $z\mapsto \varphi(z+y)$. Recall that $\nabla u$ is the distributional derivative of $u$ on $\Omega$ but yet also the distributional derivative of $u$ on $\R^n$ (see Lemma~\ref{W=barW}). Therefore
\begin{equation} \label{lconv23} \begin{split}
\int_{\R^n} \nabla u_\epsilon \cdot \varphi 
& = \int_{B(0,\epsilon)} \rho_\epsilon(y) \int_{\R^n} \nabla u(z) \cdot \varphi(z+y) dz \, dy \\
& = \int_{\R^n} \rho_\epsilon(y) \int_{\R^n} \nabla u(x-y) \cdot \varphi(x) dx \, dy = \int_{\R^n} (\rho_\epsilon * \nabla u) \cdot \varphi, 
\end{split} \end{equation}
which gives (iii).

From there, our
point (iv), that is the convergence of $\rho_\epsilon * u$ to $u$ in $W$, can be deduced with, for instance, \cite[Lemma~1.5]{Kilpelai}. 
The latter states that, under our assumptions on $\rho$, the convergence $\rho_\epsilon*g \to g$ holds in $L^2(\R^n,w)$ whenever $g\in L^2(\R^n,w)$ and $w$ is in the Muckenhoupt class $\mathcal A_2$ (we already proved this fact, see Lemma~\ref{lwA2}). 
Note that Kilpelai's result is basically a consequence of a result from Muckenhoupt about the boundedness of the (unweighted) Hardy-Littlewood maximal function in weighted $L^p$.

Finally we need to prove (v). 
Just notice that $u \in L^1_{loc}(\R^n)$, and apply the standard proof of the fact that 
$\rho_\epsilon* u \to u$ in $L^1$ for $f\in L^1$. The lemma follows
\ep

\begin{lemma} \label{lmult}
Let $u \in W$ and $\varphi \in C^\infty_0(\R^n)$. Then $u\varphi \in W$ and for any point $x \in \Gamma$ satisfying \eqref{eqLeb}
\begin{equation} \label{lmult1}
T(u\varphi)(x) = \varphi(x) Tu(x).
\end{equation}
\end{lemma}

\bp 
The function $u$ lies in $L^1_{loc}(\R^n)$ and thus defines
a distribution on $\R^n$ (see Lemma~\ref{W=barW}). Multiplication by smooth functions and (distributional) derivatives are always defined for distributions and, in the sense of distribution, $\nabla (u \varphi) = \varphi \nabla u + u \nabla \varphi$. Let $B \subset \R^n$ be a big ball such that $\supp \, \varphi \subset B$. Then
\begin{equation} \begin{split}
\|u\varphi\|_W & \leq \|\varphi\|_\infty \|\nabla u\|_{L^2(\Omega,w)} + \|\nabla \varphi\|_\infty \Big\|u-\fint_B u\Big\|_{L^2(B,w)} + \|\nabla \varphi\|_\infty \Big\|\fint_B u\Big\|_{L^2(B,w)} \\
& \leq \|\varphi\|_\infty \|\nabla u\|_{L^2(\Omega,w)} + C_B  \|\nabla \varphi\|_\infty \|u\|_{W} + C_B \|\nabla \varphi\|_\infty \|u\|_{L^1(B)}< +\infty
\end{split} \end{equation}
by the Poincar\'e inequality \eqref{1.5-bis}. We deduce $u\varphi \in W$. 

\medskip

Let take a Lebesgue point $x$ satisfying \eqref{eqLeb}. We have
\[\begin{split}
 \fint_{B(x,r)} |u(z)\varphi(z) & - \varphi(x) Tu(x)|\, dz \nn\\
&\leq \fint_{B(x,r)} |u(z)-Tu(x)| |\varphi(z)|\, dz + |Tu(x)| \fint_{B(x,r)} |\varphi(z)-\varphi(x)| \, dz
\\
&\leq \|\varphi\|_\infty \fint_{B(x,r)} |u(z)-Tu(x)|\, dz + |Tu(x)| \fint_{B(x,r)} |\varphi(z)-\varphi(x)|\, dz.
\end{split}\]
The first term of the right-hand side converges to 0 because $x$ is a Lebesgue point. The second term in the right-hand side converges to 0 because $\varphi$ is continuous. The equality \eqref{lmult1} follows.
\ep

\medskip

Let $F$ be a closed set in $\R^n$ and $E = \R^n\setminus F$. 
In the sequel, we let
\begin{equation} \label{Cinftycdef}
C^\infty_c(E) = \big\{ f \in C^\infty(E), \, \exists \epsilon>0 \text{ such that $f(x)=0$ whenever }
\dist(x,F) \leq \varepsilon \big\}
\end{equation}
denote
 the set of functions in $C^\infty(E)$ that equal 0 in a neighborhood of $F$. Furthermore, we use the notation $C^\infty_0(E)$ for the set of functions that are compactly supported in $E$, that is
\begin{equation} \label{Cinfty0def}
C^\infty_0(E) = \{ f \in C^\infty_c(E), \, \exists R>0: \, \supp f \subset B(0,R) \}.
\end{equation}

\medskip

\begin{lemma}\label{ldens0} 
The completion of $C_0^{\infty}(\Omega)$ for the norm $\|.\|_W$ 
is the set
\begin{equation} \label{2.47a}
W_0 = \big\{ u\in W \, ; \, Tu = 0 \big\}
\end{equation}
of \eqref{2.47}.
Moreover, if $u\in W_0$ is supported in a compact subset of the
open ball $B \subset \R^n$, then $u$ can be approximated 
in
the $W$-norm by functions of
$C^\infty_0(B \setminus \Gamma)$.
\end{lemma}

\bp The proof of this result will use two main steps, where 
\begin{enumerate}[(i)]
\item we use cut-off functions $\varphi_r$ to approach any function $u\in W_0$ 
by functions in $W$ that equal 0 on a neighborhood of $\Gamma$;
\item we use cut-off functions $\phi_R$ to approach any function $u \in W_0$ by functions in $W$ 
that are compactly supported in $\R^n$.
\end{enumerate} 

\noindent{\bf Part (i):}  For $r>0$ small, we choose a smooth function $\varphi_r$ 
such that $\varphi(x) = 0$ when $\delta(x) \leq r$, $\varphi(x) = 1$ when $\delta(x) \geq 2 r$, 
$0 \leq \varphi \leq 1$ everywhere, and $|\nabla \varphi(x)| \leq 10 r^{-1}$ everywhere.

Let $u \in W_0$ be given. We want to show that for $r$ small, 
$\varphi_r u$ lies in $W$ and 
\begin{equation}\label{a1}
\lim_{r \to 0} \| u-\varphi_r u \|_W^2 = 0.
\end{equation}
Notice that $\varphi_r u \in L^1_{loc}(\Omega)$, just like $u$, 
and its distribution gradient on $\Omega$ is locally in $L^2$ and given by
\begin{equation} \label{2.47bbis}
\nabla(\varphi_r u)(x) = \varphi_r(x) \nabla u(x) + u(x) \nabla \varphi_r(x).
\end{equation} 
So we just need to show that
\begin{equation} \label{2.48}
\lim_{\begin{subarray}{c} r \to 0\end{subarray}} \int |\nabla(\varphi_r u)(x) -\nabla u(x) |^2 w(x) dx 
= \lim_{r \to 0} \int |u(x) \nabla \varphi_r(x) + (1-\varphi_r(x))\nabla  u(x) |^2 w(x) dx =0.
\end{equation}

Now $\int |\nabla  u(x) |^2 w(x) dx = \|u\|_W^2 < +\infty$, so 
$\int |(1-\varphi_r) \nabla  u(x) |^2 w(x) dx$ tends to $0$, by the dominated convergence theorem,
and
it is enough to show that
\begin{equation} \label{2.49}
\lim_{r \to 0} \int |u(x) \nabla \varphi_r(x)|^2 w(x) dx = 0.
\end{equation} 
Cover $\Gamma$ with balls $B_j$, $j\in J$, of radius $r$, centered on $\Gamma$, 
and such that the $3B_j$ have bounded overlap, and notice that the region where 
$\nabla\varphi_r \neq 0$ is contained in $\cup_{j\in J} 3B_j$. 
In addition, if $x \in 3B_j$ is such that $\nabla\varphi_r \neq 0$, then  
$|\nabla \varphi_r(x)| \leq 10 r^{-1}$, so that
\begin{equation} \label{2.50}
\int_{3B_j}  |u(x) \nabla \varphi_r(x)|^2 w(x) dx
\leq 100 r^{-2} \int_{3B_j}  |u(x)|^2 w(x) dx
\leq C \int_{3B_j}  |\nabla u(x)|^2 w(x) dx,
\end{equation}
where the last part comes from \eqref{1.5-bisbis}, applied with $p=2$
and justified by the fact that $Tu=0$ on the whole $\Gamma$.
We may now sum over $j$. Denote by $A_r$ the union of the $3B_j$; then
\begin{eqnarray} \label{2.52}
\int_\Omega |u(x) \nabla \varphi_r(x)|^2 w(x) dx
&\leq& \sum_{j\in J} \int_{3B_j}  |u(x) \nabla \varphi_r(x)|^2 w(x) dx
\leq C \sum_{j\in J} \int_{3B_j} |\nabla u(x)|^2 w(x) dx
\nn\\
&\leq& C \int_{A_r} |\nabla u(x)|^2 w(x) dx
\end{eqnarray}
because the $3B_j$ have bounded overlap. 
The right-hand side of \eqref{2.52}
tends to $0$, because $\int_\Omega |\nabla u(x)|^2 w(x) dx = \| u \|_W^2< +\infty$ and by the
dominated convergence theorem. The claim \eqref{2.49} follows, 
and so does \eqref{a1}. This completes Part (i).

\medskip

\noindent {\bf Part (ii).} 
By translation invariance, we 
may assume that
$0\in \Gamma$.  Let 
$R$ be a big radius; we want to define a cut-off function $\phi_R$.
 
If we used
the classical cut-off function built as $\bar \phi_R = \bar \phi\left(\frac xR\right)$ with $\bar \phi$ 
supported in $B(0,1)$, the convergence 
would work with the help of Poincar\'e's  inequality on annuli. 
But since we did not prove this inequality, we will proceed differently and use the
`better' cut-off functions defined as follows.
 
Set $\phi_R(x) = \phi\left( \frac{\ln |x|}{\ln R} \right)$,  where $\phi$ is a smooth function 
defined on $[0,+\infty)$, supported in $[0,1]$ and such that $\phi \equiv 1$ on $[0,1/2]$.  
In particular, one can see that $\nabla \phi_R(x) \leq \frac{C}{\ln R} \frac{1}{|x|}$ and that 
$\nabla \phi_R$ is supported in $\{x\in \R^n, \, \sqrt R \leq |x| \leq R\}$. 
We take $\wh u := \phi_R u$ and we want to show that $\wh u \in W$ and $\|u-\wh u\|_W$ is small. 
Notice that $\wh u \in W_0$, by Lemma \ref{lmult}, and in addition $\wh u$ is supported in
$B(0,R)$. We want to show that
\begin{equation}\label{a2}
\lim_{R \to +\infty} \| \wh u- u \|_W^2 = 0.
\end{equation}
But 
$\wh u \in L^1_{loc}(\Omega)$, just like $u$,  and 
its distribution gradient on $\Omega$ is locally in $L^1$ and given by
\begin{equation} \label{2.521}
\nabla\wh u(x) = \phi_R(x) \nabla u(x) + u(x) \nabla \phi_R(x).
\end{equation}
Hence
\begin{eqnarray}\label{2.522}
\| \wh u- u \|_W^2 &=& \int |\nabla \wh u(x) -\nabla u(x) |^2 w(x) dx 
\nn \\
&=& \int |u(x) \nabla \phi_R(x) + (1-\phi_R(x))\nabla  u(x) |^2 w(x) dx.
\end{eqnarray}
Now $\int |\nabla  u(x) |^2 w(x) dx = \|u\|_W^2 < +\infty$, so 
$\int |(1-\phi_R) \nabla  u(x) |^2 w(x) dx$ tends to $0$, by the dominated convergence theorem,
and it is enough to show that
\begin{equation} \label{2.523}
\lim_{R \to +\infty} \int |u(x) \nabla \phi_R(x)|^2 w(x) dx = 0. 
\end{equation}
Let $C_j$ be the annulus
$\{x\in \R^n, \, 2^{j} < |x| \leq 2^{j+1}\}$. The bounds on $\nabla \phi_R$ yield
\begin{equation} \label{2.52b}
\int |u(x) \nabla \phi_R(x)|^2 w(x) dx 
\leq \frac{C}{(\ln R)^2}\sum_{j=0}^{ 1 + \log_2 R}
2^{-2j}\int_{C_j} |u(x)|^2 w(x) dx.
\end{equation}
The integral on the annulus
$C_j$ is smaller than the integral in the ball $B(0,2^{j+1})$. Since 
$u \in W_0$ and
$0\in \Gamma$, \eqref{1.5-bisbis} yields
\begin{equation} \label{2.52c} \begin{split}
\int |u(x) \nabla \phi_R(x)|^2 w(x) dx 
& \leq \frac{C}{(\ln R)^2}\sum_{j=0}^{1 + \log_2 R} \int_{B(0,2^{j+1})} |\nabla u(x)|^2 w(x) dx \\
& \leq \frac{C}{(\ln R)^2} \|u\|_W^2 \sum_{j=0}^{1 + \log_2 R} 1 \leq \frac{C}{|\ln R|}  \|u\|_W. \\
\end{split} \end{equation}
Thus $\int |u(x) \nabla \phi_R(x)|^2 w(x) dx$ converges to 0 as $R$ goes to $+\infty$, which proves \eqref{2.523} and ends Part (ii).

\medskip

We are now ready to prove the lemma.
If $u\in W_0$ and $\varepsilon > 0$ is given, we can find $R$ such that $\| \phi_R u- u \|_W^2 \leq \varepsilon$
(by \eqref{a2}). Notice that $\phi_R u \in W_0$, by Lemma \ref{lmult}, and now 
we can find $r$ such that $\| \varphi_r \phi_R u- \phi_R u \|_W^2 \leq \varepsilon$ (by \eqref{a1}).
In turn $\varphi_r \phi_R u$ is compactly supported away from $\Gamma$,
and we may now use Lemma~\ref{lconvol} to approximate it with smooth functions
with compact support in $\Omega$.
It follows that $W_0$ is included in the completion of $C^\infty_0(\Omega)$. Since $W_0$ is complete (see Lemma~\ref{lcomp2}), the reverse inclusion is immediate.

For the second part of the lemma, we are given $u\in W_0$ with a compact support inside $B$, 
we can use Part (i) to approximate it by some $\varphi_r u$ with a compact support inside $B$.
A convolution as in Lemma~\ref{lconvol} then makes it smooth without destroying the support property;
Lemma \ref{ldens0} follows.
\ep

\begin{remark} \label{rdens0} We don't know how to prove 
exactly the same result for the spaces $W_{0,B}$ of \eqref{2.47b}. 
However, we have the following weaker result. Let $B\subset \R^n$ be a ball and $B_\frac12$ denotes the ball with same center as $B$ but half its radius. For any function $u\in W_{0,B}$, there exists a sequence $(u_k)_{k\in \bN}$ of functions in $C^\infty_c(\R^n \setminus \overline{B_\frac12 \cap \Gamma})$ such that $\|u_k - u\|_W$ converges to 0. 

\medskip

Indeed, take $\eta \in C^\infty_0(B)$ such that $\eta = 1$ on $B_\frac34$. Write $u = \eta u + (1-\eta) u$; 
it is enough to prove that both $\eta u$ and $(1-\eta)u$ can be approximated by functions in 
$C^\infty_c(\R^n \setminus \overline{B_\frac12 \cap \Gamma})$. 
Notice first that $\eta u \in W_0$ and thus can be approximated by functions in 
$C^\infty_0(\Omega) \subset C^\infty_c(\R^n \setminus \overline{B_\frac12 \cap \Gamma})$, 
according to Lemma~\ref{ldens0}. Besides, $(1-\eta)u$ is supported outside of $B_\frac34$ and thus, 
if $\epsilon$ is smaller than a quarter of the radius of $B$, then the functions $\rho_\epsilon * [(1-\eta)u]$ are in 
$C^\infty_c(\R^n \setminus \overline{B_\frac12}) \subset  C^\infty_c(\R^n \setminus \overline{B_\frac12 \cap \Gamma})$. Lemma~\ref{lconvol} gives then that the family $\rho_\epsilon * [(1-\eta)u]$ approaches $(1-\eta)u$ as $\epsilon$ goes to 0.
\end{remark}

\ms
Next we worry about the completion of $C^\infty_0(\R^n)$ for the norm $\|.\|_W$.
We start with the case when $d>1$; when $0 < d \leq 1$, things are a little different and 
they will be discussed in Lemma~\ref{ldensdleq1}.

\begin{lemma}\label{ldensd>1}
Let $d>1$. Choose $x_0 \in \Gamma$ and write $B_j$ for $B(x_0,2^j)$. 
Then for any $u\in W$
\begin{equation} \label{d>1a}
u^0:= \lim_{j\to +\infty} \fint_{B_j} u(z)\, dz \text{ exists and is finite.}
\end{equation}

The completion of $C^\infty_0(\R^n)$ for the norm $\|.\|_W$ can be identified to a subspace of $L^1_{loc}(\R^n)$, which is
\begin{equation} \label{d>1b}
W^0 = \{u\in W, \, u^0 = 0\}.
\end{equation}
\end{lemma}

\ms
\begin{remark} 
Since $C_0^\infty(\Omega) \subset C_0^\infty(\R^n)$, Lemmata \ref{ldens0} and \ref{ldensd>1} imply that $W_0 \subset W^0$. In particular, 
we get that 
\begin{equation}\label{a3}
\lim_{j\to +\infty}  \fint_{B_j} u(z)\, dz = 0 \ \ \text{ for }  u\in W_0.
\end{equation}
\end{remark}

\begin{remark}
Since the completion of $C^\infty_0(\R^n)$ doesn't depend on our choice of $x_0$, the value $u^0$ doesn't depend on $x_0$ 
either. Similarly, with a small modification in the proof, we could replace $(2^j)$ with
any other sequence that tends to $+\infty$.
\end{remark}

\begin{remark}
The lemma immediately implies the following result: for any $u\in W$, $u-u^0 \in W^0$ and thus can be approximated in $L^1_{loc}(\R^n)$ and in the $W$-norm by function in $C^\infty_0(\R^n)$.
\end{remark}

\bp Let $d>1$ and choose $u\in W$. Let us first prove that $u^0$ is well defined. By translation invariance, we can choose $x_0=0$, that is $B_j = B(0,2^j)$. For $j\in \bN$, set $u_j = \fint_{B_j} f(z)\, dz$ and $V_j = \int_{B_j} dm$. The bounds \eqref{ISob8b} give that $V_j$ is equivalent to $2^{j(1+d)}$ and \eqref{ldoub8} gives that for any $z\in B_j$, $\frac{V_j}{|B_j|} \leq C w(z)$. 
Then by Lemma~\ref{lSob}
\begin{equation} \label{lD1} \begin{split}
|u_{j+1} - u_j| & \leq C \fint_{B_{j+1}} |u(z) - u_{j+1}|\, dz \leq C V_{j+1}^{-1} \int_{B_{j+1}} |u(z)-u_{j+1}| dm(z) \\
& \leq C 2^{j(1-\frac{d+1}{2})} \left( \int_{B_{j+1}} |\nabla u(z)|^2 dm(z) \right)^\frac12 \leq C2^{j\frac{1-d}{2}} \|u\|_{W}.
\end{split} \end{equation}
Since $d > 1$, $(u_j)_{j\in \bN}$ is a Cauchy sequence and converges to some value
\begin{equation} \label{lD2} 
u^0 = \lim_{j\to +\infty} u_j.
 \end{equation}
 Moreover \eqref{lD1} also entails
\begin{equation} \label{lD3} 
|u_j - u^0| \leq C 2^{j\frac{1-d}{2}} \|u\|_W.
 \end{equation}
Let us prove additional properties on $u^0$. Set $v= |u|$. 
Notice that
\begin{equation} \label{lD3a} 
|u_j| \leq v_j := \fint_{B_j} |u(z)|\, dz \leq |u_j| + \fint_{B_j} |u(z)-u_j|\, dz \leq |u_j| + C 2^{j\frac{1-d}{2}} \|u\|_W,
 \end{equation}
where the last inequality 
follows from \eqref{lD1} (with $j-1$).
As a consequence, for any $j\geq 1$, $ |v_j -  |u_j|| \leq C 2^{j\frac{1-d}{2}} \|u\|_W$ and by taking 
the limit as $j\to +\infty$,
\begin{equation} \label{lD3b} 
|u^0| = \lim_{j\to +\infty} \fint_{B_j} |u(z)|\, dz.
 \end{equation}
In addition,
\begin{equation} \label{lD3c} \begin{split}
\left|\fint_{B_j} |u(z)|\, dz - |u^0| \right|&  \leq |v_j - |u_j|| + ||u_j|- |u^0|| \leq |v_j - |u_j|| + |u_j- u^0| 
\leq C 2^{j\frac{1-d}{2}} \|u\|_W.
\end{split}  \end{equation}

\medskip

Let us show that $\|.\|_W$ is a norm for $W^0$. Let $u\in W^0$ be such that $\|u\|_W = 0$, 
then since $W^0 \subset W$, $u \equiv c$ is a constant function. Yet, observe that in this case, $u^0 = c$. 
The assumption $u\in W^0$ forces $u \equiv c \equiv 0$, that is $\|.\|_W$ is a norm on $W^0$.  

We now prove
that $(W^0,\|.\|_W)$ is complete. Let $(v_k)_{k\in \bN}$ be a Cauchy sequence in $W^0$. 
Since $(v_k-v_l)^0 = 0$, we deduce from \eqref{lD3c} that 
for $j\geq 1$ and $k,l\in \bN$,
\begin{equation} \label{lD3d} 
\fint_{B_j} |v_k(z) - v_l(z)|\, dz \leq C 2^{j\frac{1-d}{2}} \|v_k - v_l\|_W.
\end{equation}
Consequently, $(v_k)_{k\in \bN}$ is a Cauchy sequence in $L^1_{loc}$ and thus
there exists $u\in L^1_{loc}(\R^n)$ such that $v_k \to u$ in $L^1_{loc}(\R^n)$. 
Since $(\nabla v_k)_{k\in \bN}$ is also a Cauchy sequence in $L^2(\Omega,w)$, 
there exists $V \in L^2(\Omega,w)$ such that $\nabla v_k \to V$ in $\in L^2(\Omega,w)$. 
It follows that $v_k$ and $\nabla v_k$ converge in the sense of distribution to respectively $u$ and $V$, 
thus $u$ has a distributional derivative in $\Omega$ and $\nabla u$ equals $V \in L^2(\Omega,w)$. 
In particular $u\in W$. It remains to check that $u^0 = 0$. Yet, 
notice that
\begin{equation} \label{lD3e} 
|u^0| \leq \left|u^0 - \int_{B_j} u(z)\, dz\right| + \left|\int_{B_j} (u-v_k)(z)\, dz \right| + \left|\int_{B_j} v_k(z)\, dz\right|.
\end{equation}
The first term and the third term in the right-hand side are bounded by $C2^{j\frac{1-d}2} \|u\|_W$ and $C2^{j\frac{1-d}2} \|u_k\|_W$ respectively (thanks to \eqref{lD3}), the second by $C2^{j\frac{1-d}2} \|u-u_k\|_W$ (because of \eqref{lD3d}). By taking $k$ and $j$ big enough, we can make the right-hand side of \eqref{lD3e} as small as we want. It follows that $u^0 = |u^0| = 0$ and $u\in W^0$. The completeness of $W^0$ follows.

\medskip

It remains to check that the completion of $C^\infty_0(\R^n)$ is $W^0$. 
However, it is easy to see that
any function $u$ in $C^\infty_0(\R^n)$ satisfies $u^0 = 0$ and thus lies in
$W^0$. Together with the fact that $W^0$ is complete, we deduce that the completion of $C^\infty_0(\R^n)$ with the norm $\|.\|_W$ is included in $W^0$. The converse inclusion will hold once we establish that any function in $W^0$ can be approached in the $W$-norm by functions in $C^\infty_0(\R^n)$. Besides, thanks to Lemma~\ref{lconvol}, it is enough to prove that $u\in W^0$ can be approximated by functions in $W$ that are compactly supported in $\R^n$.

Fix $\phi \in C^\infty((-\infty,+\infty))$ such that $\phi \equiv 1$ on $(-\infty,1/2]$,
$\phi\equiv 0$ on $[1,+\infty)$. For $R>0$ define $\phi_R$ 
by
$\phi_R(x) = \phi(\ln|x|/\ln R)$. 
Observe that
that $\phi_R(x) \equiv 1$ if $|x|\leq \sqrt R$, $\phi_R(x) \equiv 0$ if $|x| \geq R$ and, for any $x\in \R^n$,
\begin{equation} \label{lD4} 
|\nabla \phi_R(x)| \leq \frac{C}{\ln R} \frac1{|x|}.
\end{equation}
The approximating functions will be the
$\phi_R u$, which are compactly supported in $\R^n$. 
Now
\begin{equation} \label{lD5} \begin{split}
\|u\phi_R - u\|_W^2 & = \|u(1-\phi_R)\|_W^2 \\
& \leq \int_{\Omega} (1-\phi_R(z))^2 |\nabla u(z)|^2 dm(z) + \int_\Omega |u(z)|^2 |\nabla \phi_R(z)|^2 dm(z) \\
& \leq \int_{|z|\geq \sqrt R} |\nabla u(z)|^2 dm(z) + \int_\Omega |u(z)|^2 |\nabla \phi_R(z)|^2 dm(z). \\
\end{split} \end{equation}
By the dominated convergence theorem, the first term of the right-hand side above converges to $0$ as $R$ 
goes to $+\infty$. It remains to check that the second term
also tends to $0$.
Set $C_j = B_{j} \setminus B_{j-1}$. We have if $R>1$,
\begin{equation} \label{lD6} \begin{split}
\int_\Omega |u(z)|^2 |\nabla \phi_R(z)|^2 dm(z) & \leq \frac{C}{|\ln R|^2} \int_{\sqrt R < |z| < R} \frac{|u(z)|^2}{|z|^2} dm(z) \\
& \leq \frac{C}{|\ln R|^2} \sum_{j=0}^{\log_2 R+1} 2^{-2j} \int_{C_j} |u(z) |^2 dm(z) \\
& \leq \frac{C}{|\ln R|^2} \sum_{j=0}^{\log_2 R+1} 2^{-2j} \int_{B_j} |u(z)|^2 dm(z) \\
& \leq \frac{C}{|\ln R|^2} \sum_{j=0}^{\log_2 R+1} 2^{-2j} \left( \int_{B_j} |u(z) - u_j|^2 dm(z) + V_j |u_j|^2 \right). \\
\end{split} \end{equation}
Lemma~\ref{lSob} gives that $ \int_{B_j} |u(z) - u_j|^2 dm(z)$ is bounded, up to a harmless constant, 
by $2^{2j} \int_{B_j} |\nabla u(z)|^2dm(z) \leq 2^{2j} \|u\|_W^2$. 
In addition, $V_j = m(B_j)$ is bounded by $C2^{j(1+d)}$ 
because of
\eqref{ISob8b} and we get 
that $|u_j|^2 \leq 2^{j(1-d)} \|u\|_W$, by \eqref{lD3}. Hence 
\begin{equation} \label{lD7} \begin{split}
\int_\Omega 
|u(z)|^2 |\nabla \phi_R(z)|^2 dm(z) 
& \leq \frac{C}{|\ln R|^2} \|u\|_W^2 \sum_{j=0}^{\log_2 R+1} 2^{-2j}\left( 2^{2j} + 2^{j(d+1)} 
2^{j(1-d)}\right)
\\& \leq  \frac{C}{|\ln R|^2} \|u\|_W^2 \sum_{j=0}^{\log_2 R+1} 1 \leq \frac{C}{|\ln R|} \|u\|_W^2,\\
\end{split} \end{equation}
which converges to 0 as $R$ goes to $+\infty$. This concludes the proof of Lemma \ref{ldensd>1}.
\ep 

As we shall see now, the situation in low dimensions is different, essentially because when $d\leq 1$,
the constant function $1$ can be approximated by functions of $C^\infty_0(\R^n)$.

\begin{lemma} \label{ldensdleq1}
Let $d\leq 1$. For any function $u$ in $W$, we can find a sequence of functions $(u_k)_{k\in \bN}$ in $C^\infty_0(\R^n)$ such that $u_k$ converges, 
in $L^1_{loc}(\R^n)$ and
and for the semi-norm $\|.\|_W$, to $u$. 
\end{lemma}

\begin{remark}
The fact that the function $1$ can be approached  with the norm $\|.\|_W$ by functions in $C_0^\infty$ means that the completion of $C_0^\infty$ with the norm $\|.\|_W$ is not a space of distributions.

We can legitimately say that the completion of $C_0^\infty$ is embedded into the space of distributions $D'=(C_0^\infty)' \supset L^1_{loc}$ if the convergence $u_k \in C_0^\infty \subset L^1_{loc}$ to $u \in W \subset L^1_{loc}$ in the norm $\|.\|_W$ implies, for $\varphi\in C_0^\infty$, that $\int u_k \varphi$ tends to $\int u \varphi$. 
Take $u_k \in C^\infty_0(\R^n)$ such that $u_k$ tends to 1 in $L^1_{loc}(\R^n)$ and $W$. Then since $\|.\|_W$ doesn't see the constants, $u_k$ tends to 0 in $W$; but the convergence of $u_k$ to 1 in $L^1_{loc}(\R^n)$ implies that $\int u_k \varphi$ tends to $\int \varphi \neq 0$ for some function $\varphi \in C^\infty_0(\R^n)$.
\end{remark}

\bp 
As before, we may assume that $0\in \Gamma$. Let us first prove that for $d \leq 1$,
the constant function $1$ 
(and thus any constant function) is the limit in $W$ and $L^1_{loc}(\R^n)$ of 
test functions. 

Choose $\phi \in C^\infty([0,+\infty))$ such that $\phi \equiv 1$ on $[0,1/2]$ and
$\phi\equiv 0$ on $[1,+\infty)$.
For $R>1$, define $\psi_R$ as $\psi_R(x) = \phi(\ln \ln |x|/\ln \ln R)$ if $|x|>1$ and $\psi_R(x) =1$ if $|x|\leq 1$. 
This cut-off function is famous for being used by Sobolev, and is useful to handle the critical case (that is, for us, $d=1$). It can be avoided if $d<1$ but we didn't want to separate the cases $d<1$ and $d=1$.
Let us return to the proof of the lemma. We have: $\psi_R(x) \equiv 1$ if $|x|\leq \exp(\sqrt{\ln R})$, $\psi_R(x) \equiv 0$ if $|x| \geq R$ and for any $x\in \R^n$ satisfying $|x|>1$,
\begin{equation} \label{lD8} 
|\nabla \psi_R(x)| \leq \frac{C}{\ln \ln R} \frac1{|x|\ln |x|}.
\end{equation}
It is easy to see that $\psi_R$ converges to 1 in $L^1_{loc}(\R^n)$ as $R$ goes to $+\infty$. We claim
that
\begin{equation} \label{lD9} 
\|\psi_R\|_W \text{ converges to 0 as $R$ goes to $+\infty$}.
\end{equation}
Let us prove \eqref{lD9}. 
As in Lemma~\ref{ldensd>1}, we write $B_j$ for $B(0,2^j)$ and $C_j$ for $B_{j} \setminus B_{j-1}$. 
Then for $R$ large,
\begin{equation} \label{lD10} \begin{split}
\|\psi_R\|_W^2 & \leq \frac{C}{|\ln \ln R|^2} 
\int_{2< |z| \leq R}
\frac1{|z|^2|\ln |z||^2} dm(z) \\ 
& \leq  \frac{C}{|\ln \ln R|^2}  
\sum_{j=1}^{+\infty}
2^{-2j} |\ln 2^{j}|^{-2} \int_{C_j} dm(z) \\
& \leq  \frac{C}{|\ln \ln R|^2}  
\sum_{j=1}^{+\infty}
\frac{1}{j^2} 2^{-2j} 2^{j(d+1)} \leq \frac{C}{|\ln \ln R|^2}. \\
\end{split} \end{equation}
Our claim follows,
and it implies that $\|1-\psi_R\|_W$ tends to $0$.

\medskip
We will prove now that any function in $W$ can be approached by functions in $C^\infty_0(\R^n)$. 
Let $u\in W$ be given. Let $u^0 = \fint_{B_0} u$ denote the average of $u$ on the unit ball.
We have just seen how to approximate $u_0$ by test functions, so it will be enough to show that
$u-u^0$ can be approached by test functions. 

For this we shall proceed as in Lemma~\ref{ldensd>1}. 
We shall use the product $\psi_R (u-u^0)$, where $\psi_R$ is the same cut-off function as above, 
and prove that $\psi_R (u-u^0)$ lies in $W$ and 
\begin{equation}\label{a7}
\lim_{R \to +\infty} \| (u-u_0) \psi_R \|_{W} = 0.
\end{equation}
Notice that $\psi_R (u-u^0)$ is compactly supported, and converges (pointwise and in $L^1_{loc}$)
to $u-u_0$. Thus, as soon as we prove \eqref{a7}, Lemma \ref{lconvol} will allow us to approximate 
$(u-u_0) \psi_R$ by smooth, compactly supported functions, and the desired approximation result will follow.

As for the proof of \eqref{a7}, of course we shall use Poincar\'e's inequality, and the 
the key point will be to get proper bounds on differences of averages of $u$.
These will not be as good as before, because now $d \leq 1$, and instead of working directly
on the balls $B_j$ we shall use strings of balls $D_j$ that do not contain the origin, so that their overlap 
is smaller.

Fix any unit vector $\xi \in \d B(0,1)$, and consider the balls
\begin{equation}\label{a5}
D = D^\xi = B(\xi,9/10) \ \text{ and, for $j \in \NN$, } \ 
D_j = D_j^\xi = B(2^j \xi,\frac{9}{10} 2^j).
\end{equation}
We will later use the $D_j^\xi$ to cover our usual annuli $C_j$, but in the mean time we fix $\xi$ and want
estimates on the numbers $m_j = \fint_{D_j} u_j$.

The Poincar\'e inequality \eqref{1.5-bis}, applied with with $p=1$, yields 
\begin{equation} \label{lD13} 
m(D_j)^{-1} \int_{D_j} |u - m_j| dm
\leq C 2^{j} \left( m(D_j)^{-1} \int_{D_j}  |\nabla u(z)|^2 dm \right)^\frac12. 
\end{equation}
Of course we have a similar estimate on $D_{j+1}$; observe also that
$D_j \cap D_{j+1}$ contains a ball $D'_j$ of radius $2^{j-2}$ (we may even take it centered at $2^j \xi$);
then
\begin{eqnarray}\label{71}
|m_j-m_{j+1}| &=& m(D'_j)^{-1} \int_{D'_j} |m_j -m_{j+1}| dm
\nn\\
&\leq& m(D'_j)^{-1} \int_{D'_j} (|u - m_j| + |u - m_{j+1}|) dm
\nn\\
&\leq& C m(D_j)^{-1} \int_{D_j} |u - m_j| dm + C m(D_{j+1})^{-1} \int_{D_{j+1}} |u - m_j| dm
\\
&\leq& C 2^{j} \left( m(D_j)^{-1} \int_{D_j \cup D_{j+1}}  |\nabla u(z)|^2 dm \right)^\frac12
\nn
\end{eqnarray}
because $m(D_j) \leq C m(D'_j)$ (and similarly for $m(D_{j+1})$), since the measure $m$ is doubling
by \eqref{doublinggen}. 
By \eqref{ISob8b}, $m(D_j) \geq C^{-1} 2^{j(d+1)}$, so \eqref{71} yields
\begin{equation}\label{72}
|m_j-m_{j+1}| \leq C 2^{-j(d-1)/2} \left(\int_{D_j \cup D_{j+1}}  |\nabla u(z)|^2 dm(z) \right)^\frac12
\end{equation}
The same estimate, run with $B_0 = B(0,1)$ and $D_0$ whose intersection also contains a large ball, yields
\begin{equation}\label{73}
|u^0-m_0| \leq C \left( \int_{B_0 \cup D_{0}}  |\nabla u(z)|^2 dm(z) \right)^\frac12
\leq C \|u\|_W.
\end{equation}
With $\xi$ fixed, the various $D_j \cup D_{j+1}$ have bounded overlap; thus by \eqref{72}
and Cauchy-Schwarz,
\begin{equation} \label{74}
\begin{split}
\left| m_{j+1} -m_0 \right|^2 
&\leq C \left(\sum_{i=0}^j 2^{-i(d-1)/2} \|\nabla u\|_{L^2(D_j\cup D_{j+1},w)} \right)^2 
\\ &\leq C(j+1) \sum_{i=0}^j 2^{i(1-d)} \|\nabla u\|_{L^2(D_j\cup D_{j+1},w)}^2 
\leq C(j+1) 2^{j(1-d)} \|u\|_W^2.
\end{split} \end{equation}
Here we used our assumption that $d \leq 1$, and we are happy about our trick with
the bounded overlap because a more brutal estimate would lead to a factor $(j+1)^2$
that would hurt us soon. Anyway, we add \eqref{73} and get that for $j \geq 0$,
\begin{equation}\label{75}
\left| m_j - u^0 \right|^2 \leq C(j+1) 2^{j(1-d)} \|u\|_W^2.
\end{equation}

We are now ready to prove \eqref{a7}. Since the first part of the proof gives that $\|u^0\psi_R\|_W$ tends to 0, 
we shall assume that $u^0 = 0$ to simplify the estimates. By Lemma \ref{lmult}, 
$(u-u_0) \psi_R = u \psi_R$ lies in $W$ and its gradient is $u \nabla \psi_R + \psi_R \nabla u$.
So we just need to show that when $R$ tends to $+\infty$,
\begin{equation} \label{lD16} 
\|u\psi_R - u\|_W \leq  \|(1-\psi_R) \nabla u\|_{L^2(\Omega,w)} + \|u\nabla \psi_R\|_{L^2(\Omega,w)}
\end{equation}
tends to $0$. The first term of the right-hand side converges to 0 as $R$ goes to $+\infty$,
thanks to the dominated convergence theorem, and for the second term we use \eqref{lD8} and the fact that
$\nabla \psi_R$ is supported in the region $Z_R$ where $\exp(\sqrt{\ln R}) \leq |x| \leq R$. Thus
\begin{equation} \label{lD17} 
\| u \nabla \psi_R\|_{L^2(\Omega,w)}^2 = \int_{\R^n} |u(z)|^2 |\nabla \psi_R(z)|^2 dm(z) 
\leq \frac{C}{|\ln \ln R|^2} \int_{Z_R} \frac{|u(z)|^2}{|z|^2 (\ln|z|)^2} dm(z) 
\end{equation}
As usual, we cut $Z_R$ into annular subregions $C_j$, and then further into balls like the $D_j$.
We start with the $C_j = B_j \sm B_{j-1}$. For $R$ large, if $C_j$ meets $Z_R$, then 
$10 \leq j \leq 1 + \log_2 R$ and
\begin{equation}\label{79}
\int_{C_j} \frac{|u(z)|^2}{|z|^2 (\ln|z|)^2} dm(z) \leq j^{-2} 2^{-2j} \int_{C_j} |u(z)|^2 dm(z).
\end{equation}
We further cut $C_j$ into balls, because we want to apply Poincar\'e's inequality.
Let the $D_j^\xi$ be as in the definition \eqref{a5}. We can find a finite set $\Xi \subset \d B(0,1)$
such that the balls $D^\xi$, $\xi\in \Xi$, cover $B(0,1) \sm B(0,1/2)$. Then for $j \geq 1$ the
$D^\xi$, $\xi\in \Xi$, cover $C_j$ and, by \eqref{lD17} and \eqref{79}, 
\begin{equation} \label{80}  
\| u \nabla \psi_R\|_{L^2(\Omega,w)}^2 
\leq \frac{C}{|\ln \ln R|^2} \sum_{j=10}^{1+ \log_2 R} j^{-2} 2^{-2j} \sum_{\xi \in \Xi}
\int_{D_j^\xi} |u(z)|^2 dm(z).
\end{equation}
Then by the Poincar\'e inequality \eqref{1.5-bis} (with $p=2$), 
\begin{equation}\label{81}
\int_{D_j^\xi} |u(z)-m_j^\xi|^2 dm(z) \leq C 2^{2j} \int_{D_j^\xi} |\nabla u(z)|^2 dm(z),
\end{equation}
where $m_j^\xi = \fint_{D_j^\xi}$ as in the estimates above. Thus 
\begin{equation}\label{82}
\int_{D_j^\xi} |u(z)|^2 dm(z) \leq C 2^{2j} \int_{D_j^\xi} |\nabla u(z)|^2 dm(z)
+ C m(D_j^\xi) (j+1) 2^{j(1-d)} \|u\|_W^2
\end{equation}
by \eqref{75} and because $u^0=0$. But $m(D_j^\xi) \leq C 2^{(d+1)j}$ by \eqref{ISob8b}, so
\begin{equation}\label{83}
\int_{D_j^\xi} |u(z)|^2 dm(z) \leq C (j+1) 2^{2j} \|u\|_W^2.
\end{equation}
We return to \eqref{80} and get that
\begin{equation} \label{84}  
\begin{split}
\| u \nabla \psi_R\|_{L^2(\Omega,w)}^2 
&\leq \frac{C}{|\ln \ln R|^2} \sum_{j=10}^{1+ \log_2 R} j^{-2} 2^{-2j} \sum_{\xi \in \Xi} (j+1) 2^{2j} \|u\|_W^2
\\
&\leq \frac{C}{|\ln \ln R|^2} \sum_{j=10}^{1+ \log_2 R} j^{-1} \|u\|_W^2
\leq \frac{C}{|\ln \ln R|} \|u\|_W^2
\end{split}
\end{equation}
because $\Xi$ is finite, and where we see that $j^{-1}$ is really useful.

We already took care of the other part of \eqref{lD16}; thus $\|u\psi_R - u\|_W$ tends to $0$.
This proves \eqref{a7} (recall that $u^0=0$), and completes our proof of Lemma \ref{ldensdleq1}.
\ep
\chapter{The Chain
Rule and Applications}

We record here some basic (and not shocking) properties concerning the derivative 
of $f \circ u$ when $u\in W$, and the fact that $uv \in W \cap L^\infty$ when $u, v \in W \cap L^\infty$.

\begin{lemma} \label{lcompo}
The following properties hold:
\begin{enumerate}[(a)]
\item Let $f\in C^1(\R)$ be such that $f'$ is bounded and let $u\in W$. Then $f\circ u \in W$ and
\begin{equation} \label{lcompo1}
\nabla (f \circ u) = f'(u) \nabla u.
\end{equation} 
Moreover, $T (f\circ u) = f\circ (Tu)$ $\sigma$-a.e. in $\Gamma$.
\item Let $u,v\in W$. Then $\max\{u,v\}$ and $\min\{u,v\}$ belong to $W$ and,
for almost every $x\in \R^n$,
\begin{equation} \label{lcompo2}
\nabla \max\{u,v\}(x) = \left\{\begin{array}{ll} \nabla u(x) 
& \text{ if } u(x) \geq v(x) \\ \nabla v(x) 
& \text{ if } v(x) \geq u(x)\end{array}\right.
\end{equation} 
and
\begin{equation} \label{lcompo3}
\nabla \min\{u,v\}(x) = \left\{\begin{array}{ll} \nabla u(x) 
& \text{ if } u(x) \leq v(x) \\ \nabla v(x) 
& \text{ if } v(x) \leq u(x).\end{array}\right.
\end{equation}
In particular, for any $\lambda\in \R$, $\nabla u = 0$ (Lebesgue) a.e. on $\{x\in \R^n, \, u(x) = \lambda\}$.

In addition, $T \max\{u,v\} = \max\{Tu, Tv\}$ and $T \min\{u,v\} = \min\{Tu, Tv\}$
$\sigma$-a.e. on $\Gamma$. Thus
$\max\{u,v\}$ and $\min\{u,v\}$ 
 lie in $W_0$ as soon as $u,v\in W_0$. 
\end{enumerate}
\end{lemma}

\begin{remark} \label{rcompo}
A consequence of Lemma~\ref{lcompo} (b) is that, for example, $|u| \in W$ (resp. $|u|\in W_0$) whenever $u\in W$ (resp. $u\in W_0$).
\end{remark}

\bp A big part of this proof follows the results from 1.18 to 1.23 in \cite{NDbook}.

\smallskip

Let us start with (a). 
More precisely, we aim for \eqref{lcompo1}. Let $f\in C^1(\R) \cap Lip(\R)$ and let $u\in W$. 
The idea of the proof is the following: 
we approximate $u$ by smooth functions $\varphi_k$, for which the result is immediate. 
Then we observe
that both $\nabla (f \circ u)$ and $f'(u) \nabla u$ are 
the limit (in the sense of distributions) of the gradient of  $f\circ \varphi_k$.

According to Lemma~\ref{lconvol}, there exists a sequence $(\varphi_k)_{k\in \bN}$ of functions in $C^\infty(\R^n) \cap W$ such that $\varphi_k \to u$ in $W$ and in $L^1_{loc}(\R^n)$. The classical (thus distributional) derivative of $f\circ \varphi_k$ is
\begin{equation} \label{lcompo4}
\nabla [f\circ \varphi_k] = f'(\varphi_k) \nabla \varphi_k.
\end{equation} 
In particular, since $\varphi_k \in W$ and $f'$ is bounded, $f\circ \varphi_k \in W$ and 
$\|f\circ \varphi_k \|_W \leq \|\varphi_k\|_W \sup|f'|$.

Notice 
that $|f(s) - f(t)| \leq |s-t|\sup|f'|$. 
Therefore, since $\varphi_k \to u$ in $L^1_{loc}(\R^n)$, for any ball $B\subset \R^n$
\begin{equation} \label{lcompo5}
\int_B |f\circ \varphi_k - f\circ u|\, dx \leq \sup|f'| \int_B |\varphi_k - u|\, dx \longrightarrow 0.
\end{equation} 
That is $f\circ \varphi_k \to f\circ u$ in $L^1_{loc}(\R^n)$, hence also in the sense of distributions.
Besides,
\begin{eqnarray} \label{lcompo6} 
&\,& \left( \int_\Omega |f'(\varphi_k) \nabla \varphi_k - f'(u) \nabla u|^2 dm \right)^\frac12 
\nn\\
&\,&\qquad \qquad \qquad  \leq \left( \int_\Omega |f'(\varphi_k) [\nabla \varphi_k - \nabla u] |^2 dm \right)^\frac12 + \left( \int_\Omega |\nabla u[f'(\varphi_k) - f'(u)]|^2 dm \right)^\frac12 
\nn\\
&\,&\qquad \qquad \qquad    \leq \sup |f'| \left( \int_\Omega |\nabla \varphi_k - \nabla u|^2 dm \right)^\frac12 + \left( \int_\Omega |\nabla u|^2 |f'(\varphi_k) - f'(u)|^2 dm \right)^\frac12. 
\end{eqnarray} 
The first term in the right-hand side is converges to 0 since $\varphi_k \to u$ in $W$. 
Besides, $\varphi_k \to u$ a.e. in $\Omega$ and 
$f'$ is continuous, so
$f'(\varphi_k) \to f'(u)$ a.e. in $\Omega$. 
Therefore, the second term also converges to 0 thanks to the dominated convergence theorem. 
It follows that $\nabla [f\circ \varphi_k] \to f'(u) \nabla u$ in $L^2(\Omega,w)$,
and hence also in the sense of distributions. We proved that $f\circ \varphi_k \to f\circ u$ and $\nabla [f\circ \varphi_k] \to f'(u) \nabla u \in L^2(\Omega,w)$ in the sense of distributions, 
and so the distributional derivative of $f\circ u$ 
lies in $L^2(\Omega,w)$ and is equal to 
$f'(u) \nabla u$. In particular,
$f\circ u \in W$. Note
that we also proved that
$f\circ \varphi_k \to f\circ u$ in $W$.

In order to finish the proof of (a), we need to prove that $T (f\circ u) = f(Tu)$ $\sigma$-a.e. in $\Gamma$.  
If $v \in W$ is also
a continuous function on $\R^n$, then it is easy to check from the definition of the trace 
that $Tv(x) = v(x)$ for every $x\in \Gamma$.
Since $f\circ \varphi_k$ and $\varphi_k$ are both continuous functions, we get that
\begin{equation} \label{lcompo112}
f\circ \varphi_k(x) = T(f\circ \varphi_k)(x) = f(T\varphi_k(x))
\ \text{ for $x\in \Gamma$ and } k\in \bN.
\end{equation}
Hence for every ball $B$ centered on $\Gamma$ and every $k\geq 0$, 
\begin{equation} \label{lcompo113} \begin{split}
\int_{B} |T(f\circ u) - f(Tu)| d\sigma & \leq \int_{B} |T(f\circ u) - T(f\circ \varphi_k)| d\sigma + \int_{B} |f(T\varphi_k) - f(Tu)| d\sigma \\
& \leq \int_{B} |T(f\circ u) - T(f\circ \varphi_k)| d\sigma + \sup|f'| \int_{B} |T\varphi_k - Tu| d\sigma.
\end{split} \end{equation}
Recall 
that each convergence $\varphi_k \to u$ and $f\circ \varphi_k \to f\circ u$ holds in both $W$ 
and $L^1_{loc}(\R^n)$. The assertion \eqref{CoA} 
then gives 
that both convergences $T\varphi_k \to T u$ and $T(f\circ \varphi_k) \to T(f\circ u)$ hold in $L^1_{loc}(\Gamma,\sigma)$.
Thus the right-hand side of \eqref{lcompo113} converges to 0 as $k$ goes to $+\infty$. 
We obtain that for every 
ball $B$ centered on $\Gamma$,
\begin{equation} \label{lcompo114}
\int_{B} |T(f\circ u) - f(Tu)| d\sigma = 0 ;
\end{equation}
in particular, $T (f\circ u) = f(Tu)$ $\sigma$-a.e. in $\Gamma$. 

\medskip

Let us turn to the proof of (b). 
Set $u^+=\max\{u,0\}$. Then
$\max\{u,v\} = (u-v)^+ + v$ and $\min\{u,v\} = u - (u-v)^+$. Thus is it enough to show that for any $u\in W$, $u^+$ lies in $W$ and satisfies \begin{equation} \label{lcompo7}
\nabla u^+(x) = \left\{\begin{array}{ll} \nabla u(x) 
& \text{ if } u(x) > 0 \\ 0 & \text{ if } u(x)\leq 0 \end{array}\right.
 \ \text{ for almost every } x\in \R^n
\end{equation}
and
\begin{equation} \label{lcompo7a}
T (u^+)  = (Tu)^+ \ \text{ $\sigma$-almost everywhere on } \Gamma.
\end{equation}
Note that in particular
\eqref{lcompo7} implies that $\nabla u = 0$ a.e. in $\{u=\lambda\}$. Indeed, since $u = \lambda + (u-\lambda)_+ - (\lambda - u)_+$, \eqref{lcompo7} implies that for almost every $x\in \Omega$,
\begin{equation} \label{lcompo8}
\nabla u(x) = \left\{\begin{array}{ll} \nabla u(x) & \text{ if } u(x) \neq \lambda \\ 0 & \text{ if } u(x)= \lambda. \end{array}\right.
\end{equation}

Let us prove 
the claim \eqref{lcompo7}. 
Define $f$ and $g= \1_{(0,+\infty)}$ by $f(t) = \max\{0,t\}$ and $g(t) = 0$
 when $t\leq 0$ and $g(t) =1$ when $t>0$.  Our aim
is to approximate $f$ by an increasing sequence of $C^1$-functions and then to conclude by using (a) and the monotone convergence theorem. Define for any integer $j\geq 1$ the function $f_j$ by
\begin{equation} \label{lcompo9}
f_j(t) = \left\{\begin{array}{ll} 0 & \text{ if } t\leq 0 \\ \frac{j}{j+1} t^{\frac{j+1}{j}} & \text{ if } 0 \leq t \leq 1 \\ t - \frac1{j+1} & \text{ if } t\geq 1. \end{array}\right.
\end{equation}
Notice that
$f_j$ is non-negative and $(f_j)$ is a nondecreasing sequence that converges pointwise to $f$.
Consequently, $f_j \circ u \geq 0$ 
and $(f_j \circ u)$  is a nondecreasing sequence that converges pointwise to 
$f\circ u = u^+ \in L^1_{loc}(\R^n)$. The 
monotone convergence theorem implies that $f_j \circ u \to u^+$ in $L^1_{loc}(\R^n)$. 

Moreover, $f_j$ lies in  
$C^1(\R)$ and its derivative is
\begin{equation} \label{lcompo10}
f'_j(t) = \left\{\begin{array}{ll} 0 & \text{ if } t\leq 0 \\ t^{\frac{1}{j}} & \text{ if } 0 \leq t \leq 1 \\ 1 & \text{ if } t\geq 1. \end{array}\right.
\end{equation}
Thus $f'_j$ is bounded and part (a) of the lemma implies $f_j \circ u \in W$ and 
$\nabla (f_j \circ u) = f'_j(u) \nabla u$ 
almost everywhere on $\R^n$. In addition, $f'_j$ converges to $g$ pointwise everywhere, so
\begin{equation} \label{lcompo11}
\nabla (f_j \circ u) = f'_j(u) \nabla u \to v:= (g\circ u) \nabla u = 
\left\{\begin{array}{ll} \nabla u & \text{ if } u>0 \\ 0 & \text{ if } u\leq 0 \end{array}\right.
\end{equation}
almost everywhere (i.e., wherever $\nabla (f_j \circ u) = f'_j(u) \nabla u$).
The convergence also occurs in $L^2(\Omega,w)$ and in $L^1_{loc}(\R^n)$, 
because $|\nabla (f_j \circ u)| \leq |\nabla u|$ and by the dominated convergence theorem, 
and therefore also in the sense of distribution.
Since $f_j \circ u$ converges to $u^+$ pointwise almost everywhere and hence (by the dominated
convergence theorem again) in $L^1_{loc}$ and in the sense of distributions, we get that 
$v = (g\circ u) \nabla u$ is the distribution derivative of $u^+$.
This completes the proof of \eqref{lcompo7}.

\medskip

Finally, let us establish \eqref{lcompo7a}. 
The plan is to prove that we can find smooth functions $\varphi_k$ such that $\varphi_k^+$ 
converges (in $L^1_{loc}(\Gamma,\sigma)$) to both $Tu^+$ and $(Tu)^+$. 
We claim that for $u\in W$ and any sequence $(u_k)$ in $W$, the following implication holds true:
\begin{equation} \label{lcompo12}
u_k \to u \  \text{ 
pointwise a.e.
and in }W \ \Longrightarrow \ u_k^+ \to u^+ \  \text{ 
pointwise a.e. and 
in }W.
\end{equation}
First we assume the claim and prove \eqref{lcompo7a}. 
With the help of Lemma~\ref{lconvol}, take $(\varphi_k)_{k\in \bN}$ be a sequence of functions in 
$C^\infty(\R^n)$ such that $\varphi_k \to u$ 
in $W$, and in $L^1_{loc}(\R^n)$. We may also replace $(\varphi_k)$ by a subsequence, and
get that $\varphi_k \to u$ pointwise a.e.
The claim \eqref{lcompo12} implies 
that $\varphi_k^+ \to u^+$ in $W$. 
In addition, $\varphi_k^+ \to u^+$ in $L^1_{loc}(\R^n)$, for instance because 
$\varphi_k$ tends to $u$ in $L^1_{loc}(\R^n)$ and by the estimate 
$|\varphi_k^+ - u^+| \leq |\varphi_k - u|$. 

Thus we may apply \eqref{CoA}, and we get that $T \varphi_k^+$ tends to $Tu^+$ in $L^1_{loc}(\Gamma)$.
Since $\varphi_k^+$ is continuous, $T \varphi_k^+ = \varphi_k^+$ and 
\begin{equation}\label{19}
\ \text{$\varphi_k^+$ tends to $Tu^+$ in $L^1_{loc}(\Gamma)$.}
\end{equation}

We also need to check that $\varphi_k^+$  converges to $(Tu)^+ $.
Notice that \eqref{CoA} also implies that $\varphi_k \to Tu$ in $L^1_{loc}(\Gamma,\sigma)$.
Together with the easy fact that $|a^+ - b^+| \leq |a-b|$ for $a, b \in \R$, this proves that 
$\varphi_k^+ \to (Tu)^+$ in $L^1_{loc}(\Gamma,\sigma)$. 

We just proved that $\varphi_k^+$ converges to both $T(u^+)$ and $(Tu)^+$ in $L^1_{loc}(\Gamma,\sigma)$.
By uniqueness of the limit, $T(u^+) = (Tu)^+$ $\sigma$-a.e. in $\Gamma$, 
as needed for \eqref{lcompo7a}. Thus the proof of the lemma will 
be complete as soon as we establish the claim \eqref{lcompo12}. 

\smallskip
First notice that
$|u_j^+ - u^+| \leq |u_j - u|$ and thus the a.e. pointwise convergence of $u_j \to u$ yields the a.e. pointwise convergence $u_j^+\to u^+$.
Let $g$ denote the characteristic function of $(0,+\infty)$; then by \eqref{lcompo7} 
 \begin{equation} \label{lcompo13} \begin{split}
& \left( \int_\Omega |\nabla u_j^+ - \nabla u^+|^2 dm \right)^\frac12 = \left( \int_\Omega |g(u_j)\nabla u_j - g(u)\nabla u|^2 dm \right)^\frac12 \\
&\qquad \qquad \qquad  \leq \left( \int_\Omega |g(u_j) [\nabla u_j - \nabla u] |^2 dm \right)^\frac12 + \left( \int_\Omega |\nabla u[g(u_j) - g(u)]|^2 dm \right)^\frac12 \\
&\qquad \qquad \qquad    \leq \left( \int_\Omega |\nabla u_j - \nabla u|^2 dm \right)^\frac12 + \left( \int_\Omega |\nabla u|^2 |g(u_j) - g(u)|^2 dm \right)^\frac12. \\
\end{split} \end{equation} 
The first term in the right-hand side converges to 0 since $u_j \to u$ in $W$. 
Call $I$ the second term;
$I$ is finite, since $u\in W$ and $|g(u_j) - g(u)| \leq 1$. 
Moreover,  thanks to \eqref{lcompo8}, $\nabla u = 0$ a.e. on $\{u=0\}$. So the square of $I$ can be written
 \begin{equation} \label{lcompo14} \begin{split}
I^2 & = \int_{\big\{|u|>0\big\}} |\nabla u|^2 |g(u_j) - g(u)|^2 dm.
\end{split} \end{equation} 
Let $x\in \big\{|u|>0\big\}$ be such that $u_j(x)$ converges to $u(x) \neq 0$;
then there exists $j_0\geq 0$ such that for 
$j\geq j_0$ the sign of $u_j(x)$ is the same as the 
sign of $u(x)$. That is, 
$g(u_j)(x)$ converges to $g(u)(x)$. 
Since $u_j \to u$ a.e. in $\Omega$, the previous argument implies that 
$g(u_j) \to g(u)$
a.e. in $\{|u|>0\}$. 
Then $I^2$ converges to $0$, by  
the dominated convergence theorem.
Going back to \eqref{lcompo13}, we obtain that $u_j^+ \to u^+$ in $W$, which concludes 
our proof of \eqref{lcompo12}; Lemma~\ref{lcompo} follows. 
\ep

\begin{lemma} \label{lalg}
Let $u,v \in W \cap L^\infty(\Omega)$. 
Then $uv \in W \cap L^\infty(\Omega)$, $\nabla (uv) = v \nabla u + u \nabla v$, and $T(uv) = Tu\cdot Tv$.
\end{lemma}

\bp Without loss of generality, we can assume that $\|u\|_\infty,\|v\|_\infty \leq 1$. 
The fact that $uv \in L^\infty(\Omega)$ is 
immediate. Let us now prove that $uv \in W$. 
According to Lemma~\ref{lconvol}, there exists two sequences $(u_j)_{j\in \bN}$ and $(v_j)_{j\in \bN}$ 
of functions in $C^\infty(\R^n) \cap W$ such that $u_j \to u$ and $v_j \to v$ in $W$, 
in $L^1_{loc}(\R^n)$, and pointwise.
Besides, the construction of $u_j,v_j$ given by Lemma~\ref{lconvol} allows us to assume that $\|u_j\|_\infty\leq \|u\|_\infty \leq 1$ and $\|v_j\|_\infty \leq \|v\|_\infty \leq 1$. The distributional derivative of $u_jv_j$ equals the classical derivative, which is
\begin{equation}\label{lalg1}
\nabla (u_jv_j) = v_j \nabla u_j + u_j \nabla v_j.
\end{equation}
Since $u_j$ and $v_j$ lie in $W$, \eqref{lalg1} says that 
$u_j v_j \in W$. The bound
\begin{equation}\label{lalg2} \begin{split}
\int_B |u_jv_j - uv| \leq \int_B |u_j| |v_j-v| + \int_B |v| |u_j - u| \leq \|v_j-v\|_{L^1(B)} + \|u_j-u\|_{L^1(B)} ,
\end{split} \end{equation}
which holds for any ball $B \subset \R^n$, shows that $u_jv_j \to uv$ in $L^1_{loc}(\R^n)$. Moreover,
\begin{equation}\label{lalg3} \begin{split}
& \left(\int_B |(u_j\nabla v_j + v_j \nabla u_j) - (u\nabla v + v \nabla u)|^2 dm \right)^\frac12  \\
& \qquad \qquad \leq \left(\int_B |u_j\nabla v_j- u\nabla v |^2 dm\right)^\frac12 + \left(\int_B |v_j \nabla u_j - v \nabla u|^2 dm \right)^\frac12 \\
& \qquad \qquad \leq  \left(\int_B |u_j|^2 |\nabla v_j- \nabla v |^2 dm\right)^\frac12 +  \left(\int_B |u_j- u|^2 |\nabla v |^2 dm\right)^\frac12 \\
& \qquad \qquad \qquad \qquad + \left(\int_B |v_j|^2  |\nabla u_j - \nabla u|^2 dm \right)^\frac12 + \left(\int_B |v_j - v|^2  |\nabla u|^2 dm \right)^\frac12. \\
\end{split} \end{equation}
The first and third terms in the right-hand side converge to 0 as $j$ goes to $+\infty$, because 
$|u_j|,|v_j| \leq 1$ and since $u_j \to u$ and $v_j \to v$ in $W$. 
The second and forth terms also converge to 0 thanks to the dominated convergence theorem 
(and the fact that $u_j \to u$ and $v_j \to v$ pointwise a.e.).
We deduce that $\nabla (u_jv_j) = u_j\nabla v_j + v_j \nabla u_j \to u\nabla v + v \nabla u$ 
in $L^2(\Omega,w)$. By the
uniqueness of the distributional derivative, $\nabla (uv) = u\nabla v + v \nabla u \in L^2(\Omega,w)$. 
In particular, $uv \in W$. Note that we also proved that $u_jv_j \to uv$ in $W$.

\medskip
It remains to prove that $T(uv) = Tu \cdot Tv$. 
Since $u_jv_j$ is continuous and $u_jv_j \to uv$ in $W$ and $L^1_{loc}(\R^n)$, then by  
\eqref{CoA}, $u_jv_j = T(u_jv_j) \to T(uv)$ in $L^1_{loc}(\Gamma,\sigma)$. 
Moreover, for any ball $B$ centered on $\Gamma$,
\begin{equation}\label{lalg4} \begin{split}
\int_B |u_jv_j - Tu\cdot Tv| d\sigma & \leq \int_B |u_j| |v_j-Tv| d\sigma + \int_B |u_j - Tu| |Tv| d\sigma \\
& \leq \int_B |v_j-Tv| d\sigma + \int_B |u_j - Tu| d\sigma \\
\end{split} \end{equation}
where the last line holds because $|u_j| \leq 1$ and $|Tv| \leq \sup|v| \leq 1$, where the later bound
either follows from Lemma \ref{lcompo} or is easily deduced
from the definition of the trace.
By construction, $u_j \to u$ and $v_j \to v$ in $W$ and $L^1_{loc}(\R^n)$. 
Then by \eqref{CoA}
the right-hand side terms in \eqref{lalg4} converge to 0.

We proved that $u_j v_j$ converges in $L^1_{loc}(\Gamma,\sigma)$ to both $T(uv)$ and $Tu \cdot Tv$. By uniqueness of the limit, $T(uv) = Tu \cdot Tv$ $\sigma$-a.e. in $\Gamma$. 
Lemma \ref{lalg} follows.
\ep
\chapter{The Extension Operator}
\label{Sextension} 
 
The main point of this section is the construction of our extension operator $E : H \to W$,
which will be done naturally, with the Whitney extension scheme that uses dyadic cubes.

Our main object will be a function $g$ on $\Gamma$, that typically lies in $H$ or in 
$L^1_{loc}(\Gamma,\sigma)$.  We start with the \ub{Lebesgue density} result for 
$g \in L^1_{loc}(\Gamma,\sigma)$ that was announced in the introduction.
\medskip

\begin{lemma} \label{lLdH}
For any $g\in L^1_{loc}(\Gamma,\sigma)$ and 
$\sigma$-almost all $x\in \Gamma$, 
\begin{equation} \label{3.1}
\lim_{r \to 0} \fint_{\Gamma \cap B(x,r)} |g(y)-g(x)| d\sigma(y) = 0.
\end{equation}
\end{lemma}

\bp
Since $(\Gamma,\sigma)$ satisfies \eqref{1.1}, the space $(\Gamma,\sigma)$ equipped with the metric induced by $\R^n$ is a doubling space. 
Indeed, let $B$ be a ball centered on $\Gamma$. According to \eqref{1.1}, 
\begin{equation}
\sigma(2B) \leq 2^d C_0 r^d \leq 2^d C_0^2 r^d \sigma(B).
\end{equation}
From there, the lemma is only a consequence of the Lebesgue differentiation theorem in doubling spaces (see for example \cite[Sections 2.8-2.9]{Federer}).
\ep

\begin{remark} \label{rHinL1loc}
We claim that $H \subset L^1_{loc}(\Gamma,\sigma)$, and hence \eqref{3.1} holds 
for $g\in H$ and $\sigma$-almost every $x\in \Gamma$.
Indeed, let $B$ be a ball centered on $\Gamma$, then 
a brutal estimate yields
\begin{equation}
\int_B \int_B |g(x) - g(y)| d\sigma(x) d\sigma(y) \leq C_B \left(\int_B \int_B |g(x) -g(y)|^2 d\sigma(x) d\sigma(y)\right)^\frac12 \leq C_B \|g\|_H < +\infty.
\end{equation}
Hence for $\sigma$-almost every $x\in B \cap \Gamma$,   $\int_B |g(x) - g(y)| d\sigma(y)<+\infty$. 
In particular, since $\sigma(B) > 0$, there exists $x\in B \cap \Gamma$ such that 
$\int_B |g(x) - g(y)| d\sigma(y)<+\infty$. 
We get that $g\in L^1(B,\sigma)$, and our claim follows.
\end{remark}

Let us now start 
the construction of the \ub{extension operator} 
$E : H \to W$. We
proceed as for the Whitney extension theorem, with only a minor
modification because averages will be easier to manipulate than specific values of $g$.

We shall use 
the family $\cW$ of dyadic Whitney cubes constructed as in the first pages of 
\cite{Stein}
and the partition of unity $\{ \varphi_Q \}, Q \in \cW$, that is usually associated to $\cW$.
Recall that $\cW$ is the family of maximal dyadic cubes $Q$ (for the inclusion) such that
$20Q \subset \Omega$, say, and the $\varphi_Q$ are smooth functions such that 
$\varphi_Q$ is supported in $2Q$, $0 \leq \varphi_Q \leq 1$, $|\nabla \varphi_Q| \leq C \diam(Q)^{-1}$,
and $\sum_Q \varphi_Q = \1_\Omega$.

Let us record a few of the simple properties of $\cW$. These are simple, but yet we refer to \cite[Chapter VI]{Stein} for
details.
It will be convenient to denote by $r(Q)$ the side length of the dyadic cube $Q$.
Also set $\delta(Q) = \dist(Q,\Gamma)$.
For $Q \in \cW$, we select a point $\xi_Q \in \Gamma$ such that $\dist(\xi_Q,Q) \leq 2 \delta(Q)$,
and set 
\begin{equation}\label{a7.6}
B_Q = B(\xi_Q,\delta(Q)).
\end{equation}

If $Q, R \in \cW$ are such that $2Q$ meets $2R$, then $r(R) \in \{ \frac12 r(Q), r(Q), 2r(Q) \}$;
then we can easily check that $R \subset 8Q$. Thus $R$ is a dyadic cube in $8Q$ whose side length 
is bigger than $\frac12 r(Q)$; there exist at most $2 \cdot 16^n$ dyadic cubes like this. 
This proves that there is a constant $C = C(n)$ such that for $Q\in \cW$,
\begin{equation}\label{a7.7}
\ \text{the number of cubes $R \in \cW$ such that $2R \cap 2Q \neq \emptyset$ is at most $C$.}
\end{equation}

The operator $E$ is defined on functions in $L^1_{loc}(\Gamma,\sigma)$ by
\begin{equation}\label{3.2}
Eg(x) = \sum_{Q \in \cW} \varphi_Q(x) y_Q,
\end{equation}
where we set
\begin{equation}\label{3.3}
y_Q = \fint_{B_Q} g(z) d\sigma(z),
\end{equation}
with $B_Q$ as in \eqref{a7.6}. For the extension of Lipschitz functions, for instance, one would 
take $y_Q = g(\xi_Q)$, but here we will use the extra regularity of the averages.

Notice that $Eg$ is continuous on $\Omega$, because the sum in \eqref{3.2} is locally finite.
Indeed, if $x\in \Omega$ and $Q \in \cW$ contains $x$, \eqref{a7.7} says that there are at most
$C$ cubes $R \in \cW$ such that $\varphi_R$ does not vanish on $2Q$; then the 
restriction of $Eg$ to $2Q$ is a finite sum of continuous functions.
Moreover, if $g$ is continuous on $\Gamma$, then $Eg$ is continuous on the whole $\R^n$.
We refer to \cite[Proposition VI.2.2]{Stein} for the easy proof.

\begin{theorem} \label{tExt}
For any $g\in L^1_{loc}(\Gamma,\sigma)$ 
(and by Remark \ref{rHinL1loc}, this applies to $g\in H$),
\begin{equation}\label{a7.11}
T(Eg) = g \quad \text{ $\sigma$-a.e. in $\Gamma$.}
\end{equation}
Moreover, there exists $C>0$ such that for any $g\in H$,
\begin{equation} \label{a7.12}
\|Eg\|_W \leq C \|g\|_H.
\end{equation}
\end{theorem}

\smallskip 

\bp
Let $g \in L^1_{loc}$ be given, and set $u = Eg$. 
We start the proof with the verification of \eqref{a7.11}.
Recall that by definition of the trace, 
\begin{equation}\label{3.12}
T(E(g))(x) = \lim_{r \to 0} \fint_{B(x,r)} u(z) \, dz
\end{equation}
for $\sigma$-almost every $x\in \Gamma$; 
we want to prove that this limit is $g(x)$ for almost every $x\in \Gamma$, 
and we can restrict to the case when $x$ is a Lebesgue point for $g$ (as in \eqref{3.1}).

Fix such an $x \in \Gamma$ and $r > 0$. Set $B = B(x,r)$, then
\begin{equation}\label{3.13}
\Big| \fint_{B(x,r)} u(z)\, dz - g(x) \Big| \leq \fint_{B} |u(z)-g(x)| dz
\leq C r^{-n} \sum_{R\in \cW(B)}
\int_R |u(z)-g(x)| dz,
\end{equation}
where we denote by $\cW(B)$ the set of cubes $R \in \cW$ that meet $B$.
 
Let $R \in \cW$ and $z\in R$ be given. 
Recall from \eqref{3.2} that $u(z) = \sum_{Q \in \cW} \varphi_Q(z) y_Q$;
the sum has less than $C$  terms, corresponding to cubes 
$Q \in \cW$ such that $z\in 2Q$.
If $Q$ is such a cube, we have seen that $\frac12 r(R) \leq r(Q) \leq 2 r(R)$, and 
since $\delta(R) \geq 10 r(Q)$ because $20Q \subset \Omega$, a small computation with \eqref{a7.6}
yields that $B_Q \subset 100B_R$. Hence
\begin{equation}\label{3.14}
|y_Q - g(x)| = \Big|\fint_{B_Q} g\,  d\sigma - g(x) \Big| \leq \fint_{B_Q} |g-g(x)| d\sigma
\leq C \fint_{100B_R} |g-g(x)| d\sigma.
\end{equation}
Since $u(z)$ is an average of such $y_Q$, we also get that 
$|u(z)-g(x)| \leq C \fint_{100B_R} |g-g(x)| d\sigma$, and \eqref{3.13} yields
\begin{equation}\label{3.15}
\Big| \fint_{B(x,r)} u\, dz - g(x) \Big| \leq C r^{-n} \sum_{R\in \cW(B)} |R| \fint_{100B_R} |g-g(x)| d\sigma.
\end{equation}
Notice that $\delta(R) = \dist(R,\Gamma) \leq \dist(R,x) \leq r$ because 
$R$ meets $B = B(x,r)$ and $x\in \Gamma$, so, by definition of $\cW$, the sidelength of $R$ 
is such that $r(R) \leq Cr$.
Let $\cW_k(B)$ be the collection of $R \in \cW(B)$ such that $r(R) = 2^k$. For each $k$,
the balls $100B_R$, $R \in \cW_k(B)$ have bounded overlap 
(because the cubes $R$ are essentially
disjoint and they have the same sidelength), and they are 
contained in $B' = B(x,Cr)$. Thus
\begin{eqnarray}\label{3.16}
\sum_{R\in \cW_k(B)} |R| \fint_{100B_R} |g-g(x)| d\sigma
&\leq& C 2^{nk} 2^{-dk} \sum_{R\in \cW_k(B)} \int_{100B_R} |g-g(x)| d\sigma
\nn\\
&\leq& C 2^{(n-d)k} \int_{B'} |g-g(x)| d\sigma.
\end{eqnarray}
We may sum over $k$ (because $2^k = r(R) \leq Cr$ when $R\in \cW_k(B)$, and the exponent
$n-d$ is positive). We get that
\begin{equation}\label{3.17}
\Big| \fint_{B(x,r)} u\, dz - g(x) \Big| \leq C r^{-n} \sum_k 2^{(n-d)k} \int_{B'} |g-g(x)| d\sigma
\leq C r^{-d}\int_{B'} |g-g(x)| d\sigma.
\end{equation}
If $x$ is a Lebesgue point for $g$,
\eqref{3.1} says that both sides of \eqref{3.17} tend to $0$ when $r$ tends to $0$. 
Recall from \eqref{3.12} that for almost every $x\in \Gamma$, 
$T(E(g))(x)$ is the limit of $\fint_{B(x,r)} u$; if in addition $x$ is a Lebesgue point, 
we get that $T(E(g))(x) = g(x)$. This completes our proof of \eqref{a7.11}.

\medskip
Now we show that for $g\in H$, $u\in W$ and even $\|u\|_W \leq C \|g\|_H$.
The fact that $u$ is locally integrable in $\Omega$ is obvious
($u$ is continuous there because
the cubes $2Q$ have bounded overlap), 
and similarly the distribution derivative is locally 
integrable, and given by
\begin{equation}\label{3.4}
\nabla u(x) = \sum_{Q \in \cW} y_Q \nabla\varphi_Q(x) 
= \sum_{Q \in \cW} [y_Q-y_R] \nabla\varphi_Q(x),
\end{equation}
where in the second part (which will be used later) we can pick for $R$
any given cube (that may depend on $x$), for instance, one to the cubes of $\cW$ that contains $x$,
and the identity holds because $\sum_Q \nabla\varphi_Q = \nabla(\sum_Q \varphi_Q) = 0$.
Thus the question is merely the computation of 
\begin{eqnarray}\label{3.5}
\|u\|^2_W &=& \int_\Omega |\nabla u(x)|^2 w(x) dx
= \sum_{R\in \cW} \int_R |\nabla u(x)|^2 w(x) dx
\nn\\
&\leq& C \sum_{R\in \cW} \delta(R)^{d+1-n}\int_R |\nabla u(x)|^2 dx
\end{eqnarray}
(because $w(x) = \delta(x)^{d+1-n} \leq \delta(R)^{d+1-n}$ when $x\in R$).
Fix $R \in \cW$, denote by $\cW(R)$ the set of cubes $Q \in \cW$ such that $2Q$ meets $R$,
and observe that for $x\in R$,
\begin{equation}\label{3.6}
|\nabla u(x)| \leq \sum_{Q \in \cW(R)} \big| [y_Q-y_R] \nabla\varphi_Q(x)\big|
\leq C \delta(R)^{-1} \sum_{Q \in \cW(R)} \big|y_Q-y_R \big|
\end{equation}
because $|\nabla\varphi_Q(x)| \leq C \delta(Q)^{-1} \leq C \delta(R)^{-1}$ by definitions and the standard
geometry of Whitney cubes. In turn,
\begin{eqnarray}\label{3.7}
\big|y_Q-y_R \big| 
&\leq& \fint_{\Gamma \cap B_Q}\fint_{\Gamma \cap B_R} |g(x)-g(y)| d\sigma(x)d\sigma(y)
\nn\\
&\leq& \Big\{\fint_{\Gamma \cap B_Q}\fint_{\Gamma \cap B_R} 
|g(x)-g(y)|^2 d\sigma(x)d\sigma(y)\Big\}^{1/2}
\nn\\
&\leq& C \delta(R)^{-d} \Big\{\int_{\Gamma \cap B_R}\int_{\Gamma \cap 100B_R} 
|g(x)-g(y)|^2 d\sigma(x)d\sigma(y)\Big\}^{1/2}
\end{eqnarray}
by \eqref{1.1} and because $B_Q \subset 100B_R$. Thus by \eqref{3.6}
\begin{eqnarray}\label{3.8}
\int_R |\nabla u(x)|^2 dx 
&\leq& C |R| \delta(R)^{-2} \delta(R)^{-2d} \int_{\Gamma \cap B_R}
\int_{\Gamma \cap 100B_R}  |g(x)-g(y)|^2 d\sigma(x)d\sigma(y)
\nn\\
&\leq& C \delta(R)^{n-2d-2} \int_{\Gamma \cap B_R}
\int_{\Gamma \cap 100B_R}  |g(x)-g(y)|^2 d\sigma(x)d\sigma(y)
\end{eqnarray}
because $\cW(R)$ has at most $C$ elements. We multiply by $\delta(R)^{d+1-n}$, sum 
over $R$, and get that
\begin{eqnarray}\label{3.9}
\|u\|^2_W &\leq& C \sum_{R\in \cW} \delta(R)^{-d-1} \int_{\Gamma \cap B_R}
\int_{\Gamma \cap 100B_R}  |g(x)-g(y)|^2 d\sigma(x)d\sigma(y)
\nn\\
&\leq& C \int_\Gamma\int_\Gamma |g(x)-g(y)|^2 h(x,y) d\sigma(x)d\sigma(y),
\end{eqnarray}
where we set
\begin{equation}\label{3.10}
h(x,y) = \sum_{R} \delta(R)^{-d-1},
\end{equation}
and we sum over $R \in \cW$ such that $x\in B_R$ and $y\in 100B_R$.
Notice that $|x-y| \leq 101\delta(R)$, so we only sum over $R$ such that 
$\delta(R) \geq |x-y|/101$. 

Let us fix $x$ and $y$, and evaluate $h(x,r)$.
For each scale (each value of $\diam(R)$), there are less than 
$C$ cubes $R\in \cW$ that are possible, because $x\in B_R$ implies that 
$\dist(x,R) \leq 3\delta(R)$. So the contribution of the cubes for which $\diam(R)$ is of the
order $r$ is less than $C r^{-d-1}$. We sum over the scales (larger than $C^{-1} |x-y|$)
and get less than $C |x-y|^{-d-1}$.
That is, $h(x,y) \leq C |x-y|^{-d-1}$ and
\begin{equation}\label{3.11}
\|u\|^2_W \leq C \int_\Gamma\int_\Gamma {|g(x)-g(y)|^2 \over |x-y|^{d+1}} d\sigma(x)d\sigma(y)
= C \|g\|_H^2,
\end{equation}
as needed for \eqref{a7.12}. Theorem \ref{tExt} follows.
\ep

\ms

We end the section with the density in $H$ of (traces of) smooth functions.

\begin{lemma} \label{ldensH}
For every $g\in H$,
we can find a sequence $(v_k)_{k\in \bN}$ 
in $C^\infty(\R^n)$ such that $Tv_k$
converges to $g$ in $H$ in $L^1_{loc}(\Gamma,\sigma)$, and $\sigma$-a.e. pointwise.
\end{lemma}

Notice that since $v_k$ is continuous across $\Gamma$, $Tv_k$ is the restriction of $v_k$ to $\Gamma$,
and we get the density in $H$ of continuous functions on $\Gamma$, for the same three convergences.

\smallskip
\bp
The quickest way to prove this will be 
to use Theorem \ref{tTr}, Theorem \ref{tExt} and the results in  Section~\ref{Scompleteness}.

Let $g\in H$
be given. Let $\rho_\epsilon$ be defined as in Lemma~\ref{lconvol}, and set 
$v_\varepsilon = \rho_\epsilon * Eg$ and $g_\varepsilon = Tv_\varepsilon$.
Theorem \ref{tExt} says that $Eg \in W$; then by Lemma~\ref{lconvol}, $v_\varepsilon = \rho_\epsilon * Eg$ 
lies in $C^\infty(\R^n) \cap W$. We still need to check that $g_\varepsilon$ tends to $g$
for the three types of convergence.

By Lemma~\ref{lconvol}, $v_\varepsilon = \rho_\epsilon*Eg$ converges to $Eg$ 
in $L^1_{loc}(\R^n)$ and in $W$, and then \eqref{CoA} implies that 
$g_\epsilon = Tv_\varepsilon$ tends to $g = T(Eg)$ in $L^1_{loc}(\Gamma,\sigma)$.

The convergence in $H$ is the consequence of the  
bounds 
\begin{equation} \label{ldensH1}
\|g-g_\epsilon\|_H \leq \|T(Eg - v_\epsilon)\|_H \leq C \|Eg - v_\epsilon\|_W
\end{equation}
that come from Theorem \ref{tExt},
plus the fact that the 
right-hand side converges to 0 thanks to Lemma~\ref{lconvol}.

For the a.e. pointwise convergence, let us cheat slightly: we know that the $g_\varepsilon$
converge to $g$ in $L^1_{loc}(\Gamma,\sigma)$; we can then use the diagonal process to
extract a sequence of $g_\varepsilon$ that converges pointwise a.e. to $g$, which is enough
for the lemma.
\ep

\chapter{Definition of Solutions}
\label{Ssolutions}

The aim of the following sections is to define the  harmonic measure on $\Gamma$. We follow the presentation of Kenig \cite[Sections 1.1 and 1.2]{KenigB}.
\medskip

In addition to $W$, we introduce a local
version of $W$. Let $E \subset \R^n$ be an open set. The set of function $\WW(E)$ is defined as 
\begin{equation} \label{defWE}
\WW(E) = \{f \in L^1_{loc}(E), \, \varphi f \in W \, \text{ for all } \varphi \in C_0^\infty(E)\}
\end{equation}
where the function $\varphi f$ is seen 
as a function on $\R^n$
(since $\varphi f$ is compactly supported in $E$, it can be extended by 0 outside $E$).
The inclusion $W \subset \WW(E)$ is given by 
Lemma~\ref{lmult}. 

Let us discuss a 
bit more about our newly defined spaces. First, 
we claim that
\begin{equation} \label{defWE2}
\WW(E)  \subset  \{f \in L^1_{loc}(E), \, \nabla f \in L^2_{loc}(E,w)\},
\end{equation}
where here $\nabla f$ denotes 
the distributional derivative 
of $f$ in $E$.
To see this,
let $f \in \WW(E)$ be given; we just need to see that $\nabla f \in L^2(K,w)$
for any relatively compact open subset $K$ of $E$. Pick 
$\varphi \in C^\infty_0(E)$ such that $\varphi \equiv 1$ on $K$, and observe that
$\varphi f \in W$ by \eqref{defWE}, so Lemma~\ref{W=barW} says that $\varphi f$ has 
a distribution derivative (on $\R^n$) that lies in $L^2(\R^n, w)$. Of course the two distributions
$\nabla f$ and $\nabla (\varphi f)$ coincide near $K$, so $\nabla f \in L^2(K,w)$ and our claim follows.

The reverse inclusion $\WW(E)  \supset  \{f \in L^1_{loc}(E), \, \nabla f \in L^2_{loc}(E,w)\}$ surely holds, but we will not use it.
Note that thanks to Lemma~\ref{W=barW}, we do not need to worry, even locally as here, 
about the difference between having a derivative in $\Omega \cap E$ that lies in $L^2_{loc}(E,w)$
and the apparently stronger condition of having a derivative in $E$ that lies in $L^2_{loc}(E,w)$. 
Also note that $\WW(\R^n) \neq W$; the difference is that $W$ demands some decay of $\nabla u$ at infinity,
while $\WW(\R^n)$ doesn't. 

\begin{lemma} \label{defTrWE}
Let $E\subset \R^n$ be an open set.
For every function $u\in \WW(E)$, we can define the trace of $u$ on $\Gamma \cap E$ by
\begin{equation} \label{defTrWE2}
Tu (x) = \lim_{r\to 0} \fint_{B(x,r)} u(z) \, dz
\qquad \text{ for $\sigma$-almost every } x\in \Gamma \cap E,
\end{equation}
and $Tu \in L^1_{loc}(\Gamma \cap E,\sigma)$.
 Moreover, for 
every choice of $f\in \WW(E)$ and $\varphi \in C^\infty_0(E)$, 
 \begin{equation} \label{defTrWE1}
T(\varphi u)(x) = \varphi(x) Tu(x)  \ \text{ for $\sigma$-almost every } x\in \Gamma \cap E.
 \end{equation}
\end{lemma}

\smallskip
\bp 
The existence of $\lim_{r\to 0} \fint_{B(x,r)} u(z)dz$ is easy. If $B$ is any relatively compact ball in $E$,
we can pick $\varphi \in C^\infty_0(E)$ such that $\varphi \equiv 1$ near $B$. 
Then $\varphi u \in W$, and the analogue of \eqref{defTrWE2} for $\varphi u$ comes
with the construction of the trace. This implies the existence
of the same limit for $f$, almost everywhere in $\Gamma \cap B$.

Next we check that $Tu \in L^1_{loc}(\Gamma \cap E,\sigma)$.
Let $K$ be a compact set in $E$; we want to show that $Tu \in L^1(K\cap \Gamma,\sigma)$. 
Take $\varphi \in C^\infty_0(E)$ such that $\varphi \equiv 1$ on $K$. Then $\varphi u \in W$ by definition of $\WW(E)$ and thus 
 \begin{equation} \label{defTrWE3}
\|Tu\|_{L^1(K\cap \Gamma,\sigma)} \leq \|T[\varphi u]\|_{L^1(K \cap \Gamma,\sigma)} \leq C_K \|\varphi u\|_W <+\infty
 \end{equation}
by 
 \eqref{L1locsigma}.
 
Let us turn to the proof of \eqref{defTrWE1}. Take $\varphi \in C^\infty_0(E)$ and then choose $\phi \in C^\infty_0(E)$ such that $\phi \equiv 1$ on $\supp \, \varphi$. 
According to Lemma~\ref{lmult}, $T(\varphi \phi u)(x) = \varphi(x) T(\phi u)(x)$ for almost every $x\in \Gamma$. The result then holds by noticing that $\varphi \phi u = \varphi u$ (i.e. $T(\varphi \phi u)(x) = T(\varphi u)(x)$) and $\phi u = u$ on $\supp \, \varphi$ (i.e. $\varphi(x) T(\phi u)(x) = \varphi(x) T(u)(x)$).
\ep

\bigskip

Let us 
remind the reader that we will be working with the differential operator $L=- \diver A \nabla$,
where $A: \Omega \to \mathbb M_n(\R)$ satisfies, for some constant
$C_1 \geq 1$,
\begin{itemize}
\item the boundedness condition
\begin{equation} \label{ABounded2}
|A(x)\xi \cdot \nu| \leq C_1 w(x) |\xi|\cdot |\nu| \qquad \forall x\in \Omega, \, \xi,\nu\in \R^n ;
\end{equation}
\item the ellipticity condition 
\begin{equation} \label{AElliptic2}
A(x)\xi \cdot \xi \geq C_1^{-1} w(x) |\xi|^2 \qquad \forall x\in \Omega, \, \xi\in \R^n.
\end{equation}
\end{itemize}
We denote the matrix
$w^{-1}A$ by $\A$, so that $\int_\Omega A\nabla u \cdot \nabla v = \int_\Omega \A \nabla u \cdot \nabla v \, dm$. The matrix $\A$ satisfies the unweighted elliptic and boundedness conditions, that is 
\begin{equation} \label{ABounded}
|\A(x)\xi \cdot \nu| \leq C_1 |\xi|\cdot |\nu| \qquad \forall x\in \Omega, \, \xi,\nu\in \R^n,
\end{equation}
and
\begin{equation} \label{AElliptic}
\A(x)\xi \cdot \xi \geq C_1^{-1} |\xi|^2 \qquad \forall x\in \Omega, \, \xi\in \R^n.
\end{equation}

Let us introduce now the bilinear form $a$ defined by
\begin{equation} \label{defofa}
a(u,v) = \int_\Omega A\nabla u \cdot \nabla v \, dz= \int_\Omega \A\nabla u \cdot \nabla v \, dm.
\end{equation}
From \eqref{ABounded} and \eqref{AElliptic}, we deduce that $a$ is a bounded on $W \times W$ 
and coercive on $W$ (hence also on $W_0$). That is,
\begin{equation}\label{aa8}
a(u,u) = \int_\Omega \A \nabla u \cdot \nabla u \, dm \geq C_1^{-1} \int_\Omega |\nabla u|^2 \, dm
= C_1^{-1} \| u \|_W^2
\end{equation}
for $u\in W$, by \eqref{AElliptic}.
\begin{definition}
Let $E \subset \Omega$ be an open set. 

We say that $u \in \WW(E)$ is a solution of $Lu=0$ in $E$ if for any $\varphi\in C^\infty_0(E)$,
\begin{equation} \label{defsol}
a(u,\varphi) = \int_\Omega A \nabla u \cdot \nabla \varphi \, dz= \int_\Omega \A \nabla u \cdot \nabla \varphi \, dm = 0.
\end{equation}
We say that $u \in \WW(E)$ is a subsolution (resp. supersolution) in $E$ if for any $\varphi\in C^\infty_0(E)$ 
such that $\varphi\geq 0$, 
\begin{equation} \label{defsubsol}
a(u,\varphi) = \int_\Omega A \nabla u \cdot \nabla \varphi \, dz= \int_\Omega \A \nabla u \cdot \nabla \varphi \, dm \leq 0 \ (\text{resp.} \geq 0).
\end{equation}
\end{definition}

In particular, subsolutions and supersolutions are always associated to the equation $Lu=0$. In the same way, 
each time we say that $u$ is a solution in $E$, it means that $u$ is in $\WW(E)$ and is a solution 
of $Lu=0$ in $E$.

\medskip

We start with the following important result, that extends the possible test functions in the definition of solutions.

\begin{lemma}  \label{rdefsol}
Let $E \subset \Omega$ be an open set and 
let $u\in \WW(E)$ be a solution of 
$Lu=0$ in $E$. 
Also denote by $E^\Gamma$ is the interior of $E \cup \Gamma$.
The identity \eqref{defsol} holds: 
\begin{itemize}
\item when $\varphi \in W_0$ is compactly supported in $E$;
\item when $\varphi \in W_0$ is compactly supported in $E^\Gamma$ \ub{and} $u \in \WW(E^\Gamma)$;
\item when $E = \Omega$, $\varphi \in W_0$, \ub{and} $u\in W$.
\end{itemize}
In addition, 
\eqref{defsubsol} holds when $u$ is a subsolution (resp. supersolution) in $E$ and $\varphi$ is a non-negative test function satisfying one of the above conditions.
\end{lemma}

\begin{remark} \label{rdefsol2}
The second statement of the Lemma will be used in the following context. 
Let $B \subset \R^n$ be a ball centered on $\Gamma$ and let $u \in \WW(B)$ be a solution of $Lu=0$ in $B \setminus \Gamma$.
Then we have
\begin{equation} \label{defsola}
a(u,\varphi) = \int_\Omega \A \nabla u \cdot \nabla \varphi \, dm = 0
\end{equation}
for any $\varphi \in W_0$ compactly supported in $B$. Similar statements can be written for subsolutions and supersolutions.
\end{remark}

\bp Let $u\in \WW(E)$ be a solution of $Lu=0$ on $E$ and let $\varphi \in W_0$ be 
compactly supported in $E$. We want to prove that $a(u,\varphi)=0$. 

Let $\wt E$ be an open set such that $\supp \, \varphi$ compact in $\wt E$ and 
$\wt E$ is relatively compact in $E$.  By 
Lemma~\ref{lconvol}, there exists a sequence $(\varphi_k)_{k\geq 1}$ of functions in $C^\infty_0(\wt E)$ such that $\varphi_k \to \varphi$ in $W$. 
Observe that the map
\begin{equation} \label{defsol1}
\phi \to a_{\wt E}(u,\phi) = \int_{\wt E} \A \nabla u \cdot \nabla \phi \, dm
\end{equation}
is bounded on $W$ thanks to \eqref{ABounded} and the fact that $\nabla u \in L^2(\wt E,w)$ 
(see \eqref{defWE2}). 
Then, since $\varphi$ and the $\varphi_k$
are supported in $\wt E$,
\begin{equation} \label{defsol2}
a(u,\varphi) = a_{\wt E}(u,\varphi) = \lim_{k\to +\infty} a_{\wt E}(u,\varphi_k) = \lim_{k\to +\infty} a(u,\varphi_k) = 0
\end{equation}
 by
\eqref{defsol}.

\medskip

Now let
$u\in \WW(E^\Gamma)$ be a solution of $Lu = 0$ on $E$ and let $\varphi \in W_0$ be compactly supported 
in $E^\Gamma$. We want to prove that $a(u,\varphi)=0$. 

Let $\wt{E^\Gamma}$ be an open set such that $\supp \, \varphi$ is compact in $\wt{E^\Gamma}$ 
and $\wt{E^\Gamma}$ is relatively compact in $E^\Gamma$. If we look at the proof of Lemma~\ref{ldens0} 
(that uses cut-off functions and the smoothing process given by Lemma~\ref{lconvol}), we can see that
our
$\varphi \in W_0$ can be approached in $W$ by functions 
$\varphi_k \in C^\infty_0(\wt{E^\Gamma} \setminus \Gamma)$. 
In addition,
the map
\begin{equation} \label{defsol3}
\phi \to a_{\wt {E^\Gamma}}(u,\phi) = \int_{\wt {E^\Gamma}} \A \nabla u \cdot \nabla \phi \, dm
\end{equation}
is bounded on $W$ thanks to \eqref{ABounded} and the fact that $\nabla u \in L^2(\wt {E^\Gamma},w)$ (that holds because $u \in \WW(E^\Gamma)$). 
Then, as before,
\begin{equation} \label{defsol4}
a(u,\varphi) = a_{\wt {E^\Gamma}}(u,\varphi) = \lim_{k\to +\infty} a_{\wt {E^\Gamma}}(u,\varphi_k) = \lim_{k\to +\infty} a(u,\varphi_k) = 0.
\end{equation}

\medskip

The proof of the last point, that is $a(u,\varphi) = 0$ if $u\in W$ and $\varphi \in W_0$, 
works the same 
way
as before.
This time, we use the facts that Lemma~\ref{lconvol} gives an 
 approximation of $\varphi$ by functions in $C^\infty_0(\Omega)$ and that $\phi \to a(u,\phi)$ 
 is bounded on $W$.

\medskip

Finally, the cases where $u$ is a subsolution or a supersolution have a similar proof. 
We just need to observe that
the smoothing provided by Lemma~\ref{lconvol} conserves the non-negativity of a test function.
\ep

The first property that 
we need to know about sub/supersolution is the following stability property.

\begin{lemma} \label{lstabsol}
Let $E\subset \Omega$ be an open set. 
\begin{itemize}
\item If $u,v\in \WW(E)$ are subsolutions in $E$, then $t = \max\{u,v\}$ is also a subsolution in $E$.
\item If $u,v\in \WW(E)$ are supersolutions in $E$, then $t = \min\{u,v\}$ is also a supersolution in $E$.
\end{itemize}

In particular if
$k\in \R$, then $(u-k)_+ := \max\{ u-k,0\}$ is a subsolution in $E$ whenever $u \in \WW(E)$ is a subsolution in $E$ and $\min\{u,k\}$ is a supersolution in $E$ whenever $u\in \WW(E)$ is a supersolution in $E$.
\end{lemma}

\bp 
It will be enough to prove the the first statement of the lemma, i.e., the fact that
$t= \max\{u,v\}$ is a subsolution when $u$ and $v$ are subsolutions.
Indeed, the statement about supersolutions
will follow at once, because it is easy to see that $u\in \WW(E)$ is a supersolution if and
only $-u$ is a subsolution.
The remaining assertions are then straightforward consequences of the first ones 
(because constant functions are solutions).

\medskip
So we need to prove the first part, and fortunately it will be easy to reduce to the classical
situation, where the desired result is proved in \cite[Theorem 3.5]{Stampacchia65}.
We need an adaptation, because Stampacchia's proof corresponds to the case where the 
subsolutions $u,v$ lie in $W$, and also we want to localize to a place where $w$ is bounded 
from above and below. 

Let $F \subset E$ be any open set with a smooth boundary and a finite number of connected components,
and whose closure is compact in $E$. We define a
set of functions $W^F$ as
\begin{equation} \label{lstabsol1}
W^F = \{ f \in L^1_{loc}(F), \, \nabla f \in L^2(F,w)\}.
\end{equation}
Let us record a few properties of $W^F$. Since $F$ is relatively compact in $E \subset \Omega$,
the weight $w$ is
bounded from above and below by a positive constant. 
Hence 
$W^F$ is the collection of functions in $L^1_{loc}(F)$ whose distributional derivative lies 
in $L^2(F)$. Since $F$ is bounded and has a smooth boundary, 
these functions lie in $L^2(F)$ (see \cite[Corollary 1.1.11]{Mazya11}).
Of course Mazya states this when $F$ is connected, but we here $F$ has a finite
number of components, and we can apply the result to each one.
So $W^F$ is the `classical' (where the weight is plain) Sobolev space on $F$. That is,
\begin{equation}\label{a8.25}
W^F = \{f \in L^2(F), \, \nabla f \in L^2(F)\}.
\end{equation}

Notice that $u$ and $v$ lie in $W^F$, so they are ``classical'' subsolutions of $L$ in $F$, where 
(since $F$ is relatively compact in $E \subset \Omega$) $w$ is bounded from above and below, 
and hence $L$ is a classical elliptic operator.
Then, by \cite[Theorem 3.5]{Stampacchia65},  $t = \max\{u,v\}$ is also a classical subsolution in $F$. 
This means that $a(t,\varphi) \leq 0$ for $\varphi \in C_0^{\infty}(F)$. 

Now we wanted to prove this for every $\varphi \in C_0^{\infty}(E)$, and it is enough to observe that
if $\varphi \in C_0^{\infty}(E)$ is given, then we can find an open set $F \subset \subset E$ 
that contains the support of $\varphi$, and with the regularity properties above.
Hence $t$ is a subsolution in $E$, and the lemma follows.
It was fortunate for this argument that the notion of subsolution does not come with precise estimates that
would depend on $w$.
\ep

In the sequel, the notation $\sup$ and $\inf$ are used for the essential supremum and essential infimum, 
since they are the only definitions that makes sense for the functions in $W$ or in $\WW(E)$, 
$E\subset \R^n$ open.  
Also, when we talk about solutions or subsolutions and don't specify, 
this will always refer to our fixed operator $L$.
We now state some classical regularity results inside the domain.

\begin{lemma}[interior Caccioppoli inequality] \label{CaccioI} 
Let $E \subset \Omega$ be an open set, and let $u\in \WW(E)$ be a non-negative subsolution in $E$. 
Then for any $\alpha \in C^\infty_0(E)$,
\begin{equation} \label{Caccio1}
\int_\Omega \alpha^2 |\nabla u|^2 dm \leq C  \int_{\Omega} |\nabla \alpha|^2 u^2 dm,
\end{equation}
where $C$ depends only upon the dimensions $n$ and $d$ and the constant $C_1$. 

In particular, if $B$ is a ball of radius $r$ such that $2B \subset \Omega$ and $u\in \WW(2B)$ is
a non-negative subsolution in $2B$, then
\begin{equation} \label{Caccio2}
\int_B |\nabla u|^2 dm \leq C r^{-2} \int_{2B} u^2 dm.
\end{equation}
\end{lemma}

\begin{proof} Let $\alpha \in C^\infty_0(E)$. We set $\varphi = \alpha^2 u$. Since $u\in \WW(E)$, 
the definition yields $\varphi \in W$. Moreover $\varphi$ is compactly supported in $E$ (and in particular
$\varphi \in W_0$).
The first item of Lemma~\ref{rdefsol} yields
\begin{equation} \label{Caccio5}
\int_{\Omega} \A \nabla u \cdot \nabla \varphi \, dm \leq 0.
\end{equation}
By the product rule, 
$\nabla \varphi = \alpha^2 \nabla u + 2 \alpha u \nabla \alpha$.
Thus \eqref{Caccio5} becomes
\begin{equation} \label{Caccio6}\begin{split}
\int_{\Omega}  \alpha^2 \A\nabla u \cdot \nabla u\,  dm & \leq -2 \int_{\Omega} \alpha u \, \A \nabla u \cdot \nabla \alpha \, dm. \\
\end{split}\end{equation}
It follows 
from this and 
the ellipticity and boundedness conditions \eqref{AElliptic} and \eqref{ABounded} that
\begin{equation} \label{Caccio7}\begin{split}
\int_{\Omega}  \alpha^2 |\nabla u|^2  dm & \leq C \int_{\Omega} |\alpha| |\nabla u|  |u| |\nabla \alpha| \, dm \\
\end{split}\end{equation}
and then
\begin{equation} \label{Caccio8}\begin{split} 
\int_{\Omega}  \alpha^2 |\nabla u|^2  dm 
& \leq C \left(\int_{\Omega} \alpha^2 |\nabla u|^2  dm\right)^\frac12  
\left(\int_{\Omega} u^2 |\nabla \alpha|^2  dm \right)^\frac12 \\
\end{split}\end{equation}
by  the Cauchy-Schwarz inequality. 
Consequently,
\begin{equation} \label{Caccio9} 
\int_{\Omega} \alpha^2 |\nabla u|^2  dm \leq C \int_{\Omega} |\nabla \alpha|^2 u^2 dm,
\end{equation}
which is \eqref{Caccio2}. Lemma~\ref{CaccioB} follows since \eqref{Caccio3} is a straightforward application of \eqref{Caccio2} when $E=2B$, $\alpha \equiv 1$ on $B$ and $|\nabla \alpha|\leq \frac2r$.
\end{proof}

\begin{lemma}[interior Moser estimate] \label{MoserI}
Let $p>0$ and $B$ be a ball such that $3B \subset \Omega$. 
If $u \in \WW(3B)$ is a non-negative subsolution in $2B$, then
\begin{equation} \label{Moser1}
\sup_{B} u \leq C \left( \frac1{m(2B)} \int_{2B} u^p \, dm \right)^\frac1p,
\end{equation}
where $C$ depends on $n$, $d$, $C_1$ and $p$.
\end{lemma}

\bp
For this lemma and the next ones, we shall use the fact that since $2B$ is far from $\Gamma$,
our weight $w$ is under control there, and we can easily reduce to the classical case.
Let $x$ and $r$ denote  
the center and the radius of $B$.
Since $3B \subset \Omega$, $\delta(x) \geq 3r$. 
For any $z\in 2B$, $\delta(x)-2r \leq \delta(z) \leq \delta(x+2r)$, hence
\begin{equation} \label{pprI1}
\frac13 \leq 1-\frac{2r}{\delta(x)} \leq \frac{\delta(z)}{\delta(x)} \leq 1+\frac{2r}{\delta(x)} \leq \frac53
\end{equation}
and consequently
\begin{equation} \label{pprI2}
C^{-1}_{n,d} \, w(x) \leq w(z) \leq C_{n,d} \, w(x).
\end{equation}

Let $u\in \WW(3B)$ be a non-negative subsolution in $2B$. 
Thanks to \eqref{defWE2} and \eqref{pprI2}, the gradient $\nabla u$ lies 
in $L^2(2B)$. By the Poincar\'e's inequality, 
$u\in L^2(2B)$ and thus $u$ lies 
in the classical (with no weight) Sobolev space $W^{2B}$ 
 of \eqref{a8.25}.

Consider
the differential operator $\wt L = -\diver \wt A \nabla$ 
with
$\wt A(z) = \A(z) \frac{w(z)}{w(x)}$.
Thanks to \eqref{pprI2}, \eqref{ABounded} and \eqref{AElliptic}, 
$\wt A(z)$ satisfies the elliptic condition and the boundedness condition \eqref{ABounded} and \eqref{AElliptic},
in the domain $2B$, and with the constant $C_{n,d} C_1$.
The condition satisfied by a subsolution (of $Lu=0$) on $2B$ can be rewritten
\begin{equation} \label{pprI3}
\int_{2B} \wt A \nabla u \cdot \nabla \varphi  \leq 0,
\end{equation}
and so 
we are back in 
the situation of the classical elliptic case.
By \cite[Lemma~1.1.8]{KenigB}, for instance,
\begin{equation} \label{pprI4}
\sup_B u \leq C \left( \fint_{2B} u^p(z) \, dz \right)^\frac1p,
\end{equation}
and \eqref{Moser1} follows from this and \eqref{L1byL1w}
\ep

\ms
\begin{lemma}[interior H\"older continuity] \label{HolderI}
Let $x\in \Omega$ and $R>0$ be such that $B(x,3R) \subset \Omega$, and let $u\in \WW(B(x,3R))$ be a solution in $B(x,2R)$. Write $\ds \osc_B u$ for $\ds \sup_B u - \inf_B u$. 
Then there exists $\alpha\in (0,1]$ and $C>0$ such that for any $0<r<R$,
\begin{equation} \label{Holder1}
\osc_{B(x,r)} u \leq C \left( \frac rR \right)^\alpha \left( \frac{1}{m(B(x,R))} \int_{B(x,R)} u^2 \, dm \right)^\frac12,
\end{equation}
where $\alpha$ and $C$ depend only on $n$, $d$, and $C_1$. 
Hence 
$u$
is (possibly after modifying it on a set of measure $0$)
locally H\"older continuous with exponent $\alpha$.
\end{lemma}

\bp
This lemma and the next one follow from the classical results 
(see for instance \cite[Section 1.1]{KenigB}, or \cite[Sections 8.6, 8.8 and 8.9]{GT}),
by the same trick as for Lemma \ref{MoserI}: we observe that $L$ is a constant times a
classical elliptic operator on $2B$.
\ep

\ms
\begin{lemma}[Harnack] \label{HarnackI}
Let $B$ be a ball such that $3B\subset \Omega$, and let $u\in \WW(3B)$ be a non-negative solution in $3B$. Then 
\begin{equation} \label{Harnack1}
\sup_B u \leq C \inf_B u,
\end{equation}
where $C$ depends only on  $n$, $d$ and $C_1$.
\end{lemma} 

\ms 
For the next lemma, we shall need the Harnack chains from Lemma~\ref{lHC}.

\begin{lemma} \label{HarnackI2}
Let $K$ be a compact set of $\Omega$ and let $u\in \WW(\Omega)$ be a non-negative solution in $\Omega$. Then 
\begin{equation} \label{Harnack2}
\sup_K u \leq C_K \inf_K u,
\end{equation}
where $C_K$ depends only on 
$n$, $d$, $C_0$, $C_1$, $\dist(K,\Gamma)$ and $\diam \, K$.
\end{lemma} 

\bp
Let $K$ be a compact set in $\Omega$. 
We can find $r>0$ and $k\geq 1$ such that $\dist(K,\Gamma) \geq r$ and $\diam\, K \leq kr$. 
Now let $x,y\in K$ be given. Notice that
$\delta(x)\geq r$, $\delta(y)\geq r$ and $|x-y| \leq kr$, 
so Lemma~\ref{lHC} implies the existence of a path of length at most 
by $(k+1)r$ that joins  
$x$ to $y$ and stays 
at a distance larger than some $\epsilon$ 
(that depends on $C_0$, $d$, $n$, $r$ and $k$) of $\Gamma$. 
That is, we can find a finite collection of 
balls $B_1, \dots, B_n$ ($n$  bounded uniformly on
$x,y\in K$) such that $3B_i \subset \Omega$, $B_1$ is centered on $x$, $B_n$ is centered on $y$,
and $B_i \cap B_{i+1} \neq \emptyset$. It remains to use $n$ times Lemma~\ref{HarnackI} to get that
 \begin{equation} \label{pprI5}
u(x) \leq C^n u(y) \leq C_K u(y).
\end{equation}
Lemma~\ref{HarnackI2} follows.\ep

\ms
We also need analogues at the boundary of the previous results. For these we cannot immediately reduce
to the classical case, but we will be able to copy the proofs. Of course we shall use our trace operator
to define boundary conditions, say, in a ball $B$, and this is the reason why we want to use the 
space is $\WW(B)$ defined by \eqref{defWE}. 
We cannot use 
$\WW(B\setminus \Gamma)$ instead, because we need some control on $u$ near $\Gamma$ to
define $T(u)$.

In the sequel, we will use
the expression `$Tu = 0$ a.e. on $B$', for a function $u\in \WW(B)$, 
to mean 
that $Tu$, which is defined on $\Gamma \cap B$ and lies in $L^1_{loc}(B\cap \Gamma,\sigma)$ thanks to Lemma~\ref{defTrWE}, is equal to $0$ $\sigma$-almost everywhere on $\Gamma \cap B$.
The expression `$Tu \geq 0$ a.e. on $B$' is defined similarly.

We 
start with the Caccioppoli inequality on the boundary.

\begin{lemma}[Caccioppoli inequality on the boundary] \label{CaccioB} 
Let $B \subset \R^n$ be a ball of radius $r$ centered on $\Gamma$, and let $u\in \WW(2B)$ be a non-negative subsolution in $2B \setminus \Gamma$ such that 
$T(u) = 0$ a.e. on $2B$. Then for any $\alpha \in C^\infty_0(2B)$,
\begin{equation} \label{Caccio3}
\int_{2B} \alpha^2 |\nabla u|^2 dm \leq C  \int_{2B} |\nabla \alpha|^2 u^2 dm,
\end{equation}
where $C$ depends only on 
the dimensions $n$ and $d$ and the constant $C_1$. In particular, we can take $\alpha \equiv 1$ on $B$ and $|\nabla \alpha| \leq \frac2r$, which gives
\begin{equation} \label{Caccio4}
\int_B |\nabla u|^2 dm \leq C r^{-2} \int_{2B} u^2 dm.
\end{equation}
\end{lemma}

\bp 
We can proceed exactly as for Lemma~\ref{CaccioI}, except that the initial estimate \eqref{Caccio5}
needs to be justified differently. Here we choose to apply the second item of Lemma~\ref{rdefsol}, as
explained in Remark \ref{rdefsol2}. That is, $E = 2B \sm \Gamma$ and $E^\Gamma = 2B$.

So we check the assumptions. We set, as before, $\varphi = \alpha^2 u$. First observe that $\varphi \in W$ because $u\in \WW(2B)$ and 
$\alpha \in C^\infty_0(2B)$. Moreover, $\varphi \in W_0$ because, if we let 
$\phi \in C^\infty_0(2B)$ be such that $\phi\equiv 1$ on a neighborhood of $\supp\, \alpha$,
Lemma~\ref{lmult} says that  $T(\varphi) = T(\alpha^2 \phi u) = \alpha^2 T(\phi u) = 0$ a.e. on $\Gamma$.
In addition, $\varphi$ is compactly supported in $2B$ because $\alpha$ is, and $u\in \WW(2B)$ by assumption.

Thus $\varphi$ is a valid test function, Lemma~\ref{rdefsol} applies, \eqref{Caccio5} holds, and
the rest of the proof is the same as for Lemma~\ref{CaccioI}.
\ep

\begin{lemma}[Moser estimates on the boundary] \label{MoserB}
Let $B$ be a ball centered on $\Gamma$. Let $u \in \WW(2B)$ be a non-negative subsolution in 
$2B \setminus \Gamma$ such that $Tu=0$ a.e. on $2B$. Then
\begin{equation} \label{Moser2}
\sup_{B} u \leq C \left(m(2B)^{-1}\int_{2B} u^2 dm \right)^\frac12,
\end{equation}
where $C$ depends only on 
the dimensions $d$ and $n$ and the constants $C_0$ and $C_1$.
\end{lemma}

\bp 
This proof will be a little longer, but we will follow
the ideas used by Stampacchia in \cite[Section 5]{Stampacchia65}. 
The aim is to use the so-called Moser iterations. 
We start with some consequences of Lemma \ref{CaccioB}. 

Pick $2^* \in (2,+\infty)$ 
in the range of $p$ satisfying the Sobolev-Poincar\'e inequality \eqref{rPoincare1}; 
for instance take
$2^* = \frac{2n}{n-1}$. 
Let $u$ be as in the statement and let 
$B=B(x,r)$ be a ball centered on $\Gamma$. 
We claim that for any $\alpha \in C^\infty_0(2B)$,
\begin{equation} \label{Moser6}
\int_{2B} (\alpha u)^2 dm \leq Cr^2 m(\supp \, \alpha u)^{1-\frac{2}{2^*}} m(2B)^{\frac{2}{2^*}-1} \int_{2B} |\nabla \alpha|^2 u^2 dm
\end{equation}
where in fact we abuse notation and set 
$\supp \, \alpha u = \{ \alpha u > 0 \}$. 
Indeed, by H\"older's inequality and the Sobolev-Poincar\'e inequality \eqref{rPoincare1},
\begin{equation} \label{Moser7}
\begin{split}
\int_{\R^n} (\alpha u)^2 dm
& \leq C m(\supp \, \alpha u)^{1-\frac{2}{2^*}} \left(\int_{2B} (\alpha u)^{2^*} dm \right)^{\frac{2}{2^*}} \\
& \leq Cr^2 m(\supp\,  \alpha u)^{1-\frac{2}{2^*}} m(2B)^{\frac{2}{2^*}-1}
\int_{2B} |\nabla [\alpha u]|^2 dm. \\
\end{split}
\end{equation}
The last integral can be estimated, using Caccioppoli's inequality (Lemma~\ref{CaccioB}), by
\begin{equation} \label{Moser8}
\begin{split}
\int_{2B} |\nabla (\alpha u)|^2 dm & \leq 2 \int_{2B} |\nabla \alpha|^2 u^2 dm + 2 \int_{2B} |\nabla u|^2 \alpha^2 dm \\
& \leq C  \int_{2B} |\nabla \alpha|^2 u^2 dm.
\end{split}
\end{equation}
Our claim claim \eqref{Moser6} follows.

Recall that $B = B(x,r)$, with $x\in \Gamma$.
Since $u$ is a subsolution in $2B  \setminus \Gamma$, 
Lemma~\ref{lstabsol} says that
$(u-k)_+ : = \max\{u-k,0\}$ is a non-negative subsolution in $2B  \setminus \Gamma$. 
For any $0 < s < t \leq 2 r$, we 
choose a smooth 
function $\alpha$ supported in $B(x,t)$, such that $0\leq \alpha \leq 1$, $\alpha \equiv 1$ on $B(x,s)$,
and $|\nabla \alpha| \leq \frac{2}{t-s}$. 
By \eqref{Moser6} (applied to $(u-k)_+$ and this function $\alpha$),
\begin{equation} \label{Moser9}
\int_{A(k,s)} |u-k|^2 dm \leq C\frac{r^2}{(t-s)^2} 
m(A(k,t))^{1-\frac{2}{2^*}}
m(2B)^{\frac{2}{2^*}-1} \int_{A(k,t)} |u-k|^2 dm
\end{equation}
where $A(k,s) = \{y \in B(x,s), \, u(y) > k\}$.
If $h>k$, we have also,
\begin{equation} \label{Moser10}
(h-k)^2 m(A(h,s)) \leq \int_{A(h,s)} |u-k|^2 dm \leq   \int_{A(k,s)} |u-k|^2 dm.
\end{equation}
Define
\begin{equation} \label{Moser11}
a(h,s) = m(A(h,s))
\end{equation}
and
\begin{equation} \label{Moser12}
u(h,s) = \int_{A(h,s)} |u-h|^2 dm;
\end{equation}
thus
\begin{equation} \label{Moser13}
\left\{\begin{split} 
& u(k,s) 
\leq \frac{C r^2 m(2B)^{\frac{2}{2^*}-1}}{(t-s)^2} \, u(k,t) [a(k,t)]^{1-\frac{2}{2^*}} \\
& a(h,s) \leq \frac{1}{(h-k)^2} \, u(k,t)
\end{split}\right.
\end{equation}
or, if we set $\kappa = 1-\frac{2}{2^*} > 0$,
\begin{equation}\label{systMoser}\left\{\begin{split} 
& u(k,s)
\leq \frac{C r^2 m(2B)^{-\kappa}}{(t-s)^2} \, u(k,t) [a(k,t)]^{\kappa} \\
& a(h,s) \leq \frac{1}{(h-k)^2} \, u(k,t).
\end{split}\right.\end{equation}
Notice also that $u(h,s) \leq u(k,s)$ because $A(h,s) \subset A(k,s)$ and
$|u-h|^2 \leq |u-k|^2$ on $A(h,s)$. 

Let $\epsilon >0 $ be given, to be chosen
later. The estimates \eqref{systMoser} yield
\begin{equation} \label{Moseriterate}
u(h,s)^\epsilon a(h,s) \leq u(k,s)^\epsilon
a(h,s)  \leq \frac{C r^{2\epsilon} m(2B)^{-\epsilon\kappa} }{(t-s)^{2\epsilon}(h-k)^2} \,
u(k,t)^{\epsilon + 1} a(k,t)^{\epsilon\kappa}.
\end{equation}

\ms
Following \cite{Stampacchia65}, we define a function of two variables $\varphi$ by
\begin{equation}\label{a8.62}
\varphi(h,s) = u(h,s)^\epsilon a(h,s)
\ \text{ for $h > 0$ and } 0 < s < 2 r.
\end{equation}
Notice that $\varphi(h,s) \geq 0$. When $s$ is fixed, $\varphi(h,s)$ is non increasing in $h$,
and when $h$ is fixed, $\varphi(h,s)$ is non decreasing in $s$. We want to show that
\begin{equation} \label{Moser3}
\varphi(h,s) \leq \frac{K}{(h-k)^\alpha (t-s)^\gamma} \left[ \varphi(k,t)\right]^\beta
\end{equation}
for some choice of positive constants $K$, $\alpha$ and $\gamma$, and some $\beta >1$,
because if we do so we shall be able to use Lemma~5.1 in \cite{Stampacchia65} directly.

It is a good idea to choose $\epsilon$ so that 
\begin{equation}\label{systepsilon}\left\{\begin{split} 
\beta \epsilon = \epsilon + 1, \\
\beta = \epsilon\kappa.
\end{split}\right.\end{equation}
for some $\beta >1$. 
Choose $\beta= \frac{1}{2} + \sqrt{\frac14 + \kappa} >1 $ and $\epsilon = \frac{\beta}{\kappa} >0$. An easy computation proves that $(\epsilon,\beta)$ satisfies \eqref{systepsilon}. 
With this choice,  
\eqref{Moseriterate} becomes
\begin{equation} \label{Moseriterate2}
\varphi(h,s) \leq \frac{C r^{2\epsilon} m(2B)^{-\epsilon\kappa} }{(t-s)^{2\epsilon}(h-k)^2} \varphi(k,t)^\beta,
\end{equation}
which is exactly \eqref{Moser3} with $K = C r^{2\epsilon} m(2B)^{-\epsilon\kappa}$, $\alpha = 2$ 
and $\gamma = 2\epsilon$. 

So we can apply Lemma~5.1 in \cite{Stampacchia65}, which says that
\begin{equation} \label{Moser4}
\varphi(\mathfrak d,r) = 0,
\end{equation}
where $\mathfrak d$ is given by 
\begin{equation} \label{Moser5}
\mathfrak d^\alpha = \dfrac{2^{\beta\frac{\alpha+\beta}{\beta-1}}K [\varphi(0,2 r)]^{\beta-1}}{r^\gamma}.
\end{equation}
We replace and get that we can take
\begin{equation} \label{Moser15}
\mathfrak d^2 =
C r^{2\epsilon} m(2B)^{-\epsilon\kappa} \ \frac{\varphi(0,2 r)^{\beta-1}} {r^{\gamma}}
= C m(2B)^{-\epsilon\kappa} \varphi(0, 2 r)^{\beta-1}.
\end{equation}
Notice that $\varphi(\mathfrak d,r) = 0$ implies that
$a(\mathfrak d,r) = 0$, which in turn implies that $u\leq \mathfrak d$ a.e. on $B=B(x,r)$. 
Moreover, by definition of $a$, we have $a(0,2r) \leq m(2B)$. Thus
\begin{equation} \label{Moser16}
\begin{split}
\sup_{B(x,r)} u &\leq \mathfrak d \leq C m(2B)^{-\epsilon\kappa/2}
u(0,2 r)^{(\beta-1) \varepsilon /2} a(0,2r)^{(\beta-1)/2}
\\&\leq C u(0,2 r)^{\epsilon(\beta-1)/2} m(2B)^{(\beta-1-\epsilon\kappa)/2}.
\end{split}
\end{equation}
The first line in \eqref{systepsilon} yields $\epsilon(\beta-1) = 1$ and the second line in \eqref{systepsilon} yields 
$\beta-1-\epsilon\kappa = -1$. 
Besides, 
$u(0,2 r) = \int_{2 B} u^2 dm$ 
because $u$ is nonnegative. Hence
\begin{equation} \label{Moser17}
\sup_{B} u \leq C \left( m(2B)^{-1} \int_{2 B} u^2 dm \right)^\frac12,
\end{equation}
which is the desired conclusion.
\ep

\ms
\begin{lemma}[Moser estimate at the boundary for general $p$] \label{MoserB2}
Let $p>0$. Let $B$ be a ball centered on $\Gamma$. Let $u \in \WW(2B)$ be a non-negative subsolution in $2B \setminus \Gamma$ such that $Tu=0$ a.e. on $2B$. Then
\begin{equation} \label{Moser2p}
\sup_{B} u \leq C_p \left(m(2B)^{-1}\int_{2B} u^p dm \right)^\frac1p,
\end{equation}
where $C_p$ depends only on 
the dimensions $n$ and $d$, the constants $C_0$ and $C_1$, and the exponent $p$.
\end{lemma}

\bp 
Lemma \ref{MoserB2} can be 
deduced 
from Lemma \ref{MoserB} by a simple iterative argument. 
The proof is fairly similar to the very end of the proof of \cite[Chapter IV, Theorem 1.1]{HLbook}. 
Nevertheless, because the proof in \cite{HLbook} doesn't hold at the boundary (and for the sake of completeness), we give a proof here.

\medskip

First, let us prove that we can improve \eqref{Moser2} into the following: if $B$ is a ball centered 
on $\Gamma$ and $u \in \WW(B)$ is a non-negative subsolution on $B\cap \Omega$ such that 
$Tu = 0$ a.e. on $B$, then for any $\theta \in (0,1)$ 
(in practice, close to $1$),
\begin{equation} \label{Moser18}
\sup_{\theta B} u \leq C (1-\theta)^{-\frac{n}2} \left( m(B)^{-1} \int_B u^2 \, dm \right)^\frac12,
\end{equation}
where $C>0$ depends only on $n$, $d$, $C_0$ and $C_1$.

Let $B$ be a ball 
centered on $\Gamma$, with 
radius $r$,
and let $\theta \in (0,1)$.
Choose $x\in \theta B$. 
Two cases may happen: either $\delta(x) \geq \frac{1-\theta}{6}r$ or 
$\delta(x) < \frac{1-\theta}{6}r$.
In the first case, if
$\delta(x) \geq \frac{1-\theta}{6}r$, we apply Lemma \ref{MoserI} to the ball 
$B(x,\frac{1-\theta}{20}r)$
(notice that $B(x,\frac{1-\theta}{10}r) \subset B \cap \Omega$). We get that
\begin{equation} \label{Moser19} \begin{split}
u(x) & \leq C \left( \frac1{m(B(x,\frac{1-\theta}{10}r))} \int_{B(x,\frac{1-\theta}{10}r)} u^2 \, dm \right)^\frac12 \\
& \leq C \left( \frac{m(B(x,2r))}{m(B(x,\frac{1-\theta}{10}r)} \right)^\frac12 \left( \frac1{m(B)} \int_B u^2 \, dm \right)^\frac12 \\
& \leq C (1-\theta)^{-\frac{n}2} \left( m(B)^{-1} \int_B u^2 \, dm \right)^\frac12
\end{split} \end{equation}
by
\eqref{doublinggen}. In the second case, when
$\delta(x) \leq \frac{1-\theta}{6}r$, we take $y \in \Gamma$ such that $|x-y| = \delta(x)$. Remark that $y \in \frac{1+\theta}{2}B$ and then $B(y,\frac{1-\theta}{2} r) \subset B$. We apply then Lemma \ref{MoserB} to the ball $B(y,\frac{1-\theta}{6} r)$ in order to get
\begin{equation} \label{Moser20} \begin{split}
u(x) \leq \sup_{B(y,\frac{1-\theta}{6} r)}  
u
& \leq C \left( \frac1{m(B(x,\frac{1-\theta}{3}r))} \int_{B(x,\frac{1-\theta}{3}r)} u^2 \, dm \right)^\frac12 \\
& \leq C \left( \frac{m(B(x,2r))}{m(B(x,\frac{1-\theta}{3}r)} \right)^\frac12 \left( \frac1{m(B)} \int_B u^2 \, dm \right)^\frac12 \\
& \leq C (1-\theta)^{-\frac{n}2} \left( m(B)^{-1} \int_B u^2 \, dm \right)^\frac12
\end{split} \end{equation}
with \eqref{doublinggen}. The claim \eqref{Moser18} follows.

\medskip

Let us prove now \eqref{Moser2p}. Without loss of generality, we can 
restrict to
the case $p<2$, since the case $p\geq 2$ can be deduced from Lemma \ref{MoserB} and 
H\"older's 
inequality. 

Let $B=B(x,r)$ be a ball and let $u \in \WW(2B)$ be a non-negative 
subsolution on $2B \sm \Gamma$
such that $T u=0$ on $2B$.
Set for $i\in \bN$, 
$$r_i := r \sum_{j=0}^i 3^{-j} = \frac32 r(1-3^{-i-1}) < \frac32 r.$$ 
Note that $\frac{r_i}{r_{i}-r_{i-1}} = \frac{3^{i+1}-1}2 \leq 3^{i+1}$. As a consequence, for any $i\in \bN^*$, \eqref{Moser18} yields
\begin{equation} \label{Moser21} \begin{split} 
\sup_{B(x,r_{i-1})} u  & \leq C3^{\frac{in}2} \bigg(\frac1{m(B(x,r_i))}\int_{B(x,r_i)} |u|^2 dm \bigg)^\frac12 \\
& \leq C3^{\frac{in}2} \bigg(\sup_{B(x,r_i)} u \bigg)^{1-\frac p2} 
\bigg(\frac1{m(B(x,r_i))} \int_{B(x,r_i)} |u|^p dm \bigg)^\frac12 \\
& \leq C3^{\frac{in}2} \bigg(\sup_{B(x,r_i)} u \bigg)^{1-\frac p2} 
\bigg(\frac1{m(2B)}\int_{B(x,r_i)} |u|^p dm \bigg)^\frac12. \\
\end{split}\end{equation}
Set $\alpha = 1-\frac p2$. By taking the power $\alpha^{i-1}$ of the inequality \eqref{Moser21}, where $i$ is a positive integer, we obtain 
\begin{equation} \label{Moser22}
\bigg(\sup_{B(x,r_{i-1})} u \bigg)^{\alpha^{i-1}} 
\leq C^{\alpha^{i-1}} (3^{\frac{in}2})^{\alpha^{i-1}} \bigg(\sup_{B(x,r_i)} u \bigg)^{\alpha^i} 
\bigg(m(2B)^{-1}\int_{B(x,r_i)} |u|^p dm \bigg)^{\frac12 \alpha^i},
\end{equation}
where $C$ is independent of 
$i$ (and also $p$, $x$, $r$ and $u$). An immediate induction gives, for any $i\geq 1$,
\begin{equation} \label{Moser23}
\sup_{B(x,r)} u \leq C^{\sum_{j=0}^{i-1} \alpha^{j}} 
\bigg(\prod_{j=1}^i 3^{\frac{jn}2\alpha^{j-1}} \bigg) 
\bigg(\sup_{B(x,r_i)} u \bigg)^{\alpha^i} 
\bigg(m(2B)^{-1}\int_{B(x,r_i)} |u|^p dm \bigg)^{\frac12 \sum_{j=0}^{i-1} \alpha^{j}},
\end{equation}
and if we apply Corollary \ref{Moser2} once more, we get that
\begin{equation} \label{Moser27}
\sup_{B} u \leq C^{\sum_{j=0}^{i} \alpha^{j}} 
\bigg(\prod_{j=1}^{i+1} 3^{\frac{jn}2\alpha^{j-1}} \bigg)
\bigg(m(2B)^{-1}\int_{2B} |u|^p dm \bigg)^{\frac12 \sum_{j=0}^{i-1} \alpha^{j}} 
\bigg( m(2B)^{-1} \int_{\frac32 B} |u|^2 dm \bigg)^{\frac{\alpha^i}2}.
\end{equation}
We want to to take the limit when $i$ goes to $+\infty$. Since $u \in \WW(2B)$, the quantity $\int_{\frac32 B} |u|^2 dm $ is finite and thus
\begin{equation} \label{Moser24}
\lim_{i\to +\infty} \left( m(2B)^{-1} \int_{\frac32 B} |u|^2 dm \right)^{\frac{\alpha^i}2} = 1
\end{equation}
because we took $p$ such that $\alpha = 1-\frac p2 <1$. Note also that
\begin{equation} \label{Moser25}
\lim_{i\to +\infty}\sum_{j=0}^{i-1} \alpha^{j} = \frac 2p \quad \text{ and } \quad \lim_{i\to +\infty} \frac 12\sum_{j=0}^{i-1} \alpha^{j} = \frac 1p.
\end{equation}
Furthermore, $\ds \prod_{j=1}^{i+1} 3^{\frac{jn}2\alpha^{j-1}}$ has a limit (that depends on $p$ and $n$) when $j<+\infty$ because, 
\begin{equation} \label{Moser26}
\sum_{j=1}^\infty \frac{jn}2\alpha^{j-1} = \frac {n} {2} \sum_{j=1}^{+\infty} j \alpha^{j-1} < +\infty.
\end{equation}
These three facts prove that the limit when $i\to +\infty$ of the right-hand side of \eqref{Moser27} exists and
\begin{equation} \label{Moser28}
\sup_{B} u \leq C_p  \left(m(2B)^{-1}\int_{2B} |u|^p dm \right)^{\frac1p},
\end{equation}
which is the desired result.
\ep

Next comes the H\"older continuity of the solutions at the boundary. We start with a boundary version of the density property. 

\begin{lemma} \label{HarnackB}
Let $B$ be a ball centered on $\Gamma$ and $u \in \WW(4B)$ be a non-negative supersolution in 
$4B  \setminus \Gamma$ such that $T u = 1$ a.e. on $4B$. 
Then 
\begin{equation} \label{Harnack3}
\inf_B u \geq C^{-1},
\end{equation}
where $C>0$ depends only on the dimensions $d$, $n$ and the constants $C_0$, $C_1$.
\end{lemma}

\begin{proof} The ideas of the proof are taken from the Density Theorem 
(Section 4.3, Theorem 4.9) in \cite{HLbook}. 
The result in \cite{HLbook} states, roughly speaking, that \eqref{Harnack3} holds whenever $u$ 
is a supersolution in $4B \subset \Omega$ 
such that 
$u\geq 1$ on a large piece of $B$; and its proof relies on a Poincar\'e inequality on balls 
for functions that equal $0$ 
on a big piece of the considered ball.

\noindent We will adapt this argument to the case where $B$ is centered on $\Gamma$ and we will rely on the Poincar\'e inequality given by Lemma~\ref{lpBry}.

\medskip

Let $B$ and $u$ be as in the statement.
Let $\delta \in (0,1)$ be small (it will be used to avoid some functions to take the value $0$)
and set $u_\delta =  \min\{1,u+\delta\}$ and $v_\delta:= -\Phi_\delta(u_\delta)$, where $\Phi$ is a smooth Lipschitz function defined on $\R$ such that $\Phi_\delta(s) = -\ln(s)$ when $s\in [\delta,1]$. 

The plan of the proof is: first we prove that $v_\delta$ is a subsolution, and then 
we use the Moser estimate and the Poincar\'e inequality given Lemma~\ref{MoserB} and \ref{lpBry} respectively. It will give that the supremum of $v_\delta$ on $B$ is bounded by the $L^2$-norm of the gradient of $v_\delta$.
Then, we will test the supersolution $u_\delta$ against an appropriate test function, which will give that the $L^2(2B)$ bound on $\nabla v_\delta$ - and thus the supremum of $v_\delta$ on $B$ - can be bounded by a constant independent of $\delta$. 
This will yield
a lower bound on $u_\delta(x)$ which is uniform in $\delta$ and $x\in B$. 

\medskip

So we start by proving that
\begin{equation} \label{Harnack4}
\text{$v_\delta \in \WW(4B)$ is a subsolution in $4B  \setminus \Gamma$ 
such that
$Tv_\delta = 0$ a.e. on $4B$.}
\end{equation}
Let $\varphi \in C^\infty_0(\Omega \cap 4B)$.
Choose 
$\phi \in C^\infty_0(\Omega \cap 4B)$
such that $\phi \equiv 1$ on $\supp \, \varphi$. Then for 
$y\in \Omega$,
\begin{equation} \label{Harnack5}
v_\delta(y) \varphi(y) = \Phi_\delta(\min\{1,(u(y)+\delta)\phi(y)\}) \varphi(y).
\end{equation}
Since $u\in \WW(4B)$, it follows that $u\phi \in W$ and thus $(u+\delta) \phi \in W$. 
Consequently, we obtain $\min\{1,(u+\delta)\phi\} \in W$ by 
Lemma~\ref{lcompo} (b), then $\Phi_\delta(\min\{1,(u+\delta)\phi\}) \in W$ by 
Lemma~\ref{lcompo} (a) and finally $v_\delta \varphi \in W$ thanks to Lemma~\ref{lmult}. 
Hence
$v_\delta\in \WW(4B)$. Using the fact that the trace is local and Lemmata \ref{lcompo} and \ref{defTrWE}, it is clear that 
\begin{equation} \label{Harnack6}
Tv_\delta = -\ln(\min\{1,T(u\phi)+\delta\}) = 0 \ \text{ a.e. on } 4B.
\end{equation}
The claim \eqref{Harnack4} will be proven if we can show that $v_\delta$ is a subsolution in $4B  \setminus \Gamma$. Let $\varphi \in C^\infty_0(4B  \setminus \Gamma)$ be a non-negative function. We have
\begin{equation} \label{Harnack7} \begin{split}
\int_\Omega \A \nabla v_\delta \cdot \nabla \varphi \, dm & = - \int_\Omega \frac{\A \nabla u_\delta}{u_\delta} \cdot \nabla \varphi \, dm \\
& = - \int_\Omega \A\nabla u_\delta \cdot \nabla \left( \frac{\varphi}{u_\delta} \right) \, dm - \int_{4B} \frac{\A \nabla u_\delta \cdot \nabla u_\delta}{u_\delta^2} \varphi \, dm.
\end{split} \end{equation}
The second term in the right-hand side is non-positive by 
the ellipticity condition \eqref{AElliptic}. So $v_\delta$ is a subsolution if we can establish that 
\begin{equation} \label{Harnack8}
\int_\Omega \A \nabla u_\delta \cdot \nabla \left( \frac{\varphi}{u_\delta} \right) \, dm \geq 0.
\end{equation}
Yet $u_\delta$ is a supersolution according to Lemma~\ref{lstabsol}. Moreover $\varphi/u_\delta$ is compactly supported in $4B \setminus \Gamma$ and, since $u_\delta \geq \delta>0$, 
we deduce from Lemma~\ref{lcompo} that $\frac{\varphi}{u_\delta} \in W$.
So \eqref{Harnack8} is just 
a consequence of Lemma~\ref{rdefsol}. The claim \eqref{Harnack4} follows.

\medskip

The function $v_\delta$ satisfies now all the assumptions of Lemma~\ref{MoserB} and thus
\begin{equation} \label{Harnack9}
\sup_{B} v_\delta \leq C \left( m(2B)^{-1} \int_{2B} |v_\delta|^2 dm \right)^{\frac12}.
\end{equation}
Since $Tv_\delta = 0$ a.e. on $2B$, the right-hand side can be bounded with 
the help of \eqref{1.5-bisbis}, 
which gives
\begin{equation} \label{Harnack10}
\sup_{B} v_\delta \leq Cr \left( m(2B)^{-1} \int_{2B} |\nabla v_\delta|^2 dm \right)^{\frac12}.
\end{equation}
We will prove that the right-hand side of \eqref{Harnack10} is bounded uniformly in $\delta$. 
Use the test function 
$\varphi = \alpha^2\big(\frac{1}{u_\delta} - 1\big)$ with 
$\alpha \in C^\infty_0(4B)$, $0\leq \alpha \leq 1$, $\alpha \equiv 1$ on $2B$ and $\nabla \alpha \leq \frac 1r$. Note that $\varphi$ is a non-negative function compactly supported in $4B$ and, by 
Lemma~\ref{lcompo}, $\varphi$ is in $W$ and has zero trace, that is $\varphi \in W_0$. 

\noindent Since $u$ is a supersolution, $u_\delta$ is also a supersolution. We test $u_\delta$ against $\varphi$ (this is allowed, 
thanks to Lemma~\ref{rdefsol}) and we get
\begin{eqnarray} \label{Harnack11}
0 &\leq& \int_{\R^n} \A\nabla u_\delta \cdot \nabla 
\Big[\alpha^2 \Big(\frac{1}{u_\delta} - 1\Big)\Big]dm 
\nn\\
&=& - \int_{\R^n} \alpha^2 \, \frac{\A\nabla u_\delta \cdot \nabla u_\delta}{u_\delta^2} dm 
+ 2 \int_{\R^n} \alpha \left(1-u_\delta\right) \, \frac{\A\nabla u_\delta \cdot \nabla \alpha}{u_\delta} \,dm,
\end{eqnarray}
hence, by the
ellipticity and the boundedness of $A$ (see \eqref{ABounded} and \eqref{AElliptic}),
\begin{equation} \label{Harnack12} \begin{split}
\int_{\R^n} \alpha^2 \frac{|\nabla u_\delta|^2}{u_\delta^2} dm 
& \leq C \int_{\R^n} \alpha^2 \frac{\A\nabla u_\delta \cdot \nabla u_\delta}{u_\delta^2} dm \\
& \leq  C \int_{\R^n} \alpha \left(1-u_\delta\right) \frac{\A\nabla u_\delta \cdot \nabla \alpha}{u_\delta} \,dm \\
& \leq C \int_{\R^n} \alpha \left(1-u_\delta\right) \frac{|\nabla u_\delta||\nabla \alpha|}{u_\delta} \,dm \\
& \leq C \left( \int_{\R^n} \alpha^2 \frac{|\nabla u_\delta|^2}{u_\delta^2} \,dm \right)^\frac12  \left( \int_{\R^n}\left(1-u_\delta\right)^2 |\nabla \alpha|^2 \,dm \right)^\frac12 \\
\end{split} \end{equation}
by Cauchy-Schwarz' inequality. Therefore,
\begin{equation} \label{Harnack13}
\int_{\R^n} \alpha^2 |\nabla \ln u_\delta|^2 dm =  \int_{\R^n} \alpha^2 \frac{|\nabla u_\delta|^2}{u_\delta^2} dm \leq C \int_{\R^n} (1-u_\delta)^2 |\nabla \alpha|^2 dm \leq C \int_{\R^n} |\nabla \alpha|^2 dm
\end{equation}
because $0\leq u_\delta \leq 1$, and then with our particular choice of $\alpha$,
\begin{equation} \label{Harnack14}
m^{-1}(2B) \int_{2B} |\nabla v_\delta|^2 dm = m^{-1}(2B) \int_{2B} |\nabla \ln u_\delta|^2 dm  \leq \frac{C}{r^2}.
\end{equation}

We inject
this last estimate in \eqref{Harnack10} 
and get that
\begin{equation} \label{Harnack15}
\sup_B v_\delta = \sup_B (-\ln u_\delta) \leq C,
\end{equation}
i.e.  $\ds \inf_B u_\delta = \inf_B \min\{1,u+\delta\} \geq e^{-C} = C^{-1}$. Since the constant doesn't depend on $\delta$, we have the right conclusion, that is
$\ds \inf_B u \geq C^{-1}.$
\end{proof}

\begin{lemma}[Oscillation estimates on the boundary] \label{OscB}
Let $B$ be a ball centered on $\Gamma$ and $u \in \WW(4B)$ be a solution in $4B \setminus \Gamma$ such that $T u $ is uniformly  bounded on $4B$. Then, there exists $\eta \in (0,1)$ such that
\begin{equation} \label{Osc1}
\osc_{B} u \leq \eta \osc_{4B} u + (1-\eta) \osc_{\Gamma \cap 4B} Tu.
\end{equation}
The constant $\eta$ depends only on the dimensions $n$, $d$ and the constants $C_0$ and $C_1$.
\end{lemma}

\begin{proof}
Set $\ds M_4 = \sup_{4B} u$, $\ds m_4 = \inf_{4B} u$, $\ds M_1 = \sup_{B} u$, $\ds m_1 = \inf_{B} u$, $\ds M = \sup_{4B \cap \Gamma} Tu$ and $\ds m = \inf_{4B \cap \Gamma} Tu$. 
Let us first prove that
\begin{equation} \label{Osc2}
{M_4-M_1}\geq c ({M_4-M}) 
\end{equation}
and
\begin{equation} \label{Osc3}
{m_1-m_4} \geq c ({m-m_4})
\end{equation}
for some $c\in (0,1]$. Notice 
that \eqref{Osc2} is trivially true if $M_4-M = 0$. Otherwise, we apply Lemma~\ref{HarnackB} to the non-negative supersolution $\min\{\frac{M_4-u}{M_4-M}, 1 \}$ whose trace equals 1 on $4B$ (with Lemma~\ref{lcompo}) and we obtain for some constant $c\in (0,1]$
\begin{equation} \label{Osc2b}
\frac{M_4-M_1}{M_4-M} \geq c
\end{equation}
which gives \eqref{Osc2} if we multiply both sides by $M_4-M$. In the same way,  \eqref{Osc3} is true if $m-m_4 = 0$ and  otherwise, we apply Lemma~\ref{HarnackB} to the non-negative supersolution $\min\{\frac{u-m_4}{m-m_4},1\}$ and we get for some $c\in (0,1]$
\begin{equation} \label{Osc3b}
\frac{m_1-m_4}{m-m_4} \geq c,
\end{equation}
which is
\eqref{Osc3}. 

We sum then \eqref{Osc2} and \eqref{Osc3} to get
\begin{equation} \label{Osc4}
[M_4 - m_4] - [M_1-m_1] \geq c [M_4 - m_4] - c[M-m],
\end{equation}
that is
\begin{equation} \label{Osc5}
[M_1-m_1] \leq (1-c) [M_4 - m_4] + c[M-m],
\end{equation}
which is exactly the desired result.
\end{proof}

We end the section with the H\"older continuity of solutions at the boundary.

\begin{lemma} \label{HolderB}
Let $B=B(x,r)$ be a ball centered on $\Gamma$ and $u \in \WW(B)$ be a solution in $B$ 
such that $Tu$ is continuous and bounded on $B$. 
There exists $\alpha >0$ such that for
$0<s <r$,
\begin{equation} \label{Holder2}
\osc_{B(x,s)} u \leq C \left(\frac{s}r\right)^\alpha \osc_{B(x,r)} u + C \osc_{B(x,\sqrt{sr}) \cap \Gamma} Tu
\end{equation}
 where the constants $\alpha,C$ depend only on the dimensions $n$ and $d$ and the constants $C_0$ and $C_1$. In particular, $u$ is continuous on $B$.
 
If, in addition, $Tu\equiv 0$ on $B$, then for any $0<s<r/2$
\begin{equation} \label{Holder3}
\osc_{B(x,s)} u \leq C \left(\frac{s}r\right)^\alpha \left( m(B)^{-1} \int_{B} |u|^2 dm \right)^\frac12.
\end{equation}
\end{lemma}

\begin{proof}
The first part of the Lemma, i.e. the estimate \eqref{Holder2}, is a straightforward consequence of Lemma~\ref{OscB} and \cite[Lemma~8.23]{GT}. Basically, \cite[Lemma~8.23]{GT} is a result on functions stating that the functional inequality \eqref{Osc1} can be turned, via iterations, into \eqref{Holder2}.

The second part of the Lemma is simply a consequence of the first part and of the Moser inequality given in Lemma~\ref{MoserB}.
\end{proof}

\chapter{Harmonic Measure}

We want to solve the Dirichlet problem
\begin{equation} \label{Dp1}
\left\{
\begin{array}{ll}
Lu=f & \text{ in } \Omega \\
u=g & \text{ on } \Gamma,
\end{array}
\right.
\end{equation}
with a notation that we explain now.
Here we require $u$ to lie in $W$, and by the second line we actually mean that
$Tu = g$ $\sigma$-almost everywhere on $\Gamma$, where $T$ is our trace operator.
Logically, we are only interested in functions $g\in H$, because we know that
$T(u) \in H$ for $u\in W$.

The condition $Lu = f$ in $\Omega$ is taken in the weak sense, i.e. we say that $u \in W$ satisfies $Lu = f$, 
where $f\in W^{-1} = (W_0)^*$, if for any $v\in W_0$, 
\begin{equation} \label{defLu=f}
a(u,v) = \int_\Omega A \nabla u \cdot \nabla v = \left<f, v \right>_{W^{-1},W_0}.
\end{equation}
Notice that when
$f \equiv 0$, a function $u \in W$ that satisfies 
\eqref{defLu=f} is a solution in $\Omega$.

\medskip

Now, we made sense of \eqref{Dp1} for at least $f\in W^{-1}$ and $g\in H$. 
The next result gives a good solution to the Dirichlet problem.

\begin{lemma} \label{lLM}
For any $f\in W^{-1}$ and any $g\in H$, there exists a unique $u \in W$ such that 
\begin{equation} \label{Dp2}
\left\{
\begin{array}{ll}
Lu=f & \text{ in } \Omega \\
Tu=g & \text{ a.e. on } \Gamma.
\end{array}
\right.
\end{equation}
Moreover, there exists $C>0$ independent of $f$ and $g$ such that
\begin{equation} \label{ubygf}
\|u\|_W \leq C (\|g\|_H + \|f\|_{W^{-1}}),
\end{equation}
where
\begin{equation} \label{W-1norm}
\|f\|_{W^{-1}}  = \sup_{\begin{subarray}{c} \varphi \in W_0 \\ \|\varphi\|_W = 1 \end{subarray}} \left< f, \varphi\right>_{W^{-1}, W_0}.
\end{equation}
\end{lemma}

\bp 
Since $g\in H$, Theorem \ref{tExt} implies that
there exists $G \in W$ such that $T(G) = g$ and 
\begin{equation} \label{lLM1}
\|G\|_W \leq C \|g\|_H.
\end{equation}
The quantity $LG$ is an element of $W^{-1}$ defined by
\begin{equation} \label{lLM2}
\left<LG, \varphi \right>_{W^{-1},W_0} := \int_\Omega A \nabla G \cdot \nabla \varphi \, dz= \int_\Omega \A \nabla G \cdot \nabla \varphi \, dm,
\end{equation}
and notice that 
\begin{equation} \label{lLM3}
\|LG\|_{W^{-1}} \leq C \|G\|_W \leq C \|g\|_H
\end{equation}
by
\eqref{ABounded} and \eqref{lLM1}.

Observe that
the conditions \eqref{ABounded} and \eqref{AElliptic} imply that the bilinear form $a$ is 
bounded and coercive on $W_0$. It follows
from the Lax-Milgram theorem that there exists a (unique) $v\in W_0$ such that $Lv = -LG - f$. 
Set $u = G-v$. It is now easy to see that $Tu = g$ a.e. on $\Gamma$ and $Lu = f$ in $\Gamma$. 
The existence of a solution of \eqref{Dp2} follows.

It remains to check the uniqueness of the solution and the bounds \eqref{ubygf}. 
Take $u_1,u_2\in W$ two solutions of \eqref{Dp2}. One has then $T(u_1-u_2) = g - g = 0$ and hence 
$u_1 - u_2 \in W_0$. Moreover, $L(u_1-u_2) = 0$. Since $a$ is bounded and coercive, 
the uniqueness in the Lax-Milgram theorem yields $u_1 - u_2 = 0$. 
Therefore \eqref{Dp2} has also a unique solution.

Finally, let us prove the bounds \eqref{ubygf}. From the coercivity of $a$, we get  that
\begin{equation} \label{lLM4}
\|v\|_{W}^2 \leq C a(v,v) \leq C \|LG + f\|_{W^{-1}} \|v\|_W,
\end{equation}
i.e., with \eqref{lLM3},
\begin{equation} \label{lLM5}
\|v\|_{W} \leq C \|LG + f\|_{W^{-1}} \leq C(\|g\|_H + \|f\|_{W^{-1}}).
\end{equation}
We conclude the proof of \eqref{ubygf} with 
\begin{equation} \label{lLM6}
\|u\|_{W} = \|G - v\|_{W} \leq C(\|g\|_H + \|f\|_{W^{-1}})
\end{equation}
by \eqref{lLM1}.
\ep

The next step in the construction of a harmonic measure associated to $L$, is to prove a maximum principle.

\begin{lemma} \label{lMP}
Let $u \in W$ be a supersolution in $\Omega$ satisfying $Tu \geq 0$ a.e. on $\Gamma$. Then $u\geq 0$ a.e. in $\Omega$.
\end{lemma} 

\bp
Set $v = \min\{u,0\} \leq 0$. According to Lemma~\ref{lcompo} (b), we have
\begin{equation} \label{lMP1}
\nabla v = \left\{ \begin{array}{ll}  \nabla u & \text{ if } u<0 \\ 0 & \text{ if } u\geq 0 \end{array}\right.
\end{equation}
and 
\begin{equation} \label{lMP2}
Tv = \min\{Tu , 0\} = 0 \quad \text{a.e. in } \Gamma.
\end{equation}
In particular, \eqref{lMP2} implies that $v \in W_0$. 
The third case of Lemma \ref{rdefsol} allows us to test $v$ 
against the supersolution $u\in W$; this gives
\begin{equation} \label{lMP3}
\int_\Omega \A \nabla u \cdot \nabla v \, dm  \leq 0,
\end{equation}
that is with \eqref{lMP1},
\begin{equation} \label{lMP4}
\int_\Omega \A \nabla v \cdot \nabla v \, dm = \int_{\{u<0\}} \A \nabla u \cdot \nabla u \, dm =  \int_\Omega \A \nabla u \cdot \nabla v \, dm \leq 0.
\end{equation}
Together with the ellipticity condition \eqref{AElliptic}, we obtain $\|v\|_W \leq 0$. Recall 
 from Lemma \ref{lcomp2}
that $\|.\|_W$ is a norm on $W_0 \ni v$, hence $v = 0$ a.e. in $\Omega$. 
We conclude from the definition of $v$ that $u \geq 0$ a.e. in $\Omega$. 
\ep

Here is a corollary of Lemma~\ref{lMP}.

\begin{lemma}[Maximum principle] \label{lMPs}
Let $u \in W$ be a solution 
of $Lu=0$
in $\Omega$. Then
\begin{equation} \label{lMP6}
\sup_\Omega u \leq \sup_\Gamma Tu
\end{equation}
and
\begin{equation} \label{lMP7}
\inf_\Omega u \geq \inf_\Gamma Tu,
\end{equation}
where we recall that $\sup$ and $\inf$ actually essential supremum and infimum.
In particular, 
if $Tu$ is essentially bounded,
\begin{equation} \label{lMP5}
\sup_\Omega |u| \leq \sup_\Gamma |Tu|.
\end{equation}
\end{lemma} 

\bp
Let us prove \eqref{lMP6}.
Write $M$ for the essential supremum of $Tu$ on $\Gamma$;
we may assume that $M < +\infty$, because otherwise \eqref{lMP6} is trivial. 
Then $M-u \in W$ and
$T(M-u) \geq 0$ a.e. on $\Gamma$. 
Lemma~\ref{lMP} yields $M-u \geq 0$ a.e. in $\Omega$, that is
\begin{equation} \label{lMP6a}
\sup_\Omega u \leq \sup_\Gamma Tu.
\end{equation}
The lower bound \eqref{lMP7} is similar, 
and \eqref{lMP5} follows.
\ep

\medskip

We want to define the harmonic measure via the Riesz representation theorem (for measures), 
that requires a linear form on 
the space of compactly supported continuous functions on $\Gamma$. We 
 denote this space by $C^0_0(\Gamma)$; that is, $g\in C^0_0(\Gamma)$ if $g$ is defined 
and continuous on $\Gamma$, 
and there exists a ball $B \subset \R^n$ centered on $\Gamma$ such that 
$\supp \, g \subset B\cap \Gamma$.  

\begin{lemma} \label{ldefhm}
There exists a bounded linear operator
\begin{equation} \label{defhm4}
U : C^0_0(\Gamma) \to  C^0(\R^n) 
\end{equation}
such that, for every every $g\in C^0_0(\Gamma)$,
\begin{enumerate}[(i)]
\item the restriction of $Ug$ to $\Gamma$ is $g$;
\item $\ds \sup_{\R^n} Ug = \sup_\Gamma g$ and $\ds \inf_{\R^n} Ug = \inf_\Gamma g$;
\item $Ug \in \WW(\Omega)$ and is a solution of $L$ in $\Omega$;
\item if $B$ is a ball centered on $\Gamma$ and $g \equiv 0$ on $B$, then
$Ug$ lies in $\WW(B)$;
\item if $g \in C^0_0(\Gamma) \cap H$, then $Ug \in W$, and it is the solution of \eqref{Dp2}, 
with $f=0$, provided by Lemma \ref{lLM}.
\end{enumerate}
\end{lemma}

\bp
This is essentially an argument of extension from a dense class by uniform continuity.
We first define $U$ on $C^0_0(\Gamma) \cap H$, by saying that $u = Ug$ is 
the solution of \eqref{Dp2}, with $f=0$, provided by Lemma \ref{lLM}. Thus $u \in W$;
but since its trace is $Tu = g$ is continuous, it follows from 
Lemmata \ref{HolderI} and \ref{HolderB} (the H\"older continuity inside and at the boundary)
that $u$ is continuous on $\R^n$. 

Next we check that $U$ is linear and bounded on $C^0_0(\Gamma) \cap H \subset C^0_0(\Gamma)$ 
(where we use the sup norm). The linearity comes from the uniqueness in Lemma \ref{lLM}, and the 
boundedness from the maximum principle: for $g,h \in C^0_0(\Gamma) \cap H$,
we can apply \eqref{lMP6a} to $u = Ug - Uh$, and we get that 
$\sup_{\R^n} |u| = \sup_{\Omega} |u| \leq \sup_{\Gamma} |Tu| = ||g-h||_\infty$
because $u$ is continuous.

It is clear that $C^0_0(\Gamma) \cap H$ is dense in $C^0_0(\Gamma)$, because
(restrictions to $\Gamma$ of) compactly supported smooth functions on $\R^n$ 
(or even Lipschitz functions, for that matter) lie in $H$: compute their norm in \eqref{defH} directly.
Thus $U$ has a unique extension by continuity to $C^0_0(\Gamma)$. 
We could even define $U$, with the same properties, on its closure (continuous functions that tend 
to $0$ at infinity), but we decided not to bother. 

We are now ready to check the various properties of $U$. Given $g\in C^0_0(\Gamma)$, select a sequence
$(g_k)$ of compactly supported smooth functions that converges to $g$ in the sup norm. Then $u_k = Ug_k$
converges uniformly in $\R^n$ to $u = Ug$, and in particular $u$ is continuous and its restriction to $\Gamma$
is $g$, as in (i). In addition, (ii) holds because
$\sup_{\R^n} u = \lim_{k \to +\infty} \sup_{\R^n} u_k \leq \lim_{k \to +\infty} \sup_{\Gamma} g_k
= \sup_{\Gamma} g$, and similarly for the infimum. 

For (iii) we first need to check that $u\in \WW(\Omega)$.  
Observe that we know these facts for the $u_k$, so we'll only need to take limits.
Let $\phi \in C^\infty_0(\Omega)$ be given. Lemma~\ref{CaccioI} (Caccioppoli's inequality)
says that, since $u_k$ is a solution, 
\begin{equation} \label{defhm3c}
\int_{\Omega} |\nabla (\phi u_k) |^2 dm  \leq C \int_{\Omega} \phi^2 |\nabla u_k |^2 dm + C \int_{\Omega} |\nabla \phi|^2 |u_k |^2 dm  \leq C \int_{\Omega} |\nabla \phi|^2 |u_k|^2 dm.
\end{equation}
The right-hand side of \eqref{defhm3c} converges to $C \int_{\Omega} |\nabla \phi|^2 |u|^2 dm$, 
since $|\nabla \phi|^2$ is bounded and compactly supported. 
So $\int_{B} |\nabla (\phi u_k)|^2 dm$ is bounded uniformly in $k$.
Since the $\phi u_k$ vanish outside of the support of $\phi$ (which lies far from $\Gamma$)
and converge uniformly to $\phi u$, we get that the $\phi u_k$ converge to $\phi u$ in $L^1$ and,
since the $|\nabla (\phi u_k)|$ are uniformly bounded in $L^2(\Omega,w)$, we can find a subsequence
for which they converge weakly to a limit $V \in L^2(\Omega,w)$. 
We easily check on test functions that $\nabla (\phi u) = V$, hence $\phi u \in W$ for any
$\phi \in C^\infty_0(\Omega)$, and $u\in \WW(\Omega)$.

Next we check that $u$ is a solution in $\Omega$, i.e., that for $\varphi \in C^\infty_0(\Omega)$, \begin{equation} \label{defhm3b}
\int_\Omega \A\nabla u \cdot \nabla \varphi \, dm = 0.
\end{equation}
Let $\varphi \in C^\infty_0(\Omega)$ be given, and choose $\phi \in C^\infty_0(\Omega)$ such that 
$\phi \equiv 1$ on $\supp \, \varphi$. We just proved that for some subsequence,  
$\nabla(\phi u_k)$ converges weakly to $\nabla (\phi u)$ in $L^2(\Omega,w)$. Since $u_k$ is a solution 
for every $k$,
\begin{eqnarray} \label{defhm3d}
\int_\Omega \A\nabla u \cdot \nabla \varphi \, dm 
&=& \int_\Omega \A\nabla (\phi u) \cdot \nabla \varphi \, dm 
=  \lim_{k\to \infty} \int_\Omega \A\nabla (\phi u_k) \cdot \nabla \varphi \, dm 
\nn\\
&=&  \lim_{k\to \infty} \int_\Omega \A \nabla u_k \cdot \nabla \varphi \, dm = 0.
\end{eqnarray}
This proves \eqref{defhm3b} and (iii) follows.

For (iv), suppose in addition that $g \equiv 0$ on a ball $B$ centered on $\Gamma$; 
we want similar results in $B$ (that is, across $\Gamma$).
Notice that it is easy to approximate it (in the supremum norm) by smooth, 
compactly supported functions $g_k$ that also vanish on $\Gamma \cap B$.
Let use such a sequence $(g_k)$ to define $Ug = \lim_{k \to +\infty} Ug_k$.

Let $\varphi \in C^\infty_0(B)$ be given, and let us check that $\varphi u \in W$. 
Set $K = \supp \, \varphi$, suppose $K \neq \emptyset$,  and set $\delta = \dist(K,\d B) > 0$.
Cover $K \cap \Gamma$ by a finite number of balls balls $B_i$ of radius $10^{-1}\delta$ centered on 
$K \cap \Gamma$, and then cover $K' = K \sm \cup_{i} B_i$ by a finite number of balls $B_j$ of radius
$10^{-2}\delta$ centered on that set $K'$. We can use a partition of unity composed of smooth functions
supported in the $2B_i$ and the $2B_j$ to reduce to the case when $\varphi$ is supported 
on a $2B_i$ or a $2B_j$. 

Suppose for instance that $\varphi$ is supported in $2B_i$. We can apply Lemma~\ref{CaccioB}
(Caccioppoli's inequality at the boundary) to $u_k = Ug_k$ on the ball $2B_i$ , 
because its trace $g_k$ vanishes on $4B_i$ . We get that
\begin{equation} \label{defhm3f}
\int_{2B_i} |\nabla (\varphi u_k)|^2 dm  \leq C \int_{2B_i} (|\varphi \nabla u_k)|^2 
+ |u_k \nabla \varphi|^2) dm \leq \int_{4B_i} |\nabla \varphi|^2 |u_k|^2 dm.
\end{equation}
With this estimate, we can proceed as with \eqref{defhm3c} above to prove that $\varphi u \in W$
and its derivative is the weak limit of the $\nabla(\varphi u_k)$. When instead $\varphi$ is supported in 
a $2B_j$, we use the interior Caccioppoli inequality (Lemma \ref{CaccioI} and proceed as above).

Thus $u = Ug$ lies in $W_r(B)$, and this proves (iv). We started the proof with (v), so this completes our proof of Lemma \ref{ldefhm}.
\ep

\medskip

Our next step is the construction of 
the harmonic measure. Let $X\in \Omega$. By Lemma~\ref{ldefhm}, the linear form
\begin{equation} \label{defhm5}
g  \in C^0_0(\Gamma) \to 
Ug(X)
\end{equation}
is bounded and positive 
(because $u = Ug$ is nonnegative when $g\geq 0$).
The following statement is thus a direct consequence of the Riesz representation theorem (see for instance \cite[Theorem 2.14]{Rudin}).

\begin{lemma}
There exists a unique positive regular Borel measure $\omega^X$ on $\Gamma$ such that 
\begin{equation} \label{defhm}
Ug(X) = 
\int_\Gamma g(y) d\omega^X(y)
\end{equation}
for any $g \in C^0_0(\Gamma)$. Besides, for any Borel set $E \subset \Gamma$,
\begin{equation} \label{defhm6}
\omega^X(E) = \sup\{ \omega^X(K): \, E \supset K, \, K \text{ compact}\} 
= \inf\{ \omega^X(V): \, E \subset V, \, V \text{ open}\}. 
\end{equation}
\end{lemma}

\medskip

The harmonic measure is a probability measure, as proven in the following result.

\begin{lemma} \label{defhm12} 
For any $X\in \Omega$,
\[\omega^X(\Gamma) = 1.\]
\end{lemma}

\bp Let $X\in \Omega$ be given. 
Choose $x\in \Gamma$ such that $\delta(X) = |X-x|$. Set then $B_j = B(x,2^j \delta(X))$. 
According to \eqref{defhm6},  
\begin{equation} \label{defhm8} 
\omega^X(\Gamma) = \lim_{j\to +\infty} \omega^X(\overline {B_j}).
\end{equation}
Choose, for  
$j\geq 1$, $\bar g_j \in C^\infty_0(B_{j+1})$ such that $0\leq \bar g_j \leq 1$ and $\bar g_j \equiv 1$ on $\overline{B_{j}}$ and then define $g_j = T(\bar g_j)$. Since the harmonic measure is positive, we have
\begin{equation} \label{defhm9} 
\omega^X(B_{j}) \leq \int_\Gamma g_j(y) d\omega^X(y) \leq \omega^X(B_{j+1}).
\end{equation}
Together with \eqref{defhm8},
\begin{equation} \label{defhm10} 
\omega^X(\Gamma) = \lim_{j\to +\infty} \int_\Gamma g_j(y) d\omega^X(y) = \lim_{j\to +\infty} u_j(X),
\end{equation}
where $u_j$ is the image by the map \eqref{defhm4} of the function $g_j$. 
Since $g_j$ is the trace of a smooth and compactly supported function, $g_j \in H$ and so $u_j \in W$ is the solution of \eqref{Dp2} with data $g_j$. Moreover,  
$0\leq u_j \leq 1$ by 
Lemma~\ref{ldefhm} (ii). We want to show that $u_j(X) \to 1$ when $j \to +\infty$.
The function $v_j := 1-u_j \in W$ is a solution in $B_{j}$ satisfying $Tv_j \equiv 0$ on $B_j$. 
So Lemma~\ref{HolderB} says that 
 \begin{equation} \label{defhm11} 
0\leq 1-u_j(X) = v_j(X) \leq \osc_{B_1} v_j \leq C2^{-j\alpha} \osc_{B_j} v_j \leq C2^{-j\alpha},
\end{equation}
where $C>0$ and $\alpha >0$ are independent of $j$. It follows that $v_j(X)$ tends to $0$, 
and  
$u_j(X)$ tends to $1$ 
when $j$ goes to $+\infty$. 
The lemma follows from this and \eqref{defhm10}, the lemma follows.
\ep

\begin{lemma} \label{lprhm}
Let $E \subset \Gamma$ be a Borel set and define the function $u_E$ on $\Omega$ by $u_E(X) = \omega^X(E)$. 
Then
\begin{enumerate}[(i)]
\item if there exists $X\in \Omega$ such that $u_E(X) = 0$, then $u_E \equiv 0$;
\item the function $u_E$ lies 
in $\WW(\Omega)$ and is a solution in $\Omega$;
\item if $B\subset \R^n$ is a ball such that
$E\cap B = \emptyset$, then $u_E \in \WW(B)$ and $Tu_E = 0$ on $B$.
\end{enumerate}
\end{lemma}

\bp First of all, 
$0\leq u_E\leq 1$ because $\omega^X$ is a positive probability measure for any $X \in \Omega$.

\medskip

Let us prove (i). Thanks to \eqref{defhm6}, it suffices to prove the result when $E=K$ is compact. 

Let $X\in \Omega$ be such that $u_K(X) = 0$.
Let $Y\in \Omega$ and $\epsilon>0$ be given. 
By \eqref{defhm6} again,
we can find an open $U$ such that $U \supset K$ and
$\omega^X(U) < \epsilon$. 
Urysohn's lemma (see for instance Lemma~2.12 in \cite{Rudin}) gives the existence of $g \in C^0_0(\Gamma)$ such that $0\leq g \leq 1$ and $g\equiv 1$ on $K$. 
Set $u = Ug$, where $U$ is as in \eqref{defhm4}.
Thanks to the positivity of the harmonic measure, $u_K \leq u$. 
Let $Y \in \Omega$ be given, and apply the Harnack inequality \eqref{Harnack2} to $u$
(notice that $u$ lies in $\WW(\Omega)$ and is 
a solution in $\Omega$ thanks to Lemma~\ref{ldefhm}). We get that
\begin{equation} \label{prhm1}
0\leq u_K(Y) \leq u(Y) \leq C_{X,Y} u(X) \leq C_{X,Y} \epsilon.
\end{equation}
Since \eqref{prhm1} holds for any positive $\epsilon$, we have $u_K(Y) =0$. Part (i) of the lemma follows.

\medskip

We turn to the proof of (ii),
which we first do when
$E=V$ is open. 
We first check that
\begin{equation}\label{ury0}
\text{$u_V$ is a continuous function on $\Omega$. }
\end{equation}
Fix $X \in \Omega$, and build an increasing sequence of compact sets $K_j \subset V$ 
such that $\omega^X(V) < \omega^X(K_j) + \frac1j$. 
With Urysohn's lemma again, we construct $g_j\in C^0_0(V)$ such that 
$\1_{K_j} \leq g_j \leq \1_V$ and, 
without loss of generality we can choose $g_j \leq g_i$ whenever $j\leq i$. 
Set $u_j = Ug_j \in C^0(\R^n)$, as in \eqref{defhm4}, and notice that 
$u_j(X) = \int_\Gamma g_j d\omega^X$ by \eqref{defhm}. Then for $j \geq 1$, 
\begin{equation}\label{ury}
u_{K_j}(X) = \omega^X(K_j) \leq u_j(X) \leq  \omega^X(V) = u_{V}(X)
\leq \omega^X(K_j)+\frac1j
\end{equation}
by definition of $u_E$, because the harmonic measure is nondecreasing,
and since $\1_{K_j} \leq g_j \leq \1_V$.
Similarly, $(u_j)$ is a nondecreasing sequence of functions, i.e., 
\begin{equation}\label{ury1}
u_i \geq u_j  \text{ on $\Omega \, $ for } i \geq j \geq 1,
\end{equation}
by the maximum principle in Lemma \ref{ldefhm} and because $g_i \geq g_j$, so that in particular
\begin{equation}\label{ury2}
u_j(X) \leq u_i(X) \leq u_j(X) + \frac1j \ \ \text{ for } i \geq j \geq 1,
\end{equation}
by \eqref{ury}.
Now $u_i - u_j$ is a nonnegative solution (by Lemma~\ref{ldefhm}), and Lemma~\ref{HarnackI2} 
implies that for every 
compact set $J \subset \Omega$, there exists $C_J >0$ such that
 \begin{equation} \label{prhm2}
0 \leq \sup_J (u_i -u_j) \leq C_J (u_i - u_j)(X) \leq \frac{C_J}j
 \end{equation}
for $i \geq j \geq 1$.
 We deduce from this
 that $(u_j)_j$ converges uniformly on compact sets of $\Omega$ to a function $u_\infty$, 
 which is therefore continuous on $\Omega$. Thus \eqref{ury0} will follow as soon as we prove that
 $u_\infty = u_V$. 
 
 Set $K = \bigcup_j K_j$; then $u_{K_j} \leq u_K \leq u_V$ by monotonicity of the harmonic measure,
 and \eqref{ury} implies that $u_K(X) = u_V(X)$. Now $u_V - u_K = u_{V \sm K}$, so
 $u_{V \sm K}(X) = 0$. By Point (i) of the present lemma, $u_{V \sm K}(Y) = 0$ for every $Y \in \Omega$.
 But $u_V(Y) = \omega^Y(V)$, and $\omega^Y$ is a measure, so 
$u_{V \sm K}(Y) = \lim_{j \to +\infty} u_{V \sm K_j}(Y)
= u_{V}(Y) - \lim_{j \to +\infty} u_{K_j}(Y)$. 

Since $u_{K_j}(Y) \leq u_j(Y) \leq u_{V}(Y)$ by the proof of \eqref{ury}, we get that $u_j(Y)$ tends
to $u_{V}(Y)$. In other words, $u_\infty(Y) = u_{V}(Y)$, and \eqref{ury0} follows as announced.
 
 We proved that $u_V$ is continuous on $\Omega$ and that 
it is the limit, uniformly on compact subsets of $\Omega$, of a 
 sequence of functions $u_j \in C^0(\R^n) \cap \WW(\Omega)$, which are also
 solutions of $L$ in $\Omega$. We now want to prove 
 that $u_V \in \WW(\Omega)$, and we proceed as we did near \eqref{defhm3c}.
 
Let $\phi\in C^\infty_0(\Omega)$ be given. In the distributional sense, we have $\nabla(\phi u_j) = u_j \nabla \phi + \phi \nabla u_j$.  So the 
Caccioppoli inequality given by Lemma~\ref{CaccioI} yields 
 \begin{equation} \label{prhm4}
\int_\Omega |\nabla (\phi u_j)|^2 dm \leq C \int_\Omega (|\nabla \phi|^2 |u_j|^2 + \phi^2 |\nabla u_j|^2) dm \leq C \int_\Omega |\nabla \phi|^2 |u_j|^2 dm.
 \end{equation}
Since the $u_j$ converge
to $u$ uniformly on $\supp \, \phi$,  
the right-hand side of \eqref{prhm4} converges to $C \int_\Omega |\nabla \phi|^2 |u|^2 dm$.
Consequently, the left-hand side of \eqref{prhm4} is uniformly bounded in $j$ and hence there exists $v \in L^2(\Omega,w)$ such that $\nabla (\phi u_j)$ converges weakly to $v$ in $L^2(\Omega,w)$. By uniqueness of the limit, the distributional derivative $\nabla(\phi u_V)$ equals $v \in L^2(\Omega,w)$, so by definition of $W$, $\phi u_V \in W$. Since the result holds for any $\phi \in C^\infty_0(\Omega)$, we just established $u_V\in \WW(\Omega)$ as desired.
In addition, we also checked that (for a subsequence) $\nabla (\phi u_j)$
converges weakly in $L^2(\Omega,w)$ to $\nabla (\phi u_V)$.

We now establish that $u_V$ is a solution. Let
$\varphi \in C^\infty_0(\Omega)$ be given. Choose 
$\phi \in C^\infty_0(\Omega)$ such that $\phi \equiv 1$ on $\supp \, \varphi$. Thanks to the weak 
convergence of $\nabla (\phi u_j)$ to $\nabla (\phi u_j)$
 \begin{equation} \label{prhm5} \begin{split}
\int_\Omega \A \nabla u_V \cdot \nabla \varphi \, dm & = \int_\Omega \A \nabla (\phi u_V)\cdot \nabla \varphi \, dm \\
& = \lim_{j\to +\infty} \int_\Omega \A \nabla (\phi u_j)\cdot \nabla \varphi \, dm 
= \lim_{j\to +\infty} \int_\Omega \A \nabla u_j \cdot \nabla \varphi \, dm = 0 \\
\end{split} \end{equation}
because each $u_j$ is a solution. Hence $u_V$ is a solution.

This completes our proof of (ii) when $E = V$ is open.
The proof of (ii) for general Borel sets $E$ works similarly, 
but we now approximate $E$ from above by open sets. Fix $X\in \Omega$. Thanks to 
the regularity property \eqref{defhm6}, there exists a decreasing sequence $(V_j)$ of open sets that 
contain $E$, and for which $u_{V_j}(X)$ tends to $u_{E}(X)$. 

From our previous work, we know that each
$u_{V_j}$ is continuous on $\Omega$, lies in $\WW(\Omega)$, and is a solution in $\Omega$. 
Using the same process as before, we can show first that the $u_{V_j}$ converge, uniformly on 
compact sets of $\Omega$, to $u_E$, which is then continuous on $\Omega$.
Then we prove that, for any $\phi \in C^\infty_0(\Omega)$, $\nabla (\phi u_{V_j})$ converges weakly in $L^2(\Omega,w)$ to $\nabla (\phi u_E)$, from which we deduce $u_E \in \WW(\Omega)$ and then that $u_E$ is a solution.

\medskip

Part (iii) of the lemma remains to be proven. Let $B \subset \R^n$ be a ball such that $B \cap E = \emptyset$. Since $u_E$ lies in $\WW(\Omega)$ 
and is a solution, 
Lemma~\ref{HolderI} says that $u_E$ is continuous in $\Omega$. 
We first prove that if we set $u = 0$ on $B \cap \Gamma$, we get a continuous extension of $u$,
(with then has a vanishing trace, or restriction, on $B \cap \Gamma$). 

Let $x\in B\cap \Gamma$ be given.
Choose $r>0$ such that $B(x,2r) \subset B$ and then construct a function $\bar g \in C^\infty_0(B(x,2r))$ 
such that $\bar g\equiv 1$ in $B(x,r)$. Since $\bar g$ is smooth and compactly supported, 
$g : = T(\bar g)$ lies in 
$H \cap C^0_0(\Gamma)$ and then 
$u = Ug $, the image of $g$ by the map of \eqref{defhm4}, lies in
in $W \cap C^0(\R^n)$. From the positivity of the harmonic measure, we deduce that 
$0 \leq u_E \leq 1-u$. 
Since $0$ and $1-u$ are both continuous functions that are equal $0$ at $x$, 
the squeeze theorem says that 
$u_E$ is continuous (or can be extended by continuity) at $x$, and $u_E(x) = 0$.

To complete the proof of the lemma, we show that $u_E$ actually lies in $\WW(B)$. 
As for the proof of (ii), 
we first assume that $E=V$ is open. 
We take a nondecreasing sequence of compact sets $K_j \subset V$ that converges to $V$, 
and then we build $g_j \in C_0^0(V)$, 
such that $\1_{K_j} \leq g_j \leq \1_V$ and the sequence $(g_j)$ is non-decreasing. 
We then take $u_j = U g_j$ (with the map from \eqref{defhm4}), and in particular the 
sequence $(u_j)$ is non-decreasing on $\Omega$.
From the proof of (ii), we know that $u_j$ converges to $u_V$ on compact sets of $\Omega$, 
then in particular $u_j$ converges pointwise to $u_V$ in $\Omega$. 

\noindent Let $\varphi \in C^\infty_0(B)$;
we want to prove that $\varphi u_V \in W$. From Lemma~\ref{CaccioB}, we have
 \begin{equation} \label{prhm6}
\int_B |\nabla (\varphi u_j)|^2 dm \leq C \int_B (|\nabla \varphi|^2 |u_j|^2 + \varphi^2 |\nabla u_j|^2) dm \leq C \int_B |\nabla \varphi|^2 |u_j|^2 dm.
 \end{equation}
Since $u$ is continuous on $B$, $u_V \in L^2(\supp \, \varphi,w)$ and the right-hand side 
converges to $C \int_B |\nabla \varphi|^2 |u_V|^2 dm$ by 
the dominated convergence theorem. The left-hand side is thus uniformly bounded in $j$ and
$\nabla (\varphi u_j)$ converges weakly, 
maybe after extracting a subsequence,
to some $v$ in $L^2(B,w)$. 
By uniqueness of the limit, $v = \nabla (\varphi u_V) \in L^2(B,w)$. Since the result holds for all $\varphi \in C^\infty_0(B)$, we get $u_V \in \WW(B)$. 

In the general case where $E$ is a Borel set, fix $X\in \Omega$ and take a 
decreasing sequence of open sets $V_j \supset X$ such that
$u_{V_j}(X) \to u_{E}(X)$.
We can prove using part (i) of this lemma that $u_{V_j}$ converges to $u_E$ pointwise in $\Omega$.
Then we use Lemma~\ref{CaccioB} to show that for $\varphi \in C^\infty_0(B)$,
 \begin{equation} \label{prhm7}
\int_B |\nabla (\varphi u_{V_j})|^2 dm \leq C \int_B |\nabla \varphi|^2 |u_{V_j}|^2 dm
 \end{equation}
when $j$ is so large that $V_j$ is far from the support of $\varphi$.
 The right-hand side has a limit, thanks to the dominated convergence theorem, thus the left-hand side is uniformly bounded in $j$. So there exists a subsequence of $\nabla (\varphi u_{V_j})$ that converges weakly in $L^2(B,w)$, and by uniqueness to the limit, the limit has to be $\nabla (\varphi u_E)$, which thus lies 
 in $L^2(B,w)$. We deduce that
 $\varphi u_E \in W$ and then $u \in \WW(B)$.
\ep
\chapter{Green Functions}

The aim of this section is to define a Green function, that is, formally, a function $g$ 
defined on $\Omega \times \Omega$ and such that for 
$y \in \Omega$,
\begin{equation} \label{GreenI1}
\left\{ \begin{array}{l} 
Lg(.,y) = \delta_y \quad \text{ in } \Omega\\
Tg(.,y) = 0 \quad \text{ on } \Gamma.
\end{array} \right.
\end{equation}
where $\delta_y$ denotes the Dirac distribution.

Our proof of existence and uniqueness, and the estimates below, are adapted from
arguments of \cite{GW} (see also \cite{HoK} and \cite{DK}) for the classical case of codimension $1$.

\begin{lemma} \label{GreenEx}
There exists a non-negative function $g:\ \Omega \times \Omega \to \R \cup \{+\infty\}$ 
with the following properties.
\begin{enumerate}[(i)]
\item For any $y\in \Omega$ and 
any function $\alpha \in C_0^\infty(\R^n)$ 
such that $\alpha \equiv 1$ in a neighborhood of $y$
\begin{equation} \label{GreenE1}
(1-\alpha) g(.,y) \in W_0.
\end{equation}
In particular, $g(.,y) \in \WW(\R^n \setminus \{y\})$ and $T[g(.,y)] = 0$.
\item For 
every choice of $y\in \Omega$, $R>0$, and $q\in [1,\frac{n}{n-1})$,
\begin{equation} \label{GreenE4}
g(.,y) \in W^{1,q}(B(y,R)): = \{u \in L^q(B(y,R)), \, \nabla u \in L^q(B(y,R))\}.
\end{equation}
\item For 
$y\in \Omega$ and  
$\varphi \in C^\infty_0(\Omega)$,
\begin{equation} \label{GreenE6}
\int_\Omega A \nabla_x g(x,y)\cdot \nabla \varphi(x) \, dx = \varphi(y).
\end{equation}
In particular, $g(.,y)$ is a solution 
of $Lu=0$ 
in $\Omega \setminus \{y\}$. 
\end{enumerate}
In addition, the following bounds hold.
\begin{enumerate}[(i)] \setcounter{enumi}{3}
\item For 
$r>0$,
$y\in \Omega$ and 
$\epsilon>0$,
\begin{equation} \label{GreenE2}
\int_{\Omega \setminus B(y,r)} |\nabla_x g(x,y)|^2 dm(x) \leq 
\left\{\begin{array}{ll} Cr^{1-d} & \text{ if } 4r \geq \delta(y) \\ 
\frac{Cr^{2-n}}{w(y)} & \text{ if } 2r \leq \delta(y), \, n\geq 3 \\
\frac{C_\epsilon}{w(y)} \left(\frac{\delta(y)}{r} \right)^\epsilon & \text{ if } 2r \leq \delta(y), \, n=2,\end{array} \right.
\end{equation}
where $C>0$ depends on $d$, $n$, $C_0$, $C_1$ and $C_\epsilon>0$ 
depends on $d$, $C_0$, $C_1$, and $\epsilon$.
\item For 
$x,y\in \Omega$ such that 
$x\neq y$ and 
$\epsilon>0$,
\begin{equation} \label{GreenE7}
0 \leq
g(x,y) \leq \left\{\begin{array}{ll} C|x-y|^{1-d} & \text{ if } 4|x-y| \geq \delta(y) \\ 
\frac{C|x-y|^{2-n}}{w(y)} & \text{ if } 2|x-y| \leq \delta(y), \, n\geq 3 \\
\frac{C_\epsilon}{w(y)} \left(\frac{\delta(y)}{|x-y|} \right)^\epsilon & \text{ if } 2|x-y| \leq \delta(y), \, n=2,\end{array} \right.
\end{equation}
where again $C>0$ depends on $d$, $n$, $C_0$, $C_1$ and $C_\epsilon>0$ depends on $d$, $C_0$, $C_1$, $\epsilon$.
\item For
$q\in [1,\frac{n}{n-1})$ and 
$R \geq \delta(y)$,
\begin{equation} \label{GreenE5}
\int_{B(y,R)} |\nabla_x g(x,y)|^q dm(x) \leq C_q R^{d(1-q)+1},
\end{equation}
where $C_q>0$ depends on $d$, $n$, $C_0$, $C_1$, and $q$.
\item For
$y\in \Omega$, 
$R\geq \delta(y)$,
$t>0$ and 
$p\in [1,\frac{2n}{n-2}]$ (if $n\geq 3$) or $p\in [1,+\infty)$ (if $n=2$),  
\begin{equation} \label{GreenE2w}
\frac{m(\{x\in B(y,R), \, g(x,y) >t\})}{m(B(y,R))} \leq C_p \left(\frac{R^{1-d}}t\right)^\frac p2,
\end{equation}
where $C_p>0$ depends on $d$, $n$, $C_0$, $C_1$ and $p$.
\item For
$y\in \Omega$, 
$t>0$ and
$\eta \in (0,2)$,
\begin{equation} \label{GreenE5w}
m(\{x\in \Omega, \, |\nabla_x g(x,y)| >t\}) \leq \left\{\begin{array}{ll} Ct^{-\frac{d+1}{d}} & \text{ if } t \leq \delta(y)^{-d} \\ 
Cw(y)^{-\frac{1}{n-1}} t^{-\frac{n}{n-1}} & \text{ if } t \geq \delta(y)^{-d}, \, n\geq 3 \\
C_\eta w(y)^{-1} \delta(y)^{d\eta} t^{\eta-2} & \text{ if } t \geq \delta(y)^{-d}, \, n=2,\end{array} \right.
\end{equation}
where $C>0$ depends on $d$, $n$, $C_0$, $C_1$ and $C_\eta>0$ depends on $d$, $C_0$, $C_1$, $\eta$.
\end{enumerate}
\end{lemma}

\begin{remark}
When $d<1$ and $|x-y| \geq \frac12 \delta(y)$, the bound $g(x,y) \leq C|x-y|^{1-d}$ 
given in \eqref{GreenE7} can be improved into
\begin{equation} \label{GreenE8}
g(x,y) \leq C \min\{\delta(x),\delta(y)\}^{1-d}.
\end{equation}
This fact is proven in Lemma \ref{ltcp5} below.
\end{remark}

\begin{remark}
The authors believe that the bounds given in \eqref{GreenE2} and \eqref{GreenE7} when $n=2$ and $2r$ 
(or $2|x-y|$) is smaller than $\delta(y)$ are not optimal. One should be able to replace for instance the bound $\frac{C_\epsilon}{w(y)} \left(\frac{\delta(y)}{r} \right)^\epsilon$ by $\frac{C}{w(y)} \ln \left( \frac{\delta(y)}{r} \right)$ in \eqref{GreenE2} by adapting the arguments of \cite{DK} (see also \cite[Theorem 3.3]{FJK}). However, the estimates given above are sufficient for our purposes and we didn't want to make this article even longer.
\end{remark}

\begin{remark}
Note that when 
$n\geq 3$, thanks to Lemma~\ref{lwest}, the bound \eqref{GreenE7} can be gathered into a single estimate
\begin{equation} \label{GreenN}
g(x,y) \leq C \frac{|x-y|^{2}}{m(B(y,{|x-y|}))}
\end{equation}
whenever $x,y\in \Omega $, $x\neq y$. In the same way, also for $n\geq 3$, the bound \eqref{GreenE2} can be gathered into a single estimate
\begin{equation} \label{GreenN00}
\int_{\Omega \setminus B(y,r)} |\nabla_x g(x,y)|^2 dm(x) \leq C \frac{r^{2}}{m(B(y,r))}
\end{equation}
whenever $y\in \Omega$ and $r>0$.
\end{remark}

\bp This proof will adapt the arguments of \cite[Theorem 1.1]{GW}. 

Let $y \in \Omega$ be fixed. Consider again the bilinear form $a$ on $W_0 \times W_0$ defined as
\begin{equation} \label{defofa2}
a(u,v)  = \int_\Omega A\nabla u \cdot \nabla v = \int_\Omega \A \nabla u \cdot \nabla v \, dm.
\end{equation}
The bilinear form $a$ is bounded and coercive on $W_0$, thanks to \eqref{ABounded} and \eqref{AElliptic}. 

Let $\rho >0$ be small. Take, for instance, $\rho$ such that $100\rho<\delta(y)$. Write $B_\rho$ for $B(y,\rho)$. The linear form
\begin{equation} \label{Green1}
\varphi \in W_0 \to \fint_{B_\rho} \varphi(z) \, dz
\end{equation}
is bounded. Indeed, let $z$ be a point in $\Gamma$, then 
\begin{equation} \label{Green2}
\left|\fint_{B_\rho} \varphi(z) \, dz \right| \leq C_{y,z,\rho} \fint_{B(z,|y-z|+\rho)} |\varphi(z)| \, dz \leq C_{y,z,\rho} \|\varphi\|_W 
\end{equation}
by 
Lemma~\ref{lpBry}. By the Lax-Milgram theorem, there exists then a unique function $\g^\rho = g^\rho(.,y) \in W_0$ such that 
\begin{equation} \label{Green3}
a(\g^\rho,\varphi) = \int_\Omega \A \nabla \g^\rho \cdot \nabla \varphi \, dm = \fint_{B_\rho} \varphi(z) \, dz \qquad \forall \varphi \in W_0.
\end{equation}
We like $\g^\rho$, and will actually spend some time studying it, because $g(\cdot,y)$ will later
be obtained as a limit of the $\g^\rho$. By \eqref{Green3}, 
\begin{equation} \label{Green4}
\text{$\g^\rho \in W_0$ is a solution of $L\g^\rho = 0$ in $\Omega \setminus \overline{B_\rho}$.}
\end{equation}
This fact will be useful later on.

\medskip

For now, let us prove that $\g^\rho \geq 0$ a.e. on $\Omega$. Since $\g^\rho \in W_0$, Lemma~\ref{lcompo} yields $|\g^\rho| \in W_0$, $\nabla |\g^\rho| = \nabla \g^\rho$ a.e. on $\{\g^\rho >0\}$, $\nabla |\g^\rho| = -\nabla \g^\rho$ a.e. on $\{\g^\rho <0\}$ and $\nabla |\g^\rho| = 0$ a.e. on $\{\g^\rho = 0\}$. Consequently
\begin{equation} \label{Green5}
\int_\Omega \A \nabla |\g^\rho| \cdot \nabla |\g^\rho| \, dm = \int_{\{\g^\rho >0\}} \A \nabla \g^\rho \cdot \nabla \g^\rho \, dm + \int_{\{\g^\rho <0\}} \A \nabla \g^\rho \cdot \nabla \g^\rho \, dm = \int_\Omega \A \nabla \g^\rho \cdot \nabla \g^\rho \, dm
\end{equation}
and
\begin{equation} \label{Green6}
\int_\Omega \A \nabla |\g^\rho| \cdot \nabla \g^\rho \, dm = \int_{\{\g^\rho >0\}} \A \nabla \g^\rho \cdot \nabla \g^\rho \, dm - \int_{\{\g^\rho <0\}} \A \nabla \g^\rho \cdot \nabla \g^\rho \, dm = \int_\Omega \A \nabla \g^\rho \cdot \nabla |\g^\rho| \, dm,
\end{equation}
which can be rewritten $a(|\g^\rho|,|\g^\rho|) = a(\g^\rho,\g^\rho)$ and $a(|\g^\rho|,\g^\rho) = a(\g^\rho,|\g^\rho|)$. Moreover, if we use $\g^\rho \in W_0$ and $|\g^\rho| \in W_0$ as test functions in \eqref{Green3}, we obtain
\begin{equation} \label{Green7}
a(|\g^\rho|,|\g^\rho|) = a(\g^\rho,\g^\rho) = \int_{B_\rho} \g^\rho(z) \, dz \leq \int_{B_\rho} |\g^\rho(z)|\, dz = a(\g^\rho,|\g^\rho|) = a(|\g^\rho|,\g^\rho).
\end{equation}
Hence $a(|\g^\rho|-\g^\rho,|\g^\rho|-\g^\rho) \leq 0$ and, by
the coercivity of $a$, $\g^\rho = |\g^\rho| \geq 0$ a.e. on $\Omega$.

\medskip

Let $R\geq \delta(y)> 100\rho>0$. We write again $B_R$ for $B(y,R)$. Let $p$ in the range given by Lemma~\ref{lSob}, that is $p\in [1,2n/(n-2)]$ if $n\geq 3$ and $p\in [1,+\infty)$ if $n=2$.
We aim to prove that for all $t>0$,
\begin{equation} \label{Green8}
\frac{m(\{x\in B_R, \, \g^\rho(x) > t\})}{m(B_R)} \leq C t^{-\frac{p}2} R^{\frac{p}2(1-d)}
\end{equation}
with a constant $C$ independent of $\rho$, $t$ and $R$.

We use \eqref{Green3} with the test function 
\begin{equation} \label{Green9}
\varphi(x) := \left(\frac 2{t}-\frac 1{\g^\rho(x)}\right)^+ = \max\left\{0, \frac 2{t}-\frac 1{\g^\rho(x)}\right\}
\end{equation}
(and $\varphi(x) = 0$ if $\g^{\rho}(x) = 0$), which lies in $W_0$ by Lemma~\ref{lcompo}.
So if $\Omega_{s} := \{x\in \Omega, \, \g^\rho(x) > s\}$, we have
\begin{equation} \label{GreenA}
a(\g^\rho,\varphi) = \int_{\Omega_{t/2}} \frac{\A \nabla \g^\rho \cdot  \nabla \g^\rho}{(\g^\rho)^2} \, dm = \fint_{B_r} \varphi \leq \frac2{t}.
\end{equation}
Therefore, with the ellipticity condition \eqref{AElliptic},
\begin{equation} \label{GreenB}
\int_{\Omega_{t/2}} \frac{|\nabla \g^\rho|^2}{(\g^\rho)^2} \, dm \leq \frac Ct.
\end{equation}
Pick $y_0 \in \Gamma$ such that
$|y-y_0| = \delta(y)$. Set $\wt B_R$ for $B(y_0,2R) \supset B_R$. 
Also  define $v$  by 
$v(x) : = (\ln(\g^\rho(x))-\ln t + \ln 2)^+$, which lies in $W_0$ too, 
thanks to Lemma~\ref{lcompo}.
The Sobolev-Poincar\'e inequality \eqref{1.5-bisbis} implies that 
\begin{equation} \label{GreenC}
\left(\int_{\Omega_{t/2} \cap \wt B_R} |v|^p \, dm\right)^\frac1p \leq C R \,  m(\wt B_R)^{\frac1p-\frac12} \left(\int_{\Omega_{t/2} \cap \wt B_R} |\nabla v|^2 \, dm\right)^\frac12 \leq C R \, m(\wt B_R)^{\frac1p-\frac12} t^{-\frac12}
\end{equation}
by  
\eqref{GreenB}. Since $m(\wt B_R) \approx R^{d+1}$ thanks to Lemma~\ref{lwest}, one has
\begin{equation} \label{GreenD}
\int_{\Omega_{t/2} \cap B_R} \left|\ln\left(\frac{2\g^\rho}{t} \right)\right|^p \, dm \leq C R^{p+(d+1)(1-\frac p2)} t^{-\frac p2}.
\end{equation}
But the latter implies, since $v > \ln 2$ on $\Omega_t$, that 
\begin{equation} \label{GreenE}
(\ln 2)^p m(\Omega_t \cap B_R) \leq C R^{p+(d+1)(1-\frac p2)} t^{-\frac p2} = C t^{-\frac p2} R^{\frac p2(1-d)+(d+1)}.
\end{equation}
The claim \eqref{Green8} follows once we notice that, due to Lemma \ref{lwest}, 
we have $m(B_R) \approx R^{d+1}$.

\medskip

Now 
we give a pointwise estimate on $\g^\rho$ when $x$ is far from $y$. We claim that
\begin{equation} \label{GreenF}
\g^\rho(x) \leq C |x-y|^{1-d} \qquad \text{ if } 4|x-y| \geq\delta(y) >100\rho,
\end{equation}
where again $C>0$ is independent of $\rho$.
Set $R=4|x-y| > \delta(y)$. Recall \eqref{Green4}, i.e., that
$\g^\rho$ lies
in $W_0$ and is a solution in $\Omega \setminus \overline{B_\rho}$. 
So we can use the Moser estimates to get that
\begin{equation} \label{GreenG0}
\g^\rho(x) \leq C \frac1{m(B(x,R/2))} \int_{B(x,R/2)} \g^\rho \, dm.
\end{equation}
Indeed, \eqref{GreenG0} is obtained with Lemma~\ref{MoserI} when $\delta(x) \geq R/30$ (apply Moser inequality in the ball $B(x,R/90)$) and with Lemma~\ref{MoserB2} when $\delta(x) \leq R/30$ (apply Moser inequality in the ball $B(x_0,R/15)$ where $x_0$ is such that $|x-x_0| = \delta(x)$). 

\noindent We can use now the fact that $B(x,R/2) \subset B_R$ and \cite[p. 28, Proposition 2.3]{Duoandi} to get
 \begin{equation} \label{GreenG}
\g^\rho(x) \leq C \int_0^{+\infty}  \frac{m(\Omega_{t} \cap B_R)}{m(B_R)}dt
\end{equation}
Take $s>0$, to be chosen later.
By
\eqref{Green8}, applied with any
 valid $p>2$ (for instance $p=\frac{2n}{n-1}$),
 \begin{equation} \label{GreenH0} \begin{split}
\g^\rho(x) & \leq C \int_0^{s}  \frac{m(\Omega_{t} \cap B_R)}{m(B_R)}dt + C \int_s^{+\infty}  \frac{m(\Omega_{t} \cap B_R)}{m(B_R)}dt \\
& \leq Cs + C R^{\frac p2(1-d)} \int_s^{+\infty} t^{-\frac p2} dt 
\leq Cs + C R^{\frac p2(1-d)} s^{1-\frac p2}.
\end{split} \end{equation}
We minimize the right-hand 
side in $s$. We find $s \approx R^{1-d}$ and then $\g^\rho(x) \leq C R^{1-d}$. The claim \eqref{GreenF} follows.

\medskip

Let us now prove some 
pointwise estimates on $\g^\rho$ when $x$ is close to $y$. 
When $n\geq 3$, we want to show that
\begin{equation} \label{GreenH}
\g^\rho(x) \leq C \frac{|x-y|^{2-n}}{w(y)} \qquad \text{ if }  \delta(y) \geq 2|x-y| > 4\rho \text{ and } \delta(y) > 100\rho,
\end{equation}
where $C>0$ is independent of $\rho$, $x$ and $y$. 
When $n=2$,
we claim that for any $\epsilon >0$,
\begin{equation} \label{GreenH2}
\g^\rho(x) \leq C_\epsilon \frac{1}{w(y)} \left(\frac{\delta(y)}{r} \right)^{\epsilon} \qquad \text{ if }  \delta(y) \geq 2|x-y| > 4\rho \text{ and } \delta(y) > 100\rho,
\end{equation}
where $C_\epsilon >0$ is also independent of $\rho$, $x$ and $y$.
The proof works 
a little like when 
$x$ is far from $y$, but we need to be a bit more careful about the Poincar\'e-Sobolev inequality 
that we use.
Set again $r = 2|x-y|$. Lemma~\ref{MoserI} applied to the ball $B(x,r/20)$ yields 
\begin{equation} \label{GreenI}
\g^\rho(x) \leq \frac{C}{m(B(x,r/2))} \int_{B(x,r/2)} \g^\rho \, dm 
\leq  \frac{C}{m(B_{r})} \int_{B_{r}} \g^\rho \, dm
\end{equation}
and then for $s>0$ and $R>r$ to be chosen soon,
\begin{equation} \label{GreenJ}
\g^\rho(x) \leq C \int_0^s \frac{m(\Omega_t \cap B_{r})}{m(B_{r})} dt + C \frac{m(B_R)}{m(B_{r})} \int_s^{+\infty} \frac{m(\Omega_t \cap B_{R})}{m(B_{R})} dt.
\end{equation}
Take $R=\delta(y)$. The doubling property \eqref{doublinggen} allows us to estimate 
$\frac{m(B_R)}{m(B_{r})}$ by $\big( \frac{\delta(y)}{r} \big)^n$. 
Let $p$ lie in the range given by Lemma~\ref{lSob}, and apply \eqref{Green8}
to estimate $m(\Omega_t \cap B_{R})$; we get that
\begin{equation} \label{GreenK}
\frac{m(\Omega_t \cap B_{R})}{m(B_{R})} 
\leq C t^{-p/2} R^{\frac{p}2(1-d)}
\leq C_p t^{-\frac p2} \delta(y)^{\frac{p}{2}(1-d)}.
\end{equation}
The bound \eqref{GreenJ} becomes now
\begin{equation} \label{GreenL} \begin{split}
\g^\rho(x) & 
\leq Cs + C_p \left( \frac{\delta(y)}{r} \right)^n \delta(y)^{\frac p2(1-d)} \int_s^{+\infty} t^{-\frac p2}  dt 
\leq Cs + C_p \delta(y)^{\frac p2(1-d) +n }r^{-n} s^{1-\frac{p}{2}}.
\end{split}\end{equation}
We minimize then the right hand side of \eqref{GreenL} in $s$. We take
$s\approx \delta(y)^{1-d} \left( \frac{\delta(y)}{r} \right)^{\frac{2n}{p}}$ and get that
\begin{equation} \label{GreenM}
\g^\rho(x) \leq C_p \delta(y)^{1-d} \left( \frac{\delta(y)}{r} \right)^{\frac{2n}{p}}.
\end{equation}
The assertion \eqref{GreenH} follows from \eqref{GreenM} by taking $p = \frac{2n}{n-2}$ (which is possible since $n\geq 3$) and by recalling that $w(y) = \delta(y)^{d+1-n}$. When $n=2$, we have $\delta(y)^{1-d} = \delta(y)^{n-d-1} = w(y)^{-1}$ and so \eqref{GreenH2} is obtained from \eqref{GreenM} by taking $p= \frac{2n}{\epsilon} < +\infty$.

\medskip

Next we 
give a bound on the $L^q$-norm of the gradient of $\g^\rho$ for some $q>1$. 
As before, we want the bound 
to be independent of $\rho$ so that we can later
let our Green function be a weak limit of a
subsequence of $\g^\rho$.

We want to prove first the following Caccioppoli-like inequality: for any $r> 4\rho$, 
\begin{equation} \label{GreenR} 
\int_{\Omega \setminus B_{r}} |\nabla \g^\rho|^2 \, dm 
\leq C r^{-2} \int_{B_{r}\setminus B_{r/2}} (\g^\rho)^2 dm,
 \end{equation}
 where $C>0$ is a constant that depends only upon $d$, $n$, $C_0$ and $C_1$.
  
Keep $r> 4\rho$, and let $\alpha \in C^\infty(\R^n)$ be such that
$\alpha \equiv 1$ on $\R^n \sm B_{r}$, $\alpha \equiv 0$ on $B_{r/2}$ and $|\nabla \alpha|\leq \frac 4{r}$.
By construction, $g^\rho$ lies in $W_0$, and thus the function $\varphi :=\alpha^2 \g^\rho$
is supported in $\Omega \setminus \overline{B_{r/4}}$ and lies in $W_0$ thanks to Lemma~\ref{lmult}.
Since we like function with compact support, let us further multiply $\varphi$ by a smooth, compactly
supported function $\psi_R$ such that $\psi_R \equiv 1$ on a large ball $B_R$.
Then $\psi_R \varphi$ is compactly supported in $\Omega \setminus \overline{B_\rho}$,
and still lies in $W_0$ like $\varphi$.

Also, \eqref{Green4} says that $\g^\rho$ lies in $W_0$ and is a solution of $L\g^\rho = 0$ in 
 $\Omega \setminus \overline{B_\rho} \supset \Omega \setminus \overline{B_{r/4}}$. 
So we may apply the second item of Lemma \ref{rdefsol}, with $E = \Omega \setminus \overline{B_{r/4}}$,
 and we get that 
 \begin{equation} \label{preGreenR1} 
\int_\Omega \A\nabla \g^\rho \cdot \nabla (\psi_R\varphi) \, dm = 0,
 \end{equation}
but we would prefer to know that 
\begin{equation} \label{GreenR1} 
\int_\Omega \A\nabla \g^\rho \cdot \nabla \varphi \, dm = 0.
 \end{equation}
Fortunately, we proved in (ii) of Lemma \ref{ldens0} that with correctly chosen functions $\psi_R$,
the product $\psi_R\varphi$ tends to $\varphi$ in $W$; see \eqref{a2} in particular.
Then
\begin{eqnarray}\label{apreG}
\Big|\int_\Omega \A\nabla \g^\rho \cdot [\nabla \varphi - \nabla (\psi_R\varphi)]\, dm \Big|
&\leq& C ||\nabla \g^\rho||_{L^2(dm)} \, ||\nabla\varphi - \nabla(\psi_R\varphi)||_{L^2(dm)} 
\nn\\
&\leq& C ||\g^\rho||_W \, ||\varphi - (\psi_R\varphi)||_W 
\end{eqnarray}
by the boundedness property \eqref{ABounded} of $\A$. The right-hand side tends to $0$, 
so \eqref{GreenR1} follows from \eqref{preGreenR1}. Since $\varphi =\alpha^2 \g^\rho$,
\eqref{GreenR1} yields
\begin{equation} \label{GreenR2} 
\int_\Omega \alpha^2 [\A\nabla \g^\rho \cdot \nabla \g^\rho] \, dm 
= - 2 \int_\Omega  \alpha \g^\rho [\A\nabla \g^\rho \cdot \nabla \alpha] \, dm.
 \end{equation}
Together with the elliptic and boundedness conditions on $\A$ (see \eqref{AElliptic} and \eqref{ABounded}) and the Cauchy-Schwarz inequality, 
\eqref{GreenR2} becomes
\begin{equation} \label{GreenR3} \begin{split}
\int_\Omega \alpha^2 |\nabla \g^\rho|^2 dm & \leq C \int_\Omega  \alpha \g^\rho |\nabla \g^\rho||\nabla \alpha| \, dm \\
& \leq C \left( \int_\Omega \alpha^2 |\nabla \g^\rho|^2 dm \right)^\frac12 
\left( \int_\Omega (\g^\rho)^2 |\nabla \alpha|^2 dm \right)^\frac12 ,
\end{split} \end{equation}
which can be rewritten
\begin{equation} \label{GreenR4} 
\int_\Omega \alpha^2 |\nabla \g^\rho|^2 dm \leq C \int_\Omega (\g^\rho)^2 |\nabla \alpha|^2 dm.
 \end{equation}
The bound \eqref{GreenR} is then a straightforward consequence of our choice of $\alpha$.

\medskip
Set $\hat \Omega_t = \{x\in \Omega, \, |\nabla \g^\rho|>t\}$.
As before, there will be two different behaviors. 
We first check that
\begin{equation} \label{GreenQ} 
m(\hat \Omega_t) \leq C t^{-\frac{d+1}d} \qquad \text{ when } t\leq \delta(y)^{-d}.
 \end{equation}
Let $r\geq \delta(y)$ 
be given, to be chosen later. 
The Caccioppoli-like inequality \eqref{GreenR} and the pointwise bound \eqref{GreenF} give
\begin{equation} \label{GreenS} 
\int_{\Omega \setminus B_{r}} |\nabla \g^\rho|^2 \, dm 
\leq C r^{-2} \int_{B_{r}\setminus B_{r/2}} (\g^\rho)^2 dm
\leq C r^{-2d} m(B_r) \leq C r^{1-d}
\end{equation}
 by \eqref{ISob8b}, and hence
 \begin{equation} \label{GreenT} 
m(\hat \Omega_t \setminus B_{r}) \leq C t^{-2} r^{1-d}.
 \end{equation}
 This yields
\begin{equation} \label{GreenU} 
m(\hat \Omega_t) \leq C t^{-2} r^{1-d} + m(B_{r}) = C t^{-2} r^{1-d} + Cr^{1+d}
 \end{equation}
because $r\geq \delta(y)$. Take $r = t^{-\frac{1}{d}}$ 
in \eqref{GreenU} (and notice that $r \geq \delta(y)$ when $t\leq \delta(y)^{-d}$).
The claim \eqref{GreenQ} follows.

\medskip

 We also want a version of \eqref{GreenQ} when $t$ is big. We aim to prove that
 \begin{equation} \label{GreenV} 
m(\hat \Omega_t) \leq C w(y)^{-\frac{1}{n-1}} t^{-\frac{n}{n-1}} 
\qquad \text{ when $t\geq \delta(y)^{-d}$ and } \, n\geq 3
 \end{equation}
 and for any $\eta \in (0,2)$,
  \begin{equation} \label{GreenV0} 
 m(\hat \Omega_t) \leq C_\eta w(y)^{-1} \delta(y)^{d\eta} t^{\eta-2} 
 \qquad \text{when $t\geq \delta(y)^{-d}$ and } \, n= 2.
 \end{equation}

The proof of \eqref{GreenV} is similar to \eqref{GreenQ} but has an additional difficulty: 
we cannot use the Caccioppoli-like argument \eqref{GreenR} when $r$ is smaller than $4\rho$. 
So we will use another way. By 
\eqref{Green3} for the test function $\phi = \g^\rho$ and the elliptic condition \eqref{AElliptic},
\begin{equation} \label{GreenN0}
\int_\Omega |\nabla \g^\rho|^2 dm 
\leq C\int_\Omega \A\nabla \g^\rho \cdot \nabla \g^\rho \, dm 
= C\fint_{B_\rho} \g^\rho(z)\, dz \leq \frac{C}{m(B_\rho)} \int_{B_\rho} \g^\rho \, dm
\end{equation}
by 
\eqref{L1byL1w}. Let $y_0$ be such that $|y-y_0| =\delta(y)$. We use 
 H\"older's inequality,
and then the Sobolev-Poincar\'e inequality \eqref{1.5-bisbis}, 
with $p$ in the range given by Lemma~\ref{lSob}, to get 
that
\begin{equation} \label{Green0} \begin{split}
\int_\Omega |\nabla \g^\rho|^2 dm & \leq C_p m(B_\rho)^{-1} m(B_\rho)^{1-\frac{1}{p}} \left(\int_{B_\rho} (\g^\rho)^{p} \, dm \right)^\frac1{p} \\
& \leq C_p m(B_\rho)^{-\frac{1}{p}}  \left(\int_{B(y_0,2\delta(y))} (\g^\rho)^{p} \, dm \right)^\frac1{p} \\
& \leq C_p m(B_\rho)^{-\frac{1}{p}} \delta(y) m(B_{3\delta(y)})^{\frac{1}{p} - \frac{1}{2}}  \left(\int_{\Omega} |\nabla \g^\rho|^{2} \, dm \right)^\frac1{2},
\end{split} \end{equation}
that is,
\begin{equation} \label{GreenP} 
\int_\Omega |\nabla \g^\rho|^2 dm
\leq C_p m(B_\rho)^{-\frac{2}{p}} \delta(y)^2 m(B_{\delta(y)})^{\frac{2}{p} - 1}.
 \end{equation}
We use the fact that $100\rho<\delta(y)$ and Lemma~\ref{lwest} to get that 
$m(B_\rho) \approx \rho^n w(y) = \rho^n \delta(y)^{d+1-n}$. Besides, notice
that $m(B_{3\delta(y)}) \approx \delta(y)^{d+1}$. 
We end up with 
\begin{equation} \label{GreenQ0} 
\int_\Omega |\nabla \g^\rho|^2 dm
\leq C_p \rho^{-\frac{2n}p} w(y)^{-\frac 2p} \delta(y)^{2+(d+1)(\frac2p-1)} 
= C_p \left(\frac{\delta(y)}{\rho}\right)^{\frac{2n}{p}} \delta(y)^{1-d}
\end{equation}
once we recall that $w(y) = \delta(y)^{d+1-n}$. Observe 
that the right-hand 
side of \eqref{GreenQ0} is similar to the one of \eqref{GreenM}. In the same way as below \eqref{GreenM}
we take $p=\frac{2n}{n-2}$ 
when $n\geq 3$ and $p= \frac{4}{\epsilon}$ when $n=2$, and obtain that
 \begin{equation} \label{GreenQ1} 
\int_\Omega |\nabla \g^\rho|^2 dm\leq \left\{\begin{array}{ll} C w(y)^{-1} \rho^{2-n} & \text{ if } n\geq 3 \\ C_\epsilon w(y)^{-1} \left( \frac{\delta(y)}{\rho} \right)^\epsilon  & \text{ for any $\epsilon >0$ if $n=2$.}\end{array} \right. 
 \end{equation}
 
Let $r\leq \delta(y)$, to be chosen soon. Now we show that 
 \begin{equation} \label{GreenW} 
\int_{\Omega\setminus B_{r}} |\nabla \g^\rho|^2 dm 
\leq \left\{\begin{array}{ll} C w(y)^{-1} r^{2-n} & \text{ if } n\geq 3 \\ 
C_\epsilon w(y)^{-1} \left( \frac{\delta(y)}{r} \right)^\epsilon  & \text{ for any $\epsilon >0$ if $n=2$.}\end{array} \right. 
\end{equation}
When $r \leq 4\rho$, this
is a consequence of \eqref{GreenQ1},
and when $4\rho<r\leq \delta(y)$, this
can be proven as we proved \eqref{GreenS}, by using Caccioppoli-like inequality \eqref{GreenR} and the pointwise bounds \eqref{GreenH} or \eqref{GreenH2}. 
That is, we say that
\begin{equation}\label{ag62} \begin{split}
\int_{\Omega\setminus B_{r}} |\nabla \g^\rho|^2 dm 
& \leq C r^{-2} \int_{B_{r}\setminus B_{r/2}} (\g^\rho)^2 dm \\
& \leq r^{-2} m(B_r) \frac1{w(y)^2} \left\{\begin{array}{ll} Cr^{2(2-n)} & \text{ if } n\geq 3\\ 
C_\epsilon \left( \frac{\delta(y)}{r} \right)^{2\epsilon} & \text{ if } n=2, \, \epsilon>0
\end{array} \right.
\end{split} \end{equation}
and we observe that $m(B_r) \approx w(y) r^n$.

 Let $n\geq 3$. We deduce from \eqref{GreenW} that 
 $m(\hat \Omega_t \sm B_{r})
 \leq Ct^{-2} r^{2-n} w(y)^{-1}$ and then, since $m(B_{r}) \leq Cr^n w(y)$ 
and
 thanks to Lemma~\ref{lwest},
  \begin{equation} \label{GreenX} 
m(\hat \Omega_t) \leq Cw(y)^{-1} t^{-2} r^{2-n} + m(B_{r}) \leq C t^{-2} w(y)^{-1} r^{2-n} + Cr^{n} w(y).
 \end{equation}
Choose $r = [tw(y)]^{-\frac{1}{n-1}}$ (which is smaller than $\delta(y)$ if $t\geq \delta(y)^{-d}$) 
in \eqref{GreenX}. This
yields \eqref{GreenV}.

Let $n = 2$ and let $\eta \in (0,2)$ be given. Set  $\epsilon := \frac{2\eta}{2-\eta} >0$. In this case, \eqref{GreenW} gives
\begin{equation} \label{GreenX0} 
m(\hat \Omega_t \sm B_{r}) 
\leq Ct^{-2} w(y)^{-1} \left( \frac{\delta(y)}{r} \right)^\epsilon
 \end{equation}
 and then since $m(B_{r}) \leq Cr^2 w(y)$ by 
 Lemma~\ref{lwest},
 \begin{equation} \label{GreenX1} 
m(\hat \Omega_t) \leq Ct^{-2} w(y)^{-1} \left( \frac{\delta(y)}{r} \right)^\epsilon + Cr^{2} w(y).
 \end{equation}
We want to minimize the above quantity in $r$. We take 
$r = \delta(y)^{\frac{2(1-d)+\epsilon}{2+\epsilon}} t^{-\frac{2}{2+\epsilon}}$, which is smaller than $\delta(y)$ when $t\geq \delta(y)^{-d}$ and we find that
  \begin{equation} \label{GreenX2} 
m(\hat \Omega_t) \leq Ct^{-\frac{4}{2+\epsilon}} \delta(y)^{\frac{2(1-d)+\epsilon(d+1)}{2+\epsilon}} = C t^{\eta-2} \delta(y)^{1-d+\eta d},
 \end{equation}
with our choice of $\epsilon$.
Since $w(y)^{-1} = \delta(y)^{1-d}$ when $n=2$, the claim \eqref{GreenV0} follows.

\medskip

We plan to show now that $\nabla \g^\rho
\in L^q(B_R,w)$ for $1\leq q <n/(n-1)$, and the $L^q(B_R,w)$-norm of $\nabla \g^\rho$ 
can be bounded uniformly in $\rho$. 
More precisely, we claim that for 
$R\geq \delta(y)$ and 
$1\leq q < n/(n-1)$, 
  \begin{equation} \label{GreenY} 
\int_{B_R} |\nabla \g^\rho|^q dm \leq C_q R^{d(1-q)+1},
 \end{equation}
 where $C_q$ is independent of $\rho$ and $R$.
 
Let $s\in (0,\delta(y)^{-d}]$ 
be given, to be chosen soon. Then
\begin{equation} \label{GreenZ} 
\int_{B_R} |\nabla \g^\rho|^q dm \leq C \int_0^s t^{q-1} m(B_R) dt + C\int_s^{\delta(y)^{-d}} t^{q-1} m(\hat \Omega_t \cap B_R) dt + C\int_{\delta(y)^{-d}}^{+\infty} t^{q-1} m(\hat \Omega_t \cap B_R) dt.
\end{equation}
Let us call $I_1$, $I_2$ and $I_3$ the three integrals
in the right hand side of \eqref{GreenZ}. 
By Lemma~\ref{lwest}, $I_1 \leq Cs^q m(B_R) \leq Cs^q R^{d+1}$.
The second integral $I_2$ is bounded with the help of \eqref{GreenQ}, which gives
\begin{equation} \label{Greena} 
I_2 \leq C \int_s^{\delta(y)^{-d}} t^{q-1-\frac{d+1}{d}} dt \leq C \left( s^{q-\frac{d+1}{d}} - \delta(y)^{d(1-q)+1}\right).
\end{equation}
When $n\geq 3$, the last integral $I_3$ is bounded with the help of \eqref{GreenV} and we obtain, when $q<\frac{n}{n-1}$,
\begin{equation} \label{Greenb} 
I_3 \leq C w(y)^{-\frac{1}{n-1}} \int_{\delta(y)^{-d}}^{+\infty} t^{q-1-\frac{n}{n-1}} dt \leq C w(y)^{-\frac{1}{n-1}} \delta(y)^{-qd + \frac{nd}{n-1}} = C \delta(y)^{1+d(1-q)}
\end{equation}
where the last equality is obtained by using the fact that $w(y) = \delta(y)^{d+1-n}$. Note also that the same bound \eqref{Greenb} can be obtained when $n=2$ by using \eqref{GreenV0} with $\eta = \frac{2-q}{2}$. The left-hand side of \eqref{GreenZ} can be now bounded for every $n\geq 2$ by
\begin{equation} \label{Greenc} 
\int_{B_R} |\nabla \g^\rho|^q dm \leq Cs^q R^{d+1} + C \left( s^{q-\frac{d+1}{d}} - \delta(y)^{d(1-q)+1}\right) + C \delta(y)^{1+d(1-q)} =  Cs^q R^{d+1} + Cs^{q-\frac{d+1}{d}}, \end{equation}
where the third term in the middle is dominated by $s^{q - \frac{d+1}{d}}$ because $I_2 \geq 0$.
We take $s=R^{-d} \leq \delta(y)^{-d}$ in the right hand side of \eqref{Greenc} to get the claim \eqref{GreenY}.

\medskip

As we said, we want to define the Green function as a weak limit of functions $\g^\rho$, 
$0<\rho\leq \delta(y)/100$. 
We want to prove that for
$q \in (1,\frac{n}{n-1})$ and 
$R>0$, 
\begin{equation} \label{Grc1}
\|\g^\rho\|_{W^{1,q}(B_R)} \leq C_{q,R},  
\end{equation}
where $C_{q,R}$ is independent of $\rho$ (but depends, among others
things,  on $y$, $q$ and $R$).
First, it is enough to prove the result for $R \geq 2\delta(y)$. Thanks to \eqref{GreenY}, the quantity $\|\nabla \g^{\rho}\|_{L^q(B_R,w)}$ is bounded uniformly in $\rho \in (0,\delta(y)/100)$. 
Due to \eqref{L1byL1w}, the quantity $\|\nabla \g^{\rho}\|_{L^q(B_R)}$ is bounded uniformly in $\rho$. Now, due to \cite[Corollary 1.1.11]{Mazya11}, we deduce that $\g^{\rho_\eta} \in W^{1,q}(B_R)$ and hence with the classical Poincar\'e inequality on balls that
\begin{equation} \label{Grc2}
\int_{B_R} \Big| \g^\rho(z) - \fint_{B_R} \g^\rho(y)\, dy \Big|^q dz \leq C_{q,R} \|\nabla \g^{\rho}\|_{L^q(B_R)}^q 
\leq C_{q,R},  
\end{equation}
where $C_{q,R}>0$ is 
independent of $\rho$. 
Choose $y_0 \in \Gamma$ such that $|y-y_0| = \delta(y_0)$. 
Note that $B(y_0,\delta(y)/2) \subset B_R$ because $R \geq 2\delta(y)$,
so \eqref{Grc2} implies that
\begin{equation}\label{Grc2'}
\Big| \fint_{B(y_0,\delta(y)/2)} \g^\rho(z) - \fint_{B_R} \g^\rho(y)\, dy \Big|^q\, dz
\leq \int_{B_R} \Big| \g^\rho(z) - \fint_{B_R} \g^\rho(y)\, dy \Big|^q\, dz \leq C_{q,R}
\end{equation}
and hence also, by the triangle inequality, 
\begin{equation} \label{Grc3} 
\int_{B_R} \left| \g^\rho(z) - \fint_{B(y_0,\delta(y)/2)} \g^\rho(y)\, dy \right|^q \, dz
\leq C_{q,R} \int_{B_R} \left| \g^\rho(z) - \fint_{B_R} \g^\rho(y)\, dy \right|^q\, dz. 
\end{equation}
Together with \eqref{Grc2}, we obtain
\begin{equation} \label{Grc4}
\int_{B_R} \left| \g^\rho(z) \right|^q \, dz \leq C_{q,R} \left(1 + \fint_{{B(y_0,\delta(y)/2)}} |\g^\rho(y)|\, dy\right)^q  
\end{equation}
and since \eqref{GreenF} gives that $\fint_{{B(y_0,\delta(y)/2)}} |\g^\rho| \, dz\leq C \delta(y)^{1-d}$, the claim \eqref{Grc1} follows. 

Fix $q_0 \in (1,\frac{n}{n-1})$, for instance, take $q_0 = \frac{2n+1}{2n-1}$. Due to \eqref{Grc1}, for all $R>0$, the functions $(\g_\rho)_{0< 100\rho < \delta(y)}$ are uniformly bounded in $W^{1,q_0}(B_R)$. 
So a diagonal process allows us to find a sequence $(\rho_\eta)_{\eta \geq 1}$ converging to 0 and a function $\g \in L^1_{loc}(\R^n)$ such that
\begin{equation} \label{Grc5} 
\g^{\rho_\eta} \rightharpoonup \g=g(.,y) \ \text{in $W^{1,q_0}(B_R)$, for all $R>0$.} 
 \end{equation}
Let $q\in (1,\frac{n}{n-1})$ and $R>0$. The functions $\g^{\rho_\eta}$ are uniformly bounded in $W^{1,q}(B_R)$ thanks to \eqref{Grc1}. 
So we can find a subsequence $\g^{\rho_{\eta'}}$ of $\g^{\rho_\eta}$ such that $\g^{\rho_{\eta'}}$ converges weakly to some function $\g^{(q,R)} \in W^{1,q}(B_R)$. 
Yet, by uniqueness of the limit, $\g$ equals $\g^{(q,R)}$ almost everywhere in $B_R$. As a consequence, up to a subsequence (that depends on $q$ and $R$),
 \begin{equation} \label{Grc6} 
\g^{\rho_\eta} \rightharpoonup \g=g(.,y) \ \text{in $W^{1,q}(B_R)$}. 
 \end{equation}
The assertion \eqref{GreenE4} follows.

\medskip
We aim now to prove \eqref{GreenE1}, that is 
 \begin{equation} \label{Grc7} 
(1-\alpha) \g \in W_0 
 \end{equation}
 whenever $\alpha \in C_0^\infty(\R^n)$ satisfies $\alpha \equiv 1$ on $B_r$ for some $r>0$.
 
 So we choose $\alpha \in C_0^\infty(\R^n)$ and $r>0$ such that $\alpha \equiv 1$ on $B_r$. 
 Since $\alpha$ is compactly supported, we can find $R>0$ such that $\supp \, \alpha \subset B_R$.
 For any $\eta \in \bN$ such that
 $4 \rho_\eta \leq r$ 
and $100 \rho_\eta < \delta(y)$,
 \begin{equation} \label{Grc8} \begin{split}
\|(1-\alpha) \g^{\rho_\eta}\|_W & \leq \|\g^{\rho_\eta} \nabla \alpha\|_{L^2(B_R \setminus B_r,w)} 
+ \|(1-\alpha) \nabla g^{\rho_\eta}\|_{L^2(\Omega \setminus B_r,w)} \\
& \leq C_\alpha \sup_{B_R \setminus B_r} g^{\rho_\eta} 
+ C_\alpha \|\nabla g^{\rho_\eta}\|_{L^2(\Omega \setminus B_r,w)}.
 \end{split} \end{equation}
Thanks to \eqref{GreenF}, \eqref{GreenH} and \eqref{GreenH2}, the term 
$\sup_{B_R \setminus B_r} g^{\rho_\eta}$ can be bounded by a constant that doesn't depend on $\eta$, 
provided that $\rho_\eta \leq \min(r/4,\delta(y)/100)$.
In the same way, \eqref{GreenW} proves that $\|\nabla g^{\rho_\eta}\|_{L^2(\Omega \setminus B_r,w)}$ 
can be also bounded by a constant independent of $\eta$. 
As a consequence, for any $\eta$ satisfying $4\rho_\eta \leq r$,
 \begin{equation} \label{Grc9} 
\|(1-\alpha) \g^{\rho_\eta}\|_W \leq C_{\alpha} 
\end{equation}
where $C_\alpha$ is independent of $\eta$. 
Note also that for $\eta$ large enough, $(1-\alpha) \g^{\rho_\eta}$ belongs to $W_0$
because $\g^{\rho_\eta} \in W_0$ by construction, and by Lemma \ref{lmult}.
Therefore, the functions  
$(1-\alpha) \g^{\rho_\eta}$, $\eta \in \NN$ large,  lie in a fixed closed ball of the Hilbert space $W_0$. 
So, up to a subsequence, there exists $f_\alpha \in W_0$ such that $(1-\alpha) \g^{\rho_\eta} \rightharpoonup f_\alpha$ in $W_0$.  By uniqueness of the limit, we have $(1-\alpha) \g = f_\alpha \in W_0$, that is 
 \begin{equation} \label{Grc0} 
(1-\alpha) \g^{\rho_\eta} \rightharpoonup (1-\alpha) \g \quad \text{ in } W_0.
 \end{equation}
The claim \eqref{Grc7} follows.
 
Observe 
that \eqref{Grc7} implies that $\g \in \WW(\R^n \setminus \{y\})$. Indeed, take $\varphi \in C^\infty_0(\R^n \setminus \{y\})$. We can find $r>0$ such that $\varphi \equiv 0$ in $B_r$. Construct now $\alpha \in C^\infty_0(B_r)$ such that $\alpha \equiv 1$ in $B_{r/2}$ and we have
  \begin{equation} \label{GrcA} 
\varphi \g = \varphi [(1-\alpha) \g] \in W_0 \subset W
 \end{equation}
by
\eqref{Grc7} and Lemma \ref{lmult}. Hence $\g \in \WW(\R^n \setminus \{y\})$.

\medskip
Now we want to prove \eqref{GreenE6}. Fix 
$q\in (1,n/(n-1))$ and a function $\phi \in C^\infty_0(B_{\delta(y)/2})$ such that
$\phi \equiv 1$ in $B_{\delta(y)/4}$. Then let 
$\varphi$ be any function in $C^\infty_0(\Omega)$. Let us first check that
 \begin{equation} \label{Greenz1} 
a(\g,\phi \varphi) :=  \int_\Omega A \nabla \g \cdot \nabla [\phi \varphi] dx = \varphi(y)
 \end{equation}
 and
 \begin{equation} \label{Greenz2} 
 a(\g, (1-\phi)\varphi) : = \int_\Omega A \nabla \g \cdot \nabla [(1-\phi) \varphi] dx = 0.
 \end{equation}
The map $a(.,\phi \varphi)$ is a bounded linear functional on $W^{1,q}(B_{\delta(y)/2})$ and thus the weak convergence (in $W^{1,q}(B_R)$) of a subsequence $\g^{\rho_{\eta'}}$  of $\g^{\rho_{\eta}}$ yields
 \begin{equation} \label{Greend} 
a(\g, \phi \varphi) 
= \lim_{\eta' \to +\infty} a(\g^{\rho_{\eta'}},\phi \varphi) =  \lim_{\rho\to 0} \fint_{B(y,\rho)} \phi \varphi\, dx = \varphi(y),
 \end{equation}
which is \eqref{Greenz1}. Let $\alpha \in C^\infty_0(B_{\delta(y)/4})$ 
be such that
$\alpha \equiv 1$ on $B_{\delta(y)/8}$. The map $a(.,(1-\phi)\varphi)$ is bounded on $W_0$ thus the weak convergence of a subsequence of $(1-\alpha) \g^{\rho_{\eta}}$ to $(1-\alpha) \g$ in $W_0$ gives
 \begin{equation} \label{Greendbis} \begin{split}
a(\g, (1-\phi) \varphi) & = a((1-\alpha)\g,(1-\phi)\varphi) \\
& = \lim_{\eta' \to +\infty} a((1-\alpha)\g^{\rho_{\eta'}},(1-\phi) \varphi) = \lim_{\eta' \to +\infty} a(\g^{\rho_{\eta'}},(1-\phi) \varphi) \\
& =  \lim_{\rho\to 0} \fint_{B(y,\rho)} (1-\phi) \varphi \, dx= 0.
\end{split} \end{equation}
 which is \eqref{Greenz2}. The assertion \eqref{GreenE6} 
now follows from \eqref{Greenz1} and \eqref{Greenz2}. 
 
If we use \eqref{GreenE6} for the functions in $C^\infty_0(\Omega \setminus \{y\})$, 
we immediately obtain that
  \begin{equation} \label{Greenf} 
\g \text{ is a solution of $L\g=0$ on $\Omega \setminus \{y\}$.}
 \end{equation}
 
 \medskip
 
Assertions \eqref{GreenE2} and \eqref{GreenE5} come from the weak lower semicontinuity of the 
$L^q$-norms and the bounds \eqref{GreenS}, \eqref{GreenW} and \eqref{GreenY}. 
Notice also that $r^{1-d} \approx \frac{r^{2-n}}{w(y)}$ when $r$ is near $\delta(y)$, so the cut-off
between the different cases does not need to be so precise.
Let us show \eqref{GreenE7}. Let $R>0$ be a big given number. 
We have shown that the sequence $\g^{\rho_\eta}$ is uniformly bounded in $W^{1,q}(B_R)$. Then, 
by the 
Rellich-Kondrachov theorem, there exists a subsequence of $\g^{\rho_\eta}$ that also
converges strongly in $L^1(B_R)$ and then another subsequence of $\g^{\rho_\eta}$ that converges almost everywhere in $B_R$. 
The estimates \eqref{GreenF}, \eqref{GreenH} and \eqref{GreenH2} yield then
  \begin{equation} \label{Greeng} 
0 \leq \g(x) \leq \left\{\begin{array}{ll} C|x-y|^{1-d} & \text{ if } 4|x-y| \geq \delta(y) \\ 
\frac{C|x-y|^{2-n}}{w(y)} & \text{ if } 2|x-y| \leq \delta(y), \, n\geq 3 \\
\frac{C_\epsilon}{w(y)} \left(\frac{\delta(y)}{|x-y|} \right)^\epsilon & \text{ if } 2|x-y| \leq \delta(y), \, n=2,\end{array} \right. \quad \text{ a.e. on } B_R.
 \end{equation}
But by \eqref{Greenf} $\g$ is a solution of $L\g = 0$ on $\Omega\setminus \{y\}$, so it
is continuous on $\R^n \setminus \{y\}$ by Lemmas~\ref{HolderI} and \ref{HolderB}, and the bounds
\eqref{Greeng} actually hold pointwise in $\Omega \cap B_R \setminus \{y\}$. 
Since $R$ can be chosen as large as we want, the bounds \eqref{GreenE7} follow.

\medskip

It remains to check the weak estimates \eqref{GreenE2w} and \eqref{GreenE5w}. 
Set $q= \frac{2n+1}{2n-1}$, which satisfies $1 < q < \frac{n}{n-1} < \frac{n}{n-2}$. 
Let $t>0$ be given ; by the weak
lower semicontinuity of the $L^q$-norm,
\begin{equation} \label{Greenh}
t^q \, \frac{m(\{x \in B_R, \, \g(x) >t)}{m(B_R)} \leq \frac{1}{m(B_R)}\|\g\|_{L^q(B_R,w)}^q 
\leq \liminf_{\eta \to +\infty} \frac{1}{m(B_R)}\|\g^{\rho_\eta}\|_{L^q(B_R,w)}^q.
\end{equation}
Let us use \cite[p. 28, Proposition 2.3]{Duoandi}; in the case of \eqref{GreenE2w}, we could manage otherwise,
but we also want to get \eqref{GreenE5w} with the same proof. We observe that
\begin{eqnarray} \label{Greeni} 
t^q \frac{m(\{x \in B_R, \, \g(x) >t \})}{m(B_R)} 
&\leq& \liminf_{\eta \to +\infty} \left[ \int_0^t s^{q-1} \frac{m(\{x \in B_R, \, \g(x) >t, \, \g^{\rho_\eta} >s)\}}{m(B_R)} ds \right. 
\nn\\
&\,& \qquad \qquad \qquad \left. 
+ \int_t^{+\infty} s^{q-1} \frac{m(\{x \in B_R, \, \g(x) >t, \, \g^{\rho_\eta} >s\})}{m(B_R)} ds \right] 
\nn\\
&\leq& \frac{t^q}{q} \frac{m(\{x \in B_R, \, \g(x) >t\})}{m(B_R)} \\
&\,& \qquad \qquad \qquad 
+ \liminf_{\eta \to +\infty} \int_t^{+\infty} s^{q-1} \frac{m(\{x \in B_R, \g^{\rho_\eta} >s\})}{m(B_R)} ds.
\nn 
\end{eqnarray}
Let $p$ lie in the range given by Lemma \ref{lSob}. The bounds \eqref{Green8} gives
\begin{equation} \label{Greenj} \begin{split}
t^q \frac{m(\{x \in B_R, \, \g(x) >t \})}{m(B_R)} 
& \leq C \liminf_{\eta \to +\infty} \int_t^{+\infty} s^{q-1} \frac{m(\{x \in B_R, \g^{\rho_\eta} >s\})}{m(B_R)} ds \\
&\leq C_p R^{\frac p2 (1-d)}
\int_t^{+\infty} s^{q-1-\frac p2} ds 
\leq C_p R^{\frac p2 (1-d)} t^{q-\frac p2}.
\end{split} \end{equation}
The estimates \eqref{GreenE2w} follows by dividing both sides of \eqref{Greenj} by $t^q$. The same ideas are used to prove \eqref{GreenE5w} from \eqref{GreenQ}, \eqref{GreenV} and \eqref{GreenV0}. 
This finally completes the proof of Lemma \ref{GreenEx}.
\ep 

\medskip
\begin{lemma} \label{GreenLB}
Any non-negative function $g: \, \Omega \times \Omega \to \R \cup \{+\infty\}$ 
that verifies
the following conditions:
\begin{enumerate}[(i)]
\item for every 
$y\in \Omega$ and 
$\alpha \in C^\infty_0(\R^n)$ such that
$\alpha \equiv 1$ in $B(y,r)$ for some $r>0$, the function $(1-\alpha)g(.,y)$ lies 
in $W_0$,
\item for every 
$y\in \Omega$, the function $g(.,y)$ lies 
in $W^{1,1}(B(y,\delta(y)))$,
\item for 
$y\in \Omega$ and
$\varphi \in C^\infty_0(\Omega)$,
\begin{equation} \label{GreenLB1}
\int_\Omega A \nabla_x g(x,y)\cdot \nabla \varphi(x) dx = \varphi(y),
\end{equation}
\end{enumerate}
enjoys the following pointwise lower bound: 
\begin{equation} \label{GreenLB2}
g(x,y) \geq C^{-1} \frac{|x-y|^{2}}{m(B(y,|x-y|))} \approx \frac{|x-y|^{2-n}}{w(y)}
\ \text{ for $x,y\in \Omega$ such that } 0 < |x-y| \leq \frac{\delta(y)}{2}.
\end{equation}
\end{lemma}

\bp
Let $g$ satisfy the assumptions of the lemma, fix $y\in \Omega$, write $\g(x)$ for $g(x,y)$, and 
use $B_r$ for $B(y,r)$. 
Thus we want to prove that
\begin{equation} \label{Greenj2}
\g(x) \geq \frac{|x-y|^{2}}{Cm(B_{|x-y|})} \qquad \text{ whenever } 0 < |x-y| \leq \frac{\delta(y)}{2}.
\end{equation}
With our assumptions, $\g \in \WW(\R^n\setminus \{y\})$ and it is 
a solution in $\Omega \setminus \{y\}$ with zero trace;
the proof is the same as for \eqref{GrcA} and \eqref{Greenf} in Lemma \ref{GreenEx}.
Take $x\in \Omega \setminus \{y\}$ such that $|x-y|\leq \frac{\delta(y)}2$. 
Write $r$ for $|x-y|$ and let $\alpha \in C^\infty_0(\Omega \setminus \{y\})$ be
such that $\alpha = 1$ on $B_{r} \setminus B_{r/2}$, $\alpha = 0$
outside of  $B_{3r/2} \setminus B_{r/4}$,
and $|\nabla \alpha| \leq 8/r$. Using Caccioppoli's
inequality (Lemma~\ref{CaccioI}) with the cut-off function $\alpha$, we obtain
\begin{equation} \label{Greenk} \begin{split}
\int_{B_{r} \setminus B_{r/2}} |\nabla \g|^2 dm & \leq Cr^{-2} \int_{B_{3r/2} \setminus B_{r/4}} \g^2 dm \\
&  \leq Cr^{-2} m(B_{3r/2}) \sup_{B_{3r/2} \setminus B_{r/4}} \g^2 \leq Cr^{-2} m(B_{r}) \sup_{B_{3r/2} \setminus B_{r/4}} \g^2
\end{split} \end{equation} 
by
the doubling property \eqref{doublinggen}. We can cover $B_{3r/2} \setminus B_{r/4}$ by a finite 
(independent of $y$ and $r$) number of balls of radius $r/20$ centered in $B_{3r/2} \setminus B_{r/4}$. 
Then use the Harnack inequality given by Lemma~\ref{HarnackI} several times, to get that
\begin{equation} \label{Greenl} 
\int_{B_{r} \setminus B_{r/2}} |\nabla \g|^2 dm \leq C r^{-2}m(B_r) \g(x)^2.
\end{equation} 
Define 
another function $\eta \in C^\infty_0(\Omega)$ which 
is supported in $B_r$, equal to $1$ on $B_{r/2}$, and such that
$|\nabla \eta| \leq \frac4r$. 
Use $\eta$ as a test function in \eqref{GreenLB1} to get that
\begin{equation} \label{Greenm} 
1 = \int_{B_{r} \setminus B_{r/2}} \A \nabla \g \cdot \nabla \eta \,  dm 
\leq \frac Cr \int_{B_{r} \setminus B_{r/2}} |\nabla \g| \, dm,
\end{equation} 
where we used \eqref{ABounded} for the last estimate. 
Together with the Cauchy-Schwarz inequality and \eqref{Greenl}, this yields
\begin{equation} \label{Greenn} \begin{split}
1 & \leq \frac Cr \, m(B_r)^{\frac12} 
\Big( \int_{B_{r} \setminus B_{r/2}} |\nabla \g|^2 \, dm \Big)^\frac12 
\leq C r^{-2} m(B_r) \g(x) . 
\end{split} \end{equation} 
The lower bound \eqref{Greenj2} follows.
\ep

In the sequel, $A^T$ denotes the transpose matrix of $A$, 
defined by $A^T_{ij}(x) = A_{ji}(x)$ for $x\in \Omega$ and $1\leq i,j\leq n$. Thus
$A^T$ satisfies the same boundedness and elliptic conditions as $A$.
That is, it 
satisfies \eqref{ABounded2} and \eqref{AElliptic2} with the same constant $C_1$. 
We can thus define solutions to $L_T u := - \diver A^T \nabla u = 0$ for which the results given in Section \ref{Ssolutions} hold. 

Denote by $g: \Omega \times \Omega \to \R \cup \{+\infty\}$ the Green function 
defined in Lemma \ref{GreenEx}, and by  
$g^T: \Omega \times \Omega \to \R \cup \{+\infty\}$ 
the Green function defined in Lemma \ref{GreenEx}, but with 
$A$ is replaced by $A^T$.

\begin{lemma} \label{GreenSym}
With the notation above, 
\begin{equation} \label{GreenS1}
g(x,y) = g_T(y,x) \  \text{ for }
x,y\in \Omega, \, x\neq y.
\end{equation}
In particular, the functions $y\to g(x,y)$ satisfy 
the estimates given in Lemma \ref{GreenEx} and Lemma \ref{GreenLB}. 
\end{lemma}

\bp
The proof is the same as for
\cite[Theorem 1.3]{GW}. Let us review
it for completeness. Let $x,y\in \Omega$ be such that $x\neq y$. 
Set $B = B(\frac{x+y}{2}, |x-y|)$ and let $q\in (1,\frac{n}{n-1})$.

From the construction given in Lemma~\ref{GreenEx} 
(see \eqref{Grc6} in particular),
there exists two sequences $(\rho_\nu)_\nu$ and $(\sigma_\mu)_\mu$ converging to 0 such that 
$g^{\rho_\nu}(.,y)$ and $g_T^{\sigma_\mu}(.,x)$ converge weakly 
in $W^{1,q}(B)$ to $g(.,y)$ and $g_T(.,x)$ respectively.
So, up to 
 additional subsequence extractions,
$g^{\rho_\nu}(.,y)$ and $g_T^{\sigma_\mu}(.,x)$ converge to 
$g(.,y)$ and $g_T(.,x)$, strongly in $L^1(B)$, and then
pointwise a.e. in $B$.

Inserting them as test functions in \eqref{Green3}, we obtain
\begin{equation} \label{GreenS2}
\int_\Omega A \nabla g^{\rho_\nu}(z,y) \cdot \nabla g_T^{\sigma_\mu}(z,x) dz 
= \fint_{B(y,\rho_\nu)} g_T^{\sigma_\mu}(z,x) dz  
= \fint_{B(x,\sigma_\mu)} g^{\rho_\nu}(z,y) dz.
\end{equation}
We want to let 
$\sigma_\mu$ tend to 0. The term $\fint_{B(y,\rho_\nu)} g_T^{\sigma_\mu}(z,x) dz$ 
tends to 
$\fint_{B(y,\rho_\nu)} g_T(z,x) dz$ because 
$g_T(.,x)^{\sigma_\mu}$ tends to $g_T(.,x)$ in $L^1(B)$. 
When $\rho_\nu$ is small enough, the function $g^{\rho_\nu}(.,y)$ is a solution of 
$L\g^{\rho_\nu}=0$ in $\Omega \setminus \overline{B(y,\rho_\nu)} \ni x$, so it 
is continuous at $x$ thanks to Lemma~\ref{HolderI}. Therefore, the term $\fint_{B(x,\sigma_\mu)} g^{\rho_\nu}(z,y) dz$ tends to $g^{\rho_\nu}(x,y)$. We deduce, when $\nu$ is big enough so that $\rho_\nu < |x-y|$,
\begin{equation} \label{GreenS3}
\fint_{B(y,\rho_\nu)} g_T(z,x) dz = g^{\rho_\nu}(x,y).
\end{equation}
Now let $\rho_\nu$ tend to $0$
in \eqref{GreenS3}. The function $g_T(.,x)$ is a solution 
for $L_T$  
in $\Omega \setminus \{x\}$, so it is continuous on $B(y,\rho_\nu)$ for $\nu$ large. Hence
the left-hand side of \eqref{GreenS3} converges to $g_T(y,x)$. 
Thanks to Lemma~\ref{HolderI}, the functions $g^{\rho_\nu}(.,y)$ are uniformly H\"older continuous, so the a.e. pointwise convergence of $g^{\rho_\nu}(.,y)$ 
to $g(.,y)$ can
be improved into a uniform convergence on $B(x,\frac13 |x-y|)$. In particular $g^{\rho_\nu}(x,y)$ tends to $g(x,y)$ when $\rho_\nu$ goes to 0. 
We get that $g_T(y,x) = g(x,y)$,
which is the desired conclusion.
\ep

\smallskip
\begin{lemma} \label{GreenFS} Let $g: \Omega \times \Omega \to \R \cup \{+\infty\}$ 
be the non-negative function constructed in Lemma~\ref{GreenEx}. 
Then for any $f \in C^\infty_0(\Omega)$, the function $u$ defined by
\begin{equation} \label{GreenFS1}
u(x) = \int g(x,y) f(y) dy 
\end{equation}
belongs to $W_0$ and is a solution of $Lu = f$ in the sense that
\begin{equation} \label{GreenFS2}
\int_\Omega A\nabla u \cdot \nabla \varphi  \, dx= \int_\Omega \A\nabla u \cdot \nabla \varphi \, dm 
=  \int_\Omega f \varphi \ \ \text{ for every }
\varphi \in W_0. 
\end{equation}
\end{lemma}

\bp
First, let us check that \eqref{GreenFS1} make sense. 
Since $f\in C^\infty_0(\Omega)$, there exists a big ball $B$ with center $y$ and radius $R >\delta(y)$ 
such that $\supp \, f \subset B$. By 
\eqref{GreenE4} and \eqref{GreenS1}, $g(x,.)$ lies
in $L^1(B)$. Hence the integral in
\eqref{GreenFS1} is well defined.

\medskip

Let $f \in C^\infty_0(\Omega)$. Choose
a big ball $B_f$ centered on $\Gamma$ such that $\supp f \subset B_f$.  
For any $\varphi \in W_0$,
 \begin{equation} \label{GreenFS3}
\int_\Omega f \varphi \, dz \leq \|f\|_\infty \int_{B_f} |\varphi|\, dz \leq C_f \|\varphi\|_W 
\end{equation}
by
Lemma~\ref{lpBry}. So the map $\varphi \in W_0 \to \int f \varphi$ is a bounded linear
functional on $W_0$. 
Since the map $a(u,v) = \int_\Omega \A \nabla u \cdot \nabla v \, dm$ is bounded and coercive on $W_0$, 
the Lax-Milgram theorem 
yields the existence of $u \in W_0$ such that for any $\varphi \in W_0$,
 \begin{equation} \label{GreenFS4}
\int_\Omega A \nabla u \cdot \nabla \varphi  \, dz= \int_\Omega f\varphi\, dz.
\end{equation}

\medskip

We want now to show that $u(x) = \int_\Omega g(x,y) f(y) dy$. A key point of the proof uses  
the continuity of $u$, a property 
that we assume for the moment and will prove
later on. For 
every $\rho>0$, let $g^\rho_T(.,x) \in W_0$ be
the function satisfying 
 \begin{equation} \label{GreenFS5}
 \int_\Omega A^T \nabla_y g^\rho_T(y,x) \cdot \nabla \varphi(y) dy 
 = \fint_{B(x,\rho)} \varphi(y) \, dy \qquad \ \text{for every } 
 \varphi \in W_0. 
\end{equation}
We
use $g^\rho_T(.,x)$ as a test function in \eqref{GreenFS4} 
and get that
 \begin{equation} \label{GreenFS6} \begin{split}
 \int_\Omega f(y) g^\rho_T(y,x) dy & = \int_\Omega A \nabla u(y) \cdot \nabla_y g^\rho_T(y,x) dy   =  \int_\Omega A^T \nabla_y g^\rho_T(y,x) \cdot \nabla u(y) dy \\
 & = \fint_{B(x,\rho)} u(y) \, dy,
\end{split} \end{equation}
by
\eqref{GreenFS5}. We
take a 
limit as $\rho$ goes to 0. 
The right-hand side converges to $u(x)$ because, as we assumed, $u$ is continuous. 
Choose $R \geq \delta(x)$ 
so big that
$\supp \, f \subset B(x,R)$, and choose also $q\in (1,\frac{n}{n-1})$. 
According to \eqref{Grc6}, there exists a sequence $\rho_\nu$ converging to 0 such that $g^{\rho_\nu}_T(.,x)$ 
converges weakly in $W^{1,q}(B(x,R)) \subset L^1(B(x,R))$ to the function $g_T(.,x)$, 
the latter being equal to $g(x,.)$ by 
Lemma~\ref{GreenSym}. Hence 
 \begin{equation} \label{GreenFS7}
 \lim_{\nu \to +\infty} \int_\Omega f(y) g^{\rho_\nu}_T(y,x) dy = \int_\Omega f(y) g(x,y) dy
 \end{equation}
 and then \eqref{GreenFS1} holds.
 
\medskip
It remains to check what we assumed, that is the continuity of $u$ on $\Omega$. 
The quickest way to show it is to prove a version of the H\"older continuity (Lemma~\ref{HolderI}) 
when $u$ is a solution of $Lu = f$. As for the proof of Lemma~\ref{HolderI}, since we are only interested 
in the continuity \underline{inside} the domain, we can 
use the standard elliptic theory, where 
the result is well known (see for instance \cite[Theorem 8.22]{GT}).
\ep

The following Lemma states the uniqueness of the Green function. 

\begin{lemma} \label{GreenUn}
There exists a unique function $g: \Omega \times \Omega \mapsto \R \cup \{ + \infty\}$ such that 
$g(x,.)$ is continuous on $\Omega \setminus \{x\}$
and locally integrable in $\Omega$ for every
$x\in \Omega$, and such that for every
$f \in C^\infty_0(\Omega)$
the function $u$ given by
 \begin{equation} \label{GreenUn1}
u(x) : = \int_\Omega g(x,y) f(y) dy
 \end{equation}
 belongs to $W_0$ and is a solution of $Lu = f$ in the sense that
 \begin{equation} \label{GreenUn2}
\int_\Omega A\nabla u \cdot \nabla \varphi \, dx = \int_\Omega \A\nabla u \cdot \nabla \varphi \, dm 
=  \int_\Omega f \varphi \, dx\qquad \text{for every }
\varphi \in W_0. 
\end{equation}
\end{lemma}

\bp
The existence of the Green function is given by Lemma~\ref{GreenEx}, 
Lemma~\ref{GreenSym} and Lemma~\ref{GreenFS}. 
Indeed, if $g$ is the function built in Lemma~\ref{GreenEx}, the property \eqref{GreenE4} 
(together with Lemma~\ref{GreenSym})
states that $g(x,.)$ is locally integrable in $\Omega$. 
The property \eqref{GreenE6} (and Lemma~\ref{GreenSym} again) gives that $g(x,.)$ is a solution 
in $\Omega \setminus \{x\}$, and thus, by 
Lemma~\ref{HolderI}, that $g(x,.)$ is continuous in $\Omega \setminus \{x\}$. 
The last property, i.e. that fact that $u$ given by \eqref{GreenUn1} is in $W_0$ and satisfies \eqref{GreenUn2}, is exactly Lemma~\ref{GreenFS}.

\medskip

So it remains to prove the uniqueness. Assume that $\tilde g$ is another function satisfying the given 
properties. Thus for $f \in C^\infty_0(\Omega)$, the function $\tilde u$ given by
 \begin{equation} \label{GreenUn3}
\tilde u(x) : = \int_\Omega \tilde g(x,y) f(y) dy
 \end{equation}
 belongs to $W_0$ and satisfies $L\tilde u = f$. By the uniqueness of the solution of 
 the Dirichlet problem \eqref{Dp2} (see Lemma~\ref{lLM}), we must have $\tilde u = u$. 
 Therefore, for all $x\in \Omega$ and all $f \in C^\infty_0(\Omega)$,
\begin{equation} \label{GreenUn4}
\int_\Omega [ \tilde g(x,y) - g(x,y)] f(y) dy = 0.
 \end{equation}
From the continuity of $g(x,.)$ and $\tilde g(x,.)$ in $\Omega \setminus \{x\}$, we deduce that
$g(x,y) = \tilde g(x,y)$ for any $x,y\in \Omega$, $x\neq y$.
\ep

We end this section with an 
additional property of the Green function, its decay 
near the boundary. 
This property is proven in \cite{GW} under the assumption that $\Omega$ is of `S class', 
which means that we 
can find an exterior cone at any point of the boundary. 
We still can prove it in our context because the property relies on the H\"older continuity of solutions at the boundary, that holds in our context 
because we have (Harnack tubes and) Lemma \ref{HolderB}. 

\begin{lemma} \label{GreenDi}
The Green function satisfies
\begin{equation} \label{GreenDi1}
 g(x,y) \leq C \delta(x)^\alpha |x-y|^{1-d - \alpha}  \ \  \text{ for
 $x,y \in \Omega$ such that }   |x-y | \geq 4\delta(x),
\end{equation}
where $C>0$ and $\alpha>0$ depend only on $n$, $d$, $C_0$ and $C_1$.
\end{lemma}

\bp
Let $y\in \Omega$ be given. For any $x\in \Omega$, we write $\g(x)$ for $g(x,y)$. 
We want to prove that
\begin{equation} \label{GreenDi1b}
\g(x) \leq 
C \delta(x)^{\alpha}
|x-y|^{1-d-\alpha} \ \ 
\text{ for $x\in \Omega$ such that }
|x-y| \geq 4\delta(x), 
\end{equation}
with constants $C>0$ and $\alpha>0$ that depend only on $n$, $d$, $C_0$ and $C_1$.
By Lemma \ref{GreenEx}-(v), 
\begin{equation} \label{GreenDi2}
\g(z) \leq C  |z-y|^{1-d} 
 \ \ \text{ for } z\in \Omega \setminus B(y,\delta(y)/4).
\end{equation}

Let $x$ be such that $|x-y|\geq 4 \delta(x)$, choose $x_0  \in \Gamma$ such that $|x-x_0| = \delta(x)$,
and set $r = |x-y|$ and $B = B(x_0, |x-y|/3)$; thus $x \in B$. We shall need to know that
\begin{equation}\label{za1}
\delta(y) \leq \delta(x) + |x-y| \leq \frac r4 + r = \frac{5r}{4}.
\end{equation}
Then let $z$ be any point of $\Omega \cap B$. 
Obviously $|z-x| \leq |z-x_0| + |x_0-x| \leq \frac r3 + \delta(x)
\leq \frac r3 + \frac r4 = \frac {7r}{12}$, which implies that 
\begin{equation}\label{za2}
|y-z| \geq |y-x| - |z-x| \geq r- \frac {7r}{12} = \frac {5r}{12} \geq \frac {\delta(y)}{3}.
\end{equation}
Hence by \eqref{GreenDi2}, $\g(z) \leq C  |z-y|^{1-d}$. Notice also that
$|y-z| \leq |y-x| + |z-x| \leq r + \frac {7r}{12} = \frac {19r}{12}$, so, with \eqref{za2},
$\frac {5r}{12} \leq |y-z| \leq \frac {19r}{12}$ and 
\begin{equation}\label{za3}
\g(z) \leq C r^{1-d} = C  |x-y|^{1-d} \ \text{ for } z \in \Omega \cap B,
\end{equation}
even if $d < 1$, and with a constant $C>0$ that does not dependent on $x$, $y$, or $x_0$.

We now use the fact that $\g$ is a solution of $L \g = 0$ on $\Omega \cap B$.
Notice that its oscillation on $B$ is the same as its supremum, because it is nonnegative and,
by (i) of Lemma \ref{GreenLB}, its trace on $\Gamma \cap B$ vanishes. 
Lemma~\ref{HolderB} (the H\"older continuity of solutions at the boundary) says that for some $\alpha>0$,
that depends only on $n$, $d$, $C_0$ and $C_1$, 
\begin{equation} \label{GreenDi5} \begin{split}
g(x,y) = \g(x)  & \leq \sup_{\overline B(x_0,\delta(x))} \g = \osc_{\overline B(x_0,\delta(x))} \g
\leq C \left( \frac{3\delta(x)}{|x-y|} \right)^\alpha \osc_{B(x_0,|x-y|/3)} \g  \\
& = C \left( \frac{3\delta(x)}{|x-y|} \right)^\alpha \sup_{B(x_0,|x-y|/3)} \g
\leq C \left( \frac{\delta(x)}{|x-y|} \right)^\alpha |x-y|^{1-d}.
\end{split} \end{equation}
because $B = B(x_0, |x-y|/3)$.
The lemma follows.
\ep
\chapter{The Comparison Principle} 

In this section, we prove two 
versions of the comparison principle: one for the
harmonic measure (Lemma \ref{lCP2}) and 
one for locally defined solutions (Lemma \ref{lCP3}). A big technical difference is that the former is a globally defined solution, while the latter is local.

\medskip

At the moment we write this
manuscript, the proofs of the comparison principle in codimension 1 that 
we are aware of
cannot be straightforwardly adapted to the case of higher codimension.  To be more precise, we can indeed prove the comparison principle (in higher codimension) for harmonic measures on $\Gamma$ by only slightly modifying the arguments of \cite{CFMS,KenigB}. However,  
 the proof of the comparison principle for solutions 
(of $Lu =0$) defined on a {\it subset} $D$ of $\Omega$ in the case of codimension 1 relies on the use of the harmonic measure
on the boundary $\partial D$ (see for instance \cite{CFMS,KenigB}). In our setting, in the case where the considered functions are non-negative and solutions to $Lu=0$ only on a subset $D \subsetneq \Omega$, we are lacking a definition for harmonic measures with mixed boundaries (some parts in codimension 1 and some parts in higher codimension). The reader can imagine a ball $B$ centered at a point of $\po=\Gamma$. The boundary of the $B\cap \Omega$ consists of $\Gamma\cap B$ and $\partial B$, the sets of different co-dimension.
For those reasons, our proof of the comparison principle (in higher codimension) for locally defined functions nontrivially differs 
from the one in \cite{CFMS,KenigB}.  Therefore, in a first subsection, we illustrate our arguments in the case of codimension 1 to build reader's intuition.

\section{Discussion of the comparison theorem in codimension 1}

We present here two proofs of the comparison principle in the codimension 1 case. 
The first proof of the one we can find in \cite{CFMS,KenigB} and the second one is our alternative proof. 
We consider in this subsection that the reader knows or is able to see the results in the three first sections of 
\cite{KenigB}, that contain the analogue in codimension 1 of the results proved in the previous sections. 

For simplicity, the domain $\Omega \subset \R^n$ that we study is a special Lipschitz domain, that is 
$$\Omega = \{ (y,t) \in \R^{n-1} \times \R, \, \varphi(y) < t\}$$
where $\varphi: \, \R^{n-1} \to \R$ is a Lipschitz function. The elliptic operator that we consider is $L=-\diver A \nabla$, where $A$ is a matrix with bounded measurable coefficients satisfying the classical elliptic condition (see for instance (1.1.1) in \cite{KenigB}).
Yet, the change of variable $\rho: \R^{n-1} \times \R \to \R^{n-1} \times \R$ defined by
$$\rho(y,t) = (y,t-\varphi(y))$$
maps $\Omega$ into $\wt \Omega = \R^n_+:= \{(y,t)\in \R^{n-1} \times \R, \, t>0\}$ and changes the elliptic operator $L$ into $\wt L = - \diver \wt A \nabla$, where $\wt A$ is also a matrix with bounded measurable coefficients satisfying the elliptic condition (1.1.1) in \cite{KenigB}. 
Therefore, in the sequel, we reduce our choices of $\Omega$ and 
$\Gamma = \d \Omega$ 
to $\R^n_+$ and $\R^{n-1} = \{(y,0)\in \R^n, \, y \in \R^{n-1}\}$ respectively. 

\medskip

Let us recall some facts that also 
hold in the present context.
If $u \in W^{1,2}(D)$ and $D$ is a Lipschitz set, then $u$ has a trace on the boundary of $D$ and 
hence we can give a sense
to the expression $u=h$ on $\partial D$. If in addition, the function $u$ is a solution of 
$Lu = 0$ in $D$ (the notion of solution is taken in the weak sense, see for instance \cite[Definition 1.1.4]{KenigB}) and $h$ is continuous on $\partial D$, then $u$ is continuous on $\overline{D}$.

The Green function (associated to the domain $\Omega = \R^n_+$ and the elliptic operator $L$) is denoted 
by $g(X,Y)$ - with
$X,Y\in \Omega$ - and the harmonic measure (associated to $\Omega$ and $L$) is written $\omega^X(E)$ 
- with
$X\in \Omega$ and $E \subset \Gamma$. The notation $\omega^X_D(E)$ - where $X\in D$ and $E \subset \partial D$ - denotes the harmonic measure associated to the domain $D$ (and the operator $L$).

When $x_0 = (y_0,0) \in \Gamma$ and $r>0$, we use the notation $A_r(x_0)$ for $(y_0,r)$.

\medskip

In this context, the comparison principle given in \cite[Lemma 1.3.7]{KenigB} is

\begin{lemma}[Comparison principle, codimension 1] \label{lcpcod1}
Let $x_0\in \Gamma$ and $r>0$. Let $u,v \in W^{1,2}(\Omega \cap B(x_0,2r))$ be two non-negative solutions of $Lu = Lv = 0$ in $\Omega \cap B(x_0,2r)$ satisfying $u=v= 0$ on $\Gamma \cap B(x_0,2r)$. 
Then for any $X \in \Omega \cap B(x_0,r)$, we have
\begin{equation}\label{cod1cp3}
C^{-1} \frac{u(A_r(x_0))}{v(A_r(x_0))} \leq \frac{u(X)}{v(X)} \leq C \frac{u(A_r(x_0))}{v(A_r(x_0))},
\end{equation}
where $C>0$ depends only on the dimension $n$ and the ellipticity constants of the matrix $A$.
\end{lemma}

\bp We recall quickly the ideas of the proof of the comparison principle found in \cite{KenigB}.

\medskip

Let $x_0\in \Gamma$ and $r>0$ be given.  
We denote $A_r(x_0)$ by $X_0$ and, for
$\alpha >0$, $B(x_0,\alpha r)$ by $B_\alpha$.

The proof of \eqref{cod1cp3} is reduced to the proof of the upper bound
\begin{equation}\label{cod1cp3a}
\frac{u(X)}{v(X)} \leq C \frac{u(X_0)}{v(X_0)} \ \ \text{ for }
X\in B_1 \cap \Omega
\end{equation}
because of the symmetry of the role of $u$ and $v$.

\medskip 

\noindent {\bf Step 1:} Upper bound on $u$.

\smallskip

By definition of the harmonic measure,
\begin{equation}\label{cod1cp1}
u(X) = \int_{\partial (\Omega \cap B_{3/2})} u(y) d\omega^X_{\Omega \cap B_{3/2}}(y)
\ \ \text{ for } X \in B_{3/2} \cap \Omega.
\end{equation}
Note that $\partial (\Omega \cap B_{3/2}) = (\partial B_{3/2} \cap \Omega) \cup (\Gamma \cap B_{3/2})$. Hence, for any $X \in B_{3/2} \cap \Omega$,
\begin{equation}\label{cod1cp1b}
u(X) = \int_{\partial B_{3/2} \cap \Omega} u(y) d\omega_{\Omega \cap B_{3/2}}^X(y) + \int_{\Gamma \cap B_{3/2}} u(y) d\omega_{\Omega \cap B_{3/2}}^X(y) = \int_{\partial B_{3/2} \cap \Omega} u(y) d\omega^X_{\Omega \cap B_{3/2}}(y) 
\end{equation}
because, by assumption, $u = 0$ on $\Gamma \cap B_2$.
Lemma 1.3.5 in \cite{KenigB} gives now, for any $Y \in B_{7/4}$, the bound $u(Y) \leq C u(X_0)$ with a constant $C>0$ which is
independent of $Y$. So by the
positivity of the harmonic measure, we have for any $X \in B_{3/2} \cap \Omega$
\begin{equation}\label{cod1cp1c} \begin{split}
u(X) & \leq C u(X_0) \int_{\partial B_{3/2} \cap \Omega} d\omega^X_{\Omega \cap B_{3/2}}(y) 
\leq C u(X_0) \omega^X_{\Omega \cap B_{3/2}}(\partial B_{3/2} \cap \Omega).
\end{split} \end{equation}

\medskip 

\noindent {\bf Step 2:} Lower bound on $v$.

\smallskip

First, again by definition of the harmonic measure, we have that for
$X \in B_{3/2} \cap \Omega$,
\begin{equation}\label{cod1cp2}
v(X)  = \int_{\partial (\Omega \cap B_{3/2})} v(y) d\omega^X_{\Omega \cap B_{3/2}}(y). 
\end{equation}
Set $E = \{y\in \partial B_{3/2} \cap \Omega \, ; \, \dist(y,\Gamma) \geq \frac12 r\}$. 
By assumption, $v \geq 0$ on $\partial (\Omega \cap B_{3/2})$. 
In addition, thanks to the Harnack inequality, 
$v(y) \geq C^{-1} v(X_0)$ for every $y\in E$,
with a constant $C>0$ 
that is
independent of $y$. So the positivity of the harmonic measure yields, for
$X \in B_{3/2} \cap \Omega$,
\begin{equation}\label{cod1cp2b} \begin{split}
v(X) & \geq C^{-1} v(X_0) \int_{E} d\omega^X_{\Omega \cap B_{3/2}}(y)  
\geq C^{-1} v(X_0) \omega^X_{\Omega \cap B_{3/2}}(E).
\end{split} \end{equation}

\medskip 

\noindent {\bf Step 3:} Conclusion.

\smallskip

From steps 1 and 2, we deduce that 
\begin{equation}\label{cod1cp9}
\frac{u(X)}{v(X)} 
\leq C \frac{u(X_0)}{v(X_0)} \frac{\omega^X_{\Omega \cap B_{3/2}}(\partial B_{3/2} \cap \Omega)}{\omega^X_{\Omega \cap B_{3/2}}(E)} \ \ \ \text{ for } 
X \in \Omega \cap B_{3/2}.
\end{equation}
The inequality \eqref{cod1cp3a} is now a consequence of the doubling property of the harmonic measure (see for instance (1.3.7) in \cite{KenigB}), that gives
\begin{equation}\label{cod1cp10}
\omega^X_{\Omega \cap B_{3/2}}(\partial B_{3/2} \cap \Omega) 
\leq C \omega^X_{\Omega \cap B_{3/2}}(E) \ \ \ \text{ for } 
X \in \Omega \cap B_{1}.
\end{equation}
The lemma follows.
\ep

The proof above
relies on the use of the harmonic measure for the domain $\Omega \cap B_{3/2}$. We want 
to avoid this, and use only the
Green functions and harmonic measures related to
the domain $\Omega$ itself.

\medskip 

First, we need a way to compare two functions in a domain, that is a suitable maximum principle. In the previous proof of Lemma \ref{lcpcod1}, the maximum principle was replaced/hidden by the positivity of the harmonic measure,
whose proof makes a crucial use of the maximum principle for solutions.
See \cite[Definition 1.2.6]{KenigB} for 
the construction of the harmonic measure, 
and \cite[Corollary 1.1.18]{KenigB} for the maximum principle.
The maximum principle that we will use is the following.

\begin{lemma} \label{lMPcod1}
Let $F \subset E \subset \R^n$ be two sets such that $F$ is closed, $E$ is open, and $\dist(F,\R^n \setminus E) >0$.  
Let $u$ be a solution in $E \cap \Omega$ such that 
\begin{enumerate}[(i)]
\item $\ds \int_{E} |\nabla u|^2 \, dx < +\infty$,
\item $u \geq 0$ on $\Gamma \cap E$,
\item $u \geq 0$ in $(E \setminus F) \cap \Omega$.
\end{enumerate}
Then $u\geq 0$ in $E \cap \Omega$.
\end{lemma}

In a more `classical' maximum principle, assumption (iii) would be replaced by 
$$\text{\em (iii') $u \geq 0$ in $\partial E \cap \Omega$.}$$
Since this subsection aims to illustrate what we will do in the next subsection, we 
state here a
maximum principle 
which is
as close as possible to the one we will actually prove in higher codimension. Let us 
mention
that using (iii) instead of (iii') will not make computations harder or 
easier. However, 
(iii) is much easier to define and use in the higher codimension case
(to the point that we did not even try to give a precise meaning to (iii')).

We do not prove Lemma \ref{lMPcod1} here, because the proof is the same as for
Lemma \ref{lMPg} below, which is its higher codimension version.

Notice that Lemma \ref{lMPcod1} is really a maximum principle where we use the values of $u$
on a boundary $(\Gamma \cap E) \cup (\Omega \cap F \sm E)$
that surrounds $E$ to control the values of $u$ in $\Omega \cap E$, 
but here the boundary also has a thick part, $\Omega \cap F\sm E$. 
This makes it easier to define Dirichlet conditions on that thick set, which is the main 
point of (iii). 

The first assumption (i) is a technical hypothesis, it can be seen as a way to control $u$ at infinity, 
which is needed because we actually do not require $E$ or even $F$ to be bounded.

Lemma \ref{lMPcod1} will be used in different situations. For instance, we will use it when 
$E = 2B$ and $F= B$, where $B$ is a ball centered on $\Gamma$.

\bigskip

\noindent {\bf Step 1 (modified):} We want to find an upper estimate for 
$u$ that avoids using 
the measure $\omega^X_{\Omega \cap B_{3/2}}$. 
Lemma 1.3.5 in \cite{KenigB} gives, as before, that $u(X) \leq Cu(X_0)$ for any $X \in B_{7/4} \cap \Omega$.  
The following result states the non-degeneracy of the harmonic measure.
\begin{equation}\label{cod1cp11}
\omega^X(\Gamma \setminus B_{5/4}) \geq C^{-1}  \ \ \ \text{ for }
X \in \Omega \setminus B_{3/2},
\end{equation}
where $C>0$ is independent of $x_0$, $r$ or $X$.
Indeed, when $X \in \Omega \setminus B_{3/2}$ is close to the boundary, the lower bound \eqref{cod1cp11} can be seen as a consequence of the H\"older continuity of solutions. The proof for all $X \in \Omega \setminus B_{3/2}$ is then obtained with the Harnack inequality. See 
\cite[Lemma 1.3.2]{KenigB} or Lemma \ref{ltcp4} below for the proof.

From there, we deduce that
\begin{equation}\label{cod1cp12}
u(X) \leq Cu(X_0) \omega^X(\Gamma \setminus B_{5/4})  \ \ \ \text{ for } 
X \in \Omega \cap [ B_{7/4} \setminus B_{3/2}].
\end{equation}
We want to use the maximum principle given above (Lemma \ref{lMPcod1}), with $E = B_{7/4}$ and 
$F = B_{3/2}$. However, the function $X \to \omega^X(\Gamma \setminus B_{5/4})$ doesn't satisfy the assumption (i) of Lemma \ref{lMPcod1}. 
So we take $h \in C^\infty(\R^n)$ such that $0\leq h \leq 1$, 
$h \equiv 1$ on $\R^n \setminus B_{5/4}$, and $h \equiv 0$ on $\R^n \setminus B_{9/8}$. Define 
$u_h$ as the only solution of 
$Lu_h =0$ in $\Omega$ with the Dirichlet condition $u_h = h$ on $\Gamma$. We have $u_h(X) \geq \omega^X(\Gamma \setminus B_{5/4})$ by the positivity of the harmonic measure,
and thus the bound \eqref{cod1cp12} yields the existence of $K_0>0$ (independent of $x_0$, $r$, $X$) such that 
\begin{equation}\label{cod1cp12b}
u_1(X):=  K_0 u(X_0) u_h(X) - u(X) \geq 0   \ \ \ \text{ for } 
X \in \Omega \cap [ B_{7/4} \setminus B_{3/2}].
\end{equation}
It would be easy to 
check that $u_1 \geq 0$ on $\Gamma \cap B_{7/4}$ and 
$\int_E |\nabla u_1| < +\infty$, but we leave the details because they will be done in the larger codimension case. So
Lemma \ref{lMPcod1} gives that $u_1 \geq 0$ in $\Omega \cap B_{7/4}$, that is
\begin{equation}\label{cod1cp12c}
u(X) \leq K_0 u(X_0) u_h(X) \leq K_0 u(X_0) \omega^X(\Gamma \setminus B_{9/8})  
\ \ \ \text{ for }
X \in \Omega \cap B_{7/4},
\end{equation}
by definition of $h$ and positivity of the harmonic measure.

\medskip 

\noindent {\bf Step 2 (modified):} In the same way, we want to adapt Step 2 of the proof of 
Lemma~\ref{lcpcod1}. If we want to 
proceed as in Step 1,
we would like to find and use a function $f$ that keeps the main properties of the object 
$\omega^X_{\Omega \cap B_{7/4}}(E)$, where 
$E = \big\{y\in \dr B_{7/4}, \, ; \, \dist(y,\Gamma) \geq r/2\big\}$.
For instance, $f$ such that 
\begin{enumerate}[(a)]
\item $f$ is a solution of 
$Lf = 0$ in $\Omega \cap B_{7/4}$,
\item $f \leq 0$ in $\Gamma \cap B_{7/4}$,
\item $f\leq 0$ in $\{X\in \Omega, \, \dist(X,\Gamma) < r/2\} \cap [B_{7/4} \setminus B_{3/2}]$,
\item $f(X) \approx \omega^X(\Gamma \setminus B_{9/8})$ in $\Omega \cap B_1$, 
in particular $f>0$ in $\Omega \cap B_1$.
\end{enumerate}
The last point is important to be able to conclude (in Step 3). 
It is given by the doubling property of the harmonic measure \eqref{cod1cp10} in the previous proof of Lemma \ref{lcpcod1}.

\smallskip

We were not able to find such a function $f$. 
However, we can construct an $f$ that satisfies some conditions close to
(a), (b), (c) and (d) above. Since $f$ fails to verify exactly (a), (b), (c) and (d), extra computations are needed.

\smallskip

First, note that it is enough to prove that there exists $M>0$ depending only on
$n$ and the ellipticity constants of $A$, 
such that for
$y_0\in \Gamma$, 
$s>0$, and any non-negative solution $v$ to $Lv = 0$ in 
$B(y_0,Ms)$
\begin{equation} \label{cod1cp13}
v(X) \geq C^{-1} v(A_s(y)) \omega^X(\Gamma \setminus B(y_0,2s)) \ \ \ \text{ for }
X \in \Omega \cap B(y_0,s),
\end{equation}
where here the the corkscrew point $A_s(y)$ is just $A_s(y) = (y,s)$.
Indeed, if we have \eqref{cod1cp13}, then we can prove 
that, in the situation of Step 2,
\begin{equation} \label{cod1cp13a}
v(X) \geq C^{-1} v(X_0) \omega^X(\Gamma \setminus B_2) \ \ \ \text{ for }
X \in \Omega \cap B_1
\end{equation}
by using a proper covering of the domain $\Omega \cap B_1$ (if 
$X \in \Omega \cap B_1$ lies within $\frac{1}{4M}$ of $\Gamma$, say,  
we use \eqref{cod1cp13} with $y_0 \in \Gamma$ close to $X$ and 
$s = \frac 1{2M}$,
and then the Harnack inequality; 
if instead $X \in \Omega \cap B_1$
is far from the boundary $\Gamma$, 
\eqref{cod1cp13a} is only a consequence of the Harnack inequality). 

The conclusion \eqref{cod1cp3a} comes then from \eqref{cod1cp12c}, \eqref{cod1cp13a} and the doubling property of the harmonic measure (see for instance (1.3.7) in \cite{KenigB}). 

\medskip

It remains to prove the claim \eqref{cod1cp13}. 
Let $y_0 \in \Gamma$, $s>0$, and $v$ be given.
Write $Y_0$ for $A_s(y_0)$ and, for 
$\alpha>0$,
write $B'_{\alpha}$ for $B(y_0,\alpha s)$. Let $K_1$ and $K_2$ be 
some positive constants that are independent of $y_0$, $s$, and $X$, and will be chosen later.
Pick 
$h_{K_2} \in C^\infty(\R^n)$ such that $h_{K_2} \equiv 1$ on $\R^n \setminus B'_{K_2}$, 
$0\leq h_{K_2} \leq 1$ 
everywhere,
and $h_{K_2} \equiv 0$ on $B_{K_2/2}$.
Define $u_{K_2}$ as the solution of 
$Lu_{K_2} = 0$ in $\Omega$ with the Dirichlet condition $h_{K_2}$ on $\Gamma$, that will serve as a 
smooth substitute for 
$X \to \omega^X(\Gamma \setminus B'_{K_2})$.
Define a 
function $f_{y_0,s}$ on $\Omega \setminus \{Y_0\}$ by
\begin{equation}\label{cod1cp14a}
f_{y_0,s}(X) = s^{n-2} g(X,Y_0) - K_1 u_{K_2}.
\end{equation}

When
$|X-Y_0| \geq s/8$, the term  $s^{n-2} g(X,Y_0)$ is uniformly bounded: this fact can be found 
in \cite{HoK} (for $n\geq 3$) and \cite{DK} (for $n=2$).
In addition, due to the non-degeneracy of the harmonic measure 
(same argument as for \eqref{cod1cp11}, similar to \cite[Lemma 1.3.2]{KenigB}), 
there exists $C>0$ (independent of $K_2>0$) such that 
$\omega^X(\Gamma \setminus B'_{K_2}) \geq C^{-1}$ for
$X \in \Omega \setminus B'_{2K_2}$. 
Hence
we can find $K_1>0$ such that for any choice of $K_2 >0$, we have
\begin{equation}\label{cod1cp15}
f_{y_0,s}(X) \leq 0 \ \ \ \text{ for }
X \in \Omega \setminus B'_{2K_2}.
\end{equation}
 
 \medskip
 
 For the sequel, we state an important result. There holds
 \begin{equation} \label{cod1cp15a}
 C^{-1} s^{n-2} g(X,Y_0) \leq \omega^X(\Gamma \setminus B'_2) \leq  C s^{n-2} g(X,Y_0) 
 \ \ \ \text{ for }
 X \in \Omega \cap [B'_1 \setminus B(Y_0,s/8)],
 \end{equation}
 where $C>0$ depends only on $n$ and the ellipticity constant of the matrix $A$.
This result can be seen as an analogue of \cite[Corollary 1.3.6]{KenigB}. 
It is proven in the higher codimension case in Lemma \ref{ltcp3} below. 
The equivalence \eqref{cod1cp15a} can be seen as a weak version of the comparison principle, dealing only with harmonic measures and Green functions. It can be proven, like \cite[Corollary 1.3.6]{KenigB}, before the full comparison principle by using the specific properties of the Green functions and 
harmonic measures.

 \smallskip
 
We want to take $K_2>0$ so large that
\begin{equation}\label{cod1cp16}
f_{y_0,s}(X) \geq \frac{1}{2} s^{n-2} g(X,Y_0) \ \ \ \text{ for }
X \in \Omega \cap [B'_1 \setminus B(Y_0,s/8)].
\end{equation}
We build a smooth substitute $u_4$ for $\omega^X(\Gamma \setminus B'_2)$, namely the solution of
$Lu_4 = 0$ in $\Omega$ with the Dirichlet condition $u_4 = h_4$ on $\Gamma$, 
where $h_4\in C^\infty(\R^n)$,
$h_4 \equiv 1$ on $\R^n \setminus B'_4$, $0\leq h_4 \leq 1$ 
everywhere,
and $h_4 \equiv 0$ on $B'_2$. Thanks to the H\"older continuity of solutions and the non-degeneracy of the harmonic measure, we have that for 
$X\in \Omega \cap [B'_{10} \setminus B'_5]$ and any $K_2 \geq 20$,
\begin{equation}\label{cod1cp16a}
C^{-1} u_{K_2}(X) \leq (K_2)^{-\alpha} \leq C (K_2)^{-\alpha} u_4(X),
\end{equation}
with constants $C,\alpha>0$ independent of $K_2$, $y_0$, $s$ or $X$. 
Since the functions $u_{K_2}$ and $u_4$ are smooth
enough, 
and $C^{-1} u_{K_2} = 0 \leq C (K_2)^{-\alpha} u_4(X)$ on $\Gamma \cap B'_{10}$, the maximum principle (Lemma \ref{lMPcod1}) implies that
\begin{equation}\label{cod1cp16b}
u_{K_2}(X) \leq C (K_2)^{-\alpha} u_4(X) \ \ \ \text{ for }
X \in \Omega \cap B'_{10}.
\end{equation}
We use \eqref{cod1cp15a} to get that for $K_2 \geq 20$,
\begin{equation}\label{cod1cp16c}
K_1 u_{K_2}(X) \leq C K_1 (K_2)^{-\alpha} s^{n-2} g(X,Y_0)  \ \ \ \text{ for }
X \in \Omega \cap [B'_1 \setminus B(Y_0,s/8)].
\end{equation}
The inequality \eqref{cod1cp16} can be now obtained by taking $K_2 \geq 20$ so that $C K_1 (K_2)^{-\alpha} \leq \frac12$. From \eqref{cod1cp16} and \eqref{cod1cp15a}, we deduce that
\begin{equation}\label{cod1cp16d}
f_{y_0,s}(X) \geq C^{-1} \omega^X(\Gamma \setminus B'_2) \ \ \ \text{ for }
X \in B'_1 \setminus B(Y_0,s/8),
\end{equation}
where $C>0$ depends only on $n$ and the ellipticity constants of the matrix $A$. 

\medskip 

Recall that 
our goal is to prove
the claim \eqref{cod1cp13}, which will be established with $M= 4K_2$. 
Let $v$ be a non-negative solution of
$Lv = 0$ in $\Omega \cap B'_{4K_2}$. We can find $K_3>0$ (independent of $y_0$, $s$ and $X$) such that
\begin{equation}\label{cod1cp17}
v(X) \geq K_3 v(Y_0) f_{y_0,s}(X) \ \ \ \text{ for } 
X \in B(Y_0,\frac14 s) \setminus B(Y_0,\frac18 s).
\end{equation}
Indeed $f_{y_0,s}(X) \leq s^{n-2} g(X,Y_0) \leq C$ when $|X-Y_0| \geq s/8$, 
thanks to the pointwise bounds on the Green function (see \cite{HoK}, \cite{DK}) 
and $ v(X) \geq C^{-1} v(Y_0)$ when $|X-Y_0| \leq s/4$ because of the Harnack inequality. Also, 
thanks to \eqref{cod1cp15},
\begin{equation}\label{cod1cp18}
v(X) \geq 0 \geq K_3 v(Y_0) f_{y_0,s}(X) \ \ \ \text{ for } 
X \in \Omega \cap [B'_{4K_2} \setminus B'_{2K_2}]
\end{equation}
and it is easy to check that
\begin{equation}\label{cod1cp19}
v(y) \geq 0 \geq K_3 v(Y_0) f_{y_0,s}(y) \ \ \ \text{ for } 
y \in \Gamma \cap B'_{4K_2}.
\end{equation}
We can apply our maximal principle, that is Lemma \ref{lMPcod1}, with $E = B'_{4K_2} \setminus B(Y_0,\frac18 s)$ and $F = B'_{2K_2} \setminus B(Y_0,\frac14 s)$ and get that
\begin{equation}\label{cod1cp20}
v(X) \geq K_3 v(Y_0) f_{y_0,s}(X) \ \ \ \text{ for } 
X \in \Omega \cap [B'_{4K_2} \setminus B(Y_0,\frac18 s)].
\end{equation}
In particular, thanks to \eqref{cod1cp16d},
\begin{equation}\label{cod1cp21}
v(X) \geq C^{-1} v(Y_0) \omega^X(\Gamma \setminus B'_2) \ \ \ \text{ for }
X \in \Omega \cap [B'_{1} \setminus B(Y_0,\frac18 s)].
\end{equation}
Since both $v$ and $X \to \omega^X(\Gamma \setminus B'_2)$ are solutions in $\Omega \cap B'_2$, the Harnack inequality proves
\begin{equation}\label{cod1cp23}
v(X) \geq C^{-1} v(Y_0) \omega^X(\Gamma \setminus B'_2) \ \ \ \text{ for }
X \in \Omega \cap B'_{1}.
\end{equation}
The claim \eqref{cod1cp13} follows, which ends our alternative proof of Lemma \ref{lcpcod1}.

\section{The case of codimension higher than 1}

We need first the following version of the maximum principle.

\begin{lemma} \label{lMPg}
Let $F \subset \R^n$ be a closed set and $E \subset \R^n$ an open set such that
$F \subset E \subset \R^n$ and $\dist(F,\R^n \setminus E)>0$. 
Let $u \in \WW(E)$ be a supersolution 
for $L$
in $\Omega \cap E$ such that 
\begin{enumerate}[(i)]
\item $\ds \int_{E} |\nabla u|^2 \, dm < +\infty$,
\item $Tu \geq 0$ a.e. on $\Gamma \cap E$,
\item $u \geq 0$ a.e. in $(E \setminus F) \cap \Omega$.
\end{enumerate}
Then $u\geq 0$ a.e. in $E \cap \Omega$.
\end{lemma}

\bp The present proof is a slight variation of the proof of Lemma~\ref{lMP}.

Set $v := \min\{u,0\}$ in $E \cap \Omega$ and $v:=0$ in $\Omega \setminus E$. Note that $v\leq 0$. We want to use $v$ as a test function. We claim that
\begin{equation} \label{GreenU13}
\text{$v$ lies 
in $W_0$ and is supported in $F$. }
\end{equation}
Pick 
$\eta \in C^\infty_0(E)$ such that $\eta = 1$ in $F$ and $\eta \geq 0$ everywhere, it is indeed possible because $\dist(F,\R^n \setminus E)>0$. 
Since $u\in \WW(E)$, we have $\eta u \in W$, from which we deduce $\min\{0, \eta u\} \in W$ by
Lemma \ref{lcompo}. 
By (iii), $v = \min\{0, \eta u\}$ almost everywhere and hence $v\in W$. 

Notice 
that $T(\eta u) \geq 0$ because of Assumption (ii) 
(and Lemma~\ref{defTrWE}). Hence 
$v = \min\{\eta u,0\} \in W_0$. 
And since (iii) also proves that $v$ is supported in $F$, the claim \eqref{GreenU13} follows.

\medskip

Since $v$ is in $W_0$, Lemma~\ref{ldens0} proves that $v$ can be approached in $W$ by a sequence of 
functions $(v_k)_{k\geq 1}$ in $C^\infty_0(\Omega)$
(i.e., that are compactly supported in $\Omega$; see \eqref{Cinfty0def}).
Note also that the construction used in Lemma~\ref{ldens0} allows us, 
since $v\leq 0$ is supported in $F$, to take 
$v_k \leq 0$ and compactly supported in $E$.
Definition \ref{defsubsol} gives
\begin{equation} \label{GreenU14}
\int_{E} \A \nabla u \cdot \nabla v_k \, dm = \int_{\Omega} \A \nabla u \cdot \nabla v_k \, dm \leq 0
\end{equation}
and since the map
\begin{equation} \label{GreenU15}
\varphi \in W \to \int_{E} \A \nabla u \cdot \nabla \varphi \, dm
\end{equation}
is bounded on $W$ thanks to assumption (i) and \eqref{ABounded}, we deduce that
\begin{equation} \label{GreenU16}
\int_{E} \A \nabla u \cdot \nabla v \, dm \leq 0.
\end{equation}

\medskip

Now Lemma~\ref{lcompo} gives
\begin{equation} \label{GreenU17}
\nabla v = \left\{ \begin{array}{ll}  \nabla u & \text{ if } u<0 \\ 0 & \text{ if } u\geq 0 \end{array}\right.
\end{equation}
and so \eqref{GreenU16} becomes
\begin{equation} \label{GreenU18}
\int_\Omega \A \nabla v \cdot \nabla v \, dm = \int_{E} \A \nabla u \cdot \nabla v \, dm  \leq 0.
\end{equation}
Together with the ellipticity condition \eqref{AElliptic}, we obtain $\|v\|_W \leq 0$. Recall that $\|.\|_W$ is a norm on $W_0 \ni v$, hence $v = 0$ a.e. in $\Omega$. We conclude from the definition of $v$ that $u \geq 0$ a.e. in $E \cap \Omega$.
\ep

Let us use the maximum principle above to prove the following result
on the Green function.

\begin{lemma} \label{ltcp5}
We have
\begin{equation} \label{tcp5}
 g(x,y) \leq C \min\{\delta(y),\delta(x)\}^{1-d}  \ \ \ \text{ for }
 x,y\in \Omega \text{ such that } |x-y| \geq \delta(y)/4,
\end{equation}
where the constant $C>0$ depends only on $d$, $n$, $C_0$ and $C_1$.
\end{lemma}

\begin{remark}
Lemma \ref{ltcp5} is an improvement on the pointwise bounds \eqref{GreenE7} only when $d<1$.
\end{remark}

\bp Let $y \in \Omega$. Lemma \ref{GreenEx} (v) gives 
\begin{equation} \label{tcp34a}
 g(x,y) \leq K_1 \delta(y)^{1-d}  \ \ \ \text{ for }
 x \in B(y,\delta(y)/4) \setminus B(y,\delta(y)/8)
\end{equation}
for some $K_1>0$ 
that is
independent of $x$ and $y$.
Define 
$u$ on $\Omega \setminus \{y\}$ by
$u(x) = K_1\delta(y)^{1-d} - g(x,y)$. Notice 
that $u$ is a solution in $\Omega \setminus \overline{B(y,\delta(y)/4)}$, 
by \eqref{GreenE6}. 
Also, 
thanks to Lemma~\ref{GreenEx} (i),  the integral 
$\ds \int_{\Omega \setminus B(y,\delta(y)/8)} |\nabla u|^2 \, dm  =
\int_{\Omega \setminus B(y,\delta(y)/8)} |\nabla g(.,y)|^2 \, dm$ is finite, and 
$Tu = K_1 
\delta(y)^{1-d}$ 
is non-negative a.e. on $\Gamma$.
In addition, due to \eqref{tcp34a}, we have $u\geq 0$ on $B(y,\delta(y)/4) \setminus B(y,\delta(y)/4)$.
Thus
$u$ satisfies all the assumption of 
Lemma \ref{lMPg} (the maximum principle), where we choose 
$E = \R^n \setminus \overline{B(y,\delta(y)/8)}$ and $F = \R^n \setminus B(y,\delta(y)/8)$,
and which yields
\begin{equation} \label{tcp34b}
 g(x,y) \leq C \delta(y)^{1-d}  \ \ \ \text{ for }
 x \in \Omega \setminus B(y,\delta(y)/8).
\end{equation}

\medskip

It remains to prove that
\begin{equation} \label{tcp34}
 g(x,y) \leq C \delta(x)^{1-d}  \ \ \ \text{ for }
 x,y \in \Omega \text{ such that }   |x-y | \geq \delta(y)/4.
\end{equation}
But
Lemma \ref{GreenSym} says that 
$g(x,y) = g_T(y,x)$, where $g_T$ is the Green function associated to the operator $L_T = - \diver A^T \nabla$. The above argument proves that
\begin{equation} \label{tcp35}
 g(x,y) = g_T(y,x) \leq C \delta(x)^{1-d}  \ \ \ \text{ for }
 x,y \in \Omega \text{ such that }   |x-y| \geq \delta(x)/8,
\end{equation}
which is \eqref{tcp34} once we remark that $|x-y | \geq \delta(y)/4$ implies that
$|x-y | \geq \delta(x)/8$.
\ep

Let us prove the existence of ``corkscrew points'' in $\Omega$.

\begin{lemma} \label{lNTA}
There exists $\epsilon>0$, that depends only upon the dimensions $d$ and $n$ and the constant $C_0$,
such that for 
$x_0\in \Gamma$ and
$r>0$, there exists a point $A_r(x_0) \in \Omega$ such that
\begin{enumerate}[(i)]
\item $|A_r(x_0) - x_0| \leq r$,
\item $\delta(A_r(x_0)) \geq \epsilon r$.
\end{enumerate}
In particular, 
$\delta(A_r(x_0)) \approx |A_r(x_0) - x_0| \approx r$.
\end{lemma}

In the sequel, for any $s>0$ and $y\in \Gamma$, $A_s(y)$ will denote 
any point in $\Omega$ satisfying the conditions (i) and (ii) of Lemma \ref{lNTA}.

\ms

\bp
Let $x_0 \in \Gamma$ and $r>0$ be given. 
Let $\epsilon \in (0,1/8)$ be small, to be chosen soon.
Let $z_1, \cdots z_N$ be a maximal collection of points of $B(x_0,(1-2\epsilon)r)$
that lie at mutual distances at least $4\epsilon r$. Set $B_i = B(z_i,\epsilon r)$;
notice that the $2B_i = B(z_i,2\epsilon r)$ are disjoint and contained in $B(x_0, r)$, 
and the $5B_i$ cover $B(x_0,(1-2\epsilon)r)$ (by maximality), so 
$\sum_{i} | 5B_i | \geq |B(x_0,(1-2\epsilon)r)| \geq C^{-1} r^n$ and hence
$N \geq C^{-1} \epsilon^{-n}$.

Suppose for a moment that every $B_i$ meets $\Gamma$. Pick $y_i \in \Gamma \cap B_i$,
notice that $B(y_i,\epsilon r) \subset 2B_i$, and then use the Ahlfors-regularity property \eqref{1.1}
to prove that
\begin{equation}\label{a11.46}
C_0^{-1} (\epsilon r)^{d} N \leq \sum_{i=1}^N   \H^d(\Gamma \cap B(y_i,\epsilon r))
\leq  \sum_{i=1}^N \H^d(\Gamma \cap 2B_i)
\leq \H^d(\Gamma \cap B(x_0,r)) \leq C_0 r^d
\end{equation}
because the $2B_i$ are disjoint and contained in $B(x_0,r)$. Thus
$N \leq C_0^2 \epsilon^{-d}$, which makes our initial estimate on $N$ impossible if we choose 
$\epsilon$ such that $\epsilon^{n-d} < C^{-1} C_0^{-2}$.

We pick $\epsilon$ like this, and by contraposition get that at least one $B_i$ does not meet $\Gamma$. 
We choose $A_r(x_0) = z_i$, and notice that $\delta(x_i) \geq \epsilon r$ because 
$B_i \cap \Gamma = \emptyset$, and $|z_i- x_0| \leq r$ by construction. The lemma follows.
\ep

We also need the following 
slight improvement of Lemma \ref{lNTA}.

\begin{lemma} \label{lNTA2}
Let $M_1 \geq 1$ be given. There exists $M_2> M_1$ (depending on $d$, $n$, $C_0$ and $M_1$) 
such that for any ball $B$ of radius $r$ and centered on $\Gamma$
and any $x\in B \setminus \Gamma$ such that $\delta(x) \leq \frac r{M_2}$, we can find
$y\in B \setminus \Gamma$ such that
\begin{enumerate}[(i)]
\item $\delta(y) \geq M_1 \delta(x)$,
\item $|x-y| \leq M_2 \delta(x)$.
\end{enumerate}
\end{lemma}

\bp
The proof is almost the same.
Let $M_1 \geq 1$ be given, and let $M_2 \geq 10M_1$ be large, to be chosen soon.
Then let $B= B(x_0,r)$ and $x\in B$ be as in the statement.
Set $B' = B(x_0, r- 2M_1 \delta(x)) \cap B(x, (M_2-2M_1)\delta(x))$; notice that
the two radii are larger than $M_2 \delta(x)/2$, because $r \geq M_2 \delta(x) \geq 10M_1\delta(x)$,
so $|B'| \geq C^{-1} (M_2\delta(x))^n$ for a constant $C>0$ independent of $x$.

Pick a maximal family $(z_i)$, $1 \leq i \leq N$, of points of $B'$ that lie at mutual distances at least 
$M_1 \delta(x)$ from each other, and set $B_i = B(z_i,M_1\delta(x))$ for $1 \leq i \leq N$.
The $5B_i$ cover $B'$ by maximality, so
$N \geq C^{-1 }(M_1\delta(x))^{-n}|B'|
\geq C^{-1}(M_2/M_1)^n$.

Suppose for a moment that every $B_i$ meets $\Gamma$. Then pick $y_i \in B_i \cap \Gamma$
and use the Ahlfors regularity property \eqref{1.1} and the fact that the $2B_i$ contain the
$B(y_i,M_1\delta(x))$ and are disjoint to prove that
\begin{eqnarray}\label{a11.46b}
C_0^{-1} (M_1\delta(x))^{d} N &\leq& \sum_{i=1}^N   \H^d(\Gamma \cap B(y_i,M_1\delta(x)))
\nn\\
&\leq&  \sum_{i=1}^N \H^d(\Gamma \cap 2B_i)
\leq \H^d(\Gamma \cap B(x,M_2 \delta(x))) \leq C_0 (M_2 \delta(x))^d.
\end{eqnarray}
That is, $M_1^d N \leq C_0^2 M_2^d$, and 
this contradicts our other bound for $N$ if $M_2/M_1$ is large enough.
We choose $M_2$ like this; then some $B_i$ doesn't meet $\Gamma$, and we can take $y = z_i$.
\ep

Before we prove the comparison theorem, we  
 need a substitute for \cite[Lemma 1.3.4]{KenigB}. 

\begin{lemma} \label{ltcp2}
Let $x_0 \in \Gamma$ and $r>0$ 
be given, and let $X_0 := A_r(x_0)$ be as
in Lemma~\ref{lNTA}. Let $u \in \WW(B(x_0,2r))$ be a non-negative,
non identically zero,
solution of $Lu=0$ in $B(x_0 ,2r) \cap \Omega$,
such that 
$Tu \equiv 0$ on $B(x_0,2r) \cap \Gamma$. Then
\begin{equation} \label{tcp6}
u(X) \leq C u(X_0)   \ \ \ \ \text{ for } X \in B(x_0,r),
\end{equation}
where $C>0$ depends only on $d$, $n$, $C_0$ and $C_1$.
\end{lemma}

\bp
We follow the proof of \cite[Lemma~4.4]{KJ}. 

\medskip

Let $x\in \Gamma$ and $s>0$ such that $Tu \equiv 0$ on $B(x,s) \cap \Gamma$. Then the H\"older continuity of solutions given by Lemma~\ref{Holder1} proves the existence of $\epsilon>0$ (that depends only on $d$, $n$, $C_0$, $C_1$) such that
\begin{equation} \label{tcp7}
\sup_{B(x,\epsilon s)} u \leq \frac12 \sup_{B(x,s)} u.
\end{equation}
Without loss of generality, we can choose $\epsilon < \frac12$.

A rough idea of the proof of \eqref{tcp6} is that $u(x)$ should not be near the maximum of $u$
when $x$ lies close to $B(x_0,r) \cap \Gamma$, because of \eqref{tcp7}. Then we are left with
points $x$ that lie far from the boundary, and we can use the Harnack inequality to control $u(x)$.
The difficulty is that when $x \in B(x_0,r)$ lies close to $\Gamma$, $u(x)$ can be bounded by values 
of $u$ \underline{inside the domain}, and not by values of $u$ near $\Gamma$ but from the exterior 
of $B(x_0, r) $.
We will prove this latter fact by contradiction: we show that if $\sup_{B(x_0,r)} u$ exceeds a certain bound, 
then we can construct a sequence of points $X_k \in B(x_0,\frac32 r)$ such that $\delta(X_k) \to 0$ and 
$u(X_k) \to +\infty$, and hence we 
contradict the H\"older continuity of solutions at the boundary.

\medskip

Since $u(X) > 0$ somewhere,
the Harnack inequality (Lemma~\ref{HarnackI}), maybe applied a few times, 
yields $u(X_0) > 0$. We can rescale $u$ and assume that $u(X_0) = 1$. We claim 
that there exists
$M>0$ such that for any integer $N \geq 1$ and
$Y \in B(x_0,\frac32 r)$,
\begin{equation} \label{tcp8}
\delta(Y) \geq \epsilon^N r \Longrightarrow u(Y) \leq M^N  ,
\end{equation}
where $\epsilon$ 
comes from 
\eqref{tcp7} and the constant $M$ depends only upon $d$, $n$, $C_0$, $C_1$.
We will prove this by induction.
The base case (and in fact we will manage to start directly from some large integer $N_0$)
is given by the following. Let $M_2>0$ be the value given by Lemma~\ref{lNTA2} when $M_1 := \frac1\epsilon$.  Let $N_0 \geq 1$ be the smallest integer such that $M_2 \leq \epsilon^{-N_0}$.
We want to show the existence of $M_3 \geq 1$ such that
\begin{equation} \label{tcp9}
u(Y) \leq M_3 \ \ \ \ \ \text{ for every }
Y \in B(x_0,\frac32 r) \text{ such that } \delta(Y) \geq \epsilon^{N_0}r.
\end{equation}
Indeed, if $Y \in B(x_0,\frac32 r)$ satisfies $\delta(Y) \geq \epsilon^{N_0}r$, Lemma~\ref{lHC2} 
and the fact that $|x_0 - X_0| \approx r$ (by 
Lemma~\ref{lNTA}) imply 
the existence of a Harnack chain linking $Y$ to $ X_0$.
More precisely, we can find balls $B_1, \dots, B_h$ with a same radius, such that 
$Y \in B_1$, $X_0 \in B_h$, $3B_i \subset B(x_0,2r) \setminus \Gamma$
for $i\in \{1,\dots,h\}$, and $B_i \cap B_{i+1} \neq \emptyset$ for $i \in \{1,\dots,h-1\}$,
and in addition $h$ is bounded independently of $x_0$, $r$ and $Y$.
Together with the Harnack inequality (Lemma~\ref{HarnackI}), we obtain \eqref{tcp9}.
This proves \eqref{tcp8} for $N = N_0$, but also directly for $1 \leq N \leq N_0$, if we choose
$M \geq M_3$.

For any point $Y\in B(x_0,\frac32 r)$ such that
$\delta(Y) \leq \epsilon^{N_0} r \leq  \frac{r}{M_2} $, Lemma~\ref{lNTA2} (and our choice of $M_2$) 
gives the existence of $Z\in B(x_0,\frac32 r) \cap B(Y,M_2\delta(Y))$ such that $\delta(Z) \geq M_1 \delta(Y)$. 
Since $Z \in B(Y,M_2\delta(Y))$ and $\delta(Z)> \delta(Y) >0$, Lemma~\ref{lHC2} implies the existence of a Harnack chain whose length is bounded by a constant depending on $d$, $n$, $C_0$ (and $M_2$ - but $M_2$ depends only on the three first parameters) 
and together with the Harnack inequality (Lemma~\ref{HarnackI}), we obtain the existence of $M_4 \geq 1$ 
(that depends only on $d$, $n$, $C_0$ and $C_1$) such that $u(Y) \leq M_4 u(Z)$. 
So we just proved that
\begin{equation} \label{tcp10} \begin{array}{l}
\text{for any $Y \in B(x_0,\frac32 r)$ such that 
$\delta(Y) \leq \epsilon^{N_0} r$, } \\
\qquad \text{there exists $Z\in B(x_0,\frac32r)$ such that $\delta(Z) \geq M_1 \delta(Y)$ and $u(Y) \leq M_4 u(Z)$.}
\end{array} \end{equation}

We turn
to the main
induction step. Set $M = \max\{M_3, M_4\} \geq 1$ and let $N \geq N_0$ be given. 
Assume, by induction hypothesis, that for any $Z \in B(x_0,\frac32 r)$ satisfying $\delta(Z) \geq \epsilon^{N} r$, we have $u(Z) \leq M^N$. 
Let 
$Y \in B(x_0,\frac32 r)$ be such that 
$\delta(Y) \geq \epsilon^{N+1} r$ . 
The assertion \eqref{tcp10} yields the existence of $Z\in B(x_0,\frac32 r)$ such that $\delta(Z) \geq M_1 \delta(Y) = \epsilon^{-1} \delta(Y) \geq \epsilon^N r$ and $u(Y) \leq M_4 u(Z) \leq Mu(Z)$. 
By the induction hypothesis, $u(Y) \leq M^{N+1}$. 
This completes our induction step, and the proof of \eqref{tcp8} for every $N \geq 1$.

\medskip

Choose an integer $i$ such that $2^i \geq M$, where $M$ is the 
constant of \eqref{tcp8} that we just found, and then set $M' = M^{i+3}$.
We want to prove by contradiction that
\begin{equation} \label{tcp11}
u(X) \leq M' u(X_0) = M' \ \ \ \text{ for every }
X \in B(x_0,r).
\end{equation}
So we assume that  
\begin{equation} \label{tcp12}
\text{there exists $X_1 \in B(x_0,r)$ such that $u(X_1) > M'$}
\end{equation}
and we want to prove by induction that for every 
integer $k \geq 1$,
\begin{equation} \label{tcp13}
\text{there exists $X_k \in B(x_0,\frac 32r)$ such that $u(X_k) > M^{i+2+k}$ and $|X_k - x_0| \leq \frac32 r - 2^{-k}r$.}
\end{equation}
The base step of the induction is given by \eqref{tcp12} and we want to do  
the induction step. Let $k\geq 1$ be given and assume that 
\eqref{tcp13} holds.
From the contraposition of \eqref{tcp8}, we deduce that $\delta(X_k) < \epsilon^{i+2+k}r$. 
Choose 
$x_k \in \Gamma$ such that $|X_k-x_k| = \delta(X_k) < \epsilon^{i+2+k}r$.
By the induction hypothesis,
\begin{equation} \label{tcp14}
|x_k - x_0| \leq |x_k-X_k| + |X_k-x_0| \leq \frac{3r}{2}  - 2^{-k}r + \epsilon^{i+2+k} r 
\end{equation}
and, 
since $\epsilon \leq \frac12$,
\begin{equation} \label{tcp15}
|x_k - x_0| \leq \frac{3r}{2} - 2^{-k}r + 2^{-2-k} r. 
\end{equation}
Now, due to \eqref{tcp7}, we can find $X_{k+1} \in B(x_k,\epsilon^{2+k} r)$ such that 
\begin{equation} \label{tcp16}
u(X_{k+1}) \geq 2^i \sup_{X\in B(x_k,\epsilon^{i+2+k} r)} u(X) \geq 2^i u(X_k) \geq M^{i+2+(k+1)}.
\end{equation}
The induction step will be complete if we can prove that $|X_{k+1}-x_0| \leq \frac32 r - 2^{-(k+1)}r$. Indeed, 
\begin{equation} \label{tcp17} \begin{split}
|X_{k+1}-x_0| & \leq |X_{k+1} - x_k| + |x_k-x_0 | \leq \epsilon^{2+k} r + \frac{3r}{2} - 2^{-k}r  + 2^{-2-k} r \\
& \leq  \frac{3r}{2}- 2^{-k}r  + 2^{-1-k} r = \frac32 r - 2^{-k-1} r
\end{split}\end{equation}
by \eqref{tcp15} and because $\epsilon \leq \frac12$.

\medskip
Let us sum up. We assumed the existence of  $X_1 \in B(x_0,r)$ such that $u(X_1)>M'$ and we end up 
with \eqref{tcp13}, that is a sequence $X_k$ of values in $B(x_0,\frac32r)$ such that $u(X_k)$ 
increases to $+\infty$. 
Up to a subsequence, we can thus find a point in $\overline{B(x_0,\frac32r)}$ where $u$ is not continuous, 
which contradicts Lemma~\ref{HolderB}. Hence 
$u(X) \leq M' = M' u(X_0)$ for
$X \in B(x_0,r)$. Lemma~\ref{ltcp2} follows.
\ep

\begin{lemma} \label{ltcp1}
Let $x_0 \in \Gamma$ and $r>0$ be given, and set $X_0 := A_r(x_0)$ 
as in Lemma~\ref{lNTA}. Then for all 
$X \in \Omega \setminus B(X_0,\delta(X_0)/4)$, 
\begin{equation} \label{tcp1}
r^{d-1} g(X,X_0) \leq C \omega^X(B(x_0,r) \cap \Gamma)
\end{equation}
and
\begin{equation} \label{tcp1b}
r^{d-1} g(X,X_0) \leq C \omega^X(\Gamma \setminus B(x_0,2r)),
\end{equation}
where $C > 0$ depends
only on $d$, $n$, $C_0$ and $C_1$.
\end{lemma}

\bp
We prove \eqref{tcp1} first. Let $h \in C^\infty_0(B(x_0,r))$ satisfy $h\equiv 1$ on $B(x_0,r/2)$ 
and $0\leq h \leq 1$.
Define then $u \in W$ as the solution of $Lu=0$ with data $Th$ given by Lemma~\ref{lLM}. 
Set $v(X)= 1-u(X) \in W$ and observe
that $0\leq v \leq 1$ and $Tv = 0$ on $B(x_0,r/2) \cap \Gamma$.

By
Lemma~\ref{HolderB}, we can find $\epsilon>0$ (that depends only on $d$, $n$, $C_0$, $C_1$) such that $v(A_{\epsilon r}(x_0)) \leq \frac12$, i.e. $u(A_{\epsilon r}(x_0)) \geq \frac12$. 
The existence of Harnack chains (Lemma~\ref{lHC2}) and the Harnack inequality (Lemma~\ref{HarnackI}) give
\begin{equation} \label{tcp2}
 C^{-1} \leq u(X) \ \ \ \text{ for }
 X \in B(X_0,\delta(X_0)/2).
\end{equation}

\medskip

By Lemma~\ref{GreenEx} (v),
$g(X,X_0) \leq C |X-X_0|^{1-d}$ for 
$X \in \Omega \setminus B(X_0,\delta(X_0)/4)$. 
Since $\delta(X_0) \approx r$ by construction of $X_0$, 
\begin{equation} \label{tcp3a}
r^{d-1} g(X,X_0) \leq C 
 \ \ \ \text{ for } X \in B(X_0,\delta(X_0)/2) \setminus B(X_0,\delta(X_0)/4).
\end{equation}
The combination of \eqref{tcp2} and \eqref{tcp3a} yields the existence of $K_1>0$ (depending only on $n$, $d$, $C_0$ and $C_1$) such that
\begin{equation} \label{tcp3}
r^{d-1} g(X,X_0) \leq K_1 u(X)   \ \ \ \text{ for } 
X \in B(X_0,\delta(X_0)/2) \setminus B(X_0,\delta(X_0)/4).
\end{equation}
We claim that
$K_1 u(X) - r^{d-1} g(X,X_0)$ satisfies the assumptions of Lemma~\ref{lMPg}, with 
$E = \R^n \setminus \overline{B(X_0,\delta(X_0)/4)}$ and $F = \R^n \setminus B(X_0,\delta(X_0)/2)$. Indeed Assumption (i) of Lemma~\ref{lMPg}  is satisfied because $u\in W$ and by
Lemma \ref{GreenEx} (i). Assumption (ii) of Lemma~\ref{lMPg} holds because
$T u = h \geq 0$ by construction and also
$Tg(.,X_0) = 0$ thanks to  Lemma \ref{GreenEx} (i). 
Assumption (iii) of Lemma~\ref{lMPg} is given by \eqref{tcp3}. The lemma yields
\begin{equation} \label{tcp4}
r^{d-1} g(X,X_0) \leq K_1 u(X)  \ \ \ \text{ for }
X \in \Omega \setminus B(X_0,\delta(X_0)/4).
\end{equation}
By the positivity of the harmonic measure, $u(X) \leq \omega^X(B(x_0,r) \cap \Gamma)$ for
$X \in \Omega$;
\eqref{tcp1} follows.

\bigskip

Let us turn to the proof of \eqref{tcp1b}. We want to find two points 
$x_1,x_2 \in \Gamma \cap [B(x_0,Kr) \setminus B(x_0,4r)]$,
where the constant $K \geq 10$ depends only on $C_0$ and $d$,
 such that $X_1 := A_r(x_1)$ and $X_2:= A_r(x_2)$ satisfy 
\begin{equation} \label{tcpb1}
B(X_1,\delta(X_1)/4) \cap B(X_2,\delta(X_2)/4) = \emptyset.
\end{equation}
To get such points, we use the fact that $\Gamma$ is Ahlfors regular to find $M \geq 3$ 
(that depends only on $C_0$ and $d$) such that 
$ \Gamma_1: = \Gamma \cap [B(x_0,2Mr) \setminus B(x_0,6r)] \neq  \emptyset$ and 
$\Gamma_2 : = \Gamma \cap [B(x_0,2M^2r) \setminus B(x_0,6Mr)] \neq  \emptyset$. 
Any choice of points $x_1 \in \Gamma_1$ and $x_2\in \Gamma_2$ verifies \eqref{tcpb1}.

\medskip

Let $X \in \Omega \setminus B(X_0,\delta(X_0)/4)$. Thanks to \eqref{tcpb1}, 
there exists $i\in \{1,2\}$ such that $X \notin B(X_i,\delta(X_i)/4)$. 
The existence of Harnack chains (Lemma \ref{lHC2}), the Harnack inequality (Lemma \ref{HarnackI}), 
and the fact that $Y \to g(X,Y)$ is a solution of $L_Tu := -\diver A^T \nabla u = 0$ in $\Omega \setminus \{X\}$
(Lemma \ref{GreenEx} and Lemma \ref{GreenSym}) yield
\begin{equation} \label{tcpb2}
r^{d-1} g(X,X_0) \leq C r^{d-1} g(X,X_i).
\end{equation}
By 
\eqref{tcp1} and the positivity of the harmonic measure,
\begin{equation} \label{tcpb3}
r^{d-1} g(X,X_0) \leq C r^{d-1} g(X,X_i) \leq Cw^X(B(x_i,r)\cap \Gamma) \leq C w^X (\Gamma \setminus B(x_0,r)).
\end{equation}
The lemma follows.
\ep

We turn now to 
the non-degeneracy of the harmonic measure.

\begin{lemma} \label{ltcp4}
Let $\alpha >1$, $x_0 \in \Gamma$, and $r>0$ be given, and let
$X_0:= A_r(x_0) \in \Omega$ be as in Lemma \ref{lNTA}.
Then 
\begin{equation} \label{tcp33a}
\omega^X(B(x_0,r) \cap \Gamma) \geq C_\alpha^{-1} \ \ \ \text{ for } 
X \in B(x_0,r/\alpha),
\end{equation}
\begin{equation} \label{tcp33}
\omega^X(B(x_0,r) \cap \Gamma) \geq C^{-1}_\alpha \ \ \ \text{ for }  
X \in B(X_0,\delta(X_0)/\alpha),
\end{equation}
\begin{equation} \label{tcp33b}
\omega^X(\Gamma \setminus B(x_0,r)) \geq C^{-1}_\alpha \ \ \ \text{ for } 
X \in \Omega \setminus B(x_0,\alpha r),
\end{equation}
and
\begin{equation} \label{tcp33c}
\omega^X(\Gamma \setminus B(x_0,r)) \geq C^{-1}_\alpha \ \ \ \text{ for } 
X \in B(X_0,\delta(X_0)/\alpha),
\end{equation}
where $C_\alpha>0$ depends only upon $d$, $n$, $C_0$, $C_1$ and $\alpha$.
\end{lemma}

\bp 
Let us first 
prove \eqref{tcp33a}. 
Set $u(X) = 1- \omega^X(B(x_0,r) \cap \Gamma)$. 
By Lemma \ref{lprhm}, $u$ lies
in $\WW(B(x_0,r))$,
is a solution of $Lu=0$ in $\Omega \cap B(x_0,r)$, and has a vanishing trace on $\Gamma \cap B(x_0,r)$.
So the H\"older continuity of solutions at the boundary (Lemma \ref{HolderB}) 
gives the existence of an
$\epsilon>0$, that depends only on $d$, $n$, $C_0$, $C_1$ and $\alpha$, 
such that $u(X) \leq \frac12$
for every 
$X \in B(x_0,\frac 12 [1+ \frac 1{\alpha}]r)$ such that 
$\delta(X) \leq \epsilon r$. 
Thus $v(X) := \omega^X(B(x_0,r) \geq \frac12$ for 
$X \in B(x_0,\frac 12 [1+ \frac 1{\alpha}] r)$ such that
$\delta(X) \leq \epsilon r$. We now
deduce \eqref{tcp33a} from
the existence of Harnack chains (Lemma \ref{lHC2}) and the Harnack inequality (Lemma \ref{HarnackI}).

The assertion \eqref{tcp33} follows from
\eqref{tcp33a}. Indeed,
\eqref{tcp33a} implies that 
$\omega^{A_{r/2}(x_0)}(B(x_0,r) \cap \Gamma) \geq C^{-1}$.  The existence of Harnack chains (Lemma \ref{lHC2}) and the Harnack inequality (Lemma \ref{HarnackI}) allow us to conclude.
Finally \eqref{tcp33b} and \eqref{tcp33c} can be proved as above, and we leave the details to the reader.
\ep

\begin{lemma} \label{ltcp3}
Let $x_0 \in \Gamma$ and $r>0$
be given, and set $X_0 = A_r(x_0)$. Then
\begin{equation} \label{tcp18}
C^{-1} r^{d-1} g(X,X_0) \leq \omega^X(B(x_0,r) \cap \Gamma) \leq C r^{d-1} g(X,X_0) 
\ \ \ \text{ for }
X\in \Omega \setminus B(x_0,2r),
\end{equation}
and
\begin{equation} \label{tcp18b}
\begin{split}
C^{-1} r^{d-1} g(X,X_0) \leq \omega^X(\Gamma \setminus B(x_0,2r)) 
&\leq  C r^{d-1} g(X,X_0) 
\\ &\text{ for }
X\in B(x_0,r) \setminus B(X_0,\delta(X_0)/4),
\end{split}
\end{equation}
where $C>0$ depends only upon $d$, $n$, $C_0$ and $C_1$.
\end{lemma}

\bp
The lower bounds are a consequence of Lemma~\ref{ltcp1};
the one in \eqref{tcp18} also requires to notice that $\delta(X_0) \leq r$ and thus 
$B(X_0,\delta(X_0)/4) \subset B(x_0,2r)$.
\medskip

It remains to check the upper bounds. But we first prove
an intermediate result. 
We claim that for $\phi \in C^\infty(\R^n) \cap W$ and $X \notin \supp \, \phi$,
\begin{equation} \label{tcp20}
u_\phi(X) = - \int_\Omega A \nabla \phi(Y) \cdot \nabla_y g(X,Y) dY,
\end{equation}
where $u_\phi \in W$ is the solution of $Lu_\phi = 0$, with the Dirichlet condition $Tu_\phi = T\phi$ 
on $\Gamma$, given by Lemma \ref{lLM}.
Indeed, recall that by \eqref{ABounded} and \eqref{AElliptic} the map
\begin{equation} \label{tcp21}
u,v \in W_0 \to \int_\Omega A \nabla u \cdot \nabla v \, dY = \int_\Omega \A \nabla u \cdot \nabla v \, dm
\end{equation}
is bounded and coercive on $W_0$
and the map
\begin{equation} \label{tcp22}
\varphi \in W_0 \to \int_\Omega A \nabla \phi \cdot \nabla \varphi \, dY
= \int_\Omega \A \nabla \phi \cdot \nabla \varphi \, dm
\end{equation}
is bounded on $W_0$. So the Lax-Milgram theorem yields the existence of $\mathfrak v \in W_0$ such that
\begin{equation} \label{tcp23}
\int_\Omega A \nabla \phi \cdot \nabla \varphi \, dY
= \int_\Omega A \nabla \mathfrak v \cdot \nabla \varphi \, dY \qquad \forall \varphi \in W_0.
\end{equation}

Let $s>0$ such that $B(X,2s) \cap (\supp \, \phi \cup \Gamma) = \emptyset$. 
For any $\rho>0$
we define, as we did in \eqref{Green3},
the function $\g^\rho_T = g^\rho_T(.,X)$ on $\Omega$ as the only function in $W_0$ such that
\begin{equation} \label{tcp24}
 \int_\Omega A \nabla \varphi \cdot \nabla \g^\rho_T \, dY = \int_\Omega A^T \nabla \g^\rho_T \cdot \nabla \varphi \, dY  = \fint_{B(X,\rho)} \varphi\, dY \qquad \forall \varphi \in W_0.
\end{equation}
We take $\varphi = \g^\rho_T$ in \eqref{tcp23} to get
\begin{equation} \label{tcp25}
\int_\Omega A \nabla \phi \cdot \nabla \g^\rho_T\, dY = \int_\Omega A \nabla \mathfrak v \cdot \nabla \g^\rho_T \, dY= \fint_{B(X,\rho)} \mathfrak v\, dY.
\end{equation}
We aim to take the limit as $\rho \to 0$ in \eqref{tcp25}. Since $\mathfrak v$ satisfies
\begin{equation} \label{tcp26}
\int_\Omega A \nabla \mathfrak v \cdot \nabla \varphi\, dY = \int_\Omega A \nabla \phi \cdot \nabla \varphi \, dY= 0\qquad  \forall \varphi  \in C^\infty_0(B(X,2s)),
\end{equation}
$\mathfrak v$ is a solution of $L\mathfrak v = 0$ on $B(X,2s)$ and thus Lemma \ref{HolderI} proves that $\mathfrak v$ is continuous at $X$. As a consequence,
\begin{equation} \label{tcp27}
\lim_{\rho \to 0} \fint_{B(X,\rho)} \mathfrak v \, dY=  \mathfrak v(X).
\end{equation}
Recall that the $\g^\rho_T$, $\rho>0$, are the same functions as in 
in the proof of Lemma \ref{GreenEx}, but
for the transpose matrix $A^T$. 
Let $\alpha \in C^\infty_0(B(x,2s))$ be such that $\alpha \equiv 1$ on $B(x,s)$. By 
\eqref{Grc0} and Lemma \ref{GreenSym}, there exists a sequence $(\rho_\eta)$ tending to 0,
such that $(1-\alpha)\g^{\rho_\eta}_T$ converges weakly to 
$(1-\alpha)g_T(.,X) = (1-\alpha)g(X,.)$ in $W_0$. As a consequence,
\begin{eqnarray} \label{tcp28} 
\lim_{\eta \to +\infty} \int_\Omega A \nabla \phi \cdot \nabla \g^{\rho_\eta}_T \, dY
& = & \lim_{\eta \to +\infty} \int_{\Omega} A \nabla \phi \cdot \nabla [(1-\alpha)\g^{\rho_\eta}_T] \, dY
\nn\\
& = &  \int_\Omega A \nabla \phi(Y) \cdot \nabla_y [(1-\alpha)g(X,Y)] dY \nn\\
& = &  \int_\Omega A \nabla \phi(Y) \cdot \nabla_y g(X,Y) dY.
\end{eqnarray}
The combination of \eqref{tcp25}, \eqref{tcp27} and \eqref{tcp28} yields
\begin{equation} \label{tcp29}
\int_\Omega A \nabla \phi(Y) \cdot \nabla_y g(X,Y) dY = \mathfrak v(X).
\end{equation}
Since $\mathfrak v \in W_0$ satisfies \eqref{tcp23}, the function $u_\phi = \phi - \mathfrak v$ lies
in $W$ and is a solution of $Lu_\phi = 0$ with the Dirichlet condition $Tu_\phi = T\phi$. 
Hence
\begin{equation} \label{tcp29bis}
\int_\Omega A \nabla \phi(Y) \cdot \nabla_y g(X,Y) dY = \mathfrak v(X) = \phi(X) - u_\phi(X) = -u_\phi(X),
\end{equation}
by \eqref{tcp29} and 
because $X \notin \supp\, \phi$. The claim \eqref{tcp20} follows.

\medskip

We turn to the proof of the upper bound in \eqref{tcp18}, that is,  
\begin{equation} \label{tcp19}
\omega^X(B(x_0,r) \cap \Gamma) \leq C r^{d-1} g(X,X_0) \ \ \ \text{ for }
X\in \Omega \setminus B(x_0,2r).
\end{equation}
Let $X\in \Omega \setminus B(x_0,2r)$ be given, and choose $\phi \in C^\infty_0(\R^n)$ 
such that $0\leq \phi \leq 1$, $\phi \equiv 1$ on $B(x_0,r)$, 
$\phi \equiv 0$ on $\R^n \setminus B(x_0,\frac{5}4r)$, and $|\nabla \phi| \leq \frac{10}{r}$. 
We get that 
\begin{equation} \label{tcp30}
u_\phi(X) \leq \frac{C}r \int_{B(x_0,\frac{5}4r)} |\nabla_y g(X,Y)| dm(Y)
\end{equation}
by \eqref{tcp20} and \eqref{ABounded}, and since $\omega^X(B(x_0,r) \cap \Gamma) \leq u_\phi(X)$
by the positivity of the harmonic measure,
\begin{equation} \label{tcp30bis} \begin{split}
\omega^X(B(x_0,r)\cap \Gamma) & \leq \frac{C}r \int_{B(x_0,\frac{5}4r)} |\nabla_y g(X,Y)| dm(Y). \\
& \leq \frac{C}{r} r^{\frac{d+1}{2}} \left( \int_{B(x_0,\frac{5}4r)} |\nabla_y g(X,Y)|^2 dm(Y)   \right)^\frac12
\end{split} \end{equation}
by Cauchy-Schwarz' 
inequality and Lemma \ref{lwest}. 
Since
$X \in \Omega \setminus B(x_0,2r)$, Lemma \ref{GreenSym} and Lemma \ref{GreenEx} (iii)
say that the function $Y \to g(X,Y)$ is a solution of $L_T u := -\diver A^T \nabla u$ on $B(x_0,2r)$, 
with a vanishing trace on $\Gamma \cap B(x_0,2r)$. 
So the Caccioppoli inequality at the boundary (see Lemma \ref{CaccioB}) applies and yields
\begin{equation} \label{tcp31} \begin{split}
\omega^X(B(x_0,r)\cap \Gamma) & \leq \frac{C}{r^2} \, r^{\frac{d+1}{2}}  
\left( \int_{B(x_0,\frac{3}2r)} |g(X,Y)|^2 dm(Y)   \right)^\frac12.
\end{split} \end{equation}
Then by Lemma \ref{ltcp2}, 
\begin{equation} \label{tcp32} \begin{split}
\omega^X(B(x_0,r)\cap \Gamma) & \leq \frac{C}{r^2} \, r^{d+1}  g(X,X_0) = C r^{d-1} g(X,X_0);
\end{split} \end{equation}
the bound \eqref{tcp19} follows.

\medskip

It remains to prove the upper bound in \eqref{tcp18b}, i.e., that
\begin{equation} \label{tcpc1}
 \omega^X(\Gamma \setminus B(x_0,2r)) \leq C r^{d-1} g(X,X_0) 
 \ \ \ \text{ for }
 X\in B(x_0,r) \setminus B(X_0,\delta(X_0)/4).
 \end{equation}
 The proof will be 
 similar to the upper bound in \eqref{tcp18} once we choose an appropriate function $\phi$ in \eqref{tcp20}. 
Let us do this rapidly. Let $X\in B(x_0,r) \setminus B(X_0,\delta(X_0)/4)$ be given and take
 $\phi \in C^\infty(\R^n)$ such that $0\leq \phi \leq 1$, $\phi \equiv 1$ on $\R^n \setminus B(x_0,\frac{8}{5}r)$, $\phi \equiv 0$ on $B(x_0,\frac7{5}r)$ and $|\nabla \phi| \leq \frac{10}r$. 
Notice that $X \notin \supp(\phi)$, so \eqref{tcp20} applies and yields
 \begin{equation} \label{tcpc2}
u_\phi(X) \leq \frac{C}r \int_{B(x_0,\frac{8}{5}r) \setminus B(x_0,\frac{7}{5}r)} |\nabla_y g(X,Y)| dm(Y).
 \end{equation}
By the positivity of the harmonic measure, 
$\omega^X(\Gamma \setminus B(x_0,2r)) \leq u_\phi(X)$. 
We use the Cauchy-Schwarz
and Caccioppoli inequalities (see Lemma \ref{CaccioB}), as above, and get that
 \begin{eqnarray} \label{tcpc3} 
\omega^X(\Gamma \setminus B(x_0,2r)) 
&\leq& \frac{C}r m(B(x_0,\frac85 r)) \left( \frac{1}{m(B(x_0,\frac85 r))} 
\int_{B(x_0,\frac{8}{5}r) \setminus B(x_0,\frac{7}{5}r)} |\nabla_y g(X,Y)|^2 dm(Y) \right)^\frac12 
\nn\\
&\leq& \frac{C}{r^2} r^{d+1} \left(  \frac{1}{m(B(x_0,\frac85 r))} 
\int_{B(x_0,\frac{9}{5}r) \setminus B(x_0,\frac{6}{5}r)} |g(X,Y)|^2 dm(Y) \right)^\frac12.
\end{eqnarray} 
We claim that
 \begin{equation} \label{tcpc5}
g(X,Y) \leq C g(X,X_0) \qquad \forall Y \in B(x_0,\frac{9}{5}r) \setminus B(x_0,\frac{6}{5}r)
 \end{equation}
 where $C>0$ depends only on $d$, $n$, $C_0$ and $C_1$.
Two cases may happen. 
If $\delta(Y) \geq \frac r{20}$, 
\eqref{tcpc5} is only a consequence of the existence of Harnack chains (Lemma \ref{lHC2}) and the Harnack inequality (Lemma \ref{HarnackI}). 
Otherwise, if $\delta(Y) < \frac{r}{20}$ then Lemma \ref{ltcp2} says that
$g(X,Y) \leq Cg(X,X_Y)$ for some point $X_Y \in B(x_0,\frac{9}{5}r) \setminus B(x_0,\frac{6}{5}r)$
that lies at distance at least $\epsilon r$ from $\Gamma$. Here $\epsilon$ comes from Lemma \ref{lNTA}
and thus depends only on $d$, $n$ and $C_0$. 
Together with the existence of Harnack chains (Lemma \ref{lHC2}) and the Harnack inequality (Lemma \ref{HarnackI}), we find that $g(X,X_Y)$, or $g(X,Y)$, is bounded by $Cg(X,X_0)$. 

We use \eqref{tcpc5} in the right hand side of \eqref{tcpc3} to get that
 \begin{equation} \label{tcpc4}
\omega^X(\Gamma \setminus B(x_0,2r))  \leq C r^{d-1} \frac{m(B(x_0,\frac95 r))}{m(B(x_0,\frac85 r))}  g(X,X_0) \leq C r^{d-1}  g(X,X_0),
 \end{equation}
by the doubling property of the measure $m$. The second and last assertion of the lemma follows.
\ep

\begin{lemma}[Doubling volume property for the harmonic measure] \label{ldphm}
For $x_0 \in \Gamma$ and $r>0$, we have
\begin{equation} \label{dphm1} 
\omega^X(B(x_0,2r)\cap \Gamma) \leq C \omega^X (B(x_0,r) \cap \Gamma) 
 \ \ \ \text{ for }
X \in \Omega \setminus B(x_0,4r)
 \end{equation}
 and
 \begin{equation} \label{dphm1b} 
\omega^X(\Gamma \setminus B(x_0,r)) \leq C \omega^X (\Gamma \setminus B(x_0,2r)) 
 \ \ \ \text{ for } 
X \in B(x_0,r/2),
 \end{equation}
 where $C>0$ depends only on $n$, $d$, $C_0$ and $C_1$.
\end{lemma}

\bp
Let us prove \eqref{dphm1} first. 
Lemma \ref{ltcp3} says that for
$X \in \Omega \setminus B(x_0,4r)$,
\begin{equation} \label{dphm2} 
\omega^X(B(x_0,2r)\cap \Gamma)  \approx r^{d-1} g(X,A_{2r}(x_0))
 \end{equation}
 and
 \begin{equation} \label{dphm3} 
\omega^X(B(x_0,r)\cap \Gamma)  \approx r^{d-1} g(X,A_r(x_0)),
 \end{equation}
 where $A_{2r}(x_0)$ and $A_{r}(x_0)$ are the points of 
 $\Omega$ given by Lemma \ref{lNTA}.
 The bound \eqref{dphm1} will be thus proven if we can show that
\begin{equation} \label{dphm4} 
g(X,A_{2r}(x_0)) \approx g(X,A_{r}(x_0)) \ \ \ \text{ for }
X \in \Omega \setminus B(x_0,4r).
 \end{equation}
Yet, since $Y \to g(X,Y)$ belongs to $\WW(\Omega \setminus \{X\})$ and is a solution of 
$L_T u := - \diver A^T \nabla u = 0$ in $\Omega \setminus \{X\}$ 
(see Lemma \ref{GreenEx} and Lemma \ref{GreenSym}), the equivalence in
\eqref{dphm4} is an easy consequence of the properties of $A_r(x_0)$ (Lemma \ref{lNTA}), 
the existence of Harnack chains (Lemma \ref{lHC2}) and the Harnack inequality (Lemma \ref{HarnackI}).

\medskip

We turn to the proof of \eqref{dphm1b}. 
Set $X_1 := A_r(x_0)$ and $X_{\frac12} := A_{r/2}(x_0)$. Call $\Xi$ the set of points 
$X \in B(x_0,r/2)$ such that $|X-X_1| \geq \frac14 \delta(X_1)$ and 
$|X-X_\frac12| \geq \frac14 \delta(X_\frac12)$, and first consider $X\in \Xi$. 
By Lemma~\ref{ltcp3} again,
\begin{equation} \label{dphm5} 
\omega^X(\Gamma \setminus B(x_0,2r))  \approx r^{d-1} g(X,X_1)
 \end{equation}
 and
 \begin{equation} \label{dphm6} 
\omega^X(\Gamma \setminus B(x_0,r))  \approx r^{d-1} g(X,X_\frac12).
 \end{equation}
 Since $\delta(X_1) \approx \delta(X_{\frac12}) \approx r$ and $Y\to g(X,Y)$ is a solution of 
 $L_T u = - \diver A^T \nabla u = 0$, the existence of Harnack chains (Lemma \ref{lHC2}) 
 and the Harnack inequality (Lemma \ref{HarnackI}) give $g(X,X_1) \approx g(X,X_\frac12)$ 
for $X \in \Xi$. Hence
  \begin{equation} \label{dphm7} 
\omega^X(\Gamma \setminus B(x_0,2r)) \approx \omega^X(\Gamma \setminus B(x_0,r)),
 \end{equation}
 with constants
that do not depend on $X$, $x_0$, or  $r$. 
 The equivalence in \eqref{dphm7}
 also holds for all $X \in B(x_0,r/2)$, and not only for $X \in \Xi$, by Harnack's inequality 
 (Lemma \ref{HarnackI}). This proves \eqref{dphm1b}.
\ep

\begin{remark} \label{rdphm}
The following results also hold for every $\alpha >1$. For
$x_0 \in \Gamma$ and $r>0$,
\begin{equation} \label{dphm1c} 
\omega^X(B(x_0,2r)\cap \Gamma) \leq C_\alpha \omega^X (B(x_0,r) \cap \Gamma) 
\ \ \ \text{ for }
X \in \Omega \setminus B(x_0,2\alpha r),
 \end{equation}
 and
 \begin{equation} \label{dphm1d} 
\omega^X(\Gamma \setminus B(x_0,r)) \leq C_\alpha \omega^X (\Gamma \setminus B(x_0,2r)) 
\ \ \ \text{ for }
X \in B(x_0,r/\alpha),
 \end{equation}
 where $C_\alpha>0$ depends only on $n$, $d$, $C_0$, $C_1$ and $\alpha$.
 
 \medskip
This can be deduced
from Lemma \ref{ldphm} - that corresponds to the case $\alpha =2$ - by 
applying it to smaller balls.

Let us prove for instance \eqref{dphm1d}. 
Let $X \in B(x_0,r/\alpha)$ be given. 
We only need to prove \eqref{dphm1d} when $\delta(X) < \frac r4(1-\frac 1\alpha)$, 
because as soon as we do this, the other case when $\delta(X) \geq \frac r4(1-\frac 1\alpha)$
follows, by Harnack's inequality (Lemma \ref{HarnackI}).

Let $x\in \Gamma$ such that $|x - X| = \delta(X)$; then set $r_k = 2^{k-1} r [1-\frac 1\alpha]$ and 
$B_k = B(x,r_k)$ for $k \in \ZZ$. We wish to apply the doubling property \eqref{dphm1b}
and get that 
\begin{equation} \label{dphm8a} 
\omega^X(\Gamma \setminus B_k) \leq C \omega^X(\Gamma \setminus B_{k+1}),
 \end{equation}
and we can do this as long as $X \in B_{k-1}$. With our extra assumption that
$|x - X| =\delta(X) < \frac r4(1-\frac 1\alpha)$, this is possible for all $k \geq 0$.
Notice that 
\begin{equation} \label{dphm8b} 
|x-x_0| \leq \delta(X) + |X-x_0| \leq \frac r4(1-\frac 1\alpha) + \frac{r}{\alpha}
\leq \frac r2(1-\frac 1\alpha) + \frac{r}{\alpha}
= \frac{r}2 [1+\frac 1\alpha]
 \end{equation}
and then $|x-x_0| + r_0 \leq \frac{r}2 [1+\frac 1\alpha] + \frac{r}2 [1-\frac 1\alpha] = r$,
so $B_0 = B(x,r_0) \subset B(x_0,r)$ and, by the monotonicity of the harmonic measure,
\begin{equation} \label{dphm8c} 
\omega^X(\Gamma \setminus B(x_0,r)) \leq \omega^X(\Gamma \setminus B_0).
\end{equation}
Let $k$ be the smallest integer such that $2^{k-1}(1-\frac 1\alpha) \geq 3$; obviously
$k$ depends only on $\alpha$, and $r_k \geq 3r$.
Then $|x-x_0| + 2r < 3r \leq r_k$ by \eqref{dphm8b}, hence $B(x_0,2r) \subset B_k$ and
$\omega^X(\Gamma \setminus B_k) \leq \omega^X(\Gamma \setminus B(x_0,2r))$
because the harmonic measure is monotone. Together with \eqref{dphm8c} and \eqref{dphm8a},
this proves that $\omega^X(\Gamma \setminus B(x_0,r)) \leq C^k\omega^X(\Gamma \setminus B(x_0,2r))$,
and \eqref{dphm1d} follows because $k$ depends only on $\alpha$.
The proof of \eqref{dphm1c} would be similar.
\end{remark}

\begin{lemma}[Comparison principle for global solutions] \label{lCP}
Let $x_0 \in \Gamma$ and $r>0$
be given, and let 
$X_0 := A_r(x_0) \in \Omega$ be the point given in Lemma \ref{lNTA}.
Let $u,v\in W$ be two  non-negative,
non identically zero, 
solutions of $Lu = Lv = 0$ \underline{in $\Omega$} such that
$Tu = Tv = 0$ on $\Gamma \setminus B(x_0,r)$.
Then
  \begin{equation} \label{CP1} 
C^{-1} \frac{u(X_0)}{v(X_0)} \leq \frac{u(X)}{v(X)} \leq C \frac{u(X_0)}{v(X_0)}
\ \ \ \text{ for } X \in \Omega \setminus B(x_0,2r),
 \end{equation}
where $C>0$ depends only on $n$, $d$, $C_0$ and $C_1$. 
\end{lemma}

\begin{remark}
We also have \eqref{CP1} for any $X\in \Omega \setminus B(x_0,\alpha r)$, where $\alpha >1$. In this case, the constant $C$ depends also on $\alpha$. We let the reader check that the proof below can be easily adapted to prove this too.
\end{remark}

\bp 
By symmetry and as before, it is enough to prove that
  \begin{equation} \label{CP2} 
\frac{u(X)}{v(X)} \leq C \frac{u(X_0)}{v(X_0)}  \ \ \ \text{ for }
X \in \Omega \setminus B(x_0,2r).
 \end{equation}
 Notice also that 
 thanks to the Harnack inequality (Lemma \ref{HarnackI}), $v(X)>0$ on the whole 
 $\Omega \setminus \overline{B(x_0,r)}$, so we we don't need to be careful when we divide by $v(X)$.
 
 \medskip
 Set 
$\Gamma_1 := \Gamma \cap B(x_0,r)$ and $ \Gamma_2 : = \Gamma \cap  B(x_0,\frac{15}8 r)$.
Lemma \ref{ldphm} - or more exactly \eqref{dphm1c} - gives the following fact that will be of use later on:
  \begin{equation} \label{CP3} 
\omega^X(\Gamma_2) \leq C \omega^X(\Gamma_1) \qquad \forall X \in \Omega \setminus B(x_0,2r).
 \end{equation}
 with a constant $C>0$ which depends only on $d$, $n$, $C_0$ and $C_1$. 
 
 \medskip
 
We claim that
  \begin{equation} \label{CP4} 
v(X) \geq C^{-1} \omega^X(\Gamma_1) v(X_0) \ \ \ \text{ for }
X \in \Omega \setminus B(x_0,2r).
 \end{equation}
Indeed, by 
Harnack's inequality (Lemma \ref{HarnackI}), 
   \begin{equation} \label{CP5} 
v(X) \geq C^{-1}  v(X_0) \ \ \ \text{ for }
X \in B(X_0,\delta(X_0)/2).
 \end{equation}
 Together with Lemma \ref{ltcp5}, which states that $g(X,X_0) \leq C\delta(X_0)^{1-d} \leq Cr^{1-d}$ for any $X \in \Omega \setminus B(X_0,\delta(X_0)/4)$, we deduce the existence of $K_1>0$ (that depends only on $d$, $n$, $C_0$ and $C_1$) such that
    \begin{equation} \label{CP6} 
v(X) \geq K_1^{-1} r^{d-1}  v(X_0) g(X,X_0) \ \ \ \text{ for }
X \in B(X_0,\frac12\delta(X_0)) \setminus B(X_0,\frac14\delta(X_0))
 \end{equation}
Let us apply 
the maximum principle (Lemma \ref{lMPg}, with 
$E = \R^n \setminus \overline{B(X_0,\delta(X_0)/4)}$ and 
$F = \R^n \setminus B(X_0,\delta(X_0/2))$), to 
the function $X \to v(X) - K_1^{-1} r^{d-1}  v(X_0) g(X,X_0)$. 
The assumptions are satisfied
because of \eqref{CP6}, the properties of the Green function given in Lemma \ref{GreenEx}, and the fact that $v\in W$ is a non-negative solution of $Lv = 0$ on $\Omega$. We get that
\begin{equation} \label{CP7} 
v(X) \geq K_1^{-1} r^{d-1}  v(X_0) g(X,X_0) \ \ \ \text{ for }
X \in \Omega \setminus B(X_0,\frac14\delta(X_0)) 
\supset \Omega \setminus B(x_0,2r).
\end{equation}
The claim \eqref{CP4} is now a straightforward consequence of \eqref{CP7} and Lemma \ref{ltcp3}.

\medskip

We want to prove now that
\begin{equation} \label{CP8} 
u(X) \leq Cu(X_0)\omega^X(\Gamma_2) \ \ \ \text{ for }
X \in \Omega \setminus B(x_0,2r).
\end{equation}
First, we need to prove that 
\begin{equation} \label{CP9} 
u(X) \leq Cu(X_0) \ \ \ \text{ for } 
X \in \left[B(x_0,\frac{13}{8}r) \setminus B(x_0,\frac{11}{8}r)\right] \cap \Omega  .
\end{equation}
We split $\left[B(x_0,\frac{13}{8}r) \setminus B(x_0,\frac{11}{8}r)\right] \cap \Omega$ into two sets:
\begin{equation} \label{CP10} 
\Omega_1 : = \Omega \cap \{X\in B(x_0,\frac{13}{8}r) \setminus B(x_0,\frac{11}{8}r), \, \delta(X) < \frac18 r\}
\end{equation}
and
\begin{equation} \label{CP11} 
\Omega_2 : =  \{X\in B(x_0,\frac{13}{8}r) \setminus B(x_0,\frac{11}{8}r), \, \delta(X) \geq \frac18 r\}.
\end{equation}
The proof of \eqref{CP9} for
$X \in \Omega_2$ is a consequence of the existence of Harnack chain (Lemma \ref{lHC2}) 
and the Harnack inequality (Lemma \ref{HarnackI}). So it remains to prove \eqref{CP9} for
$X\in \Omega_1$. Let thus $X \in \Omega_1$ be given.
We can find $x\in \Gamma$ such that $X \in B(x,\frac18 r)$.
Notice that $x \in B(x_0,\frac74r)$
because $X \in B(x_0,\frac{13}{8}r)$.
Yet, since $u$ is a non-negative solution of $Lu=0$ in $B(x,\frac14 r) \cap \Omega$ satisfying $Tu = 0$ 
on $B(x,\frac14 r) \cap \Gamma$, Lemma \ref{ltcp2} gives that $u(Y) \leq C u(A_{r/8}(x))$ for 
$Y \in B(x,\frac18 r)$ and thus in particular $u(X) \leq C u(A_{r/8}(x))$. 
By the existence of Harnack chains (Lemma \ref{lHC2}) and the Harnack inequality (Lemma \ref{HarnackI})
again, $u(A_{r/8}(x)) \leq Cu(X_0)$.
The bound \eqref{CP9} for all $X \in \Omega_1$ follows. 

\medskip

We proved \eqref{CP9} and now we want to get \eqref{CP8}. Recall 
from Lemma \ref{ltcp4} that 
$\omega^X(B(x_0,\frac74r) \cap \Gamma) \geq C^{-1}$ for 
$X \in B(x_0,\frac{13}{8}r) \setminus \Gamma$.
Hence, by \eqref{CP9},
\begin{equation} \label{CP12} 
u(X) \leq Cu(X_0) \omega^X(B(x_0,\frac74r) \cap \Gamma) \ \ \ \text{ for }
X \in \left[B(x_0,\frac{13}{8}r) \setminus B(x_0,\frac{11}{8}r)\right] \cap \Omega.
\end{equation}
Let $h\in C^\infty_0(B(x_0,\frac{15}{8}r))$ be such that $0\leq h \leq 1$ and
$h \equiv 1$ on $B(x_0,\frac74 r)$. Then let 
$u_h \in W$ be the solution of $Lu_h = 0$ with the Dirichlet condition $Tu_h = Th$. 
 By the positivity of the harmonic measure,
\begin{equation} \label{CP13} 
u(X) \leq Cu(X_0) u_h(X) \ \ \ \text{ for } 
X \in \left[B(x_0,\frac{13}{8}r) \setminus B(x_0,\frac{11}{8}r)\right] \cap \Omega.
\end{equation}
The maximum principle given by Lemma \ref{lMPg} - where we take $E = \R^n \setminus \overline{B(x_0,\frac{11}{8}r)}$ and $F = \R^n \setminus B(x_0,\frac{13}{8}r)$ - yields
\begin{equation} \label{CP14} 
u(X) \leq Cu(X_0) u_h(X) \ \ \ \text{ for }
X \in \Omega \setminus B(x_0,\frac{13}8r)
\end{equation}
and hence 
\begin{equation} \label{CP15} 
u(X) \leq Cu(X_0) \omega^X(\Gamma_2) \ \ \ \text{ for }
X \in \Omega \setminus B(x_0,\frac{13}8r),
\end{equation}
where
we use again the positivity of the harmonic measure. The assertion \eqref{CP8} is now proven.

\medskip

We conclude the proof of the lemma by gathering the previous results.
Because of 
\eqref{CP4} and \eqref{CP8},
\begin{equation} \label{CP16} 
\frac{u(X)}{v(X)} \leq C\frac{u(X_0)}{v(X_0)} \frac{\omega^X(\Gamma_2)}{\omega^X(\Gamma_1)} 
\ \ \ \text{ for }
X \in \Omega \setminus B(x_0,2r),
\end{equation}
and \eqref{CP2} follows from \eqref{CP3}.
Lemma \ref{lCP} 
follows.
 \ep
 
Note 
that the functions $X \to \omega^X(E)$, where $E \subset \Gamma$ is a non-trivial Borel set, 
do not lie in $W$ 
and thus cannot be used directly in Lemma \ref{lCP}. The following lemma solves this problem.

\begin{lemma}[Comparison principle for harmonic measures / Change of poles] \label{lCP2}
Let $x_0 \in \Gamma$ and $r>0$ 
be given, and let  
$X_0 := A_r(x_0) \in \Omega$ be as in Lemma \ref{lNTA}.
Let $E,F \subset \Gamma \cap B(x_0,r)$ be two Borel subsets of $\Gamma$ 
such that $\omega^{X_0}(E)$ and $\omega^{X_0}(F)$ are positive. Then 
  \begin{equation} \label{CP17} 
C^{-1} \frac{\omega^{X_0}(E)}{\omega^{X_0}(F)} 
\leq  \frac{\omega^{X}(E)}{\omega^{X}(F)} \leq C  \frac{\omega^{X_0}(E)}{\omega^{X_0}(F)}
\ \ \ \text{ for } X \in \Omega \setminus B(x_0,2r),
 \end{equation}
where $C>0$ depends only on $n$, $d$, $C_0$ and $C_1$. 
In particular, with the choice $F = B(x_0,r) \cap \Gamma$, 
  \begin{equation} \label{CP18} 
C^{-1} \omega^{X_0}(E) \leq  \frac{\omega^{X}(E)}{\omega^{X}(B(x_0,r) \cap \Gamma)} 
\leq C \omega^{X_0}(E) \ \ \ \text{ for } X \in \Omega \setminus B(x_0,2r),
 \end{equation}
where again $C>0$ depends only on $n$, $d$, $C_0$ and $C_1$.
\end{lemma}

\bp
The second part of the lemma, that is \eqref{CP18} an immediate consequence of \eqref{CP17} and the non-degeneracy of the harmonic measure (Lemma \ref{ltcp4}). In addition, it is enough to prove
  \begin{equation} \label{CP17b} 
C^{-1} \frac{\omega^{X_0}(E)}{u(X_0)} \leq  \frac{\omega^{X}(E)}{u(X)} 
\leq C  \frac{\omega^{X_0}(E)}{u(X_0)},
 \end{equation}
where $u\in W$ is any
non-negative non-zero solution of $Lu = 0$ in $\Omega$ satisfying $Tu = 0$ on $\Gamma \setminus B(x_0,r)$, and $C>0$ depends only on $n$, $d$, $C_0$ and $C_1$. Indeed, \eqref{CP17} follows 
by applying \eqref{CP17b} to both $E$ and $F$.
Incidentally, it is very easy to find $u$ like this: just apply Lemma \ref{ldefhm}
to a smooth bump function $g$ with a small compact support near $x_0$.

\medskip

Assume first that $E=K$ is a compact set. Let $X\in \Omega \setminus B(x_0,2r)$ be given. 
Thanks to Lemma \ref{lprhm} (i), the assumption $\omega^{X_0}(K)>0$ implies that
$\omega^X(K)>0$.
By the 
the regularity of the harmonic measure (see \eqref{defhm6}), we can
find an open set 
$U_X \supset K$ such that 
\begin{equation} \label{CP19} 
\text{$\omega^{X_0}(U_X) \leq 2 \omega^{X_0}(K)$ and $\omega^{X}(U_X) \leq 2 \omega^{X}(K)$.}
 \end{equation}
Urysohn's lemma (see Lemma~2.12 in \cite{Rudin}) gives a function $h \in C^0_0(\Gamma)$ such that 
$\1_K \leq h \leq \1_{U_X}$. 
Write $v^h = U(h)$ 
for the image of the function $h$ by the map given in Lemma \ref{ldefhm}. 
We have seen for the proof of Lemma \ref{ldefhm} that $h$ can be approximated, in the supremum norm,
by smooth, compactly supported functions $h_k$, and that the corresponding solutions $v_k = U(h_k)$,
and that can also obtained through \ref{lLM}, lie in $W$ and converge to $v^h$ uniformly on $\Omega$.
Hence we can find
$k >0$ such that
 \begin{equation} \label{CP20} 
\frac12 \, v_{k} \leq v^h \leq 2v_{k} 
 \end{equation}
everywhere in $\Omega$. Write $v$ 
for $v_k$. Notice
that $v$ depends on $X$, but it has no importance. The estimates \eqref{CP19} and \eqref{CP20} give
\begin{equation} \label{CP21} 
\text{$\frac14 v(X_0) \leq  \omega^{X_0}(K) \leq 2v(X_0) \quad $ and $\quad \frac14 v(X)  \leq  \omega^{X}(K) \leq 2v(X)$.}
 \end{equation}
We can even choose $U_X \supset K$ so small, and then $g_k$ with a barely larger support, 
 so that $T v = g_k$
 is supported in $B(x_0,r)$. 
As a consequence, the solution $v$ satisfies the assumption of Lemma \ref{lCP}. Hence, the latter entails
\begin{equation} \label{CP22} 
C^{-1} \frac{v(X_0)}{u(X_0)} \leq  \frac{v(X)}{u(X)} \leq  C\frac{v(X_0)}{u(X_0)}
 \end{equation}
 with a constant $C>0$ that depends only on $d$, $n$, $C_0$ and $C_1$. 
 Together with \eqref{CP21}, we get that
 \begin{equation} \label{CP23} 
C^{-1} \frac{\omega^{X_0}(K)}{u(X_0)} \leq  \frac{\omega^{X}(K)}{u(X)} 
\leq  C\frac{\omega^{X_0}(K)}{u(X_0)}
 \end{equation}
with a constant $C>0$ that still depends only on $d$, $n$, $C_0$ and $C_1$ 
(and thus is independent of $X$).
 Thus the conclusion \eqref{CP17} holds 
 whenever $E=K$ is a compact set.

\medskip
Now let $E$ be any Borel subset of $\Gamma \cap B(x_0,r)$.
Let $X \in \Omega \setminus B(x_0,2r)$. According to the regularity of the harmonic measure \eqref{defhm6}, there exists $K_X\subset E$ (depending on $X$) such that 
\begin{equation} \label{CP24} 
\text{$ \omega^{X_0}(K_X) \leq  \omega^{X_0}(E) 
\leq 2 \omega^{X_0}(K_X) \quad $ and $\quad \omega^{X}(K_X) \leq  \omega^{X}(E) 
\leq 2 \omega^{X}(K_X)$.}
 \end{equation}
The combination of \eqref{CP24} and \eqref{CP23} (applied to 
$K_X$) yields
\begin{equation} \label{CP25} 
C^{-1} \frac{\omega^{X_0}(E)}{u(X_0)} \leq  \frac{\omega^{X}(E)}{u(X)} \leq  C\frac{\omega^{X_0}(E)}{u(X_0)}
 \end{equation}
where the constant $C>0$ depends only upon $d$, $n$, $C_0$ and $C_1$. The lemma follows.
\ep

Let us prove now a comparison principle for 
the solution that
are not defined in the whole domain $\Omega$.

\begin{theorem}[Comparison principle for locally defined functions] \label{lCP3}
Let $x_0 \in \Gamma$ and $r>0$
and let 
$X_0 := A_r(x_0) \in \Omega$ be the point given in Lemma \ref{lNTA}.
Let $u,v\in \WW(B(x_0,2r))$ be two non-negative,
not identically zero,
solutions of $Lu = Lv = 0$ in $B(x_0,2r)$,
such that  
$Tu = Tv = 0$ on $ \Gamma \cap B(x_0,2r)$.
Then
  \begin{equation} \label{CP26} 
C^{-1} \frac{u(X_0)}{v(X_0)} \leq \frac{u(X)}{v(X)} \leq C \frac{u(X_0)}{v(X_0)}
\ \ \ \text{ for } X \in \Omega \cap B(x_0,r),
 \end{equation}
  where $C>0$ depends only on $n$, $d$, $C_0$ and $C_1$. 
\end{theorem}

\bp
The plan of the proof is as follows: first, for
$y_0\in \Gamma$ and 
$s>0$, we construct a function $f_{y_0,s}$ on $\Omega$ such that (i) 
$f_{y_0,s}(X)$ is equivalent to $\omega^X(\Gamma \setminus B(y_0,2s))$ when $X\in B(y_0,s)$ is 
close to $\Gamma$ and (ii) 
$f_{y_0,s}(X)$ is negative when $X \in \Omega \setminus B(y_0,Ms)$ - with
$M$ depending only on $d$, $n$, $C_0$ and $C_1$. 
We use $f_{y_0,s}$ to prove that $v(X) \geq v(A_s(y_0)) \omega^X(\Gamma \setminus B(y_0,2s))$ whenever $X \in B(y_0,s)$ and $B(y_0,Ms) \subset B(x_0,2r)$ is a ball centered on $\Gamma$. 
We use then an appropriate covering of $B(x_0,r)$ by balls and the Harnack inequality to get the lower bound $v(X) \geq v(X_0) \omega^X(\Gamma \setminus B(x_0,4r))$, which is the counterpart of \eqref{CP4} in our context.
We conclude as in Lemma \ref{lCP} by using Lemma \ref{ltcp2} and the doubling property for the harmonic measure (Lemma \ref{ldphm})

\medskip

Let $y_0\in \Gamma$ and $s>0$. Write $Y_0$ for $A_s(y_0)$. The main idea is to take 
\begin{equation}\label{a148}
f_{y_0,s} (X) := s^{d-1} g(X,Y_0) - K_1 \omega^X(\Gamma \setminus B(y_0,K_2 s))
\end{equation}
for some $K_1,K_2>0$ that depend
only on $n$, $d$, $C_0$ and $C_1$. With good choices of $K_1$ and $K_2$, 
the function $f_{y_0,s}$ is positive in $B(y_0,s)$ and negative outside of
a big ball $B(y_0,2K_2s)$. 
However, with this definition involving the harmonic measure, the function $f_{y_0,s}$ doesn't satisfy the appropriate estimates required for the use of the maximum principle given as Lemma \ref{lMPg}. 
So we shall replace $\omega^X(\Gamma \setminus B(y_0,K_2s))$ by some solution of $Lu=0$, 
with smooth Dirichlet condition.

\medskip
Let $h\in C^\infty(\R^n)$ be such that $0\leq h \leq 1$, $h\equiv 0$ on $B(0,1/2)$ and $h\equiv 1$ on 
the complement of  $B(0,1)$. 
For $\beta >0$ (which will be chosen large), 
we define $h_{\beta}$ by 
$h_\beta(x) = h(\frac{x-y_0}{\beta s})$. 
Let $u_\beta$ be
the solution, given by Lemma \ref{lLM}, 
of $Lu_\beta=0$ with the Dirichlet condition $Tu_\beta = Th_\beta$. 
Notice that $u_\beta \in W$ because $1-u_\beta$ is the solution of $L$ with the smooth and compactly
supported trace $1-h$.
Observe
that for any $X\in \Omega$ and
$\beta>0$,
\begin{equation} \label{CPA7}
 \omega^X(\Gamma \setminus B(y_0,\beta s)) \leq u_\beta(X) \leq  \omega^X(\Gamma \setminus B(y_0,\beta s/2)).
\end{equation}
The functions $u_\beta$ will play the role of 
harmonic measures but,
unlike these, the functions $u_\beta$ lie in $W$ and are thus 
 suited for the use of Lemma \ref{lMPg}.

\medskip

By 
Lemma \ref{ltcp5}, there exists $C>0$, that depends
only on $d$, $n$, $C_0$ and $C_1$, such that
\begin{equation} \label{CPA1}
g(X,Y_0) \leq C \delta(Y_0)^{1-d} \ \ \ \text{ for } 
X \in \Omega \setminus B(Y_0,\delta(Y_0)/4).
\end{equation}
Moreover, 
since $Y_0$ comes from 
Lemma \ref{lNTA}, we have $\epsilon s \leq \delta(Y_0) \leq s$ with 
an $\epsilon > 0$ that does not depend on $s$ or $y_0$,
and hence
\begin{equation} \label{CPA2}
s^{d-1} g(X,Y_0) \leq C \ \ \ \text{ for }
X \in \Omega \setminus B(y_0,2s).
\end{equation}
From this and
the non-degeneracy of the harmonic measure (Lemma \ref{ltcp4}), we deduce
that for $\beta \geq 1$, 
\begin{equation} \label{CPA3}
s^{d-1} g(X,Y_0) \leq K_1 \omega^X(\Gamma \setminus B(y_0,\beta s)) \leq K_1 u_\beta(X) 
\ \ \ \text{ for } 
X \in \Omega \setminus B(y_0,2\beta s),
\end{equation}
where the constant $K_1>0$, depends only on $d$, $n$, $C_0$ and $C_1$.

Our aim now is to find $K_2\geq 20$ such that
\begin{equation} \label{CPA6}
K_1 u_{K_2}(X) \leq \frac12 s^{d-1} g(X,Y_0) \ \ \ \text{ for }
X \in \Omega \cap [B(y_0,s) \setminus B(Y_0,\delta(Y_0)/4)].
\end{equation}
By construction, the function $u_\beta$ is such that $T(u_\beta) = 0$ on $\Gamma \cap B(y_0,\beta s/2)$, so according to the H\"older continuity at the boundary (Lemma \ref{HolderB}), we have 
\begin{equation} \label{CPA4}
\sup_{B(y_0,10 s)} u_\beta \leq C \beta^{-\alpha}
\end{equation}
for any $\beta \geq 20$, where $C$ and $\alpha>0$ depend only 
on $d$, $n$, $C_0$ and $C_1$. 
Moreover, due to 
\eqref{CPA7} and 
the non-degeneracy of the harmonic measure (Lemma \ref{ltcp4}), 
\begin{equation} \label{CPA4a}
 u_4(X) \geq C^{-1} \ \ \ \text{ for } 
 X \in \Omega \setminus B(y_0,8s)
\end{equation}
where $u_4$ is defined as $u_\beta$ (with 
$\beta = 4$).
As a consequence, 
there exists $K_3>0$, that depends only on $d$, $n$, $C_0$, and $C_1$, such that
for $\beta \geq 20$,
\begin{equation} \label{CPA4b}
 u_\beta(X) \leq K_3 \beta^{-\alpha} u_4(X) \ \ \ \text{ for }
 X \in \Omega \cap[ B(y_0,10s) \setminus B(y_0,8s)].
\end{equation}
We just proved that for $\beta \geq 20$,
the function $u'= K_3 \beta^{-\alpha} u_4 - u_\beta$ satisfies the assumption (iii) of 
Lemma \ref{lMPg}, with 
$E = B(y_0,10s)$ and $F = \overline{B(y_0,8s)}$. 
The other assumptions of Lemma \ref{lMPg} are satisfied as well, since $u' \in W$ is smooth and $T(u') = K_3  \beta^{-\alpha} T u_4 \geq 0$ on $\Gamma \cap E$.
Therefore, 
Lemma \ref{lMPg} gives
\begin{equation} \label{CPA4c}
 u_\beta(X) \leq K_3 \beta^{-\alpha} u_4(X) \ \ \ \text{ for }
 X \in \Omega \cap B(y_0,10s).
\end{equation}
Use now \eqref{CPA7} and Lemma \ref{ltcp3} to get for
$X \in \Omega \cap [B(y_0,s) \setminus B(Y_0,\delta(Y_0)/4)]$,
\begin{equation} \label{CPA4d}
 u_\beta(X) \leq K_3 \beta^{-\alpha} \omega^X(\Gamma \setminus B(y_0,2s)) 
 \leq C \beta^{-\alpha} s^{d-1} g(X,Y_0),
\end{equation}
where $C>0$ depends only on $d$, $n$, $C_0$ and $C_1$. 
The existence of $K_2\geq 20$ satisfying \eqref{CPA6} is now immediate.

\medskip

Define the function $f_{y_0,s}$ on $\Omega \setminus \{Y_0\}$ by
\begin{equation} \label{CPA8}
f_{y_0,s}(X): = s^{d-1} g(X,Y_0) - K_1 u_{K_2}(X).
\end{equation}
The inequality \eqref{CPA3} gives
\begin{equation} \label{CPA9}
f_{y_0,s}(X) \leq 0 \ \ \ \text{ for }
X \in \Omega \setminus B(y_0,2K_2s),
\end{equation}
and the estimates \eqref{CPA6} and \eqref{tcp18b} imply that 
\begin{equation} \label{CPA0}
f_{y_0,s}(X) \geq \frac12 s^{d-1} g(X,Y_0) \geq C^{-1} \omega^X(\Gamma \setminus B(y_0,2s)) 
\ \ \ \text{ for } 
X \in \Omega \cap [B(y_0,s) \setminus B(Y_0,\delta(Y_0)/4)].
\end{equation}

\bigskip

Let us turn to the proof of the comparison principle. 
By symmetry and as
in Lemma~\ref{lCP}, it suffices to prove the upper bound in \eqref{CP26}, that is
  \begin{equation} \label{CP32} 
\frac{u(X)}{v(X)} \leq C \frac{u(X_0)}{v(X_0)} \ \ \ \text{ for } 
X \in \Omega \cap B(x_0,r).
 \end{equation}

\medskip

We claim that
  \begin{equation} \label{CP33} 
v(X) \geq C^{-1} v(X_0) \omega^X(\Gamma \setminus B(x_0,2r)) 
\ \ \ \text{ for } 
X \in \Omega \cap B(x_0,r),
 \end{equation}
 where $C>0$ depends only on $n$, $d$, $C_0$ and $C_1$.
So let $X \in \Omega \cap B(x_0,r)$ be given. Two cases may happen. 
If $\delta(X) \geq \frac r{8K_2}$, where $K_2$ comes from \eqref{CPA6} 
and is the same as in the definition of $f_{y_0,s}$, the
existence of Harnack chains (Lemma \ref{lHC2}), the Harnack inequality (Lemma \ref{HarnackI}) and the non-degeneracy of the harmonic measure (Lemma \ref{ltcp4}) give 
   \begin{equation} \label{CP34} 
v(X) \approx  v(X_0)  \approx v(X_0) 
\frac{\omega^X(\Gamma \setminus B(x_0,2r))}{\omega^{X_0}(\Gamma \setminus B(x_0,2r)) } 
\approx v(X_0) \omega^X(\Gamma \setminus B(x_0,2r))
 \end{equation}
 by
 \eqref{tcp33c}. 
The more interesting remaining case is when 
 $\delta(X) < \frac{r}{8K_2}$. 
 Take $y_0 \in \Gamma$ such that $|X-y_0| = \delta(X)$. Set 
 $s:= \frac r{8K_2}$ and $Y_0 = A_s(y_0)$.  
 The ball $B(y_0, \frac12 r) = B(y_0,4K_2s)$ is contained
 in $B(x_0,\frac 74 r)$. The following points hold :  
 \begin{itemize}
 \item 
 The quantity  $\int_{B(y_0,4K_2s) \setminus B(Y_0,\delta(Y_0)/4)} |\nabla v|^2 dm$ is finite 
 because $v \in \WW(B(x_0,2r))$. The fact that 
 $\int_{B(y_0,4K_2s) \setminus B(Y_0,\delta(Y_0)/4)} |\nabla f_{y_0,s}|^2 dm$ is finite as well 
 follows from the property \eqref{GreenE1} of the Green function. 
 \item 
 There exists $K_4>0$ (depending only on $d$, $n$, $C_0$ and $C_1$) such that 
  \begin{equation} \label{CP35} 
v(Y) - K_4 v(Y_0) f_{y_0,s}(Y) \geq 0 \ \ \ \text{ for }
Y \in B(Y_0,\delta(Y_0)/2) \setminus B(Y_0,\delta(Y_0)/4). \end{equation}
This latter inequality is due to the following two bounds: the fact that
 \begin{equation} \label{CP36} 
f_{y_0,s}(Y) \leq s^{1-d} g(Y,Y_0) \leq C \ \ \ \text{ for }
Y \in B(Y_0,\delta(Y_0)/2) \setminus B(Y_0,\delta(Y_0)/4),
 \end{equation}
which is a consequence of
the definition \eqref{CPA8} and 
\eqref{GreenE7}, and the bound
  \begin{equation} \label{CP37} 
v(Y) \geq C^{-1}v(Y_0) \ \ \ \text{ for }
Y \in B(Y_0,\delta(Y_0)/2),
 \end{equation}
which comes from the Harnack inequality (Lemma \ref{HarnackI}).
 \item 
 The function $v - K_4 v(Y_0) f_{y_0,s}$
is nonnegative on 
 $\Omega \cap [B(y_0,4K_2s) \setminus B(y_0,2K_2s)]$. Indeed, $v\geq 0$ on $B(y_0,4K_2s) $ and, thanks to \eqref{CPA9}, $f_{y_0,s} \leq 0$ on $\Omega \setminus B(y_0,2K_2s)$.
 \item 
 The trace of $v-K_4 v(Y_0) f_{y_0,s}$ is non-negative on 
$B(y_0,4K_2s) \cap \Gamma$
 again because $Tv = 0$ on 
$B(y_0,4K_2s) \cap \Gamma$ 
 and $T[f_{y_0,s}] \leq 0$ on 
$B(y_0,4K_2s) \cap \Gamma$ 
 by construction.
 \end{itemize} 
 The previous points prove that $v-K_4 v(Y_0) f_{y_0,s}$ satisfies the assumptions of 
 Lemma \ref{lMPg} with
 $E = B(y_0,4K_2s) \setminus \overline{B(Y_0,\delta(Y_0)/4)}$ and $F = \overline{B(y_0,2K_2s)} \setminus B(Y_0,\delta(Y_0)/2)$. As a consequence, for any $Y\in B(y_0,4K_2s) \setminus B(Y_0,\delta(Y_0)/4)$
  \begin{equation} \label{CP38} 
v(Y) - K_4 v(Y_0) f_{y_0,s}(Y) \geq 0,
 \end{equation} 
 and hence, 
 for any $Y \in B(y_0,s) \setminus B(Y_0,\delta(Y_0)/4)$
   \begin{equation} \label{CP39} 
v(Y) \geq  K_4 v(Y_0) f_{y_0,s}(Y) \geq C^{-1} v(Y_0) \omega^Y(\Gamma \setminus B(y_0,2s))
 \end{equation} 
 by  
 \eqref{CPA0}. Since
 both $v$ and $Y\to \omega^Y(\Gamma \setminus B(y_0,2s))$ are solutions on $B(y_0,2s)$, 
 we can use the Harnack inequality (Lemma \ref{HarnackI}) to deduce, first, 
 that \eqref{CP39} holds for any $Y \in B(y_0,s)$ and second, that we can replace $v(Y_0)$ by $v(X_0)$
(recall that at this point, $\frac s r = \frac1{8K_2}
$ is controlled by the usual constants).
 Therefore,
    \begin{equation} \label{CP40} 
v(Y) \geq C^{-1} v(X_0) \omega^Y(\Gamma \setminus B(y_0,2s)) \ \ \ \text{ for } 
Y \in B(y_0,s).
 \end{equation} 
 In particular, with our choice of $y_0$ and $s$, the inequality is true when $X=Y$, that is,
 \begin{equation} \label{CP41} 
v(X) \geq C^{-1} v(X_0) \omega^X(\Gamma \setminus B(y_0,2s)) 
\geq C^{-1} v(X_0) \omega^X(\Gamma \setminus B(x_0,2r))
 \end{equation} 
where $C>0$ depends only on $d$, $n$, $C_0$ and $C_1$.
The claim \eqref{CP33} follows.
 
 \medskip
 Now we want to prove that
\begin{equation} \label{CP42} 
u(X) \leq C u(X_0) \omega^X(\Gamma \setminus B(x_0,\frac54 r)) 
\ \ \ \text{ for } 
X \in \Omega \cap B(x_0,r).
 \end{equation} 
By
Lemma \ref{ltcp2},
\begin{equation} \label{CP43} 
u(X) \leq C u(X_0)  \ \ \ \text{ for } 
X \in \Omega \cap B(x_0,\frac74r).
 \end{equation} 
Pick  
$h'\in C^\infty(\R^n)$ such that $0 \leq h'\leq 1$, $h'\equiv 1$ outside of $B(x_0,\frac32 r)$, 
and $h' \equiv 0$ on $B(x_0,\frac54 r)$. 
Let $u_{h'} = U(h')$ be  
the solution of $Lu_{h'} = 0$ with the data $Tu_{h'} = Th'$ (given by Lemma~\ref{lLM}). 
As before, $u_{h'} \in W$ because $1-u_{h'} = U(1-h)$ and $1-h$ is a test function.
Also, $u_{h'}(X) \geq \omega^X(\Gamma \setminus B(x_0,\frac32 r))$ by monotonicity.
So \eqref{tcp33b}, which 
states the non-degeneracy of the harmonic measure, gives
\begin{equation} \label{CP44} 
u_{h'}(X) \geq C^{-1} \ \ \ \text{ for } 
X\in \Omega \setminus B(x_0,\frac{13}{8}r).
 \end{equation} 
The combination of \eqref{CP43} and \eqref{CP44} yields the existence of $K_5>0$ (that depends
only on $d$, $n$, $C_0$ and $C_1$) such that $K_5 u(X_0) u_{h'} - u \geq 0$ on $\Omega \cap [B(x_0,\frac74 r) \setminus B(x_0,\frac{13}{8}r)]$. 
It is easy to check that $K_5 u(X_0) u_{h'} - u $ satisfies all the assumptions of Lemma \ref{lMPg}, with 
$E = B(x_0,\frac74 r)$ and $F =  \overline{B(x_0,\frac{13}{8}r)}$.
This is because $u \in \WW(B(x_0,2r))$, $u_{h'} \in W$, $T u_{h'} \geq 0$, and $Tu = 0$ on 
$\Gamma \cap B(x_0,2r)$.
Then by Lemma \ref{lMPg}
\begin{equation} \label{CP45} 
u \leq K_5 u(X_0) u_{h'} \ \ \ \text{ for }
X\in \Omega \cap B(x_0,\frac{7}{4}r),
 \end{equation} 
and since $u_{h'}(X) \leq \omega^X(\Gamma \setminus B(x_0,\frac{5}{4}r))$
for all 
$X\in \Omega$, 
\begin{equation} \label{CP46} 
u(X) \leq C u(X_0) \omega^X(\Gamma \setminus B(x_0,\frac{5}{4}r)) 
\ \ \ \text{ for } 
X\in \Omega \cap B(x_0,\frac{7}{4}r).
 \end{equation} 
The claim \eqref{CP42} follows.
  
The bounds \eqref{CP33} and \eqref{CP42} imply that
\begin{equation} \label{CP47} 
\frac{u(X)}{v(X)} \leq C \frac{u(X_0)}{v(X_0)} \frac{\omega^X(\Gamma \setminus B(x_0,\frac{5}{4}r))}{\omega^X(\Gamma \setminus B(x_0,2r))} \ \ \ \text{ for }
X\in \Omega \cap B(x_0,r).
 \end{equation} 
 The
 bound \eqref{CP32} is now
 a consequence of the above inequality and the doubling property of the harmonic measure (Lemma \ref{ldphm}, or more exactly \eqref{dphm1d}).
\ep


\backmatter
\bibliographystyle{amsalpha}

\printindex

\end{document}